\documentclass{ucbthesis}

\usepackage{amsmath}
\renewcommand\arraystretch{0.8}
\usepackage{amssymb}
\usepackage{amsxtra}
\usepackage[T1]{fontenc}
\usepackage{enumerate}
\usepackage{amscd}
\usepackage{slashed}
\usepackage{tikz}
\usetikzlibrary{matrix,arrows}
\usepackage{imakeidx}
\usepackage[amsmath,thmmarks,hyperref]{ntheorem}
\usepackage{color}
\definecolor{blue-black}{rgb}{0,0,0.8}
\definecolor{green-black}{rgb}{0,0.75,0}
\definecolor{red-black}{rgb}{0.8,0,0}
\usepackage{todonotes}
\usepackage[colorlinks,
             hyperindex,
             linkcolor=blue-black,
             urlcolor=red-black,
             citecolor=green-black,
             pdftitle={Quantum Algebras Associated to Irreducible Generalized Flag Manifolds},
             pdfauthor={Matthew Tucker-Simmons},
             pdfsubject={Author's PhD Thesis},
             pdfkeywords={lie groups, quantum groups, homogeneous spaces, quantum symmetric algebras, quantum clifford algebras, noncommutative geometry},
             pdfproducer={pdfLaTeX},
             pdfpagemode=None,
             bookmarksopen=true
             bookmarksnumbered=true]{hyperref}
\let\bibsection\relax
\usepackage[alphabetic]{amsrefs}
\usepackage[capitalize]{cleveref}

\usepackage{mathtools}
\DeclarePairedDelimiter\abs{\lvert}{\rvert}%
\DeclarePairedDelimiter\norm{\lVert}{\rVert}%
\makeatletter
\let\oldabs\abs
\def\abs{\@ifstar{\oldabs}{\oldabs*}}
\let\oldnorm\norm
\def\norm{\@ifstar{\oldnorm}{\oldnorm*}}
\makeatother

\usepackage[all]{xy}

\newcommand{\DynkinDiagram}{\xymatrix@M=0pt@=20pt}

\newcommand{\Edge}[2][1]{\ar@[|(1)]@#1{-}[#2]}

\newcommand{\DirectedEdge}[2][1]
{\Edge[#1]{#2} 
\ar@[|(1)]@#1{}[#2]
|>>>>{\SelectTips{eu}{}\object@#1{>}}}

\newcommand{\NodeNumberAbove}[1]{\save+<0pt,8pt>*{{\scriptstyle #1}}\restore}
\newcommand{\NodeNumberBelow}[1]{\save-<0pt,8pt>*{{\scriptstyle #1}}\restore}
\newcommand{\NodeNumberLeft}[1]{\save-<8pt,0pt>*{{\scriptstyle #1}}\restore}
\newcommand{\NodeNumberRight}[1]{\save+<8pt,0pt>*{{\scriptstyle #1}}\restore}

\numberwithin{equation}{subsection}
\setsecnumdepth{subsection}

\theoremstyle{changebreak}
\newtheorem{thm}[equation]{Theorem}
\theoremstyle{change}
\newtheorem{prop}[equation]{Proposition}
\newtheorem{lem}[equation]{Lemma}

\newtheorem{cor}[equation]{Corollary}
\newtheorem{conj}[equation]{Conjecture}
\theorembodyfont{\upshape}
\newtheorem{dfn}[equation]{Definition}
\newtheorem{dfns}[equation]{Definitions}
\newtheorem{conv}[equation]{Convention}
\newtheorem{notn}[equation]{Notation}
\newtheorem{quest}[equation]{Question}
\theoremsymbol{\ensuremath{\Diamond}}
\newtheorem{eg}[equation]{Example}
\newtheorem{rem}[equation]{Remark}
\newtheorem{warn}[equation]{Warning}
\makeatletter
\newtheoremstyle{MyNonumberplain}%
  {\item[\theorem@headerfont\hskip\labelsep ##1\theorem@separator]}%
  {\item[\theorem@headerfont\hskip\labelsep ##3\theorem@separator]}
\makeatother
\theoremstyle{MyNonumberplain}
\theorembodyfont{\upshape}
\theoremsymbol{\ensuremath{\Box}}
\newtheorem{proof}{Proof}

\crefname{equation}{equation}{equations}
\crefname{cor}{Corollary}{Corollaries}
\crefname{conj}{Conjecture}{Conjectures}
\crefname{eg}{Example}{Examples}
\crefname{prop}{Proposition}{Propositions}
\crefname{dfn}{Definition}{Definitions}
\crefname{dfns}{Definitions}{Definitions}
\crefname{notn}{Notation}{Notations}
\crefname{lem}{Lemma}{Lemmas}
\crefname{lemmadef}{Lemma/Definition}{Lemma/Definitions}
\crefname{thm}{Theorem}{Theorems}
\crefname{section}{Section}{Sections}
\crefname{subsection}{Section}{Sections}
\crefname{chapter}{Chapter}{Chapters}
\crefname{rem}{Remark}{Remarks}
\crefname{part}{Part}{Parts}
\crefname{conv}{Convention}{Conventions}
\crefname{warn}{Warning}{Warnings}
\crefname{quest}{Question}{Questions}
\creflabelformat{equation}{#2(#1)#3} 
\crefformat{equation}{#2equation~(#1)#3} 
\crefformat{eg}{#2Example~#1#3} 
\crefformat{dfn}{#2Definition~#1#3} 
\crefformat{dfns}{#2Definitions~#1#3} 
\crefformat{notn}{#2Notation~#1#3}
\crefformat{prop}{#2Proposition~#1#3} 
\crefformat{thm}{#2Theorem~#1#3} 
\crefformat{lem}{#2Lemma~#1#3} 
\crefformat{lemmadef}{#2Lemma/Definition~#1#3} 
\crefformat{cor}{#2Corollary~#1#3} 
\crefformat{conj}{#2Conjecture~#1#3} 
\crefformat{section}{#2Section~#1#3} 
\crefformat{subsection}{#2Section~#1#3} 
\crefformat{rem}{#2Remark~#1#3} 
\crefformat{chapter}{#2Chapter~#1#3} 
\crefformat{part}{#2Part~#1#3} 
\crefformat{warn}{#2Warning~#1#3}
\crefformat{conv}{#2Convention~#1#3}
\crefformat{quest}{#2Question~#1#3}
\Crefformat{equation}{#2Equation~(#1)#3} 
\Crefformat{eg}{#2Example~#1#3} 
\Crefformat{dfn}{#2Definition~#1#3}
\Crefformat{notn}{#2Notation~#1#3}
\Crefformat{prop}{#2Proposition~#1#3} 
\Crefformat{thm}{#2Theorem~#1#3} 
\Crefformat{lem}{#2Lemma~#1#3} 
\Crefformat{lemmadef}{#2Lemma/Definition~#1#3} 
\Crefformat{cor}{#2Corollary~#1#3} 
\Crefformat{conj}{#2Conjecture~#1#3} 
\Crefformat{section}{#2Section~#1#3} 
\Crefformat{subsection}{#2Section~#1#3} 
\Crefformat{rem}{#2Remark~#1#3} 
\Crefformat{chapter}{#2Chapter~#1#3} 
\Crefformat{part}{#2Part~#1#3} 

\newcommand\fB{\ensuremath{\mathfrak B}}

\newcommand\fR{\ensuremath{\mathfrak R}}

\newcommand\fb{\ensuremath{\mathfrak b}}

\newcommand\fe{\ensuremath{\mathfrak e}}
\newcommand\ff{\ensuremath{\mathfrak f}}
\newcommand\fg{\ensuremath{\mathfrak g}}
\newcommand\fh{\ensuremath{\mathfrak h}}

\newcommand\fk{\ensuremath{\mathfrak k}}
\newcommand\fl{\ensuremath{\mathfrak l}}

\newcommand\fn{\ensuremath{\mathfrak n}}
\newcommand\fo{\ensuremath{\mathfrak o}}
\newcommand\fp{\ensuremath{\mathfrak p}}
\newcommand\fq{\ensuremath{\mathfrak q}}

\newcommand\fs{\ensuremath{\mathfrak s}}

\newcommand\fu{\ensuremath{\mathfrak u}}

\newcommand\cA{\ensuremath{\mathcal A}}

\newcommand\cC{\ensuremath{\mathcal C}}
\newcommand\cD{\ensuremath{\mathcal D}}

\newcommand\cH{\ensuremath{\mathcal H}}

\newcommand\cO{\ensuremath{\mathcal O}}
\newcommand\cP{\ensuremath{\mathcal P}}
\newcommand\cQ{\ensuremath{\mathcal Q}}

\newcommand\cS{\ensuremath{\mathcal S}}
\newcommand\cT{\ensuremath{\mathcal T}}
\newcommand\cU{\ensuremath{\mathcal U}}
\newcommand\cV{\ensuremath{\mathcal V}}

\newcommand\cZ{\ensuremath{\mathcal Z}}

\newcommand\bbC{\ensuremath{\mathbb C}}

\newcommand\bbF{\ensuremath{\mathbb F}}

\newcommand\bbM{\ensuremath{\mathbb M}}
\newcommand\bbN{\ensuremath{\mathbb N}}

\newcommand\bbP{\ensuremath{\mathbb P}}
\newcommand\bbQ{\ensuremath{\mathbb Q}}
\newcommand\bbR{\ensuremath{\mathbb R}}

\newcommand\bbZ{\ensuremath{\mathbb Z}}

\newcommand\bA{\ensuremath{\mathbf A}}
\newcommand\bB{\ensuremath{\mathbf B}}
\newcommand\bC{\ensuremath{\mathbf C}}
\newcommand\bD{\ensuremath{\mathbf D}}
\newcommand\bE{\ensuremath{\mathbf E}}
\newcommand\bF{\ensuremath{\mathbf F}}
\newcommand\bG{\ensuremath{\mathbf G}}

\newcommand\bj{\ensuremath{\mathbf j}}
\newcommand\bk{\ensuremath{\mathbf k}}

\newcommand{\eqdef}{\overset{\mathrm{def}}{=\joinrel=}}
\newcommand{\grf}{\gr^F}
\newcommand{\grfc}{\gr^{\Fc}}
\newcommand{\zp}{\bbZ_+}
\newcommand{\ox}{\bar{x}}
\newcommand{\oy}{\bar{y}}
\newcommand{\gc}{\Gamma^\circ}
\newcommand{\Fc}{F^\circ}
\renewcommand{\epsilon}{\varepsilon}
\renewcommand{\phi}{\varphi}
\renewcommand{\rhd}{\triangleright}
\renewcommand{\bar}{\overline}

\newcommand{\e}{\mathrm{ev}}
\renewcommand{\o}{\mathrm{odd}}
\newcommand{\wot}{\widehat\otimes}
\newcommand{\gp}{\gamma_+}
\newcommand{\gm}{\gamma_-}
\newcommand{\gpm}{\gamma_{\pm}}
\newcommand{\delbar}{\overline{\partial}}
\newcommand{\dsl}{\slashed{D}}

\newcommand{\cop}{\Delta}
\newcommand{\counit}{\epsilon}
\newcommand{\antipode}{S}

\newcommand{\fgl}{\mathfrak{gl}}
\newcommand{\fsl}{\mathfrak{sl}}
\newcommand{\fso}{\mathfrak{so}}
\newcommand{\fsp}{\mathfrak{sp}}
\newcommand{\pplus}{\cP^+}
\newcommand{\qplus}{\cQ^+}
\newcommand{\os}{\omega_s}
\newcommand{\peq}{\preceq}
\newcommand{\seq}{\succeq}
\newcommand{\up}{\fu_+}
\newcommand{\um}{\fu_-}
\newcommand{\upm}{\fu_\pm}
\newcommand{\rtsys}{\Phi}
\newcommand{\posrts}{\rtsys^+}
\newcommand{\negrts}{\rtsys^-}
\newcommand{\simprts}{\Pi}
\newcommand{\fhk}{\fh_{\fk}}

\newcommand{\Wl}{W_{\fl}}
\newcommand{\Wul}{W^{\fl}}
\newcommand{\wl}{w_{\fl}}
\newcommand{\wzl}{w_{0,\fl}}
\newcommand{\wz}{w_0}
\newcommand{\fbg}{\fB_\fg} 

\newcommand{\qnum}[2][q]{\left[ #2  \right]_{#1}}
\newcommand{\qfact}[2][q]{\left[ #2 \right]_{#1}!}
\newcommand{\qbinom}[3][q]{{#2 \brack #3}_{#1}}
\newcommand{\uqg}{U_q(\fg)}
\newcommand{\uqzero}[1][q_0]{U_{#1}(\fg)}
\newcommand{\uqlg}[1][\lambda]{U_q^{#1}(\fg)}
\newcommand{\uqh}{U_q(\fh)}

\newcommand{\uql}{U_q(\fl)}
\newcommand{\uqbm}{U_q(\fb_-)}
\newcommand{\uqnp}{U_q(\fn_+)}

\newcommand{\uqnm}{U_q(\fn_-)}
\newcommand{\uqnpm}{U_q(\fn_\pm)}
\newcommand{\uqbpm}{U_q(\fb_\pm)}
\newcommand{\uvq}{U_{\nu}^{\bbQ}(\fg)}
\newcommand{\uvz}{U^{\bbZ}_{\nu}(\fg)}
\newcommand{\uqsl}[1][2]{U_q(\fsl_{#1})}
\newcommand{\Qq}{\bbQ(q)}
\newcommand{\Qv}{\bbQ(\nu)}
\newcommand{\Zv}{\bbZ[\nu,\nu^{-1}]}
\newcommand{\vlc}{V(\lambda)^\bbC}
\newcommand{\ejln}{E_j^\lambda(\nu)}
\newcommand{\fjln}{F_j^\lambda(\nu)}
\newcommand{\hjln}{H_j^\lambda(\nu)}
\newcommand{\kmuln}{K_\mu^\lambda(\nu)}
\newcommand{\ejlq}[1][q]{E_j^\lambda(#1)}
\newcommand{\fjlq}[1][q]{F_j^\lambda(#1)}
\newcommand{\kmulq}[1][q]{K_\mu^\lambda(#1)}
\newcommand{\hjlq}[1][q]{H_j^\lambda(#1)}

\newcommand{\ext}{\Lambda}
\newcommand{\sym}{S}
\newcommand{\extq}[1][q]{\Lambda_{#1}}
\newcommand{\symq}[1][q]{S_{#1}}

\newcommand{\vect}{\mathrm{Vect}}
\newcommand{\svect}{\mathrm{\cS Vect}}
\newcommand{\oq}[1][q]{\cO_{#1}}
\newcommand{\ov}{\cO_\nu}
\newcommand{\soq}{\cS \cO_q}

\newcommand{\oqint}[1][q]{\cO_{#1}^{\mathrm{int}}}
\newcommand{\soqint}[1][q]{\cS\cO_{#1}^{\mathrm{int}}}
\newcommand{\kg}[1][\fg]{K(#1)}
\newcommand{\kpg}[1][\fg]{K^+(#1)}
\newcommand{\kqg}[1][q]{K_{#1}(\fg)}
\newcommand{\kqpg}[1][q]{K_{#1}^+(\fg)}

\newcommand{\rhat}{\widehat{R}}
\newcommand{\rhatuvq}[2][q]{\rhat_{#2}(#1)}
\newcommand{\sigmauvq}[2][q]{\sigma_{#2}(#1)}
\newcommand{\aql}{a^\lambda(q)}
\newcommand{\bql}{b^\lambda(q)}
\newcommand{\ppql}{p_+^\lambda(q)}
\newcommand{\pmql}{p_-^\lambda(q)}
\newcommand{\psiql}{\psi^\lambda(q)}
\newcommand{\cl}{\bbC\mathrm{l}}
\newcommand{\clq}{\cl_q}

\DeclareMathOperator{\wt}{wt}
\DeclareMathOperator{\hgt}{ht}
\DeclareMathOperator{\id}{id}
\DeclareMathOperator{\Hom}{Hom}
\DeclareMathOperator{\End}{End}
\DeclareMathOperator{\ran}{ran}
\DeclareMathOperator{\gr}{gr}
\DeclareMathOperator{\spn}{span}
\DeclareMathOperator{\tr}{tr}
\DeclareMathOperator{\ad}{ad}
\DeclareMathOperator{\Gr}{Gr}
\DeclareMathOperator{\im}{im}
\DeclareMathOperator{\rk}{rank}
\DeclareMathOperator{\low}{low}
\DeclareMathOperator{\op}{op}
\DeclareMathOperator{\spin}{spin}

\newcommand{\spinc}{\spin^c}

\makeindex[name=term,title=Index of terms,intoc=true]
\makeindex[name=notn,title=Index of notation,intoc=true]

\begin{document}


\title{Quantum Algebras Associated to Irreducible Generalized Flag Manifolds}
\author{Matthew Tucker-Simmons}
\degreesemester{Fall}
\degreeyear{2013}
\degree{Doctor of Philosophy}
\chair{Professor Marc A. Rieffel}
\othermembers{Professor Nicolai Reshetikhin \\
  Professor Ori Ganor}
\numberofmembers{3}
\prevdegrees{B.Math (University of Waterloo) 2003 \\
  M.Math (University of Waterloo) 2005}
\field{Mathematics}
\campus{Berkeley}


\maketitle
\copyrightpage

\begin{abstract}
  In the first main part of this thesis we investigate certain properties of the quantum symmetric and exterior algebras associated to Type 1 representations of quantized universal enveloping algebras of semisimple Lie algebras, which were defined and studied by Berenstein and Zwicknagl in \cite{BerZwi08,Zwi09}.
  We define a notion of a commutative algebra object in a coboundary monoidal category, and we prove that, analogously to the classical case, the quantum symmetric algebra associated to a module is the universal commutative algebra generated by that module. 
  That is, the functor assigning to a module its quantum symmetric algebra is left adjoint to an appropriate forgetful functor, and likewise for the quantum exterior algebra.
  We also prove a strengthened version of a conjecture of Berenstein and Zwicknagl, which states that the quantum symmetric and exterior cubes exhibit the same amount of ``collapsing'' relative to their classical counterparts.
  More precisely, the difference in dimension between the classical and quantum symmetric cubes of a given module is equal to the difference in dimension between the classical and quantum exterior cubes.
  We prove that those quantum exterior algebras that are ``flat deformations'' of their classical analogues are Frobenius algebras.
  We also develop a rigorous framework for discussing continuity and limits of the structures involved as the deformation parameter \(q\) varies along the positive real line.
  
  The second main part of the thesis is devoted to quantum analogues of Clifford algebras and their application to the noncommutative geometry of certain quantum homogeneous spaces; this is where the thesis gets its name.
  We introduce the quantum Clifford algebra through its ``spinor representation'' via creation and annihilation operators on one of the flat quantum exterior algebras discussed in the first main part of the thesis.
  The proof that the spinor representation is irreducible, and hence that the creation and annihilation operators generate the full endomorphism algebra of the quantum exterior algebra, relies in an essential way on the Frobenius property for the quantum exterior algebra.
  We then use this notion of quantum Clifford algebra to revisit Kr\"ahmer's construction in \cite{Kra04} of a Dolbeault-Dirac-type operator on the canonical \(\spinc\) structure over a quantized irreducible flag manifold.
  This operator is of the form \(\eth + \eth^\ast\), and we prove that \(\eth\) can be identified with the boundary operator for the Koszul complex of a certain quantum symmetric algebra, which shows that \(\eth^2 = 0\).
  This is a first step toward establishing a Parthasarathy-type formula for the spectrum of the square of the Dirac operator, and hence toward proving that it satisfies the technical conditions to be part of a spectral triple in the sense of Connes.

  Parts of this work appear in the preprints \cite{ChiTuc12,KraTuc13}.
\end{abstract}

\begin{frontmatter}

\begin{dedication}
\null\vfil
\begin{center}
Dedicated to Lisa\\
\vspace{3in}
\emph{Thus is our treaty written; thus is agreement made.\\
Thought is the arrow of time; memory never fades.\\
What was asked is given.  The price is paid.\\}
\vspace{12pt}
--Robert Jordan, \emph{The Shadow Rising}
\end{center}
\vfil\null
\end{dedication}

\tableofcontents
\listoftables

\begin{acknowledgements}
  There are many people that deserve thanks here for their contributions, direct or otherwise, to this thesis.
  First and foremost is my wonderful wife Lisa, who encouraged me along the way and especially helped to push me along in the final months.
  I can only hope that I will be as supportive of her during the completion of her doctorate as she has been during the completion of mine.
  I thank also my family, who have always inspired me with their dedication to knowledge, scholarship, and research, and who helped with key grammatical questions.
  
  On the mathematical side of things, I thank my advisor, Marc Rieffel, who helped me to find research questions to pursue, but in doing so gave me the freedom to explore my own interests.
  His helpful advice on matters mathematical and professional, and generous support over the years through NSF grant DMS-1066368, were instrumental in bringing this project to completion.
  I thank my co-advisor, Nicolai Reshetikhin, for his helpful comments and questions, and for pointing out useful references in the vast literature on quantum groups.
  
  It is a great pleasure to thank Ulrich Kr\"ahmer.
  Our chance meeting at the 2009 Spring School on Noncommutative Geometry at the IPM in Tehran led not only to our collaboration, upon which part of this work is based, but also to our friendship.

  I benefited greatly from many discussions with my professors and colleagues at Berkeley.
  I would like to recognize especially Professor Vera Serganova, who answered many questions on Lie theory and taught an excellent course on the subject.
  Scott Morrison helped me in learning to use \textsc{Mathematica}.
  I thank Alex Chirvasitu for his great willingness to discuss technical details of almost anything.
  Our collaboration was very enjoyable and led to part of the results of this thesis.

  Finally, I thank the Hillegass-Parker Co-op in Berkeley.
  My six years spent there taught me a lot about life and about people, and gave me the opportunity to interact with and learn from dozens of brilliant scholars, in disciplines ranging from history and political science to physics and materials science.
  The friends that I made there helped me in so many ways during this journey.
  I cannot possibly list them all, but I would be remiss not to mention at least these few: Megan Williams, Jesse Hart Fischer-Engel, Adam Ganes, Andrew Friedman, Tim Ruckle,  and Sarah Daniels.
\end{acknowledgements}

\end{frontmatter}


\pagestyle{headings}

\chapter{Introduction}
\label{chap:introduction}

In this brief chapter we describe the contents of the rest of this thesis.
We begin in \cref{sec:introduction-motivation} by giving a broad outline of the background and motivation for the problems studied here.
Then in \cref{sec:lie-groups-intro,sec:quantum-symmetric-and-exterior-algebras-intro,sec:quantum-clifford-algebras-and-flag-mflds-intro} we describe the content of the three following chapters.
Finally in \cref{sec:standard-notation} we set notation that will be used throughout the entire thesis.

\section{Motivation}
\label{sec:introduction-motivation}

The original motivation for the results of this thesis was to further understand Kr\"ahmer's construction in \cite{Kra04} of a Dirac-type operator on certain quantum homogeneous spaces, and to make explicit certain aspects of that construction.
The general setting is that of noncommutative geometry, in which the main theme is to understand a noncommutative algebra as the algebra of functions on a ``quantum space,'' with Hopf algebras playing the role of symmetry groups.
This project lies at the crossroads of the representation-theoretic and metric approaches to noncommutative geometry.
To set the stage for the results, we give some brief background on these two viewpoints before describing the contents of the thesis.

\subsection{Homogeneous spaces, representation theory, and quantum groups}
\label{sec:motivation-homogeneous-spaces-and-quantum-groups}

The representation-theoretic viewpoint can be traced back to Felix Klein's Erlangen Program \cite{Kle93}, in which he proposed that geometries should be classified and understood by their groups of symmetries.
A modern perspective on this idea is that we should study \emph{homogeneous spaces}, that is, spaces of the form \(G/H\), where \(G\) is a Lie group (or an algebraic group) and \(H\) is a closed (or Zariski-closed) subgroup.

The fact that many spaces of independent interest can be realized in this way makes the study of homogeneous spaces attractive.
For instance, spheres, projective spaces, Grassmannians, and flag manifolds are all homogeneous spaces.

The fact that homogeneous spaces carry a transitive group action allows one to make global constructions from local data.
For instance, the machinery of \emph{induced representations} allows one to construct vector bundles over the space \(G/H\) from representations of the isotropy group \(H\).
Then the techniques of representation theory can be applied to study the geometry of these bundles.
This is especially fruitful in the case of semisimple groups, where the representation theory is well understood.
For instance, the Bott-Borel-Weil Theorem \cite{Bot57} (see also \cite{BasEas89}*{Ch.~5}) allows one to calculate sheaf cohomology of certain homogeneous vector bundles in terms of representation-theoretic data, and it was shown by Marlin that the \(K\)-theory ring of \(G/H\) is a quotient of the representation ring of \(H\), or in other words that the \(K\)-theory is generated by the classes of homogeneous vector bundles \cite{Mar76}.

The theory of quantum groups allows us to construct noncommutative analogues of many of these structures.
Take \(\fg\) to be a finite-dimensional semisimple Lie algebra over the complex numbers.
The quantized enveloping algebra \(\uqg\) is a noncommutative, noncocommutative Hopf algebra whose representation theory closely parallels that of \(\fg\) itself.
It can be given a \(\ast\)-structure corresponding to the compact real form of \(\fg\).
The matrix coefficients of the so-called Type 1 representations form a Hopf \(\ast\)-algebra \(\bbC_q[G]\), whose \(C^\ast\)-completion is one of Woronowicz's compact quantum groups \cite{Wor87}.
Viewed through the lens of the Peter-Weyl theorem \cite{Kna02}*{Theorem~4.20}, the algebra \(\bbC_q[G]\) is viewed as a deformation of the complex coordinate ring of the real affine algebraic group \(G_0\), where \(G_0\) is the compact real form of \(G\); see \cite{KliSch97}*{Chapter~9} or \cite{Jan96}*{Chapter~7}.
The \(C^\ast\)-completion is then a deformation of the algebra of continuous functions on \(G_0\).

For certain subgroups \(H \subseteq G\), there are corresponding quotient Hopf \(\ast\)-algebras \(\pi_H : \bbC_q[G] \to \bbC_q[H]\).
The quantized homogeneous space can then be defined as the right coideal \(\ast\)-subalgebra of coinvariant elements for the quotient map, i.e.\ 
\[
\bbC_q[G/H] \eqdef \{ f \in \bbC_q[G] \mid (\id \otimes \pi_H)\cop(f) = f \otimes 1 \}.
\]
As \(\pi_H\) is the quantum analogue of the restriction map from functions on \(G\) to functions on \(H\), this definition is analogous to the description of functions on \(G/H\) as functions on \(G\) that are invariant under right-translation by elements of \(H\); see \cite{Rie08} \cite{KliSch97}*{\S 11.6} for the classical and quantum cases, respectively.
The quantum analogues of vector bundles over \(G/H\) are finitely generated projective modules over \(\bbC_q[G/H]\) (or its \(C^\ast\)-completion), and these can be formed by induction from corepresentations of \(\bbC_q[H]\).
This construction is explained in \cite{GovZha99}.

In this thesis we will be concerned with \emph{parabolic subgroups}, i.e.\ subgroups \(P \subseteq G\) that contain a Borel subgroup, although we will deal with these only at the level of Lie algebras; see \cref{sec:lie-background-parabolics}.
We will moreover restrict ourselves to those parabolics which are of so-called \emph{cominuscule type} (see \cref{sec:cominuscule-parabolics}), which form a particularly nice class: the corresponding homogeneous spaces \(G/P\) are exactly the irreducible compact Hermitian symmetric spaces, or \emph{irreducible flag manifolds}.
We discuss this point further in \cref{sec:dolbeault-operator-classical-motivation} below.
The representation theory of the \(C^\ast\)-completion of \(\bbC_q[G/P]\) was studied in \cite{DijSto99}; the irreducible \(\ast\)-representations are parametrized by the symplectic leaves of the underlying Poisson structure on \(G/P\).

\subsection{Dirac operators and spectral triples}
\label{sec:motivation-dirac-ops-spectral-triples}

In Alain Connes' approach to noncommutative geometry, the primary object is a \emph{spectral triple} \((A,\cH,D)\).
Here \(A\) is a \(\ast\)-algebra of operators acting on a Hilbert space \(\cH\), and \(D\) (the \emph{Dirac operator}) is a densely defined, unbounded, self-adjoint operator on \(\cH\) that has compact resolvent and bounded commutators with elements of \(A\).
Several other technical conditions (regularity, summability, Poincar\'e duality) may be required, along with other structures such as a grading operator \(\gamma\) and a real structure \(J\).
As spectral triples are not our main focus here, we leave aside the technical details and refer to \cite{GraVarFig01}*{Ch.~9, Ch.~10} or \cite{ConMar08}*{Ch.~10}.

The motivating commutative example is \((C^\infty(X),L^2(X,\cS),\dsl)\).
Here \(X\) is a compact oriented Riemannian spin (or \(\spinc\)) manifold and \(C^\infty(X)\) is the algebra of smooth complex-valued functions on \(X\).
The complex vector bundle \(\cS\) is a \emph{spinor bundle} over \(X\), and \(L^2(X,\cS)\) is the Hilbert space of square-integrable sections.
In algebraic terms, the smooth sections of \(\cS\) constitute a Morita-equivalence bimodule between \(C^\infty(X)\) and the algebra of smooth sections of the Clifford algebra bundle \(\cl(X)\) constructed from the complexified Riemannian metric of \(X\).  
Finally, \(\dsl\) is the \emph{Dirac operator} on \(L^2(X,\cS)\) associated to the Levi-Civita connection.
This is a first-order elliptic differential operator, and it encodes the metric of \(X\) completely; see \cite{GraVarFig01}*{Proposition~9.12}.
The Dirac operator is an important object in differential geometry, representation theory, noncommutative geometry, index theory, \(K\)-theory and \(K\)-homology, and more; see \cite{Con94,Par72,Fri00,HigRoe00,LawMic89,BerGetVer04,GraVarFig01}.
The paper \cite{Con13} proves a reconstruction theorem for commutative spectral triples, i.e.\ it reconstructs the manifold, Riemannian metric, and spin structure from a spectral triple satisfying appropriate conditions.

In the case when \(X\) is just \(\spinc\) rather than \(\spin\), the connection defining the Dirac operator is not uniquely defined, so the connection must be specified.
In general, every almost complex manifold is \(\spinc\); the bundle of antiholomorphic differential forms is a canonical \(\spinc\) structure.
As indicated above, in this thesis we deal only with certain homogeneous spaces \(G/P\).
These are complex manifolds, as the Lie groups involved are complex, and we describe the construction of the canonical \(\spinc\) structure for these manifolds, in terms of induced representations, in \cref{sec:dolbeault-operator-classical-motivation}.
For this construction we use the fact that \(G/P\) is also a homogeneous space of the compact real form \(G_0\) of \(G\); in this guise it can be seen as a \emph{coadjoint orbit}.
Taking a global perspective, the paper \cite{Rie09} gives a detailed construction of Dirac operators for invariant metrics on coadjoint orbits.

\subsection{Dirac operators on quantum homogeneous spaces}
\label{sec:motivation-dirac-ops-on-quantum-homogeneous-spaces}

In \cite{Par72}, Parthasarathy discovered a formula for the square of the Dirac operator over an irreducible flag manifold \(G/P\) in terms of quadratic Casimir elements for \(\fg\) and \(\fl\), where \(\fl\) is the Levi factor of the parabolic subalgebra \(\fp\) corresponding to \(P\).
This formula allows one to compute the eigenvalues of the Dirac operator in terms of representation-theoretic data.
In this thesis we present an algebraic approach to defining a Dirac-type operator on the quantized analogue of \(G/P\), and we make a first step toward establishing a Parthasarathy-type formula.

The remarks of \cref{sec:motivation-homogeneous-spaces-and-quantum-groups} notwithstanding, we do not work directly with the quantized function algebras \(\bbC_q[G]\) or their quotients or subalgebras.
Rather, our Dirac operator will appear as a certain canonically defined element \(\dsl \in \uqg \otimes \clq\), where \(\uqg\) is the quantized enveloping algebra of \(\fg\) and \(\clq\) is an appropriate quantum Clifford algebra that we construct in \cref{chap:quantum-clifford-algebras}.
In the classical setting, this algebraic approach to the Dirac operator is explained in \cite{HuaPan06}*{Ch.~3}; see also \cite{Kos99,Agr03} and \cite{Kna01}*{Lemma 12.12}.
We leave the details to \cref{sec:dolbeault-operator-on-quantized-flag-manifolds}; see \cref{sec:dolbeault-operator-classical-motivation} for more explanation of the classical case.

\section{Lie groups, Lie algebras, and quantum groups}
\label{sec:lie-groups-intro}

In \cref{chap:lie-background} we set notation and recall the elements of Lie theory and quantum-group theory that will be used throughout the rest of the thesis.
\cref{sec:lie-algebras-notation} is all standard material on Lie algebras; we use \cite{Kna02} and \cite{Hum78} as our main references.
For convenience, in \cref{sec:simple-lie-exercises}  we include the proofs of two exercises from \cite{Kna02} that we need later on.
\cref{sec:lie-background-parabolics} discusses parabolic subalgebras of semisimple Lie algebras.
In particular, in \cref{sec:cominuscule-parabolics,sec:structure-of-cominuscule-parabolics,sec:classification-cominuscule-parabolics} we deal with the parabolic subalgebras of \emph{cominuscule type}.
\cref{prop:cominuscule-parabolic-conditions} gives a list of equivalent conditions defining this class of parabolics.
Although these conditions are well-known to experts, I have not seen in one reference the entire list of conditions together with a proof of their equivalence, so it seemed worthwhile to collect this material in one place.
Certain of the conditions are discussed in \cite{BasEas89}*{Example~3.1.10} and \cite{Kob08}*{Lemma~7.3.1}, and in Lie group form in \cite{RicRohSte92}*{Lemma 2.2}.
The very short \cref{sec:lie-groups-notation,sec:generalized-flag-manifolds} mainly set notation for the Lie groups and flag manifolds that arise in \cref{chap:quantum-clifford-algebras}.
In \cref{sec:lie-background-quantized-enveloping-algebras} we set notation and conventions for the quantized enveloping algebras that are ubiquitous throughout the remainder of the thesis.
We follow as much as possible the conventions of \cite{Jan96}.
However, for the braiding on the category of Type 1 \(\uqg\)-modules we refer to \cite{KliSch97}, which contains more details that are needed for our results in \cref{sec:canonical-bases-and-continuity} and later.
Most of the material in \cref{sec:lie-background-quantized-enveloping-algebras} is standard, except for possibly the discussion of the coboundary structure on the category of Type 1 modules in \cref{sec:coboundary-structure}.
In the final \cref{sec:quantum-schubert-cells-and-their-twists} we review the quantum Schubert cells defined by De Concini, Kac, and Procesi in \cite{ConKacPro95}, and their twisted versions defined by Zwicknagl in \cite{Zwi09}.

\section{Quantum symmetric and exterior algebras}
\label{sec:quantum-symmetric-and-exterior-algebras-intro}

In \cref{chap:quantum-symmetric-and-exterior-algebras} we turn to the study of \emph{quantum symmetric and exterior algebras}.
These will be used in the construction of the canonical \(\spinc\) structure and the corresponding Dolbeault-Dirac operator over the quantized flag manifolds that we introduced in \cref{sec:motivation-homogeneous-spaces-and-quantum-groups}.
In particular, the quantum exterior algebra will play the role of the typical fiber of the spinor bundle, and the corresponding quantum symmetric algebra will become a ``quantum commutative'' algebra of invariant differential operators on the quantized flag manifold; see \cref{rem:on-schubert-cell-for-cominuscule-case}.

We begin by discussing the classical versions of the algebras in order to set the stage for the quantum analogues.
As the exterior algebras are just \(\bbZ/2\bbZ\)-graded versions of symmetric algebras, we restrict the discussion to symmetric algebras here.
Throughout the main text we will state our results for both symmetric and exterior algebras.

\subsection{Classical symmetric algebras}
\label{sec:intro-classical-symmetric-algebras}

Let \(\fg\) be a finite-dimensional complex semisimple Lie algebra and let \(V\) be a finite-dimensional representation.
The classical symmetric algebra of \(V\) is the algebra \(\sym(V)\), which by definition is generated by \(V\) subject to the relations ensuring that all generators commute: \(xy = yx\) for all \(x,y \in V\).
If \(\{ v_1, \dots, v_n \}\) is any basis for \(V\), then \(\sym(V)\) is canonically isomorphic to the polynomial ring \(\bbC[v_1, \dots, v_n]\).

The action of \(\fg\) on \(V\) extends uniquely to an action of \(\fg\) on \(S(V)\) by derivations; in the language of Hopf algebras, \(\sym(V)\) is a \(U(\fg)\)-module algebra, where \(U(\fg)\) is the universal enveloping algebra of \(\fg\) with its usual Hopf algebra structure.
As \(\fg\)-modules, the graded components of \(\sym(V)\) are canonically isomorphic to the symmetric powers of \(V\) of the corresponding degree.

Finally, \(\sym(V)\) is the \emph{universal} commutative \(U(\fg)\)-algebra generated by \(V\): the functor \(V \mapsto S(V)\) is left adjoint to the forgetful functor from commutative \(U(\fg)\)-module algebras to \(\fg\)-modules, and \(S(V)\) has a corresponding universal mapping property.

\subsection{A quantum analogue of the symmetric algebra}
\label{sec:intro-quantum-analogues}

It is natural to ask whether there is a quantum analogue of this notion: that is, given a finite-dimensional \(\uqg\)-module \(V\), can we find a \(\uqg\)-module algebra \(\symq(V)\) that is generated by \(V\) and which in some sense resembles \(\sym(V)\)?

The immediate problem is that the action of \(\uqg\) on \(V\) does not extend to one on \(\sym(V)\) except in trivial circumstances.
More precisely, the algebra \(\sym(V)\) is constructed as the quotient
\begin{equation}
  \label{eq:classical-symmetric-algebra}
  \sym(V) = T(V)/ \langle x \otimes y - y \otimes x \mid x,y \in V \rangle,
\end{equation}
where \(T(V)\) is the tensor algebra of \(V\).
The coproduct of \(\uqg\) allows us to turn \(T(V)\) into a \(\uqg\)-module algebra.
However, the relations in \eqref{eq:classical-symmetric-algebra} are not invariant under the action of \(\uqg\) because the coproduct is not symmetric, and so the action does not descend to \(\sym(V)\).
Thus, the problem reduces to finding an appropriate analogue of these relations that are invariant, i.e.\ that form a \(\uqg\)-submodule of \(V \otimes V\).

The general solution, proposed by Berenstein and Zwicknagl in \cite{BerZwi08}, is to view the relations in \eqref{eq:classical-symmetric-algebra} as being the \(-1\)-eigenspace of the tensor flip \(\tau : V \otimes V \to V \otimes V\).
The fact that the coproduct of \(\uqg\) is not symmetric is equivalent to saying that \(\tau\) is not a module map.
But \(\tau\) can be replaced by the \emph{braiding} of \(V \otimes V\) with itself, which is a module map whose eigenvalues are plus or minus powers of \(q\).
More precisely, we have the decomposition
\begin{equation}
  \label{eq:intro-sym2-plus-ext2-equals-v2}
  V \otimes V = \symq^2V \oplus \extq^2 V,
\end{equation}
where \(\symq^2V\) (respectively \(\extq^2 V\)) is the span of the eigenspaces of the braiding for \emph{positive} (respectively \emph{negative}) eigenvalues of the braiding, i.e.\ those of the form \(+q^t\) (respectively of the form \(-q^t\)) as \(t\) varies in \(\bbQ\).
Then \(\symq(V)\) is constructed as the quotient algebra of \(T(V)\) by the ideal generated by \(\extq^2 V\).
We describe this construction in more detail in \cref{sec:defining-quantum-symmetric-and-exterior-algebras}.

The simplest example is when \(\fg = \fsl_2\) and \(V\) is the two-dimensional irreducible representation.
In that case the quantum symmetric algebra is generated by two elements \(x,y\) with the single relation \(yx = q^{-1}xy\).
More generally, the quantum symmetric algebra of the natural \(n\)-dimensional representation of \(\uqsl[n]\) is the \emph{quantum polynomial algebra} generated by \(\{ x_1, \dots, x_n \}\) with relations \(x_j x_i = q^{-1} x_i x_j\) for \(j > i\).
This example is already discussed in \cite{FadResTak89}, along with other examples corresponding to the defining representations of the infinite series of simple Lie algebras; these examples can also be found in \cite{KliSch97}*{Ch.~9}.
Corollary 4.26 of \cite{Zwi09} lists further examples of quantum symmetric algebras.

\subsection{Quantum commutativity}
\label{sec:intro-qsas-commutativity}

In \cref{sec:intro-classical-symmetric-algebras} above we mentioned that symmetric algebras can be viewed as universal objects, i.e.\ the functor that assigns to a vector space its symmetric algebra is left adjoint to a forgetful functor.
In contrast to this, the quantum symmetric algebras are defined explicitly as a quotient of the tensor algebra.
In \cref{sec:quantum-symmetric-algebras-are-commutative} we show that these algebras also can be defined via a universal property.

More precisely, we define a notion of \emph{quantum commutativity} for algebras in coboundary monoidal categories (see \cref{sec:coboundary-structure,sec:quantum-symmetric-preliminaries}).
When the coboundary category is actually a \emph{symmetric} monoidal category, our definition reduces to the usual definition of commutativity.
We show that the quantum symmetric algebras have this property, and moreover that they are universal with respect to it.
The proof of quantum commutativity is the more difficult part; once that is established, universality is straightforward.
We also compare our definition with the related notion of \emph{braided} commutativity.
The main result is \cref{thm:quantum-symmetric-algebra-is-enveloping-comm-alg}.

\subsection{Flatness, collapsing, and related properties}
\label{sec:intro-qsas-flatness-and-collapsing}

The examples constructed in \cite{FadResTak89} and discussed in Chapter 9 of \cite{KliSch97} are atypical of general quantum symmetric algebras in the sense that they are \emph{flat deformations} of the classical versions, i.e.\ their graded components all have classical dimension.
This is not the case for most examples, even when the generating module is simple; the main result of \cite{Ros99} is that the quantum symmetric algebra of the four-dimensional irreducible representation of \(\uqsl\) is not a flat deformation.
We address this issue in \cref{sec:flatness-and-related-properties,sec:collapsing-in-degree-three}.

In \cite{BerZwi08} the authors proved that the graded components of the quantum symmetric algebra of a given module are no larger than those of their classical counterparts.
That is, the \emph{Hilbert series} of the quantum symmetric algebra (i.e.\ the generating function for the dimensions of the graded components) is coefficient-wise no larger than the classical one.
However, Berenstein and Zwicknagl worked over the rational function field \(\bbC(q)\), whereas we take our deformation parameter \(q\) to be a positive real number.
Thus in \cref{sec:flatness-and-parameter-values} we spend some time examining how the quantum symmetric algebras depend on the parameter value.
We find that the dependence is quite mild: the Hilbert series is the same for all \(q\) lying in a dense open subset of \((0,\infty)\).
This uses ideas of Drinfeld from \cite{Dri92}, which were also used in a related but more general context in Chapter 6, Section 2 of \cite{PolPos05}.

Zwicknagl in \cite{Zwi09} has classified the simple representations of reductive Lie algebras whose corresponding quantum symmetric algebras are flat deformations in this sense.
With one exception, his list corresponds to those representations arising as the (abelian) nilradicals of cominuscule parabolic subalgebras in simple Lie algebras.
We briefly review this classification in \cref{sec:flatness-and-cominuscule-parabolics}.
Then in \cref{sec:quantum-sym-algs-as-quantum-schubert-cells} we recall his embedding of these flat quantum symmetric algebras into the ambient \(\uqg\).
We use this embedding in \cref{sec:filtrations-on-flat-quantum-algebras} to construct certain filtrations on these quantum symmetric algebras whose associated graded algebras are fairly simple.
Finally, in \cref{sec:quantum-exterior-algebras-are-frobenius} we use the theory of \emph{Koszul duality} to transfer this filtration to the associated quantum exterior algebra and prove that it is a Frobenius algebra.
We use this fact later on in \cref{sec:clifford-alg-quantum-case} when we define the spinor representation of the quantum Clifford algebra and prove that it is irreducible.

In \cref{sec:collapsing-in-degree-three} we prove Conjecture 2.26 of \cite{BerZwi08}, which the authors call ``numerical Koszul duality.''
This conjecture deals with the amount of collapsing in the degree three components of the quantum symmetric and exterior algebras in the non-flat case.
More precisely, it states that the difference in dimension between the classical and quantum symmetric cubes of a module is the same as the difference between the classical and quantum exterior cubes.
In fact, we prove somewhat more: we show that in a precise sense these ``differences'' are composed of the same \emph{submodules}.
This is most conveniently phrased in terms of the Grothendieck ring of the category of Type 1 modules, which we discuss in \cref{sec:grothendieck-ring}.
Then in \cref{sec:three-fruit-cactus-group} we prove some technical results before proving the conjecture in \cref{sec:numerical-koszul-duality}; the main result is \cref{thm:collapsing-in-degree-three}.

\subsection{Technical tools}
\label{sec:intro-qsas-technical-tools}

We now briefly discuss the main technical tools that we use in our investigation of quantum symmetric and exterior algebras.
These are the \emph{coboundary structure} on the category of Type 1 \(\uqg\)-modules, Lusztig's theory of \emph{canonical bases}, and the theory of \emph{quadratic algebras}.

The coboundary structure on the module category is a modification of the braided structure.
More precisely, the coboundary maps (\emph{commutors}, following the terminology of \cite{KamTin09}) are the unitary parts of the polar decompositions of the braidings.
They are also self-adjoint, so their eigenvalues are only \(\pm 1\).
The commutors are useful because the \(+1\) and \(-1\) eigenspaces are exactly the spaces \(\symq^2V\) and \(\extq^2V\) from \eqref{eq:intro-sym2-plus-ext2-equals-v2}, respectively.
This allows for a more uniform definition of the quantum symmetric and exterior algebras.
Coboundary monoidal categories were introduced by Drinfeld in connection with quasi-Hopf algebras \cite{Dri89}*{\S 3}.
Symmetric monoidal categories give rise to representations of the symmetric groups, and braided monoidal categories give rise to representations of the braid groups.
Likewise, there is a family of groups associated to coboundary monoidal categories: the so-called \emph{cactus groups}.
We introduce the coboundary structures in \cref{sec:coboundary-structure} and discuss them further in \cref{sec:quantum-symmetric-preliminaries,sec:continuity-of-coboundary-structure,sec:commutative-algebras-in-coboundary-categories,sec:three-fruit-cactus-group}.
We refer also to \cite{HenKam06,KamTin09} for more information.

Canonical bases are discussed extensively in \cref{sec:canonical-bases-and-continuity}, using \cite{Lus10} as our main source.
Roughly, these are bases for \(\uqg\)-modules that are preserved by a certain integral form of the quantum group.
With respect to this basis, the matrix coefficients of the generators of the integral form are Laurent polynomials in \(q\).
These Laurent polynomials are continuous and can be specialized at any \(q > 0\), and this allows us to build what we call a \emph{universal model} for a representation: that is, a single complex vector space carrying an action of \(\uqg\) for \emph{all} \(q > 0\) simultaneously.
These actions are continuous in \(q\) in an appropriate sense, as are the associated braidings and commutors.
This allows us to make rigorous arguments using limits and continuity, which are essential parts of the proofs of \cref{thm:quantum-symmetric-algebra-is-enveloping-comm-alg,thm:collapsing-in-degree-three}.

Finally, the theory of quadratic algebras is used throughout.
In particular, any quadratic algebra has an associated quadratic dual algebra.
For instance, the quadratic dual of a classical symmetric algebra \(S(V)\) is the classical exterior algebra \(\ext(V^\ast)\), and the same is true for the quantum versions.
We use this machinery to transfer structures and results between quantum symmetric and exterior algebras.
We use \cite{PolPos05} as our main reference for this theory.

\subsection{Previous work}
\label{sec:intro-qsas-previous-work}

The definition of quantum symmetric and exterior algebras for arbitrary finite-dimensional representations of \(\uqg\) seems to be due to Berenstein and Zwicknagl in \cite{BerZwi08}, but certain cases of this construction have been considered previously.
Drinfeld discussed the case when the braiding is involutive in \cite{Dri92}, and Berger in \cite{Ber00} considered quantum symmetric algebras corresponding to \emph{Hecke-type} braidings, i.e.\ braidings satisfying a quadratic minimal polynomial.
\index[term]{Hecke relation}
O.\ Rossi-Doria investigated the quantum symmetric algebra associated to the four-dimensional irreducible representation of \(\uqsl\) in \cite{Ros99}.

In \cite{LehZhaZha11} the authors construct quantum analogues of symmetric algebras for the defining representations of the classical Lie algebras; these agree with the definition of Berenstein and Zwicknagl for simple representations, but not for direct sums of simples.
In \cite{BerGre11}, Berenstein and Greenstein discussed a class of examples of quantum symmetric algebras related to those in \cite{BerZwi08}, and furthermore they found a spectral criterion on the braiding that guarantees a flat deformation; this condition was also pointed out by Fiore in \cite{Fio98}*{Lemma~1}.
Although we do not touch on these results in this thesis, Zwicknagl investigated other aspects of quantum symmetric algebras in \cite{Zwi09b,Zwi12}.

Other authors have investigated quantum exterior algebras in the context of the \emph{covariant differential calculi} of Woronowicz \cite{Wor89}.
These are \emph{global} objects, i.e.\ they are viewed as modules of sections of the bundle of differential forms over a quantum group or quantum space, whereas our quantum exterior algebras are \emph{local}, i.e.\ they should be viewed as the typical fiber of one of the global objects.
(This viewpoint can be made precise using the quantum analogues of the notions of induced representations and homogeneous vector bundles.)
The Woronowicz-style covariant differential calculi are constructed from a braiding using an explicit antisymmetrizer defined as a sum over permutations.
This generally leads to an exterior algebra with classical dimension only in the case when the braiding satisfies a Hecke-type relation.
See \cite{HecSch98}, in which the authors address this issue and give an  alternate construction that gives the classical dimension for an exterior algebra over the quantum group \(\cO_q(O(3))\).

Finally, the paper \cite{HecKol06}, building on the authors' previous work in \cite{HecKol04}, constructs quantum exterior algebras associated to the nilradicals of cominuscule parabolics.
The construction is global in nature: they build a covariant differential calculus over the quantization of an irreducible flag manifold, but the fiber over the identity coincides with the flat quantum exterior algebras that we consider here.
The work of Heckenberger and Kolb provides an alternative proof of flatness for these quantum exterior algebras.

\subsection{Questions and future work}
\label{sec:intro-qsas-questions-and-future-work}

Here we outline some remaining open questions regarding the quantum symmetric and exterior algebras.

\subsubsection{Transcendentality and the generic set}
\label{sec:intro-qsas-questions-generic-set}

In \cref{sec:intro-qsas-flatness-and-collapsing} we briefly discussed how the Hilbert series of the quantum symmetric algebra of a module depends on the value of the parameter \(q\).
We show in \cref{prop:transcendental-quantum-algs-are-at-most-classical} that, except for possibly finitely many algebraic numbers \(q>0\), the Hilbert series of the quantum symmetric algebra \(\symq(V)\) is coefficient-wise no larger than the classical Hilbert series.
In \cref{dfn:generic-set-for-V} we define the \emph{generic set} for \(V\) to be the set of real numbers \(q>0\) where this holds.
In all of the examples that I have worked with, this holds for \emph{all} \(q > 0\).
Therefore we ask:
\begin{quest}
  \label{quest:is-generic-set-everything}
  Is the generic set for a module always all of \((0,\infty)\)?
\end{quest}

Furthermore, two of our main results, namely \cref{thm:quantum-symmetric-algebra-is-enveloping-comm-alg,thm:collapsing-in-degree-three}, require that the value of \(q\) is transcendental, and the proofs use this assumption in an essential way.
But again, computation with examples suggests that these results hold for all \(q > 0\).
This prompts:
\begin{quest}
  \label{quest:do-main-theorems-hold-for-algebraic-q}
  Can the assumption that \(q\) is transcendental be removed from \cref{thm:quantum-symmetric-algebra-is-enveloping-comm-alg,thm:collapsing-in-degree-three}?
\end{quest}

\subsubsection{Ring-theoretic and homological properties}
\label{sec:intro-qsas-questions-other-properties}

Much is known about the structure of the quantum symmetric algebras that are flat deformations.
For instance, they are PBW algebras, and hence Koszul.
We mentioned previously that filtrations can be defined on them for which the associated graded algebras have only \(q\)-commutation relations.
Certain algebraic information can then be lifted from the associated graded back to the algebras themselves: we find that the flat quantum symmetric algebras are Noetherian domains of finite global dimension.
Furthermore, the fact that the flat quantum exterior algebras are Frobenius algebras implies, via Koszul duality, that the flat quantum symmetric algebras are twisted Calabi-Yau.
Finally, carefully analyzing the eigenvalues of the braidings, one can apply Proposition 41 in Chapter 10 of \cite{KliSch97} to show that the flat quantum symmetric algebras are in fact \emph{braided} Hopf algebras in the sense of \cite{KliSch97}*{Ch.~10, Definition 11}, or \emph{braided groups} in the sense of \cite{Maj95}*{Definition 9.4.5}.

To my knowledge, comparably little is known about the quantum symmetric and exterior algebras in the non-flat cases.

\begin{quest}
  \label{quest:alg-and-hom-properties-of-qsas}
  Are there any general algebraic and/or homological properties that are satisfied for all quantum symmetric and exterior algebras?
\end{quest}

A more refined question is:

\begin{quest}
  \label{quest:alg-and-hom-properties-of-qsas-and-rep-theory}
  How do the representation-theoretic properties of the \(\uqg\)-module \(V\) affect the ring-theoretic and homological properties of \(\symq(V)\) and \(\extq(V)\)?
  In particular, if \(V = V(\lambda)\) is a simple module with highest weight \(\lambda\), can we say anything further?
\end{quest}

\section{Quantum Clifford algebras and flag manifolds}
\label{sec:quantum-clifford-algebras-and-flag-mflds-intro}

In \cref{chap:quantum-clifford-algebras} we use the theory developed in the preceding chapter to construct a quantum analogue of a Clifford algebra associated to certain hyperbolic quadratic spaces.
Then we use it to re-examine the construction in \cite{Kra04} of a Dirac-type operator on certain quantized flag manifolds.

\subsection{Classical Clifford algebras}
\label{sec:intro-qcas-classical-clifford-alg}

A \emph{hyperbolic quadratic space} is a pair \(H(V) = (V \oplus V^\ast, h_V)\), where \(V\) is a finite-dimensional vector space and
\[
h_V : (V \oplus V^\ast) \otimes (V \oplus V^\ast) \to \bbC
\]
is the canonical symmetric bilinear form coming from the evaluation pairing.
The associated \emph{Clifford algebra} (see \cref{dfn:clifford-algebra}) is the algebra \(\cl(V)\) generated by \(V\) and \(V^\ast\) with relations
\begin{equation}
  \label{eq:hyperbolic-clifford-algebra-def}
  v \cdot v = 0, \qquad \phi \cdot \phi = 0, \qquad v \cdot \phi + \phi \cdot v = \langle \phi,v \rangle,
\end{equation}
for \(v \in V\) and \(\phi \in V^\ast\).
The Clifford algebra acts on the exterior algebra \(\ext(V)\) by \emph{creation and annihilation operators}, and this is known as the \emph{spinor representation}.
One can show that the corresponding homomorphism \(\cl(V) \to \End(\ext(V))\) is in fact an isomorphism.
The creation operators generate an isomorphic copy of \(\ext(V)\) inside \(\End(\ext(V))\), while the annihilation operators generate a copy of \(\ext(V^\ast)\).
Then the commutation relations \eqref{eq:hyperbolic-clifford-algebra-def} give us a factorization of algebras \(\End(\ext(V)) \cong \ext(V^\ast) \otimes \ext(V)\).
We review this construction in \cref{sec:clifford-alg-classical-case}.

We refer the reader to \cite{Che97} for the general theory of Clifford algebras over arbitrary fields, or to \cite{GraVarFig01}*{Ch.~5} for a very concrete perspective on Clifford algebras over \(\bbR\) or \(\bbC\).
The paper \cite{Bas74} also gives an account of Clifford algebras defined for quadratic modules over arbitrary commutative rings.
(This point of view can be applied directly to the ring \(C^\infty(X)\) and the module \(\Gamma^\infty(X,TX)\) of smooth vector fields over a compact Riemannian manifold \(X\) in order to construct the Clifford bundle \(\cl(X)\).)

Then in \cref{sec:frobenius-viewpoint-on-clifford-algebra} we revisit the isomorphism \(\cl(V) \cong \End(\ext(V))\) from the perspective of Frobenius algebras.
We identify the exterior algebras \(\ext(V)\) and \(\ext(V^\ast)\) as graded dual vector spaces, and use this identification to give an alternate definition of the annihilation operators.
Then we use the Frobenius property of the exterior algebras to prove that we have the factorization \(\End(\ext(V)) \cong \ext(V^\ast) \otimes \ext(V)\).
While this is not deep, it seems to be novel, as we do not use the commutation relations between the creation and annihilation operators to establish the factorization.

\subsection{Quantum Clifford algebras}
\label{sec:intro-qcas-quantum-clifford-alg}

In \cref{sec:clifford-alg-quantum-case} we construct a quantum Clifford algebra \(\clq\) associated to a certain hyperbolic space \(\up \oplus \um\), where \(\up\) is the (abelian) nilradical associated to a cominuscule parabolic subalgebra \(\fp\) of a simple Lie algebra \(\fg\), and \(\um \cong \up^\ast\) is the nilradical of the \emph{opposite} parabolic subalgebra.
Rather than describing \(\clq\) as a quotient of the tensor algebra of \(\up \oplus \um\), as in the classical case, we construct it via quantum analogues of creation and annihilation operators on \(\extq(\up)\), the quantum exterior algebra of \(\up\).
Here \(\upm\) are viewed as representations of \(\uql\), the quantized enveloping algebra of the Levi factor \(\fl\) of \(\fp\).

Our development closely parallels that in \cref{sec:frobenius-viewpoint-on-clifford-algebra}.
We use the duality between the two quantum exterior algebras \(\extq(\upm)\) to define the annihilation operators, and the algebra factorization \(\clq \cong \extq(\um) \otimes \extq(\up)\) follows from the Frobenius property, as in the classical case.
This is \cref{thm:quantum-gamma-factorization-isomorphism}.
Moreover, this factorization \emph{implies} that commutation relations exist between the creation and annihilation operators, although we cannot predict what they are in general.
In \cref{sec:cliff-alg-for-cp2} we examine the simplest possible nontrivial example and find that the relations are not quadratic-constant but rather involve higher-order terms.
In \cref{sec:star-structure-on-quantum-clifford-algebra} we discuss inner products on \(\extq(\up)\) and the corresponding \(\ast\)-structures on the quantized Clifford algebra \(\clq\).

\subsection{The Dolbeault-Dirac operator on quantized flag manifolds}
\label{sec:intro-qcas-dolbeault-dirac}

In \cref{sec:dolbeault-operator-on-quantized-flag-manifolds} we construct our quantized Dolbeault-Dirac operator.
With notation as in \cref{sec:intro-qcas-quantum-clifford-alg}, we define the canonical element
\[
\eth \eqdef \sum_i x_i \otimes y_i \in \up \otimes \um,
\]
where \(\{ x_i \}\) and \(\{ y_i \}\) are dual bases for \(\up\) and \(\um\), respectively.
We view \(\eth\) as an element in the tensor product algebra \(\symq(\up)^{\op} \otimes \extq(\um)\).
The two factors in this tensor product are quadratic dual to one another, and \(\eth\) is the boundary operator for the Koszul complex of \(\symq(\up)\).
It follows automatically that \(\eth^2 = 0\).

Using the embedding of \(\symq(\up)\) into \(\uqg\) that we review in \cref{sec:quantum-sym-algs-as-quantum-schubert-cells} and the embedding of \(\extq(\um)\) into \(\clq\) coming from the algebra factorization discussed  above, we regard \(\eth\) as an element of \(\uqg \otimes \clq\).
We endow \(\uqg\) with the \(\ast\)-structure on \(\uqg\) known as the compact real form (see \cref{sec:quea-compact-real-form}), and an inner product on \(\extq(\up)\) induces a \(\ast\)-structure on \(\clq\).
Finally, our quantized analogue of the Dolbeault-Dirac operator is defined to be
\[
\dsl = \eth + \eth^\ast \in \uqg \otimes \clq.
\]
This algebraic object implements the Dolbeault-Dirac operator defined by Kr\"ahmer in \cite{Kra04}.
The major contribution of our approach is to show that \(\eth^2 = 0\), and hence that \(\dsl^2 = \eth \eth^\ast + \eth^\ast \eth\), which is our \cref{thm:square-of-dirac-operator}.
We view this as the first step toward a quantum Parthasarathy-type formula.

\subsection{Previous work}
\label{sec:intro-qcas-previous-work}

\subsubsection{Quantum Clifford algebras}
\label{sec:intro-qcas-previous-work-on-qcas}

Several authors have considered quantum Clifford algebras previously.
Our approach seems to be different from all of these, mainly because of the form of the relations between the creation and annihilation operators (ours are not quadratic-constant).
We now briefly survey some of these constructions.

The earliest example seems to be due to Hayashi \cite{Hay90}.
He constructed an algebra \(\cA^+_q\) generated by elements \(\psi_i, \psi_i^\dagger, \omega_i\), which he called the \(q\)-Clifford algebra.
The \(\psi_i\)'s satisfy exterior algebra relations among themselves, as do the \(\psi_i^\dagger\)'s; only the cross-relations between them involve extra \(q\)-factors.
The \(\omega_i\)'s are invertible, and their action by conjugation on the \(\psi_i\)'s and \(\psi_i^\dagger\)'s introduces further \(q\)-factors.
The algebra \(\cA^+_q\) acts on an ordinary exterior algebra, and moreover Hayashi constructs homomorphisms from \(\uqg\) to \(\cA^+_q\) for \(\fg\) of type \(A,B\), and \(D\), which he refers to as \emph{spinor representations} of the quantized enveloping algebras.
This construction defines a \emph{single} quantum Clifford algebra for each simple Lie algebra \(\fg\) rather than functorially associating one to a representation of \(\uqg\).
In \cite{HakSed93} the authors perform a similar construction for types \(B,C\), and \(G_2\) based on \cite{Hay90} and the folding procedure for simple Lie algebras.

In the short paper \cite{BrzPapRem93}, the authors construct multi-parameter deformations of the \emph{real} Clifford algebras \(\mathrm{Cl}^{p,q}\) along with some explicit representations.
These quantum Clifford algebras do not come equipped with a quantum group action.

The paper \cite{DinFre94} by Ding and Frenkel is much more substantial.
This seems to be the first approach to quantum Clifford algebras from the perspective of representation theory.
They work just with the defining representations of the quantized enveloping algebras of the classical simple Lie algebras.
They explicitly identify the quantum symmetric and exterior squares of these  representations and describe the projections as polynomials in the braiding, similarly to \cite{FadResTak89}.
Their creation and annihilation operators are constructed abstractly, not realized as multiplication or contraction operators on the quantum exterior algebra, but they do find explicit commutation relations in terms of the braiding.

The two papers \cite{BauEtal96,DurOzi94} are closest in spirit to our approach, using creation and annihilation operators on quantum analogues of an exterior algebra.
The notion of quantum exterior algebra is the one of Woronowicz from \cite{Wor87}, i.e.\ it is constructed as the quotient of the tensor algebra by the kernel of a certain \emph{antisymmetrizer} constructed from the braiding of the underlying module.
This exterior algebra agrees with the one defined by Berenstein and Zwicknagl only in the cases when the braiding satisfies a Hecke-type condition, i.e.\ a quadratic minimal polynomial.
\index[term]{Hecke relation}

Fiore (see \cite{Fio98,Fio99}, and references therein) takes an interesting approach to quantum Clifford algebras.
He shows that one can view the generators of a quantum Clifford algebra as certain polynomials in the generators of a classical Clifford algebra, thus connecting the representations of the two algebras.
He also defines commutation relations in terms of the braidings, and notes that not every representation gives rise to a deformed algebra with classical dimension; see the remarks between equations (2.6) and (2.7) in \cite{Fio99}, as well as Lemma 1 in \cite{Fio98}.
These papers deal with \(\ast\)-structures as well.

In \cite{HecSch00} the authors deal extensively with a quantum Clifford algebra associated to the vector representation of \(U_{q^2}(\fs\fo_N)\).
They give a representation on an appropriate quantum exterior algebra, which is constructed using an explicit formula for a quantum antisymmetrizer developed by the authors in \cite{HecSch99}.
The introduction also contains some remarks discussing the relation of their quantum Clifford algebra to other definitions.

In \cite{Hec03}, Heckenberger constructs global versions of quantum Clifford algebras over quantum groups, i.e.\ a quantum analogue of the module of sections of the Clifford bundle.
The input for the construction is a first-order covariant differential \(\ast\)-calculus over the quantum group, which can be thought of as a module of \(1\)-forms on a Lie group equipped with a Hermitian metric.
The quantum Clifford algebra acts on the quantum exterior algebra constructed from the given first-order calculus by the procedure in \cite{Wor89}.
He constructs Dirac and Laplace operators as well.

In \cite{Han00,Fau00,Fau02}, the authors use the term ``quantum Clifford algebra'' to mean something else, namely a sort of Clifford algebra associated to an arbitrary bilinear form, not necessarily symmetric.
However, the paper by Hannabuss discusses connections with quantum groups, R-matrices, and quantum Clifford algebras in our sense as well.

\subsubsection{Dirac operators on quantum homogeneous spaces}
\label{sec:intro-qcas-previous-work-on-homogeneous-dirac-ops}

Many authors have constructed Dirac operators on compact quantum \emph{groups}: the case of the quantum group \(SU_q(2)\) was treated in \cite{DabEtal05b} and \cite{BibKul00}, the latter of which also discusses quantum spheres.
Neshveyev and Tuset treated the general case of compact quantum groups in \cite{NesTus10}.

Most examples of Dirac operators on quantum homogeneous spaces of the type we describe have been constructed on various quantum spheres, although more recently the quantum projective spaces have also been treated.

The first example was the construction of a spectral triple over the standard Podle\'s quantum sphere by D{\c{a}}browski and Sitarz in \cite{DabSit03}.
They quantized the standard \emph{spin} structure (rather than \(\spinc\)) and Levi-Civita connection.
D{\c{a}}browski, Landi, Paschke, and Sitarz then constructed a spectral triple over the equatorial Podle\'s sphere in \cite{DabEtal05}, and later D{\c{a}}browski, D'Andrea, Landi, and Wagner treated \emph{all} of the Podle\'s spheres in \cite{DabEtal07}.

The first higher-dimensional example was the quantum projective plane \(\bbC\bbP_q^2\), which was investigated by D{\c{a}}browski, D'Andrea, and Landi in \cite{DanDabLan08}.
The higher-dimensional projective spaces were then treated in \cite{DanDab10}.
In these examples the spectrum of the Dirac operator can be computed explicitly.
However, these papers rely on the Hecke condition for the braidings, and most of the proofs are computations that use the specific features available in this case.
Thus it seems unlikely that their methods can be generalized to other flag manifolds.

The only general construction seems to be that of \cite{Kra04}.
The construction originally given in that paper gave only an abstract argument for the existence of the spinor representation of the quantum Clifford algebra, and it was neither unique nor canonical.
This aspect has been clarified in this thesis, but several questions remain open.
We outline these in the next section.

\subsection{Questions and future work}
\label{sec:intro-qcas-questions-future-work}

Most versions of quantum Clifford algebras that we discussed in \cref{sec:intro-qcas-previous-work} come with explicit commutation relations between the creation and annihilation operators, whereas our implicit definition from \cref{sec:clifford-alg-quantum-case} does not include such relations.

\begin{quest}
  \label{quest:what-are-commutation-rels-for-qca}
  What are the relations between the quantum creation and annihilation operators from \cref{dfn:quantum-creation-operators,dfn:quantum-annihilation-operators}?
\end{quest}

As we discuss in \cref{rem:on-non-uniqueness-of-star-structure}, the inner product on the quantum exterior algebra \(\extq(\um)\), and hence the \(\ast\)-structure on the associated quantum Clifford algebra, are not unique.
It is not clear if there is a ``correct'' choice for this inner product.
The one described in \cref{sec:star-structure-on-quantum-clifford-algebra} is canonical, but may not be what is needed in order to find a good analogue of the Parthasarathy formula for \(\dsl^2\).

\begin{quest}
  \label{quest:what-is-correct-star-structure-for-qca}
  Is there a choice of inner product/\(\ast\)-structure that will allow us to characterize the commutation relations between the creation operators and their adjoints?  Is there a choice that will yield quadratic commutation relations?
\end{quest}

Our last question relates to the original, ultimate goal of the project:

\begin{quest}
  \label{quest:parthasarathy-formula-for-D-squared}
  Can we find a quantum analogue of the Parthasarathy formula for \(\dsl^2\), involving quantum Casimir elements, that will allow us to compute the spectrum in terms of representation-theoretic data?
\end{quest}

Finally, we note that the constructions described in this thesis apply only to a very restrictive class of flag manifolds, namely the irreducible ones.
The fact that the relevant quantum symmetric and exterior algebras are flat, as well as the fact that the quantum symmetric algebra embeds into the ambient quantized enveloping algebra, seem to be quite specific to this case.

\begin{quest}
  \label{quest:can-we-do-it-for-other-flag-manifolds}
  Can the methods of this thesis be adapted or extended to deal with quantized flag manifolds in which the parabolic subgroup \(P\) is not of cominuscule type, i.e.\ the flag manifold \(G/P\) is not irreducible?
\end{quest}

\section{Standard notation}
\label{sec:standard-notation}

Here we set some standard notation and define some terminology that will be used without further comment throughout this thesis:

\begin{itemize}
\item \(\bbN, \bbZ, \bbQ, \bbR, \bbC\) denote the natural numbers, the integers, the rational numbers, the real numbers, and the complex numbers, respectively.  We adopt the convention that \(0 \notin \bbN\).
\item \(\zp = \{ 0,1,2,\dots \}\) denotes the non-negative integers.
\item \([N] = \{ 1, \dots, N \}\) for \(N \in \bbN\).
\item If \(X\) is an object of a category \(\cC\), we write \(X \in \cC\) rather than \(X \in \mathrm{ob}(\cC)\).
\item All vector spaces are over \(\bbC\).  We write \(\otimes\) for \(\otimes_\bbC\) and \(\dim\) for \(\dim_\bbC\) unless stated otherwise.  Likewise, \(\Hom\) and \(\End\) stand for \(\Hom_\bbC\) and \(\End_\bbC\), respectively.
\item For a vector space \(V\), the dual vector space is \(V^\ast \eqdef \Hom_\bbC(V,\bbC)\).  If \(T : V \to W\) is a linear map, then \(T^{\tr} : W^\ast \to V^\ast\) is the transpose, or dual map.  \index[term]{transpose}
\item By \emph{algebra} we mean unital associative algebra over \(\bbC\). A module for an algebra is assumed to be a left module unless stated otherwise.
\item By \emph{graded algebra} we mean a \(\zp\)-graded algebra, i.e.\ an algebra \(A\) given as a vector space direct sum \(A = \bigoplus_{k=0}^\infty A_k\), with \(A_k \cdot A_l \subseteq A_{k+l}\).  All of our graded algebras will be locally finite-dimensional, i.e.\ \(\dim A_k < \infty\) for all \(k\).
\item The \emph{tensor algebra} of a finite-dimensional vector space \(V\) is the graded algebra
  \[
  T(V) \eqdef \bigoplus_{k=0}^\infty V^{\otimes k},
  \]
  where \(V^{\otimes 0} \eqdef \bbC\), and the multiplication is given by concatenation of tensors. \index[term]{tensor algebra} \index[notn]{TV@\(T(V)\)}
\item A \emph{quadratic algebra} is a graded algebra of the form \(A = T(V)/ \langle R \rangle\), where \(V\) is a finite-dimensional vector space, \(R \subseteq V \otimes V\) is a subspace, and \(\langle R \rangle\) is the two-sided ideal generated by \(R\). \index[term]{quadratic algebra}
\item The \emph{free algebra} on generators \(\{ x_1, \dots, x_n \}\) will be denoted \(\bbC \langle x_1, \dots, x_n \rangle\); this is canonically isomorphic to the tensor algebra \(T(V)\), where \(V\) is the free vector space with basis \(\{ x_1, \dots, x_n \}\).
\item To say that \(M\) is a module over a Hopf algebra \(H\) means that \(M\) is a left module for the underlying algebra structure of \(H\).
\end{itemize}

\chapter{Lie groups, Lie algebras, and quantum groups}
\label{chap:lie-background}

In this chapter we set notation and recall the elements of Lie theory and quantum group theory that will be used throughout this thesis.
We will always work over \(\bbC\), so all Lie groups, Lie algebras, Hopf algebras, representations, etc.\ will be complex unless we explicitly specify otherwise.
All algebras are associative and unital.

We refer the reader to \cite{Hum78} or to \cite{Kna02} for the basic theory of semisimple Lie algebras; the latter also contains all of the facts about Lie groups that we need.

\section{Semisimple Lie algebras}
\label{sec:lie-algebras-notation}

For the rest of this thesis, \(\fg\)\index[notn]{g@\(\fg\)} will denote a finite-dimensional complex semisimple Lie algebra.
We fix a Cartan subalgebra \(\fh \subseteq \fg\)\index[notn]{h@\(\fh\)}.

\subsection{Root system}
\label{sec:root-system}

The Cartan subalgebra \(\fh\) determines the root system \(\rtsys \subseteq \fh^\ast\)\index[notn]{phi@\(\rtsys,\posrts,\negrts\)}.
Let \(\posrts \subseteq \rtsys\) be a positive system and denote \(\negrts = - \posrts\), so that \(\rtsys = \posrts \cup \negrts\) is a disjoint union.
For any root \(\beta \in \rtsys\) we denote by \(\fg_\beta \subseteq \fg\) the corresponding root space.
Then we set \(\fn_{\pm} = \oplus_{\beta \in \rtsys^{\pm}}\fg_\beta\), and \(\fb_{\pm} = \fh \oplus \fn_{\pm}\), so that \(\fb_\pm\) are an opposite pair of Borel (i.e.\ maximal solvable) subalgebras of \(\fg\), and \(\fn_\pm\) are the nilpotent radicals of \(\fb_\pm\), respectively.
Let \(\simprts = \{\alpha_1, \dots, \alpha_r\} \subseteq \posrts\)\index[notn]{pi@\(\simprts\)}\index[notn]{alpha@\(\alpha_j\)} be the set of simple roots, so that \(r\) is the rank\index[term]{rank} of \(\fg\).
Let
\[
\cQ = \bigoplus_{j=1}^{r} \bbZ \alpha_j \quad \text{and} \quad \qplus = \bigoplus_{j=1}^{r} \zp \alpha_j
\]
\index[notn]{q@\(\cQ, \qplus\)}
be the integral root lattice and its positive cone, respectively.
The \emph{height function}\index[term]{height function}\index[notn]{height@hgt}  \(\hgt : \qplus \to \zp\) is given by \(\hgt(\sum_{j=1}^r n_j \alpha_j) = \sum_{j=1}^r n_j\).

If there is risk of ambiguity regarding the Lie algebra we are referring to, we will write \(\rtsys(\fg), \posrts(\fg), \simprts(\fg), \cQ(\fg)\), etc.

\subsection{Killing form and Cartan matrix}
\label{sec:killing-form-and-cartan-matrix}

As \(\fg\) is semisimple, the Killing form \(K(X,Y) = \tr(\ad X \circ \ad Y)\)\index[notn]{K@\(K\)} is nondegenerate on \(\fh\) and thus induces a linear isomorphism of \(\fh\) with \(\fh^\ast\).
For \(\phi \in \fh^\ast\), let \(H_\phi \in \fh\)\index[notn]{Hphi@\(H_\phi\)} be the unique element such that \(\phi(H) = K(H_\phi,H)\) for all \(H \in \fh\), and denote \(H_j = H_{\alpha_j}\)\index[notn]{Hj@\(H_j\)}.
There is an induced symmetric bilinear form \((\cdot,\cdot)\) on \(\fh^\ast\) given by
\[
(\phi,\psi) \eqdef c K(H_\phi,H_\psi) = c \phi(H_\psi) = c \psi(H_\phi),
\]
where the constant \(c\) is chosen so that \((\alpha,\alpha) = 2\) for all short roots \(\alpha\).
For any \(\alpha \in \rtsys\) we define
\[
d_\alpha = \frac{(\alpha,\alpha)}{2}
\]
and we abbreviate \(d_j = d_{\alpha_j}\)\index[notn]{dj@\(d_j,d_\alpha\)} for \(\alpha_j \in \simprts\).
The \emph{Cartan integers} are defined by
\[
\quad a_{ij} = \frac{2(\alpha_i,\alpha_j)}{(\alpha_i,\alpha_i)}
\]
for \(1 \leq i,j \leq r\).
For \(\alpha \in \rtsys\) we define the \emph{coroot} \(\alpha^\vee\)\index[notn]{alphacheck@\(\alpha^\vee\)} by \(\alpha^\vee = \frac{2\alpha}{(\alpha,\alpha)}\).
In particular, we have \(\alpha_i^\vee = \frac{\alpha_i}{d_i}\) and \(a_{ij} = (\alpha_i^\vee,\alpha_j)\)\index[notn]{aij@\(a_{ij}\)}. 
The \emph{Cartan matrix}\index[term]{Cartan matrix} of the root system \(\rtsys\) is \((a_{ij})\).

\subsection{Chevalley basis}
\label{sec:chevalley-basis}

As \(\fg\) is semisimple, it is the direct sum of its Cartan subalgebra with its root spaces:
\[
\fg = \fh \oplus \bigoplus_{\alpha \in \rtsys} \fg_\alpha.
\]
For \(\alpha \in \posrts\) we fix elements \(E_\alpha \in \fg_\alpha\)\index[notn]{E@\(E_\alpha,E_i\)} and \(F_\alpha \in \fg_{-\alpha}\)\index[notn]{F@\(F_\alpha,F_i\)} such that \(\{ E_\alpha, F_\alpha, h_\alpha \}\) forms an \(\fsl_2\)-triple, where \(h_\alpha \eqdef H_{\alpha^\vee}\)\index[notn]{halpha@\(h_\alpha\)}.
We denote \(E_j = E_{\alpha_j}\) and \(F_j = F_{\alpha_j}\) for \(j = 1, \dots, r\).
Then \(\{ E_\alpha, F_\alpha \}_{\alpha \in \posrts} \cup \{ H_j \}_{j=1}^r\) is a \emph{Chevalley basis} of \(\fg\).

\subsection{Weyl group}
\label{sec:weyl-group}

We denote by \(W\)\index[notn]{W@\(W\)} the Weyl group \index[term]{Weyl group} of the root system \(\rtsys\) of \(\fg\).
For \(\alpha \in \rtsys\) we let \(s_\alpha\) denote the reflection of \(\fh^\ast\) in the hyperplane orthogonal to \(\alpha\), defined by
\[
s_\alpha (\phi) = \phi - (\phi,\alpha^\vee)\alpha
\]
\index[notn]{salpha@\(s_\alpha,s_i\)}
for \(\phi \in \fh^\ast\).
Then \(W\) is generated by the \(s_{\alpha_i}\) for \(\alpha_i \in \simprts\), and we denote \(s_i \eqdef s_{\alpha_i}\) for \(1 \leq i \leq r\).

We denote by \(\ell\) the length function \index[term]{Weyl group!length function} of \(W\) with respect to this generating set, and by \(\wz\)\index[notn]{wzero@\(\wz\)} the unique longest word of \(W\).
If \(w = s_{i_1} \dots s_{i_k}\) and \(\ell(w) = k\) then we say that \(s_{i_1} \dots s_{i_k}\) is a \emph{reduced expression}\index[term]{Weyl group!reduced expression in} for \(w\).

For any element \(w \in W\), the set
\begin{equation}
  \label{eq:phi-of-w-definition}
  \rtsys(w) \eqdef \posrts \cap w \negrts
\end{equation}
\index[notn]{Phiw@\(\rtsys(w)\)}
of positive roots of \(\fg\) made negative by the action of \(w^{-1}\) is of some interest.
Given a reduced expression for \(w\), this set can be constructed explicitly via the following result (see Section 1.7 of \cite{Hum90}, for example, although our \(\rtsys(w)\) is \(\Pi(w^{-1})\) in Humphreys' notation):

\begin{lem}
  \label{lem:enumeration-of-phi-of-w}
  If \(w = s_{i_1} \dots s_{i_n}\) is a reduced expression for \(w\), then the set \(\rtsys(w)\) consists of all elements of the form \(s_{i_1} \dots s_{i_{k-1}}(\alpha_{i_k})\) for \(1 \leq k \leq n\).
\end{lem}

\subsection{Compact real form}
\label{sec:compact-real-form-of-g}

The real Lie subalgebra \(\fg_0\)\index[notn]{gzero@\(\fg_0\)} of \(\fg\) spanned by \(\{ iH_\alpha, E_\alpha - F_\alpha, i(E_\alpha + F_\alpha) \}_{\alpha\in \posrts}\) is a \emph{compact real form}\index[term]{compact real form} of \(\fg\).
The associated \emph{Cartan involution}\index[term]{Cartan involution} is the involutive real Lie algebra automorphism \(\tau\)\index[notn]{tau@\(\tau\)} of \(\fg\) given by conjugation with respect to \(\fg_0\), namely \(\tau(X + iY) = X - iY\) for all \(X,Y \in \fg_0\).
More precisely, \(\tau\) is the unique complex-antilinear map such that 
 \(\tau(E_\alpha) = -F_\alpha\), \(\tau(F_\alpha) = -E_\alpha\), and \(\tau(H) = -H\) for all \(\alpha \in \posrts\) and \(H \in \fh\).

\subsection{Universal enveloping algebra}
\label{sec:universal-enveloping-algebra}

We denote by \(U(\fg)\)\index[notn]{Ug@\(U(\fg)\)} the universal enveloping algebra of \(\fg\).
It is a Hopf algebra with coproduct \(\Delta(X) = X \otimes 1 + 1 \otimes X\), counit \(\epsilon(X) = 0\), and antipode \(S(X) = -X\) for \(X \in \fg\).
We also consider \(U(\fg)\) as a Hopf \(\ast\)-algebra such that \(\ast \circ S\) extends the Cartan involution \(\tau\) of \(\fg\).
More precisely, we have \(E_\alpha^\ast = F_\alpha\), \(F_\alpha^\ast = E_\alpha\), and \(H^\ast = H\) for all \(\alpha \in \posrts\) and \(H \in \fh\).

\subsection{Weight lattice}
\label{sec:weight-lattice}

The \emph{integral weight lattice} of \(\fg\) is the set
\[
\cP = \{ \lambda \in \fh^\ast \mid (\lambda,\alpha_j^\vee) \in \bbZ \text{ for } 1 \leq j \leq r \},
\]
and the cone of \emph{dominant integral weights} is defined by
\[
\pplus = \{ \lambda \in \cP \mid (\lambda,\alpha_j^\vee) \geq 0 \text{ for } 1 \leq j \leq r \}.
\]
\index[notn]{P@\(\cP,\pplus\)}
The \emph{fundamental weights} \(\{\omega_1, \dots, \omega_r\}\)\index[notn]{omega@\(\omega_j\)} are defined to be the dual basis to \(\{ \alpha_i^\vee \}\), so that \((\alpha_i^\vee, \omega_j) = \delta_{ij}\).
Hence
\[
\cP = \bigoplus_{j=1}^{r} \bbZ \omega_j \quad \text{and} \quad \pplus = \bigoplus_{j=1}^{r} \zp \omega_j.
\]
There is a natural partial order\index[term]{ordering of weight lattice} on \(\cP\) given by \(\mu \peq \nu\)\index[notn]{\prec@\(\peq,\prec\)} if \(\nu - \mu \in \qplus\).
We write \(\mu \prec \nu\) if \(\mu \peq \nu\) and \(\mu \neq \nu\).
We denote by \(\rho \in \pplus\)\index[notn]{rho@\(\rho\)} the half-sum of positive roots of \(\fg\); equivalently, \(\rho = \sum_{i=1}^r \omega_i\).

\subsection{Irreducible representations}
\label{sec:irreducible-representations-of-g}

For \(\lambda \in \pplus\) we denote by \(V_\lambda\)\index[notn]{Vlambda@\(V_\lambda\)} the irreducible finite-dimensional representation of \(\fg\) (and of \(U(\fg)\)) of highest weight \(\lambda\).
There is a positive-definite Hermitian inner product \(\langle \cdot,\cdot \rangle_\lambda\) on \(V(\lambda)\), antilinear in the first variable, which is compatible with the \(\ast\)-structure of \(U(\fg)\) in the sense that
\[
\langle a v, w \rangle_\lambda = \langle v, a^\ast w \rangle_\lambda
\]
for \(a \in U(\fg)\) and \(v,w \in V(\lambda)\); such an inner product is unique up to a positive scalar factor.

Every finite-dimensional representation of \(\fg\) decomposes as a direct sum of irreducible representations \(V(\lambda)\); see \cite{Kna02}*{Theorem 5.29}.
We denote by \(\cO_1\)\index[notn]{Oone@\(\cO_1\)} the category whose objects are finite-dimensional representations of \(\fg\) and whose morphisms are all morphisms of representations.
(We use the subscript ``1'' here because we will introduce in \cref{sec:representations-of-quea} a category \(\oq\) for all \(q \neq 1\).)

\section{Some simple exercises}
\label{sec:simple-lie-exercises}

We require some results that are stated as exercises in \cite{Kna02}.

\subsection{An exercise on root systems}
\label{sec:exercise-on-root-systems}

The following result appears as Problem 7 in Chapter II of \cite{Kna02}.
We use this result in the proof of 
\cref{prop:cominuscule-parabolic-conditions}, so for convenience we give the statement and proof here.

\begin{lem}
  \label{lem:root-system-lemma}
  Let \(\rtsys\) be a root system, and fix a system of simple roots \(\simprts \subseteq \rtsys\).
  Show that any positive root \(\alpha\) can be written in the form
  \begin{equation}
    \label{eq:root-system-lemma}
    \alpha = \alpha_{i_1} + \alpha_{i_2} + \dots + \alpha_{i_k}, 
  \end{equation}
  with each \(\alpha_{i_j}\) in \(\simprts\) and with each partial sum from the left equal to a positive root.
\end{lem}

\begin{proof}
  We proceed by induction on \(\hgt(\alpha)\).
  If \(\hgt(\alpha)=1\) then \(\alpha = \alpha_j\) for some \(j\), and the statement about partial sums is trivial.
  
  For the general case, we prove the following statement: if \(\alpha\) is a positive root with \(\hgt(\alpha) > 1\), then there is a simple root \(\alpha_m\) such that \(\alpha - \alpha_m\) is also a positive root.
  Then applying the induction hypothesis to \(\alpha - \alpha_m\) will give the desired result.

  Since \(\alpha\) is a positive root, we can write \(\alpha = \sum_{j=1}^r c_j \alpha_j\), where all \(c_j\) are nonnegative integers.
  Since the inner product associated to the root system is positive definite, we have
  \[
  0 < (\alpha,\alpha) = \sum_{j=1}^r c_j (\alpha,\alpha_j),
  \]
  so there must be some \(m\) with \((\alpha,\alpha_m) > 0\) and \(c_m > 0\).
  Then \(d = (\alpha,\alpha_m\spcheck)\) is a positive integer.

  Now we apply the root reflection \(s_m = s_{\alpha_m}\) to the positive root \(\alpha\).
  We have assumed that \(\hgt(\alpha) > 1\), so \(\alpha \neq \alpha_m\), and hence \(s_m(\alpha)\) is also a positive root (see e.g.~ Lemma 2.61 of \cite{Kna02}).
  
  By definition we have \(  s_m(\alpha) = \alpha - (\alpha,\alpha_m\spcheck)\alpha_m = \alpha - d \alpha_m\), and we have noted above that \(d > 0\).
  But now consider the \(\alpha_m\) root string through \(\alpha\).
  Since it contains \(\alpha - d \alpha_m\) and \(\alpha\), it must contain \(\alpha - \alpha_m\) by Proposition 2.48(g) in Chapter II of \cite{Kna02} (in fancier terms we are using the fact that \(\rtsys\) is a \emph{saturated} set of roots in the sense of Definition 13.4 of \cite{Hum78}).
  
  Since \(s_m(\alpha) = \alpha - d \alpha_m\) is a positive root, then \(\alpha - \alpha_m\) is also a positive root, which is what we wanted to prove.
  This completes the proof of the lemma.
\end{proof}

\subsection{Multiplicity-one decompositions}
\label{sec:multiplicity-one-decompositions}

In this section we prove a lemma which describes a particular set of circumstances under which a tensor product \(V \otimes W\) of finite-dimensional \(\fg\)-modules decomposes into irreducibles with multiplicity one.
For completeness we make the following
\begin{dfn}
  \label{def:multiplicity}
  Let \(U\) and \(V\) be finite-dimensional representations of a complex semisimple Lie algebra \(\fg\), and suppose that \(U\) is irreducible. 
  The \emph{multiplicity of \(U\) in \(V\)}\index[term]{multiplicity} is the number
  \[
  m_U(V) = \dim \Hom_\fg (U,V).
  \]
  If \(U = V_\lambda\) is the irreducible representation of \(\fg\) with highest weight \(\lambda\), we denote \(m_U(V) = m_\lambda(V)\).
  \index[notn]{multiplicity@\(m_U(V),m_\lambda(V)\)}
  A tensor product \(V \otimes W\) of finite-dimensional representations is said to \emph{decompose with multiplicity one} if \(m_\lambda(V \otimes W) \in \{0,1\}\) for all \(\lambda \in \pplus\).
\end{dfn}

Note that the multiplicity of \(V_\lambda\) in \(V\) is exactly the dimension of the space of highest weight vectors of weight \(\lambda\) in \(V\).

The following lemma is well-known to experts, and it appears as Problem 15 in Chapter IX of \cite{Kna02}, but I have not seen a proof written anywhere.

\begin{lem}
  \label{lem:mult-one-decomp}
  Suppose that \(V\) and \(W\) are finite-dimensional representations of \(\fg\) such that:
  \begin{enumerate}[(i)]
  \item \(W\) is irreducible, and
  \item all weight spaces of \(V\) are one-dimensional.
  \end{enumerate}
  Then \(V \otimes W\) decomposes with multiplicity one.
\end{lem}

Before proving the lemma, we recall a couple of basic facts about Lie algebra actions on \(\Hom\)-spaces.
Suppose that \(V\) and \(W\) are \emph{any} finite-dimensional representations of \(\fg\).
The action of \(\fg\) on \(\Hom(W,V)\) is defined by
\[
(X \rhd T)w = X(Tw) - T(Xw)
\]
for \(X \in \fg\), \(T \in \Hom(W,V)\), and \(w \in W\).
We now examine what it means for \(T \in \Hom(W,V)\) to be a highest weight vector.

\begin{lem}
  \label{lem:weight-vects-in-hom}
  Let \(V,W\) be any finite-dimensional representations of \(\fg\) and let \(T \in \Hom(W,V)\).  Then:
  \begin{enumerate}[(a)]
  \item \(T\) is a weight vector of weight \(\mu \in \cP\) if and only if for every weight vector \(w \in W\) of weight \(\gamma \in \cP\), \(Tw\) has weight \(\mu + \gamma\).
  \item Suppose that \(T\) is a weight vector.  Then \(T\) is a highest weight vector if and only if \(T\) commutes with all \(E_j\)'s, i.e.\ \(T(E_jw) = E_j(Tw)\) for all \(w \in W\).
  \item If \(T\) is a highest weight vector and \(W\) is irreducible, then \(T\) is completely determined by its action on a lowest weight vector in \(W\).
  \end{enumerate}
\end{lem}

\begin{proof}
  \begin{enumerate}[(a)]
  \item  By definition, \(T\) has weight \(\mu\) if \(H\rhd T = \mu(H) T\) for all \(H \in \fh\).
    It suffices to check this equality on a basis of weight vectors for \(W\).
    If \(w \in W\) has weight \(\gamma\), then we have
    \[
    (H \rhd T) w = H(Tw) - T(Hw) = H(Tw) - \gamma(H) Tw,
    \]
    so we see that \((H \rhd T)w = \mu(H) Tw\) if and only if
    \[
    H(Tw) - \gamma(H) Tw = \mu(H) Tw;
    \]
    rearranging this gives \(H(Tw) = (\mu(H) + \gamma(H)) Tw\), which means by definition that \(Tw\) has weight \(\mu + \gamma\).
  \item Now suppose that \(T\) is a weight vector.  
    By definition, \(T\) is a highest weight vector if and only if \(E_j \rhd T = 0\) for all \(j\).
    This means that for any \(w \in W\) we have
    \[
    0 = (E_j \rhd T) w = E_j(Tw) - T(E_jw),
    \]
    i.e.~\(T\) is a highest weight vector if and only if \(T\) commutes with the action of all of the \(E_j\)'s on \(W\).
  \item If \(W\) is irreducible and \(w_0\) is a lowest weight vector, then \(U(\fn_+)w_0 = W\), where \(U(\fn_+)\) is the subalgebra of \(U(\fg)\) generated by the \(E_j\)'s.
    Since we have \(T(E_j)w = E_j(Tw)\) for all \(w \in W\), we see that once \(Tw_0\) is specified, by applying the \(E_j\)'s and using (b), the action of \(T\) on \(U(\fn_+)w_0\) is completely determined. 
    Since this submodule is all of \(W\), the proof is complete.
  \end{enumerate}
\end{proof}

Now we proceed with the proof of the exercise:
\begin{proof}[Proof of \cref{lem:mult-one-decomp}]
  First, note that \(W \cong (W^\ast)^\ast\) as representations of \(\fg\).  This is straightforward to check directly, but a higher-level explanation is that the antipode of \(U(\fg)\) squares to the identity.
  Thus we can write
  \begin{equation}
    \label{eq:mult-one-decomp}
    V \otimes W \cong V \otimes (W^\ast)^\ast \cong \Hom(W^\ast,V).
  \end{equation}
  Now, \(W\) is irreducible if and only if \(W^\ast\) is irreducible, so we may as well swap the roles of \(W^\ast\) and \(W\) in \eqref{eq:mult-one-decomp}.
  It thus suffices to prove that \(\Hom(W,V)\) decomposes with multiplicity one under the hypotheses of the lemma.
  
  Let \(\mu \in \pplus\).
  Our goal is to show that the space of highest weight vectors of weight \(\mu\) in \(\Hom(W,V)\) has dimension at most one.
  Let \(T,T' \in \Hom(W,V)\) be such highest weight vectors, and let \(w_0 \in W\) be a lowest weight vector, with weight \(\gamma \in \cP\).

  By part (a) of \cref{lem:weight-vects-in-hom}, we know that both \(Tw_0\) and \(T'w_0\) have weight \(\mu + \gamma\) in \(V\).
  Since the weight spaces of \(V\) are one-dimensional, this means that these two vectors are dependent, so
  \[
  c Tw_0 + c'T'w_0 = 0
  \]
  for some constants \(c,c' \in \bbC\).
  But then \((cT + c'T')w_0 = 0\).
  Since \(cT + c'T'\) is also a highest weight vector of weight \(\mu\) in \(\Hom(W,V)\), part (c) of \cref{lem:weight-vects-in-hom} shows that \(cT + c'T' = 0\).

  Thus, the space of highest weight vectors of weight \(\mu\) in \(\Hom(W,V)\) has dimension at most one for any \(\mu \in \pplus\).
  We conclude that \(\Hom(W,V)\) decomposes with multiplicity one, as desired.
\end{proof}

\section{Parabolic subalgebras}
\label{sec:lie-background-parabolics}

By definition, a \emph{parabolic subalgebra} \index[term]{parabolic subalgebra} of \(\fg\) is any subalgebra \(\fp\) containing a Borel subalgebra.
Here we mainly just describe the constructions and results that we need, referring the reader to Chapter V, Section 7 of \cite{Kna02} and Chapter 1, Section 10 of \cite{Hum90} for proofs.
However, in \cref{sec:cominuscule-parabolics} we do include a proof of the equivalence of several different characterizations of so-called ``cominuscule-type'' parabolics.
While these characterizations are well-known to experts, it seemed worthwhile to collect the material in one place.
In \cref{sec:structure-of-cominuscule-parabolics} we describe some features that arise only in the cominuscule situation, and in \cref{sec:classification-cominuscule-parabolics} we give the classification of cominuscule parabolics in terms of Dynkin diagrams.

\subsection{Standard parabolic subalgebras}
\label{sec:standard-parabolics}

Since all Borel subalgebras of \(\fg\) are conjugate by the action of \(\mathrm{Int}(\fg)\) (the connected Lie subgroup of \(GL(\fg)\) whose Lie algebra is \(\ad(\fg)\)), it is enough to consider parabolic subalgebras \(\fp\) that contain the fixed Borel subalgebra \(\fb_+\) defined in \cref{sec:root-system}.
These are called \emph{standard parabolic subalgebras}\index[term]{standard parabolic subalgebra}.
According to Proposition 5.90 of \cite{Kna02}, all standard parabolics arise via the following construction.
Given a subset \(\cS \subseteq \simprts(\fg)\)\index[notn]{S@\(\cS\)} of simple roots, define two sets of roots by
\begin{gather}
  \rtsys(\fl) = \spn(\cS)\cap \rtsys(\fg), \quad \rtsys(\up) =
  \rtsys^+(\fg) \setminus \rtsys(\fl),\\
  \intertext{and set}
  \fl = \fh \oplus \bigoplus_{\alpha \in \rtsys(\fl)} \fg_{\alpha},
  \quad \fu_{\pm} = \bigoplus_{\alpha \in \rtsys(\up)} \fg_{\pm
    \alpha}, \quad \text{and} \quad \fp = \fl \oplus \up.
\end{gather}
\index[notn]{Phiparabolic@\(\rtsys(\fl),\posrts(\fl),\rtsys(\upm)\)}
\index[notn]{l@\(\fl\)}
\index[notn]{u@\(\upm,\fu\)}
\index[notn]{p@\(\fp\)}
For convenience we also denote \(\posrts(\fl) = \rtsys(\fl) \cap \posrts(\fg)\), \(\rtsys(\um) = - \rtsys(\up)\), and set \(\fu = \up \oplus \um\).
Then \(\fp\) is called the \emph{standard parabolic subalgebra} associated to \(\cS\), \(\fl\) is its \emph{Levi factor}\index[term]{Levi factor of a parabolic}, and \(\up\) is its \emph{nilradical}\index[term]{nilradical of a parabolic}.
We also refer to \(\rtsys(\up)\) as the set of \emph{radical roots}\index[term]{radical roots}.

\begin{rem}
  \label{rem:description-of-roots-in-l}
  Note that the roots in \(\rtsys(\fl)\) are exactly those that can be expressed as linear combinations of the simple roots in \(\cS\).
  The roots in \(\rtsys(\up)\), therefore, are all of those positive roots of \(\fg\) that involve any simple root \emph{not} in \(\cS\).
\end{rem}

\begin{warn}
  \label{warning:p-depends-on-S}
    To avoid complicating the notation, we have not adorned \(\fp,\fl,\upm\), etc.\ with any subscripts indicating the dependence on \(\cS\).
\end{warn}

\subsection{Basic properties of the Levi factor and nilradical}
\label{sec:basic-properties-of-levi-and-nilradical}

In the rest of this section, we give some basic facts about the Levi factor and nilradical of a standard parabolic subalgebra associated to a subset of the simple roots.
The results listed here are either given in Chapter 5, Section 7 of \cite{Kna02} or follow immediately from those results, so we omit the proofs.
The following result is essentially Corollary 5.94 of \cite{Kna02}:

\begin{prop}
  \label{prop:basic-parabolic-facts}
  With notation as above, the following hold:
  \begin{enumerate}[(a)]
  \item \(\fp\) is a subalgebra of \(\fg\);
  \item \(\fl\) is a subalgebra of \(\fp\) (and hence of \(\fg\) as well);
  \item \(\up\) is a nilpotent ideal of \(\fp\);
  \item \([\up, \um] \subseteq \fl\).
  \end{enumerate}
  Furthermore, the decomposition
  \begin{equation}
    \label{eq:g-decomp}
    \fg = \um \oplus \fl \oplus \up
  \end{equation}
  is a splitting of \(\fg\) as \(\fl\)-modules with respect to the adjoint action.
\end{prop}

Observe that \(\up\) being an ideal of \(\fp\) means that \([\fl, \up] \subseteq \up\), and similarly \([\fl, \um] \subseteq \um\), so \(\upm\) are both \(\fl\)-modules with respect to the adjoint action.
The decomposition \eqref{eq:g-decomp} then shows that \(\fg/\fp \cong \um\) as \(\fl\)-modules.
We have the following result which describes the behavior of \(\fl\) and \(\upm\) with respect to the Killing form:
\begin{lem}
  \label{lem:l-upm-killingform}
  With respect to the Killing form of \(\fg\), we have:
  \begin{enumerate}[(a)]
  \item \(\up\) and \(\um\) are isotropic;
  \item \(\fl = \fu^\perp\), where \(\fu = \um \oplus \up\);
  \item The pairing \(\um \times \up \to \bbC\) coming from the Killing form is nondegenerate, so that \(\up\) and \(\um\) are mutually dual as \(\fl\)-modules.
  \end{enumerate}
\end{lem}

Next we examine the Levi factor in more detail:
\begin{prop}
  \label{prop:levifactor}
  The Levi factor \(\fl\) of \(\fp\) is a reductive subalgebra of \(\fg\), i.e.~we have \(\fl = \cZ(\fl) \oplus [\fl,\fl]\).
  Moreover:
  \begin{enumerate}[(a)]
  \item The center of \(\fl\)\index[term]{center of Levi factor} is given by
    \[
    \cZ(\fl) = \bigcap_{\alpha \in \rtsys(\fl)} \ker \alpha = \spn \{ H_{\omega_j} \mid \alpha_j \in \simprts \setminus \cS \};
    \]
    \index[notn]{Zl@\(\cZ(\fl)\)}
  \item The semisimple part\index[term]{semisimple part of Levi factor} \(\fk := [\fl,\fl]\) of \(\fl\) is given by
    \[
    \fk = \fhk \oplus \bigoplus_{\alpha \in \rtsys(\fl)} \fg_\alpha,
    \]
    \index[notn]{k@\(\fk\)}
    where
    \[ 
    \fhk = \fh \cap \fl = \spn \{ H_j \mid \alpha_j \in \cS \}
    \]
    \index[notn]{hk@\(\fhk\)}
    is a Cartan subalgebra of \(\fk\).
  \item The Cartan subalgebra \(\fh\) of \(\fg\) decomposes as \(\fh = \fhk \oplus \cZ(\fl)\), and this decomposition is orthogonal with respect to the Killing form of \(\fg\).
  \end{enumerate}
\end{prop}

\begin{rem}
  \label{rem:cartan-of-k}
  Part (b) of the preceding proposition describes a distinguished Cartan subalgebra of the semisimple part of the Levi factor.
  Since the adjoint action of \(\fhk\) on \(\fk\) is just the restriction of the adjoint action of \(\fh\) on \(\fg\), the root system of \(\fk\) is exactly \(\rtsys(\fl)\).
  
  This observation allows us to easily understand restrictions of (finite-dimensional) representations of \(\fg\) to \(\fl\).
  Indeed, weights of \(\fg\) are linear functionals on \(\fh\), and we can just restrict them to \(\fhk\).
  This tells us how \(\fk\) acts on the restriction of a representation of \(\fg\).
  Then the decomposition of \(\fh\) from part (c) of the proposition tells us how the center of \(\fl\) acts in the restriction.
  Bearing this in mind, we view elements of \(\cP\) as weights of \(\fl\) whenever convenient.
  In particular, when restricting representations of \(\fg\) to \(\fl\), the weights are the same.

  The important thing to note is that the notions of dominant weight for \(\fl\) and for \(\fg\) differ.
  An integral weight \(\lambda \in \cP(\fg)\) is dominant for \(\fl\) if \((\lambda, \alpha^\vee) \in \zp\) for all \(\alpha \in \rtsys(\fl)\).
  In particular, if \(\alpha_s \in \simprts \setminus \cS\) is a simple root of \(\fg\) that does not lie in \(\rtsys(\fl)\), then \(-\alpha_s\) is dominant for \(\fl\).

  See also Theorem 5.104 and Proposition 5.105 of \cite{Kna02} for a more precise description of the relationship between representations of \(\fg\) and \(\fl\).
\end{rem}

\subsection{The Weyl group of a standard parabolic subalgebra}
\label{sec:weyl-group-of-standard-parabolic}

The subset \(\rtsys(\fl)\) of \(\rtsys(\fg)\) is a root system in its own right.
The Weyl group of this root system can be identified with the \emph{parabolic subgroup}\index[term]{parabolic subgroup} \(\Wl\)\index[notn]{Wl@\(\Wl\)} of \(W\) defined by
\begin{equation}
  \label{eq:weyl-group-of-levi-factor-definition}
  \Wl \eqdef \langle s_i \mid \alpha_i \in \cS \rangle \subseteq W.
\end{equation}
We define \(\wzl\)\index[notn]{wzerol@\(\wzl\)} to be the longest word of \(\Wl\), and we define the \emph{parabolic element}\index[term]{parabolic element} \(\wl\)\index[notn]{wl@\(\wl\)} of \(W\) by
\begin{equation}
  \label{eq:parabolic-element-definition}
  \wl \eqdef \wzl \wz.
\end{equation}
Finally, define 
\begin{equation}
\label{eq:parabolic-subgroup-complement-definition}
\Wul = \{ w \in W \mid  w^{-1}(\alpha) \in \posrts(\fg) \text{ for all } \alpha \in \cS\}.
\end{equation}
\index[notn]{Wupperl@\(\Wul\)}
Translating Proposition 1.10 of \cite{Hum90} into our notation, we have:

\begin{prop}
  \label{prop:humphreys-lemma-on-parabolic-subgroups}
  \begin{enumerate}[(a)]
  \item \(\rtsys(\fl)\) is a root system with associated reflection group \(\Wl\).
  \item The length function of \(W\), restricted to \(\Wl\), corresponds to the length function of \(\Wl\) with respect to the generating set \(\{ s_\alpha \mid \alpha \in \cS \}\).
  \item For any \(w \in W\), there is a unique element \(u \in \Wl\) and a unique element \(v \in \Wul\) such that \(w = uv\).  Their lengths satisfy \(\ell(w) = \ell(u) + \ell(v)\).  Moreover, \(u\) is the unique element of smallest length in the coset \(\Wl v\).
  \end{enumerate}
\end{prop}

The parabolic element satisfies the following properties:
\begin{prop}
  \label{prop:properties-of-parabolic-element-and-whatnot}
  \begin{enumerate}[(a)]
  \item The subsets \(\rtsys(\upm)\) are preserved (not necessarily fixed pointwise) by the action of \(\Wl\).
  \item We have \(\rtsys(\wl) = \rtsys(\up)\), where \(\rtsys(\wl)\) is as defined in \cref{sec:weyl-group}.
  \item The parabolic element \(\wl\) lies in \(\Wul\), and 
    \begin{equation}
      \label{eq:factorization-of-longest-word}
      \wz = \wzl \wl
    \end{equation}
  is the unique decomposition of \(\wz\) from \cref{prop:humphreys-lemma-on-parabolic-subgroups}(c).
  \end{enumerate}
\end{prop}

Now fix a reduced expression
\begin{equation}
  \label{eq:reduced-expression-for-longest-word}
  \wz = s_{i_1} \dots s_{i_M} s_{i_{M+1}} \dots s_{i_{M+N}}
\end{equation}
for the longest word that is compatible with \eqref{eq:factorization-of-longest-word} in the sense that \(\wzl = s_{i_1} \dots s_{i_M}\) and \(\wl = s_{i_{M+1}} \dots s_{i_{M+N}}\).
For \(1 \leq k \leq N\), define a root \(\xi_k\) of \(\fg\) by
\begin{equation}
  \label{eq:radical-root-recipe}
  \xi_k \eqdef s_{i_1} \dots s_{i_M} s_{i_{M+1}} \dots s_{i_{k-1}}(\alpha_{i_k}) = \wzl s_{i_{M+1}} \dots s_{i_{k-1}}(\alpha_{i_k}).
\end{equation}
\index[notn]{xik@\(\xi_k\)}
Then combining parts (a) and (b) of \cref{prop:properties-of-parabolic-element-and-whatnot} with \cref{lem:enumeration-of-phi-of-w}, we have

\begin{lem}
  \label{lem:radical-roots-recipe}
  The set of radical roots \(\rtsys(\up)\) is precisely \(\{ \xi_1, \dots, \xi_N \}\).
\end{lem}

\subsection{Cominuscule parabolics}
\label{sec:cominuscule-parabolics}

In later chapters we will be especially concerned with a certain class of parabolic subalgebras known as the \emph{cominuscule}\index[term]{cominuscule parabolic} ones.
For this section we assume that \(\fg\) is simple rather than just semisimple.
This means that the adjoint representation of \(\fg\) is irreducible, and hence it has a highest weight, which we denote by \(\theta \in \rtsys(\fg)\)\index[notn]{theta@\(\theta\)}.
By definition, the roots of \(\fg\) are the weights of the adjoint representation, and hence we refer to \(\theta\) as the \emph{highest root}\index[term]{highest root} of \(\fg\).
It is maximal in \(\rtsys(\fg)\) with respect to the order on \(\cP\) introduced in \cref{sec:weight-lattice}.

\begin{prop}
  \label{prop:cominuscule-parabolic-conditions}
  Assume that \(\fg\) is simple and that \(\fp \subseteq \fg\) is the standard parabolic subalgebra determined by a subset \(\cS \subseteq \simprts\) as in \cref{sec:standard-parabolics}.
  The following conditions are equivalent:
  \begin{enumerate}[(a)]
  \item \(\fg/\fp\) is a simple \(\fp\)-module;
  \item \(\um\) is a simple \(\fl\)-module;
  \item \(\um\) is an abelian Lie algebra;
  \item \(\up\) is a simple \(\fl\)-module;
  \item \(\up\) is an abelian Lie algebra;
  \item \(\fp\) is maximal, i.e.\ \(\cS = \simprts \setminus \{\alpha_s\}\) for some \(1 \le s \le r\)\index[notn]{s@\(s\)}, and moreover \(\alpha_s\) has coefficient 1 in the highest root \(\theta\) of \(\fg\);
  \item \([\fu,\fu] \subseteq \fl\), where \(\fu = \um \oplus \up\);
  \item \((\fg,\fl)\) is a symmetric pair, i.e.\ there is an involutive Lie algebra automorphism \(\sigma\) of \(\fg\) such that \(\fl = \fg^\sigma\).
  \end{enumerate}
\end{prop}

\begin{proof}
  \((a) \iff (b)\)
  Write \(\fg = \um \oplus \fp\).  Then 
  \[
  [\up, \fg] = [\up, \um] \oplus [\up, \fp] \subseteq \fl \oplus \up = \fp
  \]
  by \cref{prop:basic-parabolic-facts}, so \(\up\) acts trivially on \(\fg / \fp\).
  This means that \(\fg/\fp\) is simple as a \(\fp\)-module if and only if it is simple as an \(\fl\)-module.
  Then the decomposition \eqref{eq:g-decomp} shows that \(\fg/\fp \cong \um\) as \(\fl\)-modules, proving the equivalence of (a) and (b).
  \vspace{0.05in}

  \noindent \((b) \implies (c)\)
  Since \([\fl,\um] \subseteq \um\), the Jacobi identity shows that \([\um,\um]\) is an \(\fl\)-submodule of \(\um\).
  Since \(\um\) is nilpotent, \([\um,\um]\) cannot be all of \(\um\).
  Hence if \(\um\) is simple, we must have \([\um,\um] = 0\), so \(\um\) is abelian.
  \vspace{0.05in}

  \noindent \((d) \implies (e)\) is identical to \((b) \implies (c)\).
  \vspace{0.05in}

  \noindent \((b) \iff (d)\) 
  According to \cref{lem:l-upm-killingform}(c), we have \(\um \cong \fu_+^\ast\).  Thus \(\um\) is simple if and only if \(\up\) is simple.
  \vspace{0.05in}

  \noindent \((c) \iff (e)\) 
  The statement that \(\up\) is abelian is equivalent to the statement that for any two roots \(\xi,\xi' \in \rtsys(\up)\), \(\xi + \xi'\) is not a root of \(\fg\), and the analogous statement holds for \(\um\).
  Since \(\rtsys(\um) = - \rtsys(\up)\), we see that \(\um\) is abelian if and only if \(\up\) is abelian.
  \vspace{0.05in}

  \noindent \((c) + (e) \iff (g)\)
  We have
  \[
  [\fu,\fu] = [\um,\um] \oplus [\um, \up] \oplus [\up, \up] \subseteq [\um,\um] \oplus \fl \oplus [\up, \up],
  \]
  where we have used \cref{prop:basic-parabolic-facts}(d).
  As \(\upm\) are Lie subalgebras of \(\fg\) that intersect \(\fl\) trivially, we see that \([\fu,\fu] \subseteq \fl\) if and only if both \(\um\) and \(\up\) are abelian.
  \vspace{0.05in}

  \noindent \((g) \implies (h)\)
  Define a linear map \(\sigma : \fg \to \fg\) by
  \[
  \sigma(X) =
  \begin{cases}
    X & \text{ if \(X \in \fl\),}\\
    -X & \text{ if \(X \in \fu\).} 
  \end{cases}
  \]
  It is clear that \(\sigma\) is involutive, and it is straightforward to check that \(\sigma\) is a Lie algebra homomorphism if and only if \([\fu,\fu] \subseteq \fl\).
  \vspace{0.05in}

  \noindent \((h) \implies (g)\)
  Suppose that \(\sigma\) is an involutive Lie algebra automorphism of \(\fg\) such that \(\fl = \fg^\sigma = \{ X \in \fg \mid \sigma(X) = X \}\).
  Note that \(\sigma(H) = H\) for all \(H \in \fh\) since \(\fh \subseteq \fl\).
  We claim that \(\sigma\) must then preserve the weight spaces of \(\fg\).
  Indeed, if \(X \in \fg_{\alpha}\) for \(\alpha \in \rtsys\), then for \(H \in \fh\) we have
  \[
  [H, \sigma(X)] = [\sigma(H), \sigma(X)] = \sigma([H,X]) = \alpha(H) \sigma(X),
  \]
  so that \(\sigma(X) \in \fg_\alpha\) as well.
  This means that \(\sigma(\fu) = \fu\), since \(\fu = \um \oplus \up\) is a sum of weight spaces by definition.
  Since \(\sigma\) is involutive and since \(\fl\) is its entire fixed-point set, then \(\sigma\) must act as the scalar \(-1\) on \(\fu\).
  Finally, for \(X,Y \in \fu\) we have
  \[
  \sigma([X,Y]) = [\sigma(X),\sigma(Y)] = [-X,-Y] = [X,Y].
  \]
  Thus \([X,Y]\) is a fixed point of \(\sigma\), so \([X,Y] \in \fl\).
  \vspace{0.05in}

  \noindent \((f) \implies (b)\)
  If \(\cS = \simprts \setminus \{\alpha_s\}\), then by definition we have \(F_s \in \um\).
  For any \(1 \le j \le r\) with \(j \neq s\), we have \(E_j \in \fl\), and
  \([E_j,F_s] = 0\).
  This means that \(F_s\) is a highest weight vector in \(\um\) for the action of \(\fl\).
  Since every root in \(\rtsys(\um)\) contains \(\alpha_s\) with coefficient exactly \(-1\), it follows that \(F_s\) generates \(\um\) as an \(\fl\)-module, and hence \(\um\) is simple.
  \vspace{0.05in}

  \noindent \((b) \implies (f)\)
  We prove this via the contrapositive.
  If \(\fp\) is not maximal, then there are two distinct simple roots \(\alpha_s, \alpha_t \in \simprts \setminus \cS\).
  Thus \(F_s, F_t \in \um\).
  For the Chevalley generates \(E_i\) corresponding to simple roots \(\alpha_i \in \rtsys(\fl)\), we have \([E_i,F_s] = 0\), so that \(F_s\) is a highest weight vector in \(\um\) for the action of \(\fl\).
  However, \(F_t \notin \mathrm{ad}(\fl)F_s\), so that \(\um\) is not simple.
  This implies that \(\fp\) must be maximal.

  Now we suppose that \(\fp\) is maximal, so \(\cS = \simprts \setminus \{ \alpha_s\}\) for some \(s\).  We want to show that \(\alpha_s\) has coefficient exactly \(1\) in the highest root of \(\fg\).
  If not, then define
  \[
  \Omega = \{ \beta \in \rtsys(\up) \mid \alpha_s \text{ has coefficient 1 in } \beta\},
  \]
  and set
  \[
  V = \bigoplus_{\beta \in \Omega} \fg_{-\beta} \subseteq \um.
  \]
  Our assumption implies that \(V\) is a proper nonzero subspace of \(\um\).
  But \(V\) is invariant under the adjoint action of \(\fl\), which means that \(\um\) cannot be simple.
  \vspace{0.05in}

  \noindent \((e) \implies (f)\)
  We will show that if (f) does not hold, then \(\up\) cannot be abelian.
  We claim that in this case there are two roots \(\alpha,\beta \in \rtsys(\up)\) such that \(\alpha + \beta\) is a root of \(\fg\), and hence \(\alpha + \beta \in \rtsys(\up)\).
  But then \(E_\alpha,E_\beta \in \up \) and \([E_\alpha,E_\beta] \neq 0\), so \(\up\) is not abelian.

  Now we verify the claim by using \cref{lem:root-system-lemma}.
  Let \(\theta\) be the highest root of \(\fg\).
  The lemma tells us that we can write 
  \begin{equation}
    \label{eq:highest-root-decomp}
    \theta = \alpha_{j_1} + \dots + \alpha_{j_k}
  \end{equation}
  such that each partial sum from the left is also a positive root of \(\fg\).

  Assuming that (f) fails, there are some indices \(s,t\) such that \(\alpha_s,\alpha_t \in \simprts \setminus \cS\) and such that \(\theta - \alpha_s - \alpha_t\) is a nonnegative integral combination of simple roots.  
  (It may be the case that \(s = t\).)
  Thus in the decomposition \eqref{eq:highest-root-decomp} there are indices \(m,n\) such that \(j_m = s\) and \(j_n = t\).
  Without loss of generality we may assume that \(m < n\).

  Now define \(\beta = \alpha_{j_1} + \dots + \alpha_{j_{n-1}}\).
  This is a positive root of \(\fg\) since it is one of the partial sums described in the lemma.
  But since \(m < n\) and \(j_m = s\), this means that \(\beta\) involves \(\alpha_s\) with nonzero coefficient, and hence \(\beta \in \rtsys(\up)\) according to \cref{rem:description-of-roots-in-l}.

  We also have \(\alpha_t \in \rtsys(\up)\).
  Since \(j_n = t\), then
  \[
  \beta + \alpha_t = (\alpha_{j_1} + \dots + \alpha_{j_{n-1}}) + \alpha_{j_n}
  \]
  is also a positive root of \(\fg\) since it is one of the partial sums from \cref{lem:root-system-lemma}. 
  By similar reasoning as above we see that \(\beta + \alpha_t \in \rtsys(\up)\) also.
  This completes the proof of the claim, and the proposition.
\end{proof}

\begin{dfn}
  \label{dfn:cominuscule-parabolic}
    If the equivalent conditions of \cref{prop:cominuscule-parabolic-conditions} are satisfied, then we say that \(\fp\) is of \emph{cominuscule type}, or simply that \(\fp\) is \emph{cominuscule}.
\end{dfn}

\begin{rem}
  \label{rem:cominuscule-definition}
  Using the classification of finite-dimensional simple Lie algebras over \(\bbC\), condition (f) in \cref{prop:cominuscule-parabolic-conditions} is the easiest way to determine which parabolic subalgebras are of cominuscule type.
  A table of highest roots for all simple \(\fg\) can be found in the exercises for Chapter 12 in \cite{Hum78} or \cite{Hel78}*{Ch.~X, Thm.~3.28}.

  Condition (f) is well-studied.
  The corresponding fundamental weight \(\omega_s\) is called \emph{cominuscule}.
  There is a similar notion of \emph{minuscule weight}, which overlaps with the notion of cominuscule weight in all simple types except for \(B_n\) and \(C_n\).
  These notions have consequences for the corresponding generalized flag manifolds \(G/P\); see Chapter 9 of \cite{BilLak00} for more information.

  The equivalence of conditions (d) and (f) (or rather the conditions for the corresponding Lie groups) is discussed in \cite{RicRohSte92}*{Lemma 2.2}.
\end{rem}

\subsection{Structure of cominuscule parabolics}
\label{sec:structure-of-cominuscule-parabolics}

In this section we assume that \(\fp \subseteq \fg\) is a parabolic subalgebra of cominuscule type, and we list some simple consequences of this situation for the representations \(\upm\) of \(\fl\).

\begin{lem}
  \label{lem:roots-of-uplus-and-l}
  Let \(\xi \in \rtsys(\fg)\) and write \(\xi = \sum_{j=1}^r n_j \alpha_j\).
  Then:
  \begin{enumerate}[(a)]
  \item \(\xi \in \rtsys(\up)\) if and only if \(n_s = 1\), i.e.\ if and only if \(\xi = \alpha_s + \beta\) with \(\beta \in \qplus\) and \((\beta,\os) = 0\).
  \item \(\xi \in \rtsys(\um)\) if and only if \(n_s = -1\), i.e.\ if and only if \(\xi = - \alpha_s - \beta\) with \(\beta \in \qplus\) and \((\beta,\os) = 0\).
  \item \(\xi \in \rtsys(\fl)\) if and only if \(n_s = 0\).
  \end{enumerate}
  Equivalently,
  \begin{equation}
    \label{eq:u-and-l-and-omegas-inner-products}
  (\xi, \os) =
  \begin{cases}
    d_s & \text{if and only if } \xi \in \rtsys(\up)\\
    0 & \text{if and only if } \xi \in \rtsys(\fl)\\
    -d_s & \text{if and only if } \xi \in \rtsys(\um).\\
  \end{cases}
\end{equation}
\end{lem}

\begin{proof}
  This follows immediately from the definitions of \(\rtsys(\fl)\) and \(\rtsys(\upm)\) in \cref{sec:standard-parabolics} together with \cref{prop:cominuscule-parabolic-conditions}(f).
\end{proof}

\begin{rem}
  \label{rem:roots-of-uplus-and-l}
  We emphasize that the element \(\beta \in \qplus\) in parts (a) and (b) of  \cref{lem:roots-of-uplus-and-l} is not necessarily in \(\rtsys(\fl)\) or even in the root system \(\rtsys\) of \(\fg\).
  For example, take \(\fg = \fsl_4\) with simple roots \(\alpha_1,\alpha_2,\alpha_3\) as usual, and take \(\cS = \{ \alpha_1,\alpha_3 \}\), so \(s = 2\).
  If \(\xi = \alpha_1 + \alpha_2 + \alpha_3 \in \rtsys(\up)\), then 
\(\beta = \alpha_1 + \alpha_3\) is not a root of \(\fsl_4\).
\end{rem}

Now we can say more about the structure of \(\um\) as an \(\fl\)-module:
\begin{prop}
  \label{prop:uminus-as-l-module}
  \begin{enumerate}[(a)]
  \item The weights of \(\um\) are precisely the roots in \(\rtsys(\um)\).
  \item The highest weight of \(\um\) is \(-\alpha_s\) (note that \(-\alpha_s\) is dominant and integral for \(\fl\), see \cref{rem:cartan-of-k}), and \(F_s\) is a highest weight vector.
  \item The central element \(H_{\os}\) of \(\fl\) acts as the scalar \(-d_s\) in \(\um\).
  \item The weight spaces of \(\um\) are one-dimensional.
  \end{enumerate}
\end{prop}

\begin{proof}
  Part (a) is true because the action of \(\fl\) on \(\um\) is just the adjoint action of \(\fg\), and roots are precisely the weights of the adjoint action.
  
  For part (b), note that \(F_s \in \um\), the weight of \(F_s\) is \(-\alpha_s\), and that \([E_j,F_s]=0\) for all \(E_j \in \fl\) (since \(E_s \notin \fl\)).
  Thus \(F_s\) is a highest weight vector.
  
  As \(\um\) is irreducible, Schur's Lemma implies that \(H_{\os}\) acts as a scalar, which we can then compute by acting on the highest weight vector:
  \[
  [H_{\os}, F_s] = -\alpha_s(H_{\os})F_s = -(\alpha_s, \os)F_s = - d_s F_s,
  \]
  which establishes (c).
  
  Finally, the weight spaces of \(\um\) are one-dimensional because they are precisely the root spaces of \(\fg\) corresponding to roots in \(\rtsys(\um)\), and root spaces are one-dimensional.
\end{proof}

The fact that the weight spaces of \(\um\) (and hence of \(\up\)) are one-dimensional has the following consequence:
\begin{prop}
  \label{prop:multiplicity-one-decomps}
  The four tensor products of \(\fl\)-modules \(\upm \otimes \upm, \upm \otimes \fu_\mp\) decompose into simple submodules with multiplicity one.
\end{prop}

\begin{proof}
  Since we have assumed that \(\fp\) is cominuscule, both \(\upm\) are simple modules by \cref{prop:cominuscule-parabolic-conditions}.
  By \cref{prop:uminus-as-l-module}(d) all of the weight spaces of \(\um\) are one-dimensional, and the same is true of \(\up\) because \(\up \cong (\um)^\ast\).
  Then \cref{lem:mult-one-decomp} implies that the four tensor products under consideration all decompose with multiplicity one.
\end{proof}

\subsection{Classification of cominuscule parabolics}
\label{sec:classification-cominuscule-parabolics}

In this section we give the classification of the cominuscule parabolic subalgebras of simple Lie algebras.
Condition (f) of \cref{prop:cominuscule-parabolic-conditions} gives an easy way to do this using the preexisting classification of root systems of simple Lie algebras.
We use the numbering of simple roots given in the table in Section 11.4 of \cite{Hum78}.
The list of highest roots appears as Table 2 in Section 12.2 of \cite{Hum78}; for convenience we reproduce it here in \cref{tab:dynkin-diagrams}, along with the Dynkin diagrams.

\begin{table}
  \centering
  \begin{tabular}{|c|c|c|}
    \hline
    Root system & Dynkin diagram & Highest root \\
    \hline
    \(\bA_N\) & 
    \(\DynkinDiagram{
      \circ  \Edge{r} \NodeNumberBelow{1} \NodeNumberAbove{\phantom{A}} & \circ \Edge{r} \NodeNumberBelow{2} & \,\,\cdots\,\, \Edge{r} & \circ \Edge{r} \NodeNumberBelow{{N-1}} & \circ \NodeNumberBelow{N}}\)
    & \(\alpha_1 + \dots + \alpha_N\) \\[3mm]
    \hline
    \(\bB_N\) & 
    \(\DynkinDiagram{
      \circ \Edge{r} \NodeNumberBelow{1} \NodeNumberAbove{\phantom{A}} & \circ \Edge{r} \NodeNumberBelow{2}& 
      \,\,\cdots\,\, \Edge{r} & \circ \DirectedEdge[2]{r} \NodeNumberBelow{{N-1 }}& 
      \circ \NodeNumberBelow{N}}\) 
    & \(\alpha_1 + 2(\alpha_2 + \dots + \alpha_N)\) \\[3mm]
    \hline
    \(\bC_N\) &
    \(\DynkinDiagram{
      \circ \Edge{r} \NodeNumberBelow{1} \NodeNumberAbove{\phantom{A}} & \circ \Edge{r} \NodeNumberBelow{2} & \,\,\cdots\,\, \Edge{r} & \circ \NodeNumberBelow{{N-1}} & \circ \DirectedEdge[2]{l} \NodeNumberBelow{N}}\)
    & \(2(\alpha_1 + \dots + \alpha_{N-1}) + \alpha_N\) \\[3mm]
    \hline
    \(\DynkinDiagram{ \\ \bD_N \\}\) &
    \(\DynkinDiagram{
      & & & & \circ  & \NodeNumberLeft{N-1} \NodeNumberAbove{\phantom{A}}\\
      \circ \Edge{r} \NodeNumberBelow{1}  & \circ \Edge{r} \NodeNumberBelow{2} & 
      \,\,\cdots\,\, \Edge{r} & \circ \Edge{ur} \Edge{dr}  & \NodeNumberLeft{N-2}&\\
      &&&  &\circ \NodeNumberBelow{\phantom{A}} \NodeNumberRight{N} & 
    }\) 
    & \(\DynkinDiagram{  \alpha_1 + 2(\alpha_2 + \dots + \alpha_{N-2}) \\ {\phantom{=}} + \alpha_{N-1} + \alpha_N \\ } \) \\[3mm]
    \hline
    \(\DynkinDiagram{ \\ \bE_6 \\}\) &
    \(\DynkinDiagram{ \NodeNumberAbove{\phantom{A}} & & \circ \NodeNumberLeft{2} & & \\ \circ \NodeNumberBelow{1} \Edge{r} & \NodeNumberBelow{3} \circ \Edge{r} & \circ \NodeNumberBelow{4} \Edge{r} \Edge{u}& \circ \NodeNumberBelow {5} \Edge{r} & \circ \NodeNumberBelow{6} \\  & & & & } \)
    & \(\DynkinDiagram{ \\ \alpha_1 + 2 \alpha_2 + 2 \alpha_3 + 3 \alpha_4 + 2 \alpha_5 + \alpha_6 \\ }\) \\
    \hline
    \(\DynkinDiagram{ \\ \bE_7 \\}\) &
    \(\DynkinDiagram{ \NodeNumberAbove{\phantom{A}} & & \circ \NodeNumberLeft{2} & & & \\ \circ \NodeNumberBelow{1} \Edge{r} & \NodeNumberBelow{3} \circ \Edge{r} & \circ \NodeNumberBelow{4} \Edge{r} \Edge{u}& \circ \NodeNumberBelow {5} \Edge{r} & \circ \NodeNumberBelow{6} \Edge{r} & \circ \NodeNumberBelow{7} \\  & & & & &} \)
    & \(\DynkinDiagram{  2 \alpha_1 + 2 \alpha_2 + 3 \alpha_3 + 4 \alpha_4 \\ {\phantom{==}} + 3 \alpha_5 + 2 \alpha_6 + \alpha_7 \\ }\) \\[3mm]
    \hline
    \(\DynkinDiagram{ \\ \bE_8 \\}\) &
    \(\DynkinDiagram{ \NodeNumberAbove{\phantom{A}} & & \circ \NodeNumberLeft{2} & & & & \\ \circ \NodeNumberBelow{1} \Edge{r} & \NodeNumberBelow{3} \circ \Edge{r} & \circ \NodeNumberBelow{4} \Edge{r} \Edge{u}& \circ \NodeNumberBelow {5} \Edge{r} & \circ \NodeNumberBelow{6} \Edge{r} & \circ \NodeNumberBelow{7} \Edge{r} & \circ \NodeNumberBelow{8} \\ & & & & & &} \)
    & \(\DynkinDiagram{  2 \alpha_1 + 3 \alpha_2 + 4 \alpha_3 + 6 \alpha_4 \phantom{==} \\  {\phantom{==}} +  5 \alpha_5 + 4 \alpha_6 + 3 \alpha_7 + 2 \alpha_8 \\ }\) \\[3mm]
    \hline
    \(\bF_4\) & 
    \(\DynkinDiagram{
      \circ \NodeNumberAbove{\phantom{A}} \NodeNumberBelow{1} \Edge{r} & \circ \NodeNumberBelow{2} \DirectedEdge[2]{r} & 
      \circ \NodeNumberBelow{3} \Edge{r} & \NodeNumberBelow{4} \circ}\) &
    \(2 \alpha_1 + 3 \alpha_2 + 4 \alpha_3 + 2 \alpha_4\) \\[3mm]
    \hline
    \(\bG_2\) & 
    \(\DynkinDiagram{\circ \NodeNumberAbove{\phantom{A}} \NodeNumberBelow{1}  & \DirectedEdge[3]{l} \circ \NodeNumberBelow{2} }\) & \(3 \alpha_1 + 2 \alpha_2\)\\[3mm]
    \hline 
  \end{tabular}

  \caption{Dynkin diagrams and highest roots}
  \label{tab:dynkin-diagrams}
\end{table}

We can see immediately from the expressions for the highest roots that \(\fe_8,\ff_4\), and \(\fg_2\) have no cominuscule parabolic subalgebras, while \(\fsl_{N+1}\) has \(N\), \(\fso_{2N+1}\) and \(\fsp_{2N}\) have one each, \(\fso_{2N}\) has three, \(\fe_6\) has two, and \(\fe_7\) has one.

We would like more information than just the enumeration of the cominuscule parabolics, however.
We also want to know the isomorphism type of the subalgebras \(\fl\) and \(\fk\), and the highest weight of the representation \(\um\).
This data is presented most easily in graphical form using the Dynkin diagrams.
We give an example first to illustrate the principle.

\begin{eg}
  \label{eg:gr24}
  Let \(\fg = \fsl_{4}\) with simple roots \(\alpha_1,\alpha_2,\alpha_3\), as in \cref{rem:roots-of-uplus-and-l}.
  Since the highest root is given by \(\theta = \alpha_1 + \alpha_2 + \alpha_3\), all three of the maximal parabolics are cominuscule according to \cref{prop:cominuscule-parabolic-conditions}(f).
  We take \(s = 2\) in this instance.
  Then we have \(\rtsys(\fl) = \{ \pm \alpha_1, \pm \alpha_3 \}\), while \(\rtsys(\up) = \{ \alpha_2, \alpha_1 + \alpha_2, \alpha_2 + \alpha_3, \alpha_1 + \alpha_2 + \alpha_3 \}\).
  The decomposition \(\fp = \fl \oplus \up\) looks like
  \begin{equation}
   \fp = \left\{
  \begin{pmatrix}
    * & * & * & * \\
    * & * & * & * \\
    0 & 0 & * & * \\
    0 & 0 & * & * 
  \end{pmatrix}  \in \fsl_4 \right\},\label{eq:gr24-decomp}
\end{equation}
  where the top left and bottom right \(2 \times 2\) blocks make up \(\fl\) and the top right \(2 \times 2\) block is \(\up\).
  Thus \(\fl\) is isomorphic to \(\fs(\fgl_2 \times \fgl_2) \cong \fsl_2 \times \fsl_2 \times \bbC\), and \(\fk \cong \fsl_2 \times \fsl_2\).

  By \cref{prop:uminus-as-l-module}(b) the highest weight of \(\um\) as an \(\fl\)-module is \(-\alpha_2\).
  We illustrate this situation with the following diagram:
  \begin{equation}
    \label{eq:gr24-diagram}
    \DynkinDiagram{
      \circ \NodeNumberAbove{1} \Edge{r} & \times \Edge{r} & \circ \NodeNumberAbove{1}}
  \end{equation}
  This indicates that \(\alpha_2\) has been excluded from \(\rtsys(\fl)\), and the 1's above the first and third nodes are the coefficients \((-\alpha_2,\alpha_1^\vee)\) and \((-\alpha_2,\alpha_3^\vee)\) determining the highest weight of \(\um\) as a representation of the semisimple part \(\fk\) of \(\fl\).
  The Dynkin diagram formed by the uncrossed nodes of \eqref{eq:gr24-diagram} is two copies of \(\bA_1\), which is the isomorphism type of \(\fk\).

  We note in passing that in this case the associated flag manifold \(SL_4/P\) is the Grassmannian \(\mathrm{Gr}_2(4)\).
  The tangent space at the identity (coset) is canonically isomorphic to \(\um\) as an \(\fl\)-module, and it is realized as the action
  \[
  (X,Y)\cdot M = YM - MX
  \]
  of \(\fsl_2 \times \fsl_2\) on \(\um \cong M_2(\bbC)\). 
  This can be seen by noting that \(\um\) is embedded in the lower left corner of \(\fsl_4\) as the complement to \(\fp\) in the decomposition \eqref{eq:gr24-decomp}.
  Then the action of \(\fl\) on \(\um\) is just the restriction of the adjoint action of \(\fsl_4\).
\end{eg}

In \cref{tab:irreducible-parabolics} we list the cominuscule parabolics of simple Lie algebras by diagrams similar to \eqref{eq:gr24-diagram}.
We refer to these as \emph{parabolic Dynkin diagrams}\index[term]{parabolic Dynkin diagram}, but we note that this terminology is not standard.
For clarity we omit the labels on the nodes indicating the numbering of the simple roots.
The diagrams should be read as follows:
\begin{itemize}
\item A crossed node indicates that the corresponding simple root has been excluded from \(\rtsys(\fl)\).  More precisely, the node for \(\alpha_s\) is crossed if \(\cS = \simprts \setminus \{ \alpha_s \}\).
\item The isomorphism type of the semisimple part \(\fk\) of the Levi factor \(\fl\) of \(\fp\) is given by the Dynkin diagram formed by removing the crossed node and any incident edges from the Dynkin diagram of \(\fg\).
\item The coefficients over the nodes of the parabolic Dynkin diagram indicate the marks of the highest weight of \(\um\) as an \(\fl\)-module.  
  In particular, if \(\cS = \simprts \setminus \{ \alpha_s \}\), then \cref{prop:uminus-as-l-module}(b) says that the highest weight of \(\um\) is \(-\alpha_s\).  
  Thus, for a simple root \(\alpha_j \neq \alpha_s\), the coefficient over the \(\alpha_j\) node is given by \((-\alpha_s,\alpha\spcheck_j) = - a_{js}\).
\end{itemize}

\setlength{\extrarowheight}{3mm}

\begin{table}
  \centering
  \begin{tabular}{|c|c|c|c|}
    \hline
    & \(\fg\) & Parabolic Dynkin diagram & \(\fk\)  \\ 
   \hline

    I & \( \bA_{p + q - 1} \) &
    \( \DynkinDiagram{\circ  
      \Edge{r} \NodeNumberBelow{\phantom{\sum}}
      \NodeNumberAbove{0} & \hspace{1mm} \dots \hspace{1mm} \Edge{r} & \circ 
      \NodeNumberAbove{1} & \Edge{l} \times  & \Edge{l} \circ \NodeNumberAbove{1} 
      & \Edge{l} \hspace{1mm} \dots \hspace{1mm} & \Edge{l} \circ \NodeNumberAbove{0}  
    }
    \) & \(\bA_{p-1} \times  \bA_{q-1}\) \\

    II & \( \bB_N\) &
    \(\DynkinDiagram{
      \times  \Edge{r} \NodeNumberBelow{\phantom{\sum}} \NodeNumberAbove{\phantom{\sum}} & \circ \Edge{r} \NodeNumberAbove{1}  & 
      \,\,\cdots\,\, \Edge{r} & \circ \DirectedEdge[2]{r} \NodeNumberAbove{0} & 
      \circ \NodeNumberAbove{0}  }  \) & \(\bB_{N-1}\) \\

    III & \( \bC_N\) &
    \(\DynkinDiagram{
      \circ \Edge{r} \NodeNumberAbove{0}  & \circ \NodeNumberAbove{0}  \Edge{r} & 
      \,\,\cdots\,\, \Edge{r} & \circ  \NodeNumberAbove{2} & 
      \DirectedEdge[2]{l} \times  }\) & \(\bA_{N-1}\)\\

     \(\DynkinDiagram{ \\ \text{IV} \\ } \) & \( \DynkinDiagram{ \\ \bD_N \\} \) & 
        \(\DynkinDiagram{
      & & & & \circ \NodeNumberAbove{0} &  \NodeNumberAbove{\phantom{\sum}}\\
      \times  \Edge{r}   & \circ \Edge{r} \NodeNumberAbove{1}  & 
      \,\,\cdots\,\, \Edge{r} & \circ \Edge{ur} \Edge{dr} \NodeNumberAbove{0}  & & \\
      &&&  & \circ \NodeNumberAbove{0}  \NodeNumberBelow{\phantom{\sum}} & 
    }\)  & \( \DynkinDiagram{ \\ \bD_{N-1} \\} \) \\

    \(\DynkinDiagram{ \\ \text{V} \\ } \) & \( \DynkinDiagram{ \\ \bD_N \\} \) & 
    \(\DynkinDiagram{
      & & & & \circ \NodeNumberAbove{0}   &  \NodeNumberAbove{\phantom{\sum}}\\
      \circ \NodeNumberAbove{0}   \Edge{r}   & \circ \NodeNumberAbove{0} \Edge{r}  & 
      \,\,\cdots\,\, \Edge{r} & \circ \NodeNumberAbove{1}  \Edge{ur} \Edge{dr}  & & \\
      &&&  &\times   \NodeNumberBelow{\phantom{\sum}} & 
    }\) & \( \DynkinDiagram{ \\ \bA_{N-1} \\} \) \\

     VI & \( \bE_6 \)  &      \(\DynkinDiagram{ \NodeNumberAbove{\phantom{A}} & & \circ \NodeNumberAbove{0}  & & \\ \circ  \NodeNumberAbove{0} \Edge{r} &  \circ \NodeNumberAbove{0}  \Edge{r} & \circ  \Edge{r}  \NodeNumberAbove{\quad 0} \Edge{u}& \circ  \NodeNumberAbove{1} \Edge{r} & \times  \NodeNumberBelow{\phantom{\sum}} } \) & \(\bD_5\) \\

     VII & \(\bE_7\) & 
     \(\DynkinDiagram{ \NodeNumberAbove{\phantom{A}} & & \circ \NodeNumberAbove{0}  & & & \\ \circ  \NodeNumberAbove{0}  \Edge{r} &  \circ \NodeNumberAbove{0}  \Edge{r} & \circ  \Edge{r}  \NodeNumberAbove{\quad 0} \Edge{u} & \circ  \NodeNumberAbove{0}  \Edge{r}  & \circ  \NodeNumberAbove{1}  \Edge{r} & \times  \NodeNumberBelow{\phantom{\sum}} } \)
     & \(\bE_6\) \\
    \hline
  \end{tabular}
  \caption{Parabolic Dynkin diagrams of simple Lie algebras}
  \label{tab:irreducible-parabolics}
\end{table}

\begin{rem}
  \label{rem:comments-on-table-of-parabolics}
  Some comments on \cref{tab:irreducible-parabolics}:
  \begin{enumerate}[(1)]
  \item In the \(\bA_N\) case (i.e.~\(\fg = \fsl_{N+1}\)), any node can be crossed out, including the nodes at either end (because every simple root has coefficient 1 in the highest root).  When an end-node is crossed out, the Dynkin diagram of \(\fk\) is connected and there is only one nonzero coefficient label.
  \item Strictly speaking, the table should include another parabolic Dynkin diagram for the root system of type \(\bD_N\) in which the node corresponding to \(\alpha_{N-1}\) is crossed out.  We omit this diagram from the table, as it can be obtained from the diagram in row V in which the \(\alpha_{N}\) node is crossed out by the nontrivial automorphism of the Dynkin diagram for \(\bD_N\).  The representation of \(\fl\) on \(\um\) is the same in either case.
  \item A similar statement holds for the \(\bE_6\) case, in which the nodes corresponding to \(\alpha_1\) and \(\alpha_6\) are exchanged by the nontrivial automorphism of the Dynkin diagram.
  \item Table 3.2 of \cite{BasEas89} also enumerates the irreducible parabolic subalgebras along with the representations \(\um\).
  Their conventions are slightly different, however: they instead label the nodes with the coefficients of the highest weight of \(\up\).
  Additionally, their diagram for the unique cominuscule parabolic subalgebra of \(\fe_7\) is incorrect: the 1 in their figure should be over the node we have labeled as corresponding to \(\alpha_1\) in \cref{tab:dynkin-diagrams}.
  \end{enumerate}
\end{rem}

\section{Lie groups}
\label{sec:lie-groups-notation}

We denote by \(G\) the connected, simply connected complex Lie group with Lie algebra \(\fg\).
Similarly, we denote by \(H, B_\pm, P, L\) the connected (complex) Lie subgroups of \(G\) whose Lie algebras are \(\fh, \fb_{\pm}, \fp, \fl\), respectively.
We denote by \(G_0\) the compact real form of \(G\) corresponding to the compact real form \(\fg_0\) of \(\fg\) described in \cref{sec:compact-real-form-of-g}.
\index[notn]{G@\(G,G_0\)}
\index[notn]{H@\(H\)}
\index[notn]{B@\(B_\pm\)}
\index[notn]{P@\(P\)}
\index[notn]{L@\(L\)}
The subgroup \(B_+\) is a maximal solvable subgroup of \(G\).
We have \(P \supseteq B_+\) as \(\fp \supseteq \fb_+\), and hence \(P\) is referred to as a \emph{parabolic subgroup}\index[term]{parabolic subgroup} of \(G\).
If \(\fp\) is of cominuscule type in the sense of \cref{dfn:cominuscule-parabolic} then we say that \(P\) is of cominuscule type as well.

\section{Generalized flag manifolds}
\label{sec:generalized-flag-manifolds}

A \emph{generalized flag manifold}\index[term]{generalized flag manifold}\index[notn]{GP@\(G/P\)} is a homogeneous space of the form \(G/P\), where \(G\) and \(P\) are as in \cref{sec:lie-groups-notation}
(\(P\) is not necessarily cominuscule).

As \(G\) and \(P\) are complex Lie groups, \(G/P\) is a complex manifold.
It is compact because \(P\) contains the Borel subgroup \(B_+\) of \(G\).
Alternatively, one can show that the compact real form \(G_0\) acts transitively on \(G/P\), and hence as real manifolds we have \(G/P \cong G_0/L_0\), where \(L_0 = G_0 \cap L\) (see for example \cite{BasEas89}*{\S 6.4}); this also shows that \(G/P\) is compact.
The complex tangent space to \(G/P\) at the identity coset \(eP\) is canonically isomorphic to \(\fg/\fp \cong \um\); the Cartan involution \(\tau\) from \cref{sec:compact-real-form-of-g} together with the Killing form \(K\) of \(\fg\) define a Hermitian form on \(\um\) via
\begin{equation}
  \label{eq:hermitian-form-on-uminus}
  \langle X, Y \rangle \eqdef -K(\tau(X),Y).  
\end{equation}
This Hermitian form induces a Hermitian metric\index[term]{generalized flag manifold!Hermitian metric on} on \(G/P\) via translation with the \(G\)-action, so \(G/P\) becomes a Hermitian manifold.

More can be said when \(P\) is of cominuscule type.
In this case \(G/P\) is a \emph{compact Hermitian symmetric space}, i.e.\ a compact Hermitian manifold in which every point is an isolated fixed point of an involutive holomorphic isometry; and moreover every compact Hermitian symmetric space is of the form \(G/P\) for a cominuscule parabolic \(P\).
See \cite{Hel78}*{Ch.~VIII, \S4} as well as \cite{Kos61}*{Proposition 8.2}.

\section{Quantized enveloping algebras}
\label{sec:lie-background-quantized-enveloping-algebras}

Here we define the quantized universal enveloping algebras \(\uqg\) associated to semisimple Lie algebras \(\fg\).
There are many references that cover the general structure and representation theory of \(\uqg\), many of them using quite different notation and differing conventions on both the algebra and coalgebra structures.
We have chosen the book \cite{Jan96} as our main reference for its excellent exposition and clear motivation of the results.
However, that book does not cover as much ground as some other books, especially with respect to \(\ast\)-structures, and so we make use of some other books as well, namely \cite{ConPro93,KliSch97,Lus10,ChaPre95}.

\subsection{\(q\)-numbers}
\label{sec:q-numbers}

We take \(q \in \bbC \setminus \{ 0 \}\) to be not a root of unity, and fix a complex number \(\hbar\)\index[notn]{h@\(\hbar\)} such that \(e^\hbar = q\); if \(q > 0\) then we take \(\hbar \in \bbR\) as well.
This allows us to define arbitrary complex powers of \(q\) unambiguously as \(q^z \eqdef e^{z\hbar}\).  
For \(\alpha \in \rtsys\) we define \(q_\alpha = q^{d_\alpha}\), and we abbreviate \(q_j = q_{\alpha_j} = q^{d_j}\) for \(\alpha_j \in \simprts\).\index[notn]{qj@\(q_j,q_\alpha\)}
For \(n,k \in \zp\) we define the \emph{quantum number}\index[term]{quantum!number} \([n]_q\), the \emph{quantum factorial}\index[term]{quantum!factorial} \([n]_q!\) and the \emph{quantum binomial coefficient}\index[term]{quantum!binomial coefficient} \(\qbinom{n}{k}\) by
\[
\qnum{n} \eqdef \frac{q^n-q^{-n}}{q-q^{-1}}, \quad \qfact{n} = \prod_{k=1}^n [k]_q, \quad \text{and} \quad \qbinom{n}{k} = \frac{[n]_q!}{[k]_q![n-k]_q!},
\]
respectively.

\subsection{Algebra structure}
\label{sec:quea-algebra-structure}

We define \(\uqg\)\index[notn]{Uqg@\(\uqg\)} to be the unital associative algebra over \(\bbC\) generated by elements \(E_1,\dots,E_r\), \(F_1,\dots,F_r\) and \(K_\lambda\) for \(\lambda \in \cP\), with the relations
\begin{equation*}
  \begin{gathered}
  K_\lambda E_j = q^{(\lambda,\alpha_j)} E_j K_\lambda, \quad  K_\lambda F_j = q^{-(\lambda,\alpha_j)} F_j K_\lambda, \quad   K_\lambda K_\mu = K_{\lambda + \mu},\\
  E_i F_j - F_j E_i = \delta_{ij} \frac{K_j - K_j^{-1}}{q_j - q_j^{-1}},
  \end{gathered}
\end{equation*}
where we denote \(K_j \eqdef K_{\alpha_j}\) (the \(\alpha_j\) are the simple roots, see \cref{sec:root-system}), together with the \emph{quantum Serre relations}\index[term]{quantum Serre relations}
\begin{equation}
  \label{eq:quantum-serre-relations}
  \begin{gathered}
    \sum_{k=0}^{1-a_{ij}} (-1)^k \qbinom[q_i]{1-a_{ij}}{k} E_i^{1-a_{ij}-k} E_j E_i^{k} = 0,\\
    \sum_{k=0}^{1-a_{ij}} (-1)^k \qbinom[q_i]{1-a_{ij}}{k} F_i^{1-a_{ij}-k} F_j F_i^{k} = 0,
  \end{gathered}
\end{equation}
for \(i \neq j\).

Furthermore, we denote by \(\uqnp\) the subalgebra of \(\uqg\) generated by \(E_1,\dots,E_r\), by \(\uqnm\) the subalgebra generated by \(F_1,\dots,F_r\), and by \(\uqbpm\) the subalgebras generated by \(\uqnpm\) together with all \(K_\lambda\), respectively.
Finally, we denote by \(\uqh\) the subalgebra generated by the \(K_\lambda\) for \(\lambda \in \cP\).

\begin{rem}
  We note that most sources define \(\uqg\)\label{rem:on-the-def-of-uqg} to contain \(K_\lambda\) just for \(\lambda \in \cQ\), not for all \(\lambda \in \cP\).
  The version defined here is denoted \(U_q(\fg, \cP)\) in the notation of \cite{Jan96}*{\S 4.5}.
  This version will be useful to us when we embed \(\uql\) into \(\uqg\) in \cref{sec:flatness-and-cominuscule-parabolics}, where \(\fl\) is the Levi factor of a certain parabolic subalgebra \(\fp\) of \(\fg\).
\end{rem}

\subsection{Hopf algebra structure}
\label{sec:quea-hopf-structure}

The algebra \(\uqg\) has a Hopf algebra structure with comultiplication \(\cop\), counit \(\counit\), and antipode \(\antipode\) defined on the generators by
\begin{alignat*}{6}
  \cop(E_j) & = E_j \otimes 1 + K_j \otimes E_j,    & \qquad \epsilon(E_j) = 0,     & \qquad S(E_j) = - K_j^{-1}E_j,\\
  \cop(F_j) & = F_j \otimes K_j^{-1} + 1 \otimes F_j  & \qquad \epsilon(F_j) = 0,     & \qquad S(F_j) = - F_jK_j,\\
  \cop(K_\lambda) & = K_\lambda \otimes K_\lambda,      & \qquad \epsilon(K_\lambda) = 1, & \qquad S(K_\lambda) = K_{-\lambda} = K_\lambda^{-1}.
\end{alignat*}
Note that both \(\uqbpm\) are Hopf subalgebras of \(\uqg\), but \(\uqnpm\) are not.

\begin{rem}
  \label{rem:coproduct-choice}
  It is unfortunate that the conventions for the Hopf algebra structure on \(\uqg\) are not more standardized.
  The coproduct described here is the same as the one used in \cite{Jan96}, \cite{Lus10}, and the opposite of the coproduct used in \cite{KliSch97}.
\end{rem}

\begin{notn}
  \label{notn:sweedler-notation}
  We use Sweedler's notation \(\cop(a) = a_{(1)} \otimes a_{(2)}\), with implied summation, for the coproduct of \(a\in\uqg\).
\end{notn}

\subsection{Compact real form}
\label{sec:quea-compact-real-form}

For \(q \in \bbR\) only, we define the \emph{compact real form} of the quantized enveloping algebra to be the Hopf algebra \(\uqg\) together with the \(\ast\)-structure determined by
\begin{equation}
  \label{eq:compact-real-form-of-uqg}
  E_j^\ast = K_j F_j, \qquad F_j^\ast = E_j K_j^{-1}, \qquad K_\lambda^\ast = K_\lambda.
\end{equation}
There are other \(\ast\)-structures as well; see \cite{KliSch97}*{\S 6.1.7} or \cite{Twi92}, for example.
Note that these two sources use opposite coproducts (the one in \cite{Twi92} is the same as ours) but the \(\ast\)-structures are nevertheless identical.

\subsection{Braid group action}
\label{sec:quea-braid-group-action}

The \emph{Coxeter braid group} \(\fbg\) associated to \(\fg\) is the cover of the Weyl group \(W\) of \(\fg\) whose relations are obtained from those of \(W\) by removing the restriction that the generators be involutive.
Lusztig has shown that \(\fbg\) acts on \(\uqg\) via algebra automorphisms, as well as on all Type 1 representations of \(\uqg\).
We do not define the group formally, as all we require are the automorphisms themselves; see \cite{KliSch97}*{\S 6.2.1}, \cite{Lus10}*{Ch.~37}, or \cite{Jan96}*{Ch.~8} for more details.
For \(1 \leq i \leq r\), the braid group automorphism \index[term]{braid group action on \(\uqg\)} \(T_i\)\index[notn]{Ti@\(T_i\)} of \(\uqg\) is defined by:
\begin{gather*}
  T_i(E_i) = -F_i K_i, \quad T_i(F_i) = - K_i^{-1} E_i, \quad T_i(K_\lambda) = K_{s_{\alpha_i}(\lambda)},\\
  T_i(E_j) = \sum_{k=0}^{-a_{ij}} (-1)^{k-a_{ij}} q_i^{-k} E_i^{(-a_{ij}-k)}E_j E_i^{(k)}, \quad i \neq j\\
  T_i(F_j) = \sum_{k=0}^{-a_{ij}} (-1)^{k-a_{ij}} q_i^{k} F_i^{(k)}F_jF_i^{(-a_{ij}-k)}, \quad i \neq j,
\end{gather*}
where\(E_i^{(n)},F_i^{(n)}\) denote the \emph{divided powers}\index[term]{divided power}
\[
E_i^{(n)} \eqdef \frac{E_i^n}{\qfact[q_i]{n}}, \qquad F_i^{(n)} \eqdef \frac{F_i^n}{\qfact[q_i]{n}}
\]
for \(n \in \zp\).
This definition of \(T_i\) coincides with the automorphisms denoted \(\cT_i\), \(T_{\alpha_i}\), and \(T''_{i,1}\) in \cite{KliSch97}, \cite{Jan96}, and \cite{Lus10}, respectively.

For \(w \in W\), the algebra automorphism \(T_w\)\index[notn]{Tw@\(T_w\)} of \(\uqg\) is defined by
\[
T_w = T_{i_1} \dots T_{i_k},
\]
where \(s_{i_1} \dots s_{i_k}\) is any reduced expression for \(w\).
This is independent of the choice of reduced expression precisely because the \(T_i\) obey the relations of \(\fbg\), and furthermore, we have \(T_wT_{w'} = T_{ww'}\) whenever \(\ell(ww') = \ell(w) + \ell(w')\); see \cite{Jan96}*{\S 8.18}.

\subsection{Quantum root vectors, Poincar\'e-Birkhoff-Witt theorem}
\label{sec:quantum-root-vectors}

The quantized enveloping algebra \(\uqg\) is defined via generators \(E_i,F_i\) corresponding just to the simple roots \(\alpha_i \in \simprts\), together with the \(K_\lambda\) for \(\lambda \in \cP\).
The braid group action on \(\uqg\) described above allows us to define analogues of all Chevalley basis elements \(E_\alpha,F_\lambda\) for \(\alpha \in \rtsys\).
Let
\begin{equation*}
  \label{eq:reduced-expression-for-w0}
  \wz = s_{i_1} \dots s_{i_d}
\end{equation*}
be a reduced expression for the longest word of the Weyl group of \(\fg\).
For \(1 \leq k \leq d\), define
\begin{equation}
  \label{eq:sequence-of-positive-roots}
  \beta_k = s_{i_1} \dots s_{i_{k-1}}(\alpha_{i_k}).  
\end{equation}
Then by \cref{lem:enumeration-of-phi-of-w}, the \(\beta_i\) are all distinct, and the set \(\{ \beta_1,\dots,\beta_d \}\) exhausts \(\posrts(\fg)\).
The \emph{quantum root vectors}\index[term]{quantum root vectors} \(E_{\beta_k}\) and \(F_{\beta_k}\) are defined in analogy with \eqref{eq:sequence-of-positive-roots} by

\begin{equation}
  \label{eq:quantum-root-vectors-def}
  E_{\beta_k} = T_{i_1} \dots T_{i_{k-1}}(E_{i_k}), \qquad F_{\beta_k} = T_{i_1} \dots T_{i_{k-1}}(F_{i_k})
\end{equation}
\index[notn]{Ebetak@\(E_{\beta_k}\)}\index[notn]{Fbetak@\(F_{\beta_k}\)}
for \(1 \leq k \leq d\).
The quantum root vectors depend on the choice of reduced expression for \(\wz\).
However, if \(\beta_k = \alpha_i\) for some \(i\) then \(E_{\beta_k} = E_i\), i.e.\ \(E_{\alpha_i} = E_i\).

The set of monomials of the form \(E_{\beta_1}^{e_1} \dots E_{\beta_d}^{e_d}\) with all \(e_i \in \zp\) is a basis of \(\uqnp\), and the analogous statement holds for the \(F_{\beta_k}\) and \(\uqnm\) (\cite{Jan96}*{Theorem 8.24}); moreover, the set of monomials of the form \(K_\lambda E_{\beta_1}^{e_1} \dots E_{\beta_d}^{e_d} F_{\beta_1}^{f_1} \dots F_{\beta_d}^{f_d}\) is a basis for \(\uqg\).
The latter fact is known as the Poincar\'e-Birkhoff-Witt theorem (PBW theorem) for \(\uqg\)\index[term]{PBW theorem} in analogy with the corresponding result for \(U(\fg)\).

The following result, proved originally by Levendorski{\u\i} and Soibelman and elucidated nicely in \cite{ConPro93}*{Theorem 9.3}, gives some commutation relations \index[term]{quantum root vectors!commutation relations} for the quantum root vectors:
\begin{prop}
  \label{prop:commutation-rels-for-quantum-root-vectors}
  For \(\bk = (k_1,\dots,k_d)\), define \(E_\bk = E_{\beta_1}^{k_1}\dots E_{\beta_d}^{k_d}\) and \(F_\bk = F_{\beta_d}^{k_d}\dots F_{\beta_1}^{k_1}\).
  Then for \(i < j\), there are numbers \(a_{\bk},b_{\bk} \in \bbQ[q^{\pm 1}]\) such that
  \begin{equation}
    \label{eq:commutation-rels-for-Es}
    E_{\beta_i} E_{\beta_j} - q^{(\beta_i,\beta_j)}E_{\beta_j} E_{\beta_i} = \sum_{\bk \in \zp^d} a_{\bk} E_{\bk}
  \end{equation}
  and
  \begin{equation}
    \label{eq:commutation-rels-for-Fs}
    F_{\beta_i} F_{\beta_j} - q^{-(\beta_i,\beta_j)}F_{\beta_j} F_{\beta_i} = \sum_{\bk \in \zp^d} b_{\bk} F_{\bk}.
  \end{equation}
  Moreover, \(a_{\bk} \neq 0\) only when \(\bk\) is such that \(k_t = 0\) for \(t \leq i\) and \(t \geq j\), and when 
  \[
  \sum_{l = i+1}^{j-1} k_l \beta_l = \beta_i + \beta_j.
  \]
  The same statement holds when \(a_\bk\) is replaced by \(b_\bk\).
\end{prop}

\subsection{Grading by the root lattice}
\label{sec:grading-by-root-lattice}

There is a grading of \(\uqg\) by the root lattice \(\cQ\) of \(\fg\), given by
\[
\uqg = \bigoplus_{\beta \in \cQ} \uqlg[\beta], \quad \text{where} \quad \uqlg[\beta] = \{ X \in \uqg \mid K_\lambda X K_{-\lambda} = q^{(\lambda,\beta)}X \text{ for all } \lambda \in \cP \}.
\]
It is clear that \(\uqh \subseteq \uqlg[0]\) and that \(E_{\beta_k} \in \uqlg[\beta_k]\), \(F_{\beta_k} \in \uqlg[-\beta_k]\) for \(1 \leq k \leq d\).

\subsection{Type 1 Representations}
\label{sec:representations-of-quea}

The finite-dimensional representation theory of \(\uqg\) closely parallels that of \(\fg\) itself; see Chapters 6 and 7 of \cite{KliSch97} or Chapter 5 of \cite{Jan96} for more detail.
Here we just give the definitions and mention some results that we will use below.
We use the terms representation and module interchangeably throughout; by module we mean left module.

For a finite-dimensional representation \(V\) of \(\uqg\) and a weight \(\lambda \in \cP\), we define the subspace of \emph{vectors of weight \(\lambda\)} by
\[
V_\lambda \eqdef \{ v \in V \mid K_\mu v = q^{(\lambda,\mu)}v \text{ for all } \mu \in \cP\}.
\]
We refer to \(V_\lambda\) as the \(\lambda\)\emph{-weight space} of \(V\).
A vector \(v \in V\) is a \emph{highest weight vector}\index[term]{highest weight vector} if it is a weight vector (i.e.\ lies in some weight space) such that \(\uqnp v = (0)\) and \(\uqnm v = V\), and \(V\) is called a \emph{highest weight representation of weight \(\lambda\)}\index[term]{representation!highest weight} if it has a highest weight vector of weight \(\lambda\).

For any dominant integral weight \(\lambda \in \pplus\) there is a finite-dimensional irreducible representation \(V(\lambda)\) of \(\uqg\) of highest weight \(\lambda\), and the dimensions of its weight spaces are the same as those of the corresponding representation of \(\fg\).
We denote by \(\oq\)\index[notn]{Oq@\(\oq\)} the full subcategory of \(\uqg\)-\textbf{Mod} whose objects consist of finite direct sums of the \(V(\lambda)\), and refer to these as \emph{Type 1 representations}.

When \(q \in \bbR\) there is a positive-definite Hermitian inner product \(\langle \cdot,\cdot \rangle_\lambda\) on \(V(\lambda)\), antilinear in the first variable, which is compatible with the compact real form of \(\uqg\) in the sense that
\[
\langle a v, w \rangle_\lambda = \langle v, a^\ast w \rangle
\]
for \(a \in \uqg\) and \(v,w \in V(\lambda)\).
This inner product is unique up to a positive scalar factor.
Decomposing an arbitrary \(V \in \oq\) into simple submodules induces an inner product on \(V\), and furthermore choosing inner products on \(U,V \in \oq\) induces inner products on \(U \otimes V\) and \(V \otimes U\).

\subsection{Tensor products and dual representations}
\label{sec:tensor-product-and-dual-reps}

The \emph{trivial representation}\index[term]{representation!trivial} of \(\uqg\) is \(\bbC = V(0)\) with the action given by \(a \cdot c = \epsilon(a)c\) for \(a \in \uqg\) and \(c \in \bbC\).

For any \(\uqg\)-modules \(U\) and \(V\), the tensor product \(U \otimes V\)\index[term]{representation!tensor product} (recall that all tensor products are over \(\bbC\) unless otherwise specified) is again a \(\uqg\)-module via
\[
a \cdot (u \otimes v) \eqdef a_{(1)}u \otimes a_{(2)}v
\]
for \(a \in \uqg\), \(u \in U\), and \(v \in V\).
It follows immediately from the counit axiom that scalar multiplication defines isomorphisms of representations \(\bbC \otimes V \cong V \cong V \otimes \bbC\).
Thus \(\oq\) is a monoidal category where \(\bbC\) is the monoidal unit.

The dual vector space \(V^\ast\)\index[term]{representation!dual} of a \(\uqg\)-module \(V\) is again a \(\uqg\)-module via the action
\[
(a \cdot f)(v) \eqdef f(S(a)v)
\]
for \(a \in \uqg\), \(f \in V^\ast\), and \(v \in V\).
The antipode axiom implies that the natural (evaluation) map \(V^\ast \otimes V \to \bbC\) given by \(f \otimes v \mapsto f(v)\) for \(f \in V^\ast\) and \(v \in V\) is a morphism of \(\uqg\)-modules.

We have the following interaction between tensor products and dual representations:
\begin{lem}
  \label{lem:dual-of-tensor-product}
  Let \(U,V \in \oq\).
  The pairing \(\langle f \otimes g, u \otimes v \rangle \eqdef g(u)f(v)\) for \(u \in U,v \in V\), \(g \in U^\ast\), and \(f \in V^\ast\) defines an isomorphism of representations \((U \otimes V)^\ast \cong V^\ast \otimes U^\ast\).
\end{lem}

\subsection{The braiding}
\label{sec:braiding}

For \(V,W \in \oq\) we denote by \(\rhat_{VW} = \tau \circ R : V \otimes W \to W \otimes V\)\index[notn]{Rhat@\(\rhat_{VW},\rhat_{\lambda \mu}\)} the braiding of finite-dimensional Type 1 \(\uqg\)-modules.
Here \(\tau\) is the tensor flip, and \(R\) is the universal R-matrix \index[term]{R-matrix} of \(\uqg\).
When \(V = V_\lambda\) and \(W = V_\mu\) for \(\lambda,\mu \in \pplus\) we write \(\rhat_{\lambda \mu}\) instead of \(\rhat_{V_\lambda V_\mu}\).
There is some choice involved in the R-matrix (see \cite{Jan96}*{Ch.~7} for an explanation of the choices); we choose \(R\) such that for weight vectors \(v \in V\) and \(w \in W\), we have
\begin{equation}
  \label{eq:braiding-weight-vectors}
  \rhat_{VW}(v \otimes w) = q^{(\wt(v),\wt(w))} w \otimes v + \sum_i w_i \otimes v_i,
\end{equation}
where \(\wt(w_i) \succ  \wt(w)\) and \(\wt(v_i) \prec \wt(v)\) for all \(i\), and \(\prec\) is the partial order on \(\cP\) defined in \cref{sec:weight-lattice}.
The morphisms \(\rhat_{VW}\) are \(\uqg\)-module maps, and they turn \(\oq\) into a braided monoidal category.

The braidings \(\rhat_{V,V}\) are diagonalizable, and all eigenvalues are of the form \(\pm q^{a_j}\) for certain \(a_j \in \bbQ\) \cite{KliSch97}*{Corollary 8.23}.
Moreover, the braidings \(\rhat_{V,W}\) are well-behaved with respect to duality in the sense that \((\rhat_{V,W})^{\tr} = \rhat_{V^\ast,W^\ast}\), where the superscript \(\tr\) denotes the dual map (or transpose), and we have identified \((V \otimes W)^\ast \cong W^\ast \otimes V^\ast\) via \cref{lem:dual-of-tensor-product}.

When \(q\) is real, fixing invariant inner products on \(V\) and \(W\) as in \cref{sec:representations-of-quea} and giving \(\uqg\) the \(\ast\)-structure \eqref{eq:compact-real-form-of-uqg}, the braidings are well-behaved with respect to adjoints: we have \((\rhat_{V,W})^\ast = \rhat_{W,V}\).
This is a consequence of the fact that the R-matrix is \emph{real} in the sense that \(R^\ast = R_{21} \eqdef \tau(R)\).
This is shown for the R-matrices associated to the infinite families of simple Lie algebras in \cite{KliSch97}*{\S 9.1.1, \S 9.2.4, \S 9.3.5}.
In general, one can show that \((\rhat_{V,W})^\ast = \rhat_{W,V}\) by noting that both are module maps and that they agree on highest weight vectors.
(This is true only for the compact real form of \(\uqg\).)

\subsection{The coboundary structure}
\label{sec:coboundary-structure}

The notion of a coboundary structure on a monoidal category, introduced by Drinfeld in \cite{Dri89}*{\S 3}, is similar to a braiding in the sense that it provides a natural isomorphism \(V \otimes W \overset{\sim}{\longrightarrow} W \otimes V\) for each ordered pair of objects in the category.

More precisely, if \((\cC,\otimes,1_\cC)\) is a monoidal category, then we can view the tensor product as a functor \( \otimes : \cC \times \cC \to \cC\), where \(\cC \times \cC\) is the Cartesian product category.
The tensor product functor takes a pair of objects \((C,D)\) to their tensor product \(C \otimes D\), and takes a pair of morphisms \((f,g)\) to their tensor product \(f \otimes g\).
Similarly, one can define a functor \(\otimes^{\mathrm{op}} : \cC \times \cC \to \cC\) which takes the pair of objects \((C,D)\) to \(D \otimes C\), and likewise with morphisms.

A braiding on \(\cC\) is, by definition, a natural isomorphism from \(\otimes\) to \(\otimes^{\mathrm{op}}\) satisfying certain coherence conditions (the ``hexagon axioms'', see \cite{JoyStr93}*{Definition 2.1}, although that terminology is not used there).
A coboundary structure is also a natural isomorphism from \(\otimes\) to \(\otimes^{\mathrm{op}}\), but satisfying different coherence conditions.
The main difference is that braidings are \emph{local} in the sense that the braiding of an object \(U\) over a tensor product object \(V \otimes W\) is determined by the braidings of \(U\) over \(V\) and \(U\) over \(W\); this is not true for coboundary maps.
Additionally, coboundary isomorphisms are required to be symmetric, but this is not required for a braiding.

\begin{dfn}
  \label{dfn:coboundary-category}
  A \emph{coboundary category}\index[term]{coboundary category} is a monoidal category \(\cC\) together with a natural isomorphism \(\sigma_{VW} : V \otimes W \to W \otimes V\) for all \(V,W \in \cC\) satisfying
  \begin{enumerate}[(i)]
  \item (Symmetry axiom) \(\sigma_{WV} \sigma_{VW} = \id_{V \otimes W}\);
  \item (Cactus axiom) For all \(U,V,W \in \cC\) the following diagram commutes:
    \begin{equation}
      \label{eq:coboundary-axiom}
      \begin{CD}
        U \otimes V \otimes W @>{\sigma_{UV} \otimes \id}>> V \otimes U \otimes W\\
        @V{\id \otimes \sigma_{VW}}VV @VV{\sigma_{V \otimes U,W}}V\\
        U \otimes W \otimes V @>>{\sigma_{U,W \otimes V}}> W \otimes V \otimes U
      \end{CD}
    \end{equation}
  \end{enumerate}
  Following the terminology of \cite{KamTin09}, we refer to the maps \(\sigma_{VW}\) as \emph{commutors}.
\end{dfn}
Naturality of the commutors means that for any objects \(V,V',W,W' \in \cC\) and any morphisms \(f : V \to V'\) and \(g : W \to W'\), the following diagram commutes:
\begin{equation}
  \label{eq:naturality-of-coboundary-operators}
  \begin{CD}
    V \otimes W @>{\sigma_{VW}}>> W \otimes V\\
    @V{f \otimes g}VV @VV{g \otimes f}V\\
    V' \otimes W' @>>{\sigma_{V' W'}}> W' \otimes V'
  \end{CD}
\end{equation}

Polar decomposition of the R-matrix leads to a coboundary structure on \(\oq\).
In \cite{KamTin09} the authors develop this at the level of suitable completions of \(\uqg\) and its tensor square.
We work here just at the level of representations.
For \(V,W \in \oq\), the \emph{commutor} \(\sigma_{VW}\) is defined as the unitary part of the polar decomposition
\begin{equation}
  \label{eq:def-of-commutors}
  \rhat_{VW} = \sigma_{VW} \left( (\rhat_{VW})^* \rhat_{VW} \right)^{\frac{1}{2}}
\end{equation}
of the braiding (which is independent of the choice of invariant inner products on \(V\) and \(W\), as indicated above; \(\sigma_{VW}\) is unitary rather than just a partial isometry because \(\rhat_{VW}\) is invertible.)
We record here the properties of the maps \(\sigma_{VW}\) that we will require:
\begin{prop}
  \label{prop:coboundary-struct-on-uqg-modules}
  The maps \((\sigma_{VW})_{V,W \in \oq}\) form a coboundary structure on \(\oq\).
  Moreover:
  \begin{enumerate}[(a)]
  \item The diagram
    \[
    \begin{CD}
      (W \otimes V)^* @>{(\sigma_{VW})^{\mathrm{tr}}}>> (V \otimes W)^*\\
      @V{\cong}VV @VV{\cong}V\\
      V^* \otimes W^* @>>{\sigma_{V^* W^*}}> W^* \otimes V^*
    \end{CD}
    \]    
    commutes, where the vertical arrows are the isomorphisms from \cref{lem:dual-of-tensor-product} and \((\sigma_{VW})^{\mathrm{tr}}\) is the transpose (or dual map) of the commutor \(\sigma_{VW}\).
  \item Each \(\sigma_{VW}\) is unitary by construction, and \((\sigma_{VW})^* = \sigma_{WV}\).
  \item In particular, \(\sigma_{VV}\) is a self-adjoint unitary operator with the same eigenspaces as \(\rhat_{VV}\), and \(\sigma_{VV}\) has eigenvalue \(\pm 1\) on an eigenspace according as \(\rhat_{VV}\) has eigenvalue \(\pm q^a\) for some \(a \in \bbQ\).
  \end{enumerate}
\end{prop}

\begin{rem}
  \label{rem:coboundary-equiv-def}
  As discussed in \cref{sec:braiding}, the R-matrix satisfies \(R^\ast = R_{21}\) for the compact real form of \(\uqg\).
  Thus we see that the \(*\)-structure and inner products are not necessary to define the coboundary structure.
\end{rem}

\section{Quantum Schubert cells and their twists}
\label{sec:quantum-schubert-cells-and-their-twists}

Quantum Schubert cells \index[term]{quantum Schubert cell} are subalgebras of \(\uqg\) associated to elements in the Weyl group of \(\fg\). 
They have been studied extensively from both ring-theoretic and representation-theoretic points of view.

\subsection{Quantum Schubert cells}
\label{sec:quantum-schubert-cells}

\begin{dfn}
  \label{dfn:quantum-schubert-cells}
  Let \(w\) be an element of the Weyl group \(W\) of \(\fg\), and let \(w = s_{i_1}\dots s_{i_n}\) be a reduced expression for \(w\).
  The \emph{quantum Schubert cell} \(U(w)\)\index[notn]{Uw@\(U(w)\)} is defined to be the subalgebra of \(\uqg\) generated by the elements
  \begin{equation}
    \label{eq:quantum-schubert-cell-generators}
    T_{i_1} \dots T_{i_{k-1}}(E_{i_k}), \quad 1 \le k \le n,
  \end{equation}
  where the \(T_i\) are the braid group automorphisms of \(\uqg\) defined in \cref{sec:quea-braid-group-action}. 
\end{dfn}

\begin{rem}
  \label{rem:on-the-quantum-schubert-cells}
  Although the generators \eqref{eq:quantum-schubert-cell-generators} depend on the choice of reduced expression for \(w\), the algebra \(U(w)\) is independent of this choice (see Section 9.3 of \cite{ConPro93}, for example).
  The quantum Schubert cells were introduced in \cite{ConKacPro95} and are hence often called De Concini-Kac-Procesi algebras.
  Zwicknagl shows in Section 5 of \cite{Zwi09} that \(U(w)\) can also be defined as 
  \begin{equation}
    \label{eq:alternate-def-of-quantum-schubert-cell}
    U(w) = T_w(\uqbm) \cap \uqnp
  \end{equation}
  (note that the \(w^{-1}\) appearing in Definition 5.2 of \cite{Zwi09} must be changed to \(w\), and Theorem 5.21(a) therein should say \(U'(w)\) rather than \(U(w)\)).
\end{rem}

Technicalities of the definition of quantum root vectors aside, morally the definition of quantum Schubert cells says that \(U(w)\) is generated by quantum root vectors associated to the roots in \(\rtsys(w)\); see \cref{lem:enumeration-of-phi-of-w}.
The reason that this is not a theorem is that the quantum root vectors depend on a choice of reduced expression for \(\wz\), while \(U(w)\) does not.

\subsection{Twisted quantum Schubert cells}
\label{sec:twisted-quantum-schubert-cells}

\index[term]{quantum Schubert cell!twisted}
Twisted quantum Schubert cells were introduced by Zwicknagl in \cite{Zwi09}*{\S 5} in connection with quantum symmetric algebras.
We will discuss this connection further in \cref{sec:quantum-sym-algs-as-quantum-schubert-cells}.

\begin{dfn}
  \label{dfn:twisted-quantum-schubert-cell}
  For an arbitrary element \(w \in W\), the \emph{twisted quantum Schubert cell} \(U'(w)\)\index[notn]{Uprimew@\(U'(w)\)} is defined by
  \begin{equation}
    \label{eq:twisted-quantum-schubert-cell-dfn}
    U'(w) = T_{\wz w^{-1}} U(w).
  \end{equation}
  (Recall from \cite{Jan96}*{\S 8.18} that \(T_{uv} = T_u T_v \) if \(\ell(uv) = \ell(u) + \ell(v)\), but this is not the case in general; in particular, \(\ell(w_0 u) \neq \ell(w_0) + \ell(u)\) unless \(u = e\).)
\end{dfn}

The following example examines the situation when \(w = \wl\) is the parabolic element associated to a standard parabolic subalgebra \(\fp \subseteq \fg\):

\begin{eg}
  \label{eg:twisted-quantum-schubert-cell-for-parabolic}
  Let \(\fp \subseteq \fg\) be the standard parabolic subalgebra associated to a subset of simple roots \(\cS \subseteq \simprts(\fg)\), and let \(\Wl,\wzl,\wl\) be as defined in \cref{sec:weyl-group-of-standard-parabolic}.
  Taking \(w = \wl\), we have \(\wz \wl^{-1} = \wzl\), and so in this instance the twisted quantum Schubert cell \(U'(\wl)\) is given by
  \begin{equation}
    \label{eq:twisted-quantum-schubert-cell-for-parabolic}
    U'(\wl) = T_{\wzl} U(\wl).
  \end{equation}
  Now, as in \eqref{eq:reduced-expression-for-longest-word}, we fix a reduced decomposition 
  \begin{equation*}
    \wz = s_{i_1} \dots s_{i_M} s_{i_{M+1}} \dots s_{i_{M+N}}
  \end{equation*}
  for the longest word that is compatible with the factorization \(w_0 = \wzl \wl\) in the sense that \(\wzl = s_{i_1} \dots s_{i_M}\) and \(\wl = s_{i_{M+1}} \dots s_{i_{M+N}}\).
  According to \cref{dfn:quantum-schubert-cells}, the quantum Schubert cell \(U(\wl)\) is generated by the elements
  \begin{equation}
    \label{eq:gens-of-quantum-schubert-cell}
    X_k \eqdef T_{i_{M+1}} \dots T_{i_{M+k-1}}(E_{i_{M+k}}), \quad 1 \leq k \leq N.
  \end{equation}
  Hence the twisted quantum Schubert cell \(U'(\wl)\) is generated by the elements \(T_{\wzl}(X_k)\).

  Recall from \eqref{eq:radical-root-recipe} that we defined
  \[
  \xi_k \eqdef s_{i_1} \dots s_{i_M} s_{i_{M+1}} \dots s_{i_{k-1}}(\alpha_{i_k}) = \wzl s_{i_{M+1}} \dots s_{i_{k-1}}(\alpha_{i_k}),
  \]
  and from \cref{lem:radical-roots-recipe} that \(\rtsys(\up) = \{ \xi_1,\dots,\xi_N \}\).
  Then according to the definition \eqref{eq:quantum-root-vectors-def} of the quantum root vectors, we have
  \begin{equation}
    \label{eq:quantum-root-vectors-for-radical-roots}
    E_{\xi_k} = T_{i_1} \dots T_{i_M}T_{i_{M+1}} \dots T_{i_{M+k-1}}(E_{i_{M+k}}) = T_{\wzl}(X_k).
  \end{equation}
  Thus the twisted quantum Schubert cell \(U'(\wl) = T_{\wzl} U(\wl)\) is generated by the quantum root vectors \(\{ E_{\xi_1}, \dots, E_{\xi_N} \}\), and these are the quantum root vectors associated to the roots in \(\rtsys(\up)\).  
\end{eg}

\chapter{Quantum symmetric and exterior algebras}
\label{chap:quantum-symmetric-and-exterior-algebras}

As the name suggests, this chapter deals with quantum analogues of the symmetric and exterior algebras associated to a representation of a complex semisimple Lie algebra.
In \cref{sec:quantum-symmetric-preliminaries} we discuss some preliminary notions related to coboundary structures.
Then in \cref{sec:defining-quantum-symmetric-and-exterior-algebras} we define the quantum symmetric and exterior algebras, along with analogues of the spaces of symmetric and alternating tensors.
In \cref{sec:canonical-bases-and-continuity} we use canonical bases to show that much of the infrastructure related to the quantized enveloping algebras \(\uqg\)--representations, braidings, coboundary structures--is ``continuous'' in the deformation parameter \(q\).
In \cref{sec:quantum-symmetric-algebras-are-commutative} we use these continuity results to prove that the quantum symmetric and exterior algebras are ``commutative'' in a precise sense, and that they have universal properties analogous to those of classical symmetric and exterior algebras.
\cref{sec:flatness-and-related-properties} deals with the question of flatness, i.e.~when the graded components of the quantum symmetric and exterior algebras have the same dimensions as their classical counterparts.
In the final \cref{sec:collapsing-in-degree-three}, we show that the degree three components of these algebras display the same amount of collapsing from the classical case.

\section{Preliminaries}
\label{sec:quantum-symmetric-preliminaries}

For \cref{sec:quantum-symmetric-preliminaries} only, let \((\cC, \otimes, 1_\cC)\) denote an arbitrary monoidal category; we suppress the associator and unitor isomorphisms.

\subsection{Some isomorphisms in coboundary categories}
\label{sec:coboundary-categories}

\index[term]{coboundary category}
In \cref{sec:coboundary-structure} we defined a coboundary category and showed that polar decomposition of the braidings leads to a coboundary structure on \(\oq\).
Following \cite{HenKam06}*{\S 3.1}, here we describe for general \(\cC\) some maps that can be built from compositions and tensor products of commutors and identity maps.
Let \((\sigma_{XY})_{X,Y \in \cC}\) be a coboundary structure on \(\cC\) as in \cref{dfn:coboundary-category}.

For any objects \(A_1, \dots, A_n\) in \(\cC\) and  \(1 \le p \le r < t \le n\), define an isomorphism 
\index[notn]{sigmaprt@\(\sigma_{p,r,t}\)}
\begin{multline}
  \label{eq:sigma-prq-def}
  \sigma_{p,r,t} = \id \otimes \sigma_{(A_p \dots A_r),(A_{r+1} \dots A_t)} \otimes \id     : \\ 
  A_1 \dots A_{p-1}(A_p \dots A_r)(A_{r+1} \dots A_t)A_{t+1} \dots A_n \\ 
  \to A_1 \dots A_{p-1}(A_{r+1} \dots A_t)(A_p \dots A_r)A_{t+1} \dots A_n,
\end{multline}
where we have omitted the \(\otimes\) symbols between the \(A_i\) for readability.
With the \(\sigma_{p,r,t}\) as building blocks, we recursively define isomorphisms 
\[
s_{p,t} : A_1 \dots A_{p-1}(A_p \dots A_t)A_{t+1} \dots A_n \to A_1 \dots A_{p-1}(A_t \dots A_p)A_{t+1} \dots A_n
\]
by
\index[notn]{spt@\(s_{p,t}\)}
\begin{equation}
  s_{p,p+1} = \sigma_{p,p,p+1}, \quad s_{p,t} = \sigma_{p,p,t} \circ s_{p+1,t} \text{ for \(t-p>1\).}  \label{eq:s-pq-def}
\end{equation}
The naturality properties of the coboundary structure imply that the \(\sigma_{p,r,t}\) and the \(s_{p,t}\) also satisfy appropriate naturality conditions with respect to morphisms in the category \(\cC\).

\subsection{The cactus group}
\label{sec:cactus-group}

In a symmetric monoidal category, the symmetric group \(S_n\) acts on the \(n\)-fold tensor power \(V^{\otimes n}\) of any object \(V\).
In a braided monoidal category, the \(n\)-strand braid group \(B_n\) acts on \(V^{\otimes n}\).
In a coboundary category, the symmetry group that acts is known as the \emph{\(n\)-fruit cactus group}\index[term]{cactus group}.
This name comes from the interpretation of the group as the fundamental group of a moduli space of diagrams resembling cacti; see \cite{HenKam06}.

\begin{dfn}
  \label{dfn:cactus-group}
  The \emph{\(n\)-fruit cactus group} \(J_n\) is the abstract group generated by elements \(s_{p,t}\) for \(1 \le p < t \le n\) with relations
  \begin{enumerate}[(a)]
  \item \(s_{p,t}^2 = 1\);
  \item \(s_{p,t}s_{k,l} = s_{k,l}s_{p,t}\) if \(p<t\) and \(k<l\) are disjoint, i.e. if \(t < k\) or \(l < p\);
  \item \(s_{p,t}s_{k,l} = s_{k',l'}s_{p,t}\) if \(p \le k < l \le t\), where \(k',l'\) are determined by \(k+l' = p+t = l+k' \).
  \end{enumerate}
\end{dfn}

There is a homomorphism \(J_n \to S_n\), denoted \(x \mapsto \hat{x}\), determined by 
\begin{equation}
  \label{eq:cactus-gp-hom-to-symmetric-gp}
  s_{p,t} \mapsto \hat{s}_{p,t} =
  \begin{pmatrix}
    1 & \dots & p-1 & p & \dots & t & t+1 & \dots & n\\
    1 & \dots & p -1 & t & \dots & p & t+1 & \dots & n
  \end{pmatrix},
\end{equation}
i.e.\ the involutive permutation which reverses the interval from \(p\) to \(t\).

\begin{lem}
  \label{lem:cactus-group-action-on-coboundary-category}
  The isomorphisms \(s_{p,t}\) in \(\cC\) defined in \eqref{eq:s-pq-def} satisfy the relations from \cref{dfn:cactus-group}.
  Hence for any object \(V \in \cC\), the cactus group \(J_n\) acts on \(V^{\otimes n}\) via morphisms in the category.
\end{lem}

\begin{proof}
  This is part of the content of Lemmas 3 and 4 of \cite{HenKam06}.
\end{proof}

\begin{rem}
  \label{rem:on-the-cactus-group}
  Like the braid group \(B_n\), the cactus group \(J_n\) has the symmetric group \(S_n\) as a quotient.
  As indicated above, it also can be viewed as the fundamental group of a moduli space of certain geometric diagrams.
  However, while much is known about the braid group (see \cite{KasTur08}, for example), comparatively little is known about the cactus group.
\end{rem}

\subsection{A coboundary structure on super-representations}
\label{sec:coboundary-super-representations}

In order to treat quantum symmetric and exterior algebras on the same footing, we show how to extend the coboundary structure on \(\oq\) to the corresponding category of super-representations.\index[term]{super-representations}
The construction can also be performed, \emph{mutatis mutandis}, to extend the braiding.
In more generality, this construction can be performed in any pre-additive coboundary (or braided) monoidal category with finite coproducts (direct sums) in which the coproduct distributes over the tensor product.

\begin{dfn}
  \label{dfn:super-version-of-category}
  We define the category \(\soq\)\index[notn]{soq@\(\soq\)} to be the Cartesian product \(\oq \times \oq\), i.e.\ the category whose objects are \(V = (V_0,V_1)\) of objects of \(\oq\) and whose morphisms are pairs \(f = (f_0,f_1)\) of morphisms of \(\oq\).
  The monoidal structure of \(\oq\) leads to one on \(\soq\) as follows: the tensor product  \(V \wot W\) of two objects \(V=(V_0,V_1)\) and \(W=(W_0,W_1)\) of \(\soq\) is given by
  \[  
  (V \wot W)_0 \eqdef  (V_0\otimes W_0)\oplus (V_1\otimes W_1), \quad (V \wot W)_1 \eqdef (V_0\otimes W_1)\oplus (V_1 \otimes W_0),
  \]
  and the tensor product \(f \wot g\) of two morphisms \(f = (f_0,f_1)\) and \(g = (g_0,g_1)\) is
  \[
  (f \wot g)_0 \eqdef (f_0 \otimes g_0) \oplus (-f_1 \otimes g_1), \quad (f \wot g)_1 \eqdef (f_0 \otimes g_1) \oplus (f_1 \otimes g_0).
  \]
  In other words, the objects of \(\soq\) are nothing but the \(\bbZ/2\)-graded, or super-objects of \(\oq\), and we will refer to them as such, often identifying \((V_0,V_1)\in\soq\) with the object \(V_0\oplus V_1\in \oq\), together with the information of its decomposition into an even part \(V_0\) and odd part \(V_1\).
  For any \(V \in \oq\), there are even and odd incarnations of \(V\) in \(\soq\), namely \index[notn]{Vev@\(V_\e,V_\o\)}
  \[
  V_\e \eqdef (V,0) \quad \text{and} \quad V_\o \eqdef (0,V).
  \]
\end{dfn}

Now we define the commutors on \(\soq\):
\begin{dfn}
  \label{dfn:commutors-on-super-representations}
  For objects \(V=(V_0,V_1)\) and \(W=(W_0,W_1)\) of \(\soq\), the commutor \(\sigma_{VW} : V \wot W \to W \wot V\) is defined by its restriction to the summands \(V_i \otimes W_j\) as
  \[
  \sigma_{VW}\mid_{V_i \otimes W_j} \eqdef (-1)^{ij} \sigma_{V_i W_j} : V_i \otimes W_j \to W_j \otimes V_i
  \]
  for \(i,j\in \{ 0,1 \}\), where \(\sigma_{V_i W_j}\) is the commutor in \(\oq\).
\end{dfn}

With these definitions, we have:
\begin{lem}
  \label{lem:super-representation-stuff}
  With the tensor product as in \cref{dfn:super-version-of-category}, \(\soq\) is a monoidal category with monoidal unit \(\bbC_\e\), where \(\bbC\) is the trivial representation in \(\oq\).
  Moreover, with the commutors \(\sigma_{VW}\) as in \cref{dfn:commutors-on-super-representations}, \(\soq\) is a coboundary category.
\end{lem}
\begin{proof}
  Chasing the signs, this follows immediately from the corresponding facts for \(\oq\).
\end{proof}

\section{Definitions of the algebras}
\label{sec:defining-quantum-symmetric-and-exterior-algebras}

In this section we recall the definitions of quantum symmetric and exterior algebras associated to modules in \(\oq\).
Certain examples of quantum symmetric algebras have been studied for some time, namely quantum polynomial and matrix algebras, along with their orthogonal and symplectic counterparts; see \cite{KliSch97}*{Ch.~9}.
These turn out to be the quantum symmetric algebras associated to the defining representations of \(\uqg\) for \(\fg\) belonging to the four infinite series of simple Lie algebras, and for \(\fg = \fsl_m \times \fsl_n\) in the case of quantum matrices.
As far as I know, the first general definition of quantum symmetric and exterior algebras associated to arbitrary finite-dimensional \(\uqg\)-modules is \cite{BerZwi08}*{Definition 2.7}.

\subsection{Quantum symmetric tensors}
\label{sec:quantum-symmetric-tensors}

First we define the analogues of symmetric and antisymmetric tensors for objects in \(\oq\):
\index[term]{symmetric tensors}
\index[term]{antisymmetric tensors}

\begin{dfn}[\cite{BerZwi08}*{Definition 2.1}]
  \label{dfn:symmetric-and-antisymmetric-tensors}
  Let \(V \in \oq\) and \(n \geq 2\).
  The space of \emph{symmetric \(n\)-tensors} \(\symq^n V\)\index[notn]{snqv@\(\symq^nV\)} and the space of \emph{antisymmetric \(n\)-tensors} \(\extq^n V\)\index[notn]{lambdanqv@\(\extq^nV\)} are defined as
  \begin{equation}
    \label{eq:symvectors}
    \begin{gathered}
      \symq^n V = \{ v \in V^{\otimes n} \mid s_{i,i+1} v = v \text{ for } 1 \le i \le n-1  \},\\
      \extq^n V = \{ v \in V^{\otimes n} \mid s_{i,i+1} v = -v \text{ for } 1 \le i \le n-1  \},
    \end{gathered}
  \end{equation}
  where \(s_{i,i+1}=\sigma_{i,i,i+1}\) is as defined in \cref{sec:coboundary-categories}, namely the commutor \(\sigma_{VV}\) acting in the \(i\) and \(i+1\) slots and the identity in all others.
\end{dfn}

\begin{lem}
  \label{lem:on-symmetric-tensors}
  With definitions as above, we have:
  \begin{enumerate}[(a)]
  \item The tensor square \(V \otimes V\) decomposes as \(V \otimes V = \symq^2 V \oplus \extq^2 V\).
  \item For any \(n \geq 2\), both \(\symq^n V\) and \(\extq^n V\) are \(\uqg\)-submodules of \(V^{\otimes n}\).
  \end{enumerate}  
\end{lem}

\begin{proof}
  Part (a) follows immediately from the fact that \(\sigma_{VV}\) is involutive.
  As \(\sigma_{VV}\) is a module map, this shows that both \(\symq^2 V\) and \(\extq^2 V\) are submodules of \(V \otimes V\).
  Then \(\symq^n V\) is a submodule of \(V^{\otimes n}\) for \(n > 2\) because it can be written as the intersection of submodules
  \begin{equation}
    \label{eq:sym-n-tensors-as-intersection}
    \symq^n V = \bigcap_{j=1}^{n-1} V^{\otimes (j-1)} \otimes \symq^2 V \otimes V^{\otimes (n-j-1)},
  \end{equation}
  and similarly for \(\extq^n V\).
\end{proof}

\subsection{Quantum symmetric algebras}
\label{sec:quantum-symmetric-algebras}

The quantum symmetric and antisymmetric tensors defined above are \emph{subobjects} of \(V^{\otimes n}\).
In contrast to this, and in accordance with the classical situation, the quantum symmetric and exterior algebras of \(V\) are defined as \emph{quotients} of the tensor algebra of \(V\).
\index[term]{quantum symmetric algebra}
\index[term]{quantum exterior algebra}

\begin{dfn}[\cite{BerZwi08}*{Definition 2.7}]
  \label{dfn:quantum-symmetric-and-exterior-algebras}
  Let \(V \in \oq\).
  The \emph{quantum symmetric algebra} \(\symq(V)\)\index[notn]{snv@\(\symq(V)\)} and the \emph{quantum exterior algebra} \(\extq(V)\)\index[notn]{lambdaqv@\(\extq(V)\)} are defined as
  \begin{equation}
    \label{eq:quantum-symmetric-and-exterior-algebras}
    \symq(V) \eqdef T(V) / \langle \extq^2V \rangle \quad \text{and} \quad
    \extq(V) \eqdef T(V) / \langle \symq^2V \rangle.
  \end{equation}
  As the defining ideals are homogeneous and generated by \(\uqg\)-submodules of \(V \otimes V\), both \(\symq(V)\) and \(\extq(V)\) are quadratic \(\uqg\)-module algebras.
  We denote the homogeneous components of \(\symq(V)\) and \(\extq(V)\) by \(\symq^n(V)\) and \(\extq^n(V)\), respectively, for \(n \in \zp\).
  We denote multiplication in \(\symq(V)\) and \(\extq(V)\) by juxtaposition and by the symbol \(\wedge\), respectively.
\end{dfn}

\begin{notn}
  \label{notn:subobject-vs-quotient}
  We emphasize that \(\symq^n V\), with no parentheses, is a submodule of \(V^{\otimes n}\), while \(\symq^n(V)\) is a quotient of \(V^{\otimes n}\).
  The constructions \(V \mapsto \symq^n V\) and \(V \mapsto \symq^n(V)\) are functorial in \(V\), and similarly for \(\extq^nV\) and \(\extq^n(V)\).

  We note also that \cite{BerZwi08} uses the notation \(S^n_\sigma V\), \(\ext^n_\sigma V\), rather than \(\symq^n V\), \(\extq^n V\), etc.  
  We emphasize the dependence on the parameter \(q\) because later we will need to consider what happens as \(q\) varies.
\end{notn}

\begin{rem}
  \label{rem:on-defs-of-sym-and-ext-algebras}
  According to \cref{prop:coboundary-struct-on-uqg-modules}(c), \(\ker(\sigma_{VV} - \id)\) is the span of the eigenspaces of the braiding \(\rhat_{VV}\) for \emph{positive} eigenvalues, i.e.~those of the form \(q^a\), while \(\ker(\sigma_{VV}+\id)\) is the span of eigenspaces for \emph{negative} eigenvalues, i.e.~those of the form \(-q^a\).
  (Note that this notion makes sense because \(q\) is not a root of unity.)
  Thus, speaking informally, the definitions of the quantum symmetric and exterior algebras become the classical ones when \(q=1\).
\end{rem}  

\begin{eg}
  \label{eg:quantum-symmetric-algebra-for-2-dim-rep}
  Take \(\fg = \fsl_2\), let \(V = \bbC^2\) be the two-dimensional irreducible representation of \(\uqsl\), and let \(\{ x_1,x_2 \}\) be the standard basis.
  With respect to this basis, the generators \(E,F\), and \(K = K_\alpha\) of \(\uqsl\) act on \(V\) via
  \renewcommand\arraystretch{0.6}
  \begin{equation}
    \label{eq:generators-of-uqsl2-in-2dim-rep}
    E \mapsto
    \begin{pmatrix}
      0 & 1\\ 0 & 0
    \end{pmatrix}, \quad 
    F \mapsto
    \begin{pmatrix}
      0 & 0\\ 1 & 0
    \end{pmatrix}, \quad 
    K \mapsto
    \begin{pmatrix}
      q & 0\\ 0 & q^{-1}
    \end{pmatrix}. 
  \end{equation}
  With respect to the (lexicographically ordered) tensor product basis for \(V \otimes V\), the braiding \(\rhat_{VV}\) and the commutor \(\sigma_{VV}\) are given by
  \renewcommand\arraystretch{0.8}
  \begin{equation}
    \label{eq:braiding-and-commutor-for-2-dim-rep}
    \rhat_{VV} = 
    \begin{pmatrix}
      q^{\frac12} & 0 & 0 & 0 \\
      0 & q^{\frac12} - q^{-\frac32} & q^{-\frac12} & 0 \\
      0 & q^{-1/2} & 0 & 0 \\
      0 & 0 & 0 & q^{\frac12} \\
    \end{pmatrix}, \qquad
    \sigma_{VV} =
    \begin{pmatrix}
      1 & 0 & 0 & 0 \\
      0 & \frac{q^2-1}{1+q^2} & \frac{2 q}{1+q^2} & 0 \\
      0 & \frac{2 q}{1+q^2} & \frac{1-q^2}{1+q^2} & 0 \\
      0 & 0 & 0 & 1 \\
    \end{pmatrix}.
  \end{equation}
  The relations for the quantum symmetric and exterior algebras of \(V\) are obtained by diagonalizing either \(\rhat_{VV}\) or \(\sigma_{VV}\).
  We obtain:
  \[
  \symq^2 V = \spn_{\bbC} \{ x_1 \otimes x_1, q x_1 \otimes x_2 + x_2 \otimes x_1, x_2 \otimes x_2 \},
  \]
  while
  \[
  \extq^2 V = \spn_{\bbC} \{ -q^{-1} x_1 \otimes x_2 + x_2 \otimes x_1 \}.
  \]
  Thus the quantum symmetric and exterior algebras are given by
  \[
  \symq(V) = \bbC \langle x_1,x_2 \rangle / \langle x_2 x_1 = q^{-1}x_1x_2 \rangle
  \]
  and
  \[
  \extq(V) = \bbC \langle x_1,x_2 \rangle / \langle x_1 \wedge x_1= x_2\wedge x_2 = 0, x_2 \wedge x_1 = -q x_1 \wedge x_2 \rangle,
  \]
  respectively.
  Manin used these algebras as the starting point for his approach to noncommutative geometry via quadratic algebras; see \cite{Man88}*{\S 1} or \cite{KliSch97}*{\S 4.1.3}.
  The algebra \(\symq(V)\) is often referred to as the \emph{quantum plane}. \index[term]{quantum plane}
\end{eg}

\begin{rem}
  \label{rem:quantum-exterior-algebras-are-quantum-symmetric-algebras}
  The construction of the quantum symmetric algebra makes sense for any object in an abelian coboundary category.
  In particular, it makes sense in \(\soq\).
  In that case, for an object \(V \in \oq\), the quantum symmetric algebra of \(V_\o \in \soq\) is exactly \(\extq(V)\), together with the \(\bbZ/2\)-grading given by the parity of the homogeneous components.
\end{rem}

\subsection{Quadratic duality}
\label{sec:quadratic-duality}

\index[term]{quadratic dual algebra}

Proposition 2.11(c) of \cite{BerZwi08} states that for \(V \in \oq\) the algebras \(\symq(V)\) and \(\extq(V^\ast)\) are quadratic duals of one another.
As this will be important for us in \cref{sec:quantum-exterior-algebras-are-frobenius} below, and since the definition of quadratic duality in \cite{PolPos05}*{Ch.~1, \S 2} must be modified slightly (for reasons given below in \cref{rem:on-quadratic-duality}), we go into some detail on this point.

\begin{dfn}
  \label{dfn:quadratic-dual-algebra}
  \index[notn]{Ashriek@\(A^{"!}\)}
  Let \(A = T(V)/\langle R \rangle\) be a quadratic algebra, where \(V\) is a finite-dimensional vector space and \(R \subseteq V \otimes V\) is a subspace.
  The \emph{quadratic dual algebra} \(A^!\) is the algebra
  \[
  A^! \eqdef T(V^\ast) / \langle R^\circ \rangle,
  \]
  where \(R^\circ \subseteq V^\ast \otimes V^\ast\) is the \emph{annihilator} of \(R\), defined as
  \[
  R^\circ \eqdef \{ \phi \in V^{\ast} \otimes V^{\ast} \mid \langle \phi, R \rangle =0\},
  \]
  and where the pairing between \(V \otimes V\) and \(V^\ast \otimes V^\ast\) is the one defined in \cref{lem:dual-of-tensor-product}.
\end{dfn}

\begin{rem}
  \label{rem:on-quadratic-duality}
  This definition of \(A^!\) differs from that in \cite{PolPos05} only in the definition of the pairing between \(V \otimes V\) and \(V^\ast \otimes V^\ast\).
  Polishchuk and Positselski use the other natural convention for the pairing, namely \(\langle f \otimes g, u \otimes v \rangle = f(u)g(v) \).
  This pairing still induces an isomorphism of \emph{vector spaces} \(V^\ast \otimes V^\ast \cong (V \otimes V)^\ast\), but not an isomorphism of \emph{modules}.
  Since we want to remain within the category of \(\uqg\)-modules, we are forced to make the other choice.
  The effect of making this choice is that our algebra \(A^!\) is the \emph{opposite} of the algebra denoted \(A^!\) as defined in \cite{PolPos05}.
\end{rem}

\begin{prop}
  \label{prop:quadratic-duality-of-quantum-symmetric-and-exterior-algebras}
  Let \(V \in \oq\).
  With all notation as in \cref{dfn:quadratic-dual-algebra}, we have
  \begin{enumerate}[(a)]
  \item \((\extq^2V)^\circ = \symq^2 V^\ast\) and \((\symq^2V)^\circ = \extq^2V^\ast\).
  \item \(\symq(V)^! = \extq(V^\ast)\) and \(\extq(V)^! = \symq(V^\ast)\).
  \end{enumerate}
\end{prop}

\begin{proof}
  Part (b) follows immediately from part (a) and the definition of the quadratic dual, so it is left to prove (a).
  We show only that \((\extq^2V)^\circ = \symq^2 V^\ast\), as the other equality is similar.

  Denote \(\sigma = \sigma_{VV}\).
  The point here is that, when we identify \((V \otimes V)^\ast\) with \(V^\ast \otimes V^\ast\) as in \cref{lem:dual-of-tensor-product}, we have \(\sigma^{\tr} = \sigma_{V^\ast V^\ast}\) by \cref{prop:coboundary-struct-on-uqg-modules}(a).
  Since \(\sigma\) is involutive, we have
  \[
  \extq^2V \eqdef \ker(\sigma + \id) = \ran(\sigma - \id).
  \]
  Hence for \(\phi \in V^\ast \otimes V^\ast\) we have
  \begin{align*}
    \phi \in (\extq^2V)^\circ & \iff \phi \circ (\id - \sigma) = 0\\
    & \iff (\id - \sigma)^{\tr}(\phi) = 0\\
    & \iff (\id - \sigma^{\tr})(\phi) = 0\\
    & \iff \phi \in \ker(\id - \sigma_{V^\ast V^\ast}) = \symq^2 V^\ast.
  \end{align*}
\end{proof}

\subsection{Symmetrization}
\label{sec:symmetrization}

Classically, for a vector space \(V\) one has the symmetrizer
\begin{equation}
  \label{eq:symmetrization-map}
  \mathrm{Sym}_n(v_1 \otimes \dots \otimes v_n) = \frac{1}{n!} \sum_{\pi \in S_n} v_{\pi(1)} \otimes \dots \otimes  v_{\pi(n)},
\end{equation}
which is a projection of \(V^{\otimes n}\) onto the space \(\sym^n V\) of symmetric \(n\)-tensors.
Using the symmetrizers, one can embed the symmetric algebra \(S(V)\) into \(T(V)\) and show that the graded components of the symmetric algebra are isomorphic to the corresponding spaces of symmetric tensors.

In the quantum setting, this seems to be more difficult.
Unless the braiding of \(V\) with itself satisfies a Hecke-type relation (e.g.~\((\rhat_{VV}-q)(\rhat_{VV}+q^{-1})=0\), which happens when \(V\) is the fundamental \(n\)-dimensional representation of \(U_q(\fsl_n)\)), there does not seem to be a general formula in terms of the braiding (or the commutor) for the projection of \(V^{\otimes n}\) onto \(\symq^n V\).
Nonetheless, in this section we show that there is a canonical isomorphism  \(\symq^nV \cong \symq^n(V)\) without exhibiting an explicit symmetrizer map.
This addresses Question 2.12 of \cite{BerZwi08}.

In this section we take \(q > 0\) and we consider the compact real form of \(\uqg\) as in \cref{sec:quea-compact-real-form}.
Recall that an invariant inner product \(\langle \cdot,\cdot \rangle\) on a module \(V \in \oq\) determines an invariant inner product on each \(V^{\otimes n} \) by
\begin{equation}
  \label{eq:tensor-product-inner-product}
  \langle v_1 \otimes \dots \otimes v_n, w_1 \otimes \dots \otimes w_n \rangle \eqdef \langle v_1,w_1 \rangle \dots \langle v_n,w_n \rangle.
\end{equation}

\begin{lem}
  \label{lem:symmetric-and-antisymmetric-tensors-are-orthogonal}
  Let \(q > 0\), let \(V \in \oq\), and choose a \(\uqg\)-invariant inner product on \(V\) as discussed in \cref{sec:representations-of-quea}.
  Then we have \((\symq^2 V)^\perp = \extq^2V\) with respect to the inner product on \(V \otimes V\) determined by \eqref{eq:tensor-product-inner-product}.
\end{lem}

\begin{proof}
  Denoting \(\sigma = \sigma_{VV}\), we have
  \[
  (\symq^2V)^\perp = \ker(\sigma - \id)^\perp =  \ran(\sigma^\ast-\id) = \ran(\sigma - \id) = \extq^2V;
  \]
  here we used the fact that \(\ker(T)^\perp = \ran(T^\ast)\) in any finite-dimensional Hilbert space, together with the self-adjointness of \(\sigma\) from \cref{prop:coboundary-struct-on-uqg-modules}(c).
\end{proof}

This allows us to prove the main result of this section:

\begin{prop}
  \label{prop:symmetrization-and-antisymmetrization}
  Let \(q > 0\) and let \(V \in \oq\).
  For each \(n \geq 2\), the natural composite morphisms
  \begin{equation}
    \label{eq:symmetrization-and-antisymmetrization}
    \symq^nV \hookrightarrow V^{\otimes n} \twoheadrightarrow \symq^n(V) \quad \text{and} \quad     \extq^nV \hookrightarrow V^{\otimes n} \twoheadrightarrow \extq^n(V)
  \end{equation}
  are isomorphisms in \(\oq\).
\end{prop}

\begin{proof}
  We prove the statement in the symmetric case; the second case is analogous.
  Denote the defining ideal of \(\symq(V)\) by \(J = \langle \extq^2V \rangle\) and set \(J_n = J \cap V^{\otimes n}\), so that \(\symq^n(V) = V^{\otimes n}/J_n\).
  We will show that \(V^{\otimes n} = \symq^n V \oplus J_n\), which establishes the result.
  Indeed, we have
  \[
  J_n = \sum_{j=1}^{n-1} V^{\otimes (j-1)} \otimes \extq^2 V \otimes V^{\otimes (n-j-1)},
  \]
  so that with respect to the inner product \eqref{eq:tensor-product-inner-product} we have
  \[
  J_n^\perp = \bigcap_{j=1}^{n-1} V^{\otimes (j-1)} \otimes (\extq^2 V)^\perp \otimes V^{\otimes (n-j-1)} = \bigcap_{j=1}^{n-1} V^{\otimes (j-1)} \otimes \symq^2 V \otimes V^{\otimes (n-j-1)} = \symq^n V.
  \]
  Hence \(V^{\otimes n} = \symq^n V \oplus J_n\), as claimed.
\end{proof}

\section{Canonical bases and continuity}
\label{sec:canonical-bases-and-continuity}

In this section we establish some preparatory results that will be needed to prove our main results in \cref{sec:quantum-symmetric-algebras-are-commutative,sec:collapsing-in-degree-three}.
In those sections we employ limiting arguments: we consider a module \(V\) in \(\oq\) and consider what happens as \(q \to 1\).
In order to make rigorous sense of this notion, we need to find a common vector space \(V\) on which \(\uqg\) can act for \emph{all} values of \(q\) simultaneously, and we need to understand the extent to which these actions are continuous as a function of \(q\).
We use Lusztig's notion of canonical basis \index[term]{canonical basis} as a tool in constructing this universal model for the representations.
The process is outlined in \cite{Wen98}*{\S 2.1}; for clarity we present more detail here.

In \cref{sec:formal-and-integral-versions} we introduce analogues of the quantized enveloping algebra where the deformation parameter is a formal variable rather than a complex number.
In \cref{sec:canonical-bases} we discuss canonical bases for representations, and then in \cref{sec:continuity-of-braidings,sec:continuity-of-coboundary-structure,sec:cactus-gp-action-on-q-symmetric-tensors} we use the canonical bases to address the question of continuity of the braidings and commutors.

\subsection{Formal and integral versions of the quantized enveloping algebra}
\label{sec:formal-and-integral-versions}

Denote by \(m\) the least positive integer such that \(m(\cP,\cP) \subseteq 2 \bbZ\).
Let \(\nu\) be a formal variable and \(\nu^{\frac1m}\) be an \(m^{\mathrm{th}}\) root, and define \(\nu_i = \nu^{d_i}\).\index[notn]{nu@\(\nu,\nu_i\)}
Define \(\uvq\) to be the \(\bbQ(\nu^{\frac1m})\)-algebra with generators \(E_i,F_i\), and \(K_\lambda\) with relations obtained by replacing \(q\) with \(\nu\) in all the relations for \(\uqg\); quantum numbers and factorials (and hence divided powers) are defined identically for \(\nu\) as for \(q\).
The comultiplication and counit of \(\uvq\) are defined by the same formulae as for \(\uqg\).
The braid group \(\fbg\) acts on \(\uvq\), and the quantum root vectors and their divided powers are defined as in \(\uqg\).
Finally, \(\uvz\)\index[notn]{Unug@\(\uvq,\uvz\)} denotes the \(\bbZ[\nu^{\frac1m},\nu^{-\frac1m}]\)-subalgebra of \(\uvq\) generated by the divided powers \(E_i^{(n)},F_i^{(n)}\) and the elements \(K_\lambda\).

\begin{conv}
  \label{conv:roots-of-formal-variable}
  From now on we will write \(\Qv\) and \(\Zv\) in place of  \(\bbQ(\nu^{\frac1m})\) and \(\bbZ[\nu^{\frac1m},\nu^{-\frac1m}]\), tacitly understanding that all of the relevant roots of the variable \(\nu\) are available to us when needed.
\end{conv}

We denote by \(\ov\) the category of Type 1 representations of \(\uvq\), defined analogously to the category \(\oq\) for \(\uqg\).
\index[notn]{Ov@\(\ov\)}
This category is monoidal (with tensor products taken over \(\bbQ(\nu)\)) and admits braided and coboundary structures analogous to those for \(\oq\); we will discuss this in some detail in \cref{sec:canonical-bases-and-continuity}.
See \cite{ChaPre95}*{Ch.~10} for more information on quantized enveloping algebras and their representations for formal parameters.

\subsection{Canonical bases and universal models for representations}
\label{sec:canonical-bases}

We begin our construction of our universal model for an arbitrary representation of \(\uqg\) by constructing a universal model for a simple representation \(V(\lambda)\).
\index[term]{universal model for a representation}

Our goal is to show that for any \(q > 0\) we can realize the simple \(\uqg\)-module of highest weight \(\lambda\) on the vector space \(\vlc\), and moreover that as \(q \to 1\) we obtain the corresponding classical representation of \(U(\fg)\).
We use the technique of \emph{specialization}\index[term]{specialization} extensively, i.e.\ we work with the forms of \(\uqg\) defined over \(\Qv\) and \(\Zv\), and then replace the formal parameter \(\nu\) with the complex number \(q\).
To ensure that this process is well-defined, we show that the expressions to be specialized are Laurent polynomials in \(\nu\).
The following definition is helpful:

\begin{dfn}
  \label{dfn:action-by-laurent-polynomials}
    Let \(V\) be a \(\Qv\)-vector space with a distinguished basis \(B\), and let \(T : V \to V\) be a linear map.
  If \(T\) preserves the \(\Zv\)-span of \(B\) (or equivalently the matrix coefficients of \(T\) with respect to \(B\) lie in \(\Zv\)) then we will say that \(T\) \emph{acts on \(B\) by Laurent polynomials}.
  (Keep in mind \cref{conv:roots-of-formal-variable}.)
\end{dfn}

Now we introduce distinguished bases for the simple \(\uvq\)-modules:

\begin{dfn}
  \label{dfn:universal-v-lambda}
  For \(\lambda \in \cP^+\), we let \(\vlc\)\index[notn]{Vlambda@\(\vlc\)} be the complex vector space with basis \(B(\lambda)\), where \(B(\lambda)\)\index[notn]{Blambda@\(B(\lambda)\)} is the \emph{canonical basis} for the \(\uvq\)-module \(V(\lambda)\); here we use the conventions of \cite{Lus10}*{\S 14.4.12}, except that Lusztig writes \(B(\Lambda_\lambda)\).  
  (Recall that we defined \(\uvq\) in \cref{sec:formal-and-integral-versions}.)
  In \cite{Wen98} the pair \((\vlc,B(\lambda))\) is referred to as a \emph{Weyl module}\index[term]{Weyl module}.
\end{dfn}

The following lemma summarizes the properties of the canonical bases that we require:

\begin{lem}
  \label{lem:canonical-basis-lattice-preservation}
  Let \(\lambda \in \pplus\) and let \(V(\lambda)\), \(B(\lambda)\) be as in \cref{dfn:universal-v-lambda}.
  \begin{enumerate}[(a)]
  \item The \(\Zv\)-submodule of \(V(\lambda)\) generated by the canonical basis \(B(\lambda)\) is preserved by the integral form \(\uvz\) (see \cref{sec:formal-and-integral-versions}).
    In particular, the elements \(K_\mu\) for \(\mu \in \cP\) and the divided powers \(E_i^{(n)},F_i^{(n)}\) of the generators of \(\uvq\) act by Laurent polynomials on the canonical basis \(B(\lambda)\).
  \item For all \(\beta \in \posrts\), the divided powers \(E_{\beta}^{(n)},F_\beta^{(n)}\) of the quantum root vectors act by Laurent polynomials on \(B(\lambda)\).
  \item If we define elements \(H_j \in \uvq\) for \(1 \leq j \leq r\) by
    \begin{equation}
      \label{eq:Hj-definition}
      H_j \eqdef \frac{K_j - K_j^{-1}}{\nu_j - \nu_j^{-1}} \in \uvq,
    \end{equation}
    then \(H_j\) acts by Laurent polynomials on \(B(\lambda)\).
  \end{enumerate}
\end{lem}

\begin{proof}
  Part (a) follows from the fact that \((V(\lambda),B(\lambda))\) is a \emph{based module}; see \cite{Lus10}*{\S 27.1.2, \S27.1.4}.
  
  For part (b), let \(\beta \in \posrts\).
  Then \(E_\beta = T(E_i)\) for some \(i\), where \(T\) is an automorphism coming from some element of the braid group of \(\fg\).
  We have \(\nu_\beta = \nu_i\) and hence \(E_\beta^{(n)} = T(E_i^{(n)})\), and similarly for \(F_\beta\).
  The formulas in \cite{Lus10}*{\S 41.1.2} can be adapted to show that \(\uvz\) is preserved by the braid group action, and hence \(E_\beta^{(n)},F_\beta^{(n)} \in \uvz\).
  The result then follows from part (a).

  For part (c), note that \(H_j\) scales a weight vector of weight \(\mu\) by
  \[
  \frac{\nu^{(\alpha_j,\mu)} - \nu^{-(\alpha_j,\mu)}}{\nu_j - \nu_j^{-1}} = \frac{\nu_j^{(\alpha_j^\vee,\mu)} - \nu_j^{-(\alpha_j^\vee,\mu)}}{\nu_j - \nu_j^{-1}},
  \]
  which is a Laurent polynomial in \(\nu_j\), and hence in \(\nu\).
\end{proof}

\begin{notn}
  \label{notn:u-one-g}
  For uniformity of notation, we denote \(U_1(\fg) \eqdef U(\fg)\).
\end{notn}

We will now define an action of each \(\uqg\) on \(\vlc\) for \(q > 0\).
We need to treat the \(q \neq 1\) and \(q = 1\) cases separately because the generators \(K_i\) do not have analogues in \(U_1(\fg)\).

\begin{dfns}
  \label{dfns:specializations-of-uqnu-action}
  We make the following definitions:
  \begin{enumerate}[(1)]
  \item For \(1 \leq j \leq r\) and \(\mu \in \cP\), define matrices \(\ejln,\fjln,\kmuln,\hjln\) to be the matrix representations of the actions of \(E_j,F_j,K_\mu,H_j \in \uvq\) on \(V(\lambda)\) with respect to the canonical basis \(B(\lambda)\).
    By \cref{lem:canonical-basis-lattice-preservation}, the entries of these matrices are all in \(\Zv\). (Recall \cref{conv:roots-of-formal-variable}; in particular, the \(K_\mu\)'s will act by Laurent polynomials in an \(m^{\mathrm{th}}\) root of \(\nu\).)
  \item For \(q > 0\), for \(1 \leq j \leq r\) and \(\mu \in \cP\), define \(\ejlq,\fjlq,\kmulq,\hjlq\) to be the operators on \(\vlc\) whose matrix representations with respect to \(B(\lambda)\) are given by specializing \(\ejln,\fjln,\kmuln,\hjln\), respectively, at \(\nu = q\).  
    This is well-defined because, as we just noted, the entries of the latter matrices are Laurent polynomials in (a root of) \(\nu\).
  \end{enumerate}
\end{dfns}

Now we discuss continuity of the operators just defined:

\begin{lem}
  \label{lem:continuity-of-parametrized-operators}
  Let \(\lambda \in \pplus\).
  For each \(j\), the maps \(q \mapsto \ejlq\), \(q \mapsto \fjlq\), and \(q \mapsto \hjlq\) are continuous maps \((0,\infty) \to \End(\vlc)\).
  For \(\mu \in \cP\), the map \(q \mapsto \kmulq\) is also continuous, with \(\kmulq[1] = \id\).
\end{lem}

\begin{proof}
  Continuity is clear from the fact that the entries of the matrices are Laurent polynomials.
  The fact that \(\kmulq[1] = \id\) follows from the fact that \(K_\mu\) acts on weight vectors in \(V(\lambda)\) by powers of \(\nu\) and that \(B(\lambda)\) is a weight basis.
\end{proof}

We now show that these operators define representations of \(\uqg\) on the complex vector space \(\vlc\) for all \(q > 0\), including \(q=1\):

\begin{prop}
  \label{prop:universal-model-for-v-lambda}
  With notation as in \cref{dfns:specializations-of-uqnu-action}, we have:
  \begin{enumerate}[(a)]
  \item For \(q \neq 1\), the operators \(\ejlq,\fjlq,\kmulq\) turn the complex vector space \(\vlc\) into the irreducible representation of \(\uqg\) of highest weight \(\lambda\).
  \item For \(q = 1\), the operators \(\ejlq[1],\fjlq[1],\hjlq[1]\) turn the complex vector space \(\vlc\) into the irreducible representation of \(U_1(\fg)\) of highest weight \(\lambda\).
  \end{enumerate}
\end{prop}

\begin{proof}
  Part (a) is clear, as the relations in \(\uqg\) are the same as those in \(\uvq\) with \(\nu\) replaced by \(q\).
  For part (b), we need to work a little harder; we need to show that the operators \(\ejlq[1],\fjlq[1],\hjlq[1]\) satisfy the relations of \(U_1(\fg)\).
  
  It is clear that the operators \(\{\hjlq[1]\}_{j=1}^r\) commute, because the \(K_j \in \uvq\) commute.
  By the definition \eqref{eq:Hj-definition} of the elements \(H_j\) and the defining relations for \(\uvq\), we see that  \([E_i,F_j] = \delta_{ij}H_j\) in \(\uvq\), and hence also \([E_i^\lambda(1),\fjlq[1]] = \delta_{ij}\hjlq[1]\).
  Using the fact that quantum binomial coefficients specialize to their classical values, we see that the Serre relations for the \(\ejlq[1]\) and the \(\fjlq[1]\) follow immediately from the quantum Serre relations in \(\uvq\) upon specializing \(\nu\) to \(1\).

  Finally, we need to show that \([H_i^\lambda(1), \ejlq[1]] = a_{ij}\ejlq[1]\) and \([H_i^\lambda(1), \fjlq[1]] = - a_{ij}\fjlq[1]\).
  We prove just the first equality, as the second is analogous.
  Indeed, in \(\uvq\) we have
  \begin{align*}
    [H_i, E_j] & \eqdef \left[ \frac{K_i - K_i^{-1}}{\nu_i - \nu_i^{-1}}, E_j \right]\\
    & = \frac{\nu_i^{a_{ij}}-1}{\nu_i - \nu_i^{-1}} (E_jK_i + K_i^{-1}E_j)\\
         & = \frac{\nu_i^{2a_{ij}}-1}{(\nu_i - \nu_i^{-1})(\nu_i^{a_{ij}}+1)} (E_jK_i + K_i^{-1}E_j).
  \end{align*}
  Since \(\nu_i - \nu_i^{-1}\) divides \(\nu_i^{2a_{ij}}-1\) in \(\Zv\), the fraction is specializable at \(\nu = 1\), and its value is \(\frac{a_{ij}}{2}\).
  Since \(K_i^\lambda(1)=\id\), this gives the relation \([H_i^\lambda(1), \ejlq[1]] = a_{ij}\ejlq[1]\).

  So far we have shown that these operators do define a representation of \(U_1(\fg)\), and it is manifestly finite-dimensional.
  But we need to show that it has the correct highest weight.
  Let \(v_\lambda \in B(\lambda)\) be the highest weight vector for the action of \(\uvq\).
  For \(q \neq 1\), we have
  \[
  \hjlq v_\lambda = \frac{q^{(\alpha_j,\lambda)} - q^{-(\alpha_j,\lambda)}}{q_j - q_j^{-1}} v_\lambda = \frac{q_j^{(\alpha_j^\vee,\lambda)} - q_j^{-(\alpha_j^\vee,\lambda)}}{q_j - q_j^{-1}} v_\lambda.
  \]
  From \cref{lem:continuity-of-parametrized-operators}, we see that as \(q \to 1\) we get \(\hjlq[1]v_\lambda = (\alpha_j^\vee,\lambda) v_\lambda\), so that \(v_\lambda\) is a weight vector of weight \(\lambda\) for the action of \(U_1(\fg)\).
  As \(\ejlq\) annihilates \(v_\lambda\) for \(q \neq 1\), again by \cref{lem:continuity-of-parametrized-operators} we see that \(\ejlq[1]\) annihilates \(v_\lambda\), so \(v_\lambda\) is a highest weight vector.
  As \(\vlc\) contains a highest weight vector of weight \(\lambda\) and has the same dimension as the irreducible representation with that highest weight, we conclude that \(\vlc\) \emph{is} that irreducible representation.
\end{proof}

\begin{rem}
  \label{rem:on-weyl-modules-for-non-simples}
  We have now constructed the Weyl module \(\vlc\) associated to each dominant integral weight \(\lambda \in \cP\) and shown that the actions are continuous in \(q\).  
  Next we do the same for finite-dimensional modules \(V\) that are not simple.
  This necessarily involves some measure of choice, as a general module does not have a unique decomposition into irreducibles (although the decomposition into isotypical components is unique).
  
  The philosophy for reducible modules is therefore somewhat different.
  For a simple module, we begin with a weight \(\lambda\) (which is not associated to any particular value of \(q\)) and construct the universal model.
  In the reducible situation, we begin with a representation \(V\) of some \emph{particular} \(\uqzero\) (possibly for \(q_0 = 1\)), and we construct a universal model \(V^\bbC\) for \(V\) on which all \(\uqg\) act simultaneously.
\end{rem}

\begin{prop}
  \label{prop:universal-model-for-non-simple-rep}
  Let \(q_0 > 0\) (we allow \(q = 1\)) and let \(V \in \cO_{q_0}\) be a finite-dimensional Type 1 representation of \(\uqzero\).
  Then there is a complex vector space \(V^\bbC\) carrying an action of \(\uqg\) for each \(q > 0\) such that \(V^\bbC \cong V\) as \(\uqzero\)-modules.
  The vector space \(V^\bbC\) has a distinguished basis \(B(V)\) such that the divided powers of the generators \(E_j\) and \(F_j\), and all \(H_j\) and \(K_\mu\) act via Laurent polynomials, and these actions are continuous in \(q\) as in \cref{lem:continuity-of-parametrized-operators}.
\end{prop}

\begin{proof}
  As finite-dimensional representations of \(\uqzero\) are completely reducible (see \cite{Jan96}*{Theorem 5.17 and discussion in \S 6.26}), we can choose a decomposition \(V \cong \oplus_{\lambda} V(\lambda)\) as \(\uqzero\)-modules.
  Then define \(V^\bbC = \oplus_{\lambda} \vlc\), and define the basis \(B(V)\) to be the (disjoint) union of the bases \(B(\lambda)\).
  \cref{prop:universal-model-for-v-lambda} and \cref{lem:continuity-of-parametrized-operators} together imply that \(V^\bbC\) has the desired properties.
\end{proof}

Finally, we turn to universal models for tensor products.
In addition to the choice inherent in the decomposition of a general representation into simple representations, we have the following choice: if we are given a representation \(V = V_1 \otimes \dots \otimes V_n\), we can decompose \(V\) into simple modules \(V(\lambda)\) and take the universal model \(V^\bbC\) as in \cref{prop:universal-model-for-non-simple-rep}, or we can decompose each \(V_j\) and take the tensor product of the \(V_j^\bbC\).
We choose the latter, and formalize this choice as:

\begin{dfn}
  \label{dfn:universal-model-for-tensor-product}
  When a \(\uqzero\)-module \(V \in \cO_{q_0}\) is presented to us as a tensor product \(V = V_1 \otimes \dots \otimes V_n\), we take the universal model \(V^\bbC\) to be
  \index[term]{universal model for a representation}
  \[
  V^\bbC \eqdef V_1^\bbC \otimes \dots \otimes V_n^{\bbC},
  \]
  and we take the distinguished basis \(B(V)\) to be the basis formed from the tensor products of elements in the bases \(B(V_i)\).
  When working with braidings and commutors later on, we will often simultaneously consider modules of the form \(U \otimes V\) and \(V \otimes U\).
  In that situation, we choose distinguished bases \(B(U)\) and \(B(V)\) as above and then form distinguished bases \(B(U \otimes V)\) and \(B(V \otimes U)\) from the pairwise tensor products.
  Likewise, we form a distinguished basis \(B(V^{\otimes n}\) for \(V^{\otimes n}\) from a distinguished basis for \(V\).
\end{dfn}

\begin{rem}
  \label{rem:on-the-universal-models}
  As the universal models for reducible modules and for tensor products depend on choices of decompositions into simple submodules, they are not functorial.
  However, in \cref{sec:quantum-symmetric-algebras-are-commutative,sec:collapsing-in-degree-three} we use this construction only for one module at a time, so we do not require that \(V \mapsto V^\bbC\) is a functor.
\end{rem}

\begin{conv}
  \label{conv:dropping-upper-c}
  From now on, whenever it is convenient we will forget about the superscript \(\bbC\) notation, tacitly replacing the \(\uqzero\)-module \(V\) with a universal model \(V^\bbC\) and thereby allowing all \(\uqg\) to act simultaneously on \(V\) itself.
\end{conv}

\subsection{Continuity of the braidings}
\label{sec:continuity-of-braidings}

In this section we discuss continuity and limits as \(q \to 1\) for the braidings of modules in \(\oq\).

We begin by recalling in more detail the construction of the braidings and commutors for Type 1 representations of \(\uvq\).
We define an element \(\fR(\nu)\) of the completion \(\overline{U}_\nu^\bbQ(\fg) \widehat\otimes \overline{U}_\nu^\bbQ(\fg)\) (see \cite{KliSch97}*{\S 6.3.3 and \S 8.3.3} for details, although note that the formula for the R-matrix must be modified slightly, as we use the opposite coproduct) by \index[notn]{R@\(\fR(\nu)\)}
\begin{equation}
  \label{eq:rmatrix-definition}
  \fR(\nu) \eqdef \sum_{t_1, \dots t_d = 0}^\infty \prod_{j=1}^d \frac{(1-\nu_{\beta_j}^{-2})^{t_j}}{[t_j]_{\nu_{\beta_j}}!} \nu_{\beta_j}^{t_j(t_j + 1)/2} F_{\beta_j}^{t_j} \otimes E_{\beta_j}^{t_j},
\end{equation}
where \(\beta_1,\dots,\beta_d\) are the positive roots of \(\fg\), listed as in \eqref{eq:sequence-of-positive-roots} according to some arbitrary but fixed reduced expression for \(\wz\), and the factors in the products occur in the order \(\beta_d,\beta_{d-1},\dots,\beta_1\).
As all \(E_\beta\) and \(F_\beta\) act nilpotently in any finite-dimensional representation (see \cite{Jan96}*{Proposition 5.1}), for any \(U,V \in \ov\) there is a well-defined operator \index[notn]{RUV@\(\fR_{UV}(\nu),R_{UV}(\nu),\rhat_{UV}(\nu)\)}
\begin{equation}
  \label{eq:rmatrix-action-in-tensor-product}
  \fR_{UV}(\nu) : U \otimes V \to V \otimes U, \quad u \otimes v \mapsto \fR(\nu) (u \otimes v),
\end{equation}
where the action of \(\fR\) is componentwise.
Furthermore we define the operator \(B_{UV}(\nu)\)\index[notn]{BUV@\(B_{UV}(\nu)\)} on \(U \otimes V\) by
\begin{equation}
  \label{eq:BUV-definition}
  B_{UV}(\nu)(u \otimes v) \eqdef \nu^{(\wt(u),\wt(v))} u \otimes v
\end{equation}
for weight vectors \(u\) and \(v\).
Finally, we define maps \(R_{UV}(\nu) : U \otimes V \to U \otimes V\) and \(\rhat_{UV}(\nu) : U \otimes V \to V \otimes U\) by
\begin{equation}
  \label{eq:braiding-map-definition}
  R_{UV}(\nu) \eqdef B_{UV}(\nu) \circ \fR_{UV}(\nu) \quad \text{and} \quad \rhat_{UV}(\nu) \eqdef \tau_{UV} \circ R_{UV}(\nu),
\end{equation}
where \(\tau_{UV}(u \otimes v) = v \otimes u\) is the tensor flip.
By definition, the map \(\rhat_{UV}(\nu)\) is the braiding of \(U\) with \(V\).
With this description in place, we can now prove:

\begin{lem}
  \label{lem:rmatrix-acts-by-laurent-polynomials}
  Let \(U,V \in \ov\) and let \(B(U \otimes V)\) be a distinguished basis for the tensor product \(U \otimes V\) as in \cref{dfn:universal-model-for-tensor-product}.
  Then \(R_{UV}(\nu)\) acts on \(B(U \otimes V)\) by Laurent polynomials in \(\nu\), and similarly the matrix coefficients of \(\rhat_{UV}(\nu)\) with respect to the bases \(B(U \otimes V)\) and \(B(V \otimes U)\) are Laurent polynomials in \(\nu\).
\end{lem}

\begin{proof}
    Rearranging \eqref{eq:rmatrix-definition} by dividing the \(E_{\beta_j}^{t_j}\) term by \([t_j]_{\nu_{\beta_j}}!\), we get a linear combination of terms of the form \(\displaystyle \prod_j E_{\beta_j}^{(t_j)} \otimes F_{\beta_j}^{t_j}\) with coefficients in \(\Zv\) (recall that \(E_{\beta_j}^{(t_j)}\) is the divided power of \(E_{\beta_j}\)).
    By \cref{prop:universal-model-for-non-simple-rep}, each of these terms acts by Laurent polynomials on \(B(U \otimes V)\), so the same is true of \(\fR_{UV}(\nu)\).
    Since \(B(U \otimes V)\) is a weight basis, the map \(B_{UV}(\nu)\) acts by Laurent polynomials (in fact by powers of \(\nu\)), and hence \(R_{UV}(\nu)\) does as well.
    As the tensor flip \(\tau_{UV}\) exchanges the bases \(B(U \otimes V)\) and \(B(V \otimes U)\), we see that \(\rhat_{UV}(\nu)\) also acts by Laurent polynomials with respect to these bases.
\end{proof}

\begin{dfns}
  \label{dfn:braiding-on-universal-models}
  \index[notn]{RUVq@\(\rhatuvq{UV}\)}
  Let \(U,V \in \oq\) for some \(q > 0\), replace \(U,V\) by universal models as in \cref{sec:canonical-bases} (so that all \(\uqg\) for \(q > 0\) act on \(U\) and \(V\)), and form the distinguished bases \(B(U \otimes V)\) and \(B(V \otimes U)\) as in \cref{dfn:universal-model-for-tensor-product}.

   For each \(q > 0\), \(q \neq 1\), let \(\rhatuvq{UV} : U \otimes V \to V \otimes U\) be the braiding of \(\uqg\)-modules, defined as in \eqref{eq:braiding-map-definition} except with \(\nu\) replaced by \(q\).
    For \(q = 1\), define \(\rhatuvq[1]{UV} = \tau_{UV}\) to be the tensor flip.
\end{dfns}

We obtain the following consequence of \cref{lem:rmatrix-acts-by-laurent-polynomials}:

\begin{prop}
  \label{prop:continuity-of-rmatrices}
  For \(U\) and \(V\) as above, the family of braidings \(\rhatuvq{UV}\) is continuous, i.e.\ \(q \mapsto \rhatuvq{UV}\) is a continuous map \((0,\infty) \to \Hom_{\bbC}(U \otimes V, V \otimes U)\).
\end{prop}

\begin{proof}
  By \cref{lem:rmatrix-acts-by-laurent-polynomials}, the braidings \(\rhat_{UV}(\nu)\) act by Laurent polynomials in \(\nu\).
  As the braidings \(\rhatuvq{UV}\) are the specializations of \(\rhat_{UV}(\nu)\) at \(\nu = q\), they also act by Laurent polynomials, and hence they are continuous in \(q\) except possibly at \(q=1\), where we defined \(\rhatuvq[1]{UV}\) differently.
  
  Examining \eqref{eq:rmatrix-definition} and noting that \(\lim_{q\to 1} [n]_q = n\), we see that as \(q \to 1\), the only term that survives is the one in which \(t_1 = \dots = t_d = 0\), which just acts as the identity in \(U \otimes V\).
  It is clear from \eqref{eq:BUV-definition} that the operator \(B_{UV}(q)\) (again with \(\nu\) replaced by \(q\)) tends to the identity as \(q \to 1\), and hence we see that \(\lim_{q \to 1}\rhatuvq{UV} = \tau_{UV} = \rhatuvq[1]{UV}\).
  Hence the whole family of braidings is continuous in \(q\).
\end{proof}

\subsection{Continuity of the coboundary structure}
\label{sec:continuity-of-coboundary-structure}

Now we turn to the coboundary structure and the commutors \(\sigma_{UV}\).
We aim to establish a continuity result akin to \cref{prop:continuity-of-rmatrices}.
However, as the following example shows, the commutors do not act by Laurent polynomials, and therefore our method of proof must be somewhat different.

\begin{eg}
  \label{eg:commutors-dont-act-by-laurent-polys}
  Recall from \cref{eg:quantum-symmetric-algebra-for-2-dim-rep} that the commutor for the two-dimensional irreducible representation \(V\) of \(\uqsl\) is given by
  \[
  \sigma_{VV} =
  \begin{pmatrix}
    1 & 0 & 0 & 0 \\
    0 & \frac{1-q^2}{1+q^2} & \frac{2 q}{1+q^2} & 0 \\
    0 & \frac{2 q}{1+q^2} & \frac{-1+q^2}{1+q^2} & 0 \\
    0 & 0 & 0 & 1 \\
  \end{pmatrix}.
  \]
  with respect to the standard (lexicographically ordered) basis for \(V \otimes V\).
  For this representation, the canonical basis \emph{is} the standard basis, and we see that in this instance the commutor does not act by Laurent polynomials with respect to the canonical basis.
  Note, however, that the denominators of the rational functions appearing in \(\sigma_{VV}\) are nonzero for all real \(q\).
\end{eg}

We now define the parametrized commutors, analogous to the parametrized braidings in \cref{dfn:braiding-on-universal-models}:

\begin{dfn}
  \label{dfn:numerical-coboundary-operators}
  \index[notn]{SigmaUVq@\(\sigmauvq{UV}\)}
  Let \(U,V \in \oq\) for some \(q > 0\), replace \(U,V\) by universal models as in \cref{sec:canonical-bases} (so that all \(\uqg\) for \(q > 0\) act on \(U\) and \(V\)), and form the distinguished bases \(B(U \otimes V)\) and \(B(V \otimes U)\) as in \cref{dfn:universal-model-for-tensor-product}.

  For each \(q > 0\), \(q \neq 1\), let \(\sigmauvq{UV} : U \otimes V \to V \otimes U\) be the commutor of \(\uqg\)-modules, defined as in \cref{sec:coboundary-structure} by
  \begin{equation}
    \label{eq:parametrized-coboundary-map}
    \sigmauvq{UV} \eqdef \rhatuvq{UV} \left( \rhatuvq{UV}^\ast \rhatuvq{UV} \right)^{-\frac12} = \rhatuvq{UV} \left( \rhatuvq{VU} \rhatuvq{UV} \right)^{-\frac12}.
  \end{equation}
  For \(q = 1\), define \(\sigmauvq[1]{UV} = \tau_{UV}\) to be the tensor flip.
\end{dfn}

The proof that the commutors \(\sigmauvq{UV}\) form a continuous family is not difficult.
However, later on we will require more, namely that with respect to the distinguished bases \(B(U \otimes V)\) and \(B(V \otimes U)\) the commutors act via rational functions whose denominators do not vanish on \(\bbR_{>0}\), as we saw in \cref{eg:commutors-dont-act-by-laurent-polys}.
As the proof of this fact is somewhat convoluted, we outline our strategy now:
\begin{description}
\item[Step 1] Prove that the map \(q \mapsto \sigmauvq{UV}\) is continuous for \(q > 0\).
\item[Step 2]  Explicitly construct the commutor \(\tilde{\sigma}_{UV}(\nu)\) of \(\uvq\)-modules, which by definition acts via rational functions on the distinguished bases.
\item[Step 3] Show that the specializations \(\tilde{\sigma}_{UV}(q)\) coincide with the maps \(\sigmauvq{UV}\) for \emph{transcendental} values of \(q\).
\item[Step 4] Combine Steps 1 and 3 to show that the denominators appearing in the matrix coefficients of the commutors do not vanish for \(q > 0\), and hence that the specialization \(\tilde{\sigma}_{UV}(q)\) makes sense for \emph{any} \(q > 0\), and \(\sigmauvq{UV} = \tilde{\sigma}_{UV}(q)\) for all such \(q\).
\end{description}
We begin with:

\begin{prop}[Step 1]
  \label{prop:continuity-of-commutors}
  Let \(U,V \in \oq\) for some \(q > 0\), and construct the commutors \(\sigmauvq{UV}\) as in \cref{dfn:numerical-coboundary-operators}.
  Then \(q \mapsto \sigmauvq{UV}\) is a continuous map \((0,\infty) \to \Hom_\bbC(U \otimes V, V \otimes U)\).
\end{prop}

\begin{proof}
  From \eqref{eq:parametrized-coboundary-map}, the coboundary maps are defined by
  \[
  \sigmauvq{UV} = \rhatuvq{UV} \left( \rhatuvq{UV}^\ast \rhatuvq{UV} \right)^{-\frac12};
  \]
  as the braidings are continuous by \cref{prop:continuity-of-rmatrices}, we need only show that the map 
  \begin{equation}
    \label{eq:continuity-for-coboundary-inverse-square-root}
    q \mapsto \left( \rhatuvq{UV}^\ast \rhatuvq{UV} \right)^{-\frac12}
  \end{equation}
  is continuous at each \(q_0 > 0\).

  Let \(I\) be a compact interval in \(\bbR_{> 0}\) containing \(q_0\) in its interior.
  Let \(A = C(I,\End_{\bbC}(U \otimes V))\) be the \(C^\ast\)-algebra of continuous functions from \(I\) into the endomorphism algebra of \(U \otimes V\), and define \(a : I \to \End_{\bbC}(U \otimes V)\)  by \(a(q) = (\rhatuvq{UV}^\ast \rhatuvq{UV})^{-1}\).
  Then \(a\) is continuous since multiplication, inversion, and the adjoint are continuous in the endomorphism algebra, and since \(q \mapsto \rhatuvq{UV}\) is continuous by \cref{prop:continuity-of-rmatrices}.
  Hence \(a \in A\).
  
  The element \(a\) is clearly self-adjoint, and the spectrum of \(a\) is nonnegative since the spectrum of each \(a(q)\) is strictly positive, so \(a\) itself is positive.
  Therefore \(a\) has a unique positive square root \(a^{\frac12} \in A\).
  But then \(q \mapsto a^{\frac12}(q)\) is continuous (by definition of \(A\)), 
  and we have \(a^{\frac12}(q) = (\rhatuvq{UV}^\ast \rhatuvq{UV})^{-\frac12}\), so we see that \eqref{eq:continuity-for-coboundary-inverse-square-root} is continuous.
  This completes the proof.
\end{proof}

We now proceed with Step Two of our plan, namely constructing the coboundary operators \(\tilde{\sigma}_{UV}(\nu)\) in the formal setting.
The difficulty in this step is that we do not have a notion of positivity for operators on a \(\Qv\)-vector space, so we cannot take the positive square root of the double-braiding \(\rhat_{VU}(\nu)\rhat_{UV}(\nu)\) as is done in \eqref{eq:parametrized-coboundary-map}.
We require the following result:
\begin{lem}[\cite{KliSch97}*{Ch. 8, Proposition 22}]
  \label{lem:eigenvalues-of-R21R}
  Let \(\mu,\lambda,\lambda' \in \pplus\).
  If \(V(\mu)\) occurs as a direct summand in \(V(\lambda) \otimes V(\lambda')\), then the operator \(\rhat_{\lambda' \lambda}(\nu) \rhat_{\lambda \lambda'}(\nu)\) acts as the scalar
  \begin{equation}
    \label{eq:eigenvalue-of-R21R}
    \nu^{-(\lambda,\lambda+2\rho) - (\lambda',\lambda'+2\rho) + (\mu,\mu+2\rho)}
  \end{equation}
  on the \(V(\mu)\)-isotypic component of the tensor product.
  The analogous statement holds also when \(\nu\) is replaced by the nonzero complex number \(q\) (not a root of unity).
\end{lem}

We can thus make the following:
\begin{dfns}[Step 2]
  \label{dfn:coboundary-operator-for-formal-situation}
  \begin{enumerate}[(1)]
  \item For \(\lambda,\lambda' \in \pplus\), define 
    \(A_{\lambda\lambda'}(\nu)\)
    to be the operator on \(V(\lambda) \otimes V(\lambda')\) that acts as the scalar 
    \[
    \nu^{\frac12  \left( (\lambda,\lambda+2\rho) + (\lambda',\lambda'+2\rho) - (\mu,\mu+2\rho) \right)}
    \]
    on the \(V(\mu)\)-isotypic component of the tensor product.
    (Note that this scalar is the inverse square root of \eqref{eq:eigenvalue-of-R21R}.)
  \item For arbitrary \(U,V \in \ov\) we define the operator 
    \(A_{UV}(\nu)\)
    on \(U \otimes V\) by decomposing \(U \cong \oplus_\lambda m_\lambda V(\lambda)\) and \(V \cong \oplus_{\lambda'} m_{\lambda'} V(\lambda')\) into isotypic components and taking the direct sum of the operators \(A_{\lambda \lambda'}(\nu)\) defined immediately above in (1). 
    Note that this is canonical because decompositions into isotypic components are unique.
  \item For \(U,V \in \ov\) we define the \emph{coboundary operator} \(\tilde\sigma_{UV}(\nu) : U \otimes V \to V \otimes U\) \index[notn]{sigmauv@\(\tilde\sigma_{UV}(\nu),\tilde\sigma_{UV}(q)\)} by
    \begin{equation}
      \label{eq:coboundary-operator-for-formal-situation}
      \tilde\sigma_{UV}(\nu) = \rhat_{UV}(\nu) A_{UV}(\nu),
    \end{equation}
    where \(A_{UV}(\nu)\) is as defined in (2).
  \item Let \(q > 0\) be transcendental.  
    Form universal models for \(U \otimes V\) and \(V \otimes U\) and the associated distinguished bases \(B(U \otimes V)\) and \(B(V \otimes U)\).  
    We define the \emph{specialization} \(\tilde\sigma(q) : U \otimes V \to V \otimes U\) to be the linear map whose matrix with respect to the distinguished bases is the specialization at \(\nu = q\) of the matrix for \(\tilde\sigma_{UV}(\nu)\).
    This is well-defined because the matrix coefficients of \(\tilde\sigma_{UV}(\nu)\) are rational functions in \(\nu\), and \(q\) is transcendental.
  \end{enumerate}
\end{dfns}

With these definitions in place, we now prove:
\begin{lem}[Step 3]
  \label{lem:specialized-commutors-agree-with-numerical-commutors}
  For each transcendental \(q > 0\), we have \(\sigma_{UV}(q) = \tilde\sigma_{UV}(q)\).
\end{lem}

\begin{proof}
  By construction, it is enough to prove the statement when \(U = V(\lambda)\) and \(V = V(\lambda')\) for some \(\lambda,\lambda' \in \pplus\).
  As the braidings \(\rhatuvq{UV}\) are the specializations at \(\nu = q\) of the braidings \(\rhatuvq[\nu]{UV}\), we just need to check that the specialization at \(\nu = q\) of the operator \(A_{\lambda \lambda'}(\nu)\) defined in \cref{dfn:coboundary-operator-for-formal-situation}(1) coincides with \((\rhatuvq{UV}^\ast \rhatuvq{UV})^{-\frac12}\).

  Indeed, by construction we have \(A_{\lambda \lambda'}^{-2} = \rhatuvq{UV}^\ast \rhatuvq{UV}\).
  Moreover, for \(q > 0\) the specialization \(A_{\lambda \lambda'}(q)\) has positive eigenvalues on orthogonal eigenspaces, so it is a positive operator.
  Hence we have \(A_{\lambda \lambda'}(q) = (\rhatuvq{UV}^\ast \rhatuvq{UV})^{-\frac12}\), as desired.
\end{proof}

Now we consider the matrix coefficients of the coboundary operators \(\sigma_{UV}(q)\) with respect to the distinguished bases \(B(U \otimes V)\) and \(B(V \otimes U)\).
By \cref{prop:continuity-of-commutors} we know that these matrix coefficients are continuous functions of \(q\).
By \cref{lem:specialized-commutors-agree-with-numerical-commutors}, we know that these continuous functions agree with certain \emph{rational} functions on the set of positive transcendental numbers.
The next result will allow us to conclude that the denominators of these rational functions do not vanish on \(\bbR_{>0}\):

\begin{lem}
  \label{lem:extension-of-rational-functions}
  Let \(I \subseteq \bbR\) be an open interval and let \(f : I \to \bbR\) be continuous.
  Let \(D \subseteq I\) be a dense set, and suppose that there are polynomials \(g,h \in \bbQ[q]\) such that \(\gcd(g,h)=1\) in \(\bbQ[q]\) and such that \(f(t) = \frac{g(t)}{h(t)}\) for all \(t \in D\).
  Then \(h(t) \neq 0\) for all \(t \in I\), and \(f(t) = \frac{g(t)}{h(t)}\) for all \(t \in I\).
\end{lem}

\begin{proof}
  Suppose that \(h(t_0) = 0\) for some \(t_0 \in I\).
  Let \((t_n)_{n = 1}^\infty\) be a sequence of elements of \(D\) tending to \(t_0\).
  Our hypotheses implies that \(h(t_n) \neq 0\).
  Then we have
  \begin{equation}
    \label{eq:rational-function-proof}
    f(t_0) = \lim_{n \to \infty} f(t_n) = \lim_{n \to \infty} \frac{g(t_n)}{h(t_n)}.
  \end{equation}
  If \(g(t_0) \neq 0\), then for \(n\) sufficiently large \(g(t_n)\) will be bounded away from zero, and as \(h\) is continuous we have \(h(t_n) \to h(t_0) = 0\).
  But then \(\abs{\frac{g(t_n)}{h(t_n)}} \to \infty\), which contradicts \eqref{eq:rational-function-proof}, and hence we must have \(g(t_0) = 0\) as well.
  But now the minimal polynomial of \(t_0\) over \(\bbQ\) must divide both \(g\) and \(h\), contradicting our assumption that \(\gcd(g,h) = 1\).
\end{proof}

Now we address the final step of our strategy:

\begin{prop}[Step 4]
  \label{prop:commutors-act-by-rational-functions}
  With respect to the bases \(B(U \otimes V)\) and \(B(V \otimes U)\), the matrix coefficients of the commutors \(\sigmauvq{UV}\) are rational functions in \(\Qq\) whose denominators do not vanish for \(q > 0\).
\end{prop}

\begin{proof}
  This follows from \cref{lem:extension-of-rational-functions} together with the discussion preceding it.
\end{proof}

\subsection{Rigidity for continuous families}
\label{sec:rigidity-for-continuous-families}

Continuing with the theme of continuity, now we examine continuously parametrized subspaces of a universal model for a representation.
For instance, given a universal model for \(V\), we have the family \(q \mapsto \symq^2V\) of subspaces of \(V \otimes V\), parametrized by \(\bbR_{>0}\).
Each \(\symq^2V\) is a \(\uqg\)-submodule of \(V \otimes V\), and hence has a decomposition into highest weight modules with multiplicities.
It is natural to ask whether this decomposition is independent of the parameter.
We address this question in \cref{prop:continuity-of-symmetric-and-exterior-squares}.
We also discuss decompositions of a continuous family of submodules into irreducibles in \cref{prop:decomp-of-cts-family-is-constant}. 
We begin with the following notation:
\begin{notn}
  \label{notn:full-grassmannian}
  \index[notn]{GrV@\(\Gr(V)\)}
  \index[term]{Grassmann manifold}
  For a vector space \(V\), we denote by \(\Gr(V)\) the disjoint union of all Grassmann manifolds \(\Gr_k(V)\) for \(0 < k < \dim V\).
\end{notn}

\begin{prop}
  \label{prop:continuity-of-symmetric-and-exterior-squares}
  Let \(q_0 > 0\) and \(V \in \oq[q_0]\).
  Replace \(V\) by a universal model as in \cref{prop:universal-model-for-non-simple-rep}.
  Then \(q \mapsto \symq^2V\) and \(q \mapsto \extq^2V\) are continuous functions \((0,\infty) \to \Gr(V \otimes V)\).
\end{prop}

\begin{proof}
  We prove the result only for \(\symq^2V\), as the proof for \(\extq^2V\) is similar.
  It follows immediately from \cref{dfn:symmetric-and-antisymmetric-tensors} together with the fact that \(\sigma_{VV}(q)\) is involutive that
  \[
  \symq^2V = \ker(\sigma_{VV}(q)-\id) = \ran(\sigma_{VV}(q)+\id).
  \]
  By \cref{prop:continuity-of-commutors} the map \(q \mapsto \sigma_{VV}(q)-\id\) is continuous, so the result will follow if we can show that \(\dim \symq^2V\) is constant in \(q\).
  Note that \(\frac{\sigma_{VV}(q)+\id}{2}\) is the orthogonal projection onto \(\symq^2V\), so that
  \[
  \dim \symq^2V = \tr \left( \frac{\sigma_{VV}(q)+\id}{2} \right).
  \]
  By \cref{prop:continuity-of-commutors} the projection varies continuously with \(q\), and hence \(\dim \symq^2V\) is a continuous, integer-valued function of \(q\), so it is constant.
\end{proof}

\begin{lem}
  \label{lem:continuity-of-highest-weight-spaces}
  Let \(q>0\), let \(V_1,\dots,V_n \in \oq\), and let \(\lambda \in \pplus\).
  Replace all \(V_i\) by universal models as in \cref{prop:universal-model-for-non-simple-rep}, and let \(V = V_1 \otimes \dots \otimes V_n\).
  For each \(q>0\), let \(V^\lambda_q \subseteq V\) be the space of highest weight vectors of weight \(\lambda\) for the action of \(\uqg\).
  Then \(q \mapsto V^\lambda_q\) is a continuous map \((0,\infty) \to \Gr(V)\).
\end{lem}

\begin{proof}
  Because tensor products of modules in \(\oq\) decompose with the same weight multiplicities as they do classically, the dimension of \(V^\lambda_q\) is independent of \(q\).
  Thus the range of the map \(q \mapsto V^\lambda_q\) is contained in a single component of \(\Gr(V)\).

  Now let \((q_n)_{n=1}^\infty\) be a sequence of positive real numbers converging to some \(q_0 > 0\).
  We want to show that \((V^\lambda_{q_n})_{n=1}^\infty\) converges to \(V^\lambda_{q_0}\) in \(\Gr(V)\).
  If not, there is an open neighborhood \(\cU \subseteq \Gr(V)\) containing \(V^\lambda_{q_0}\) such that some subsequence of \((V^\lambda_{q_n})_{n=1}^\infty\) avoids \(\cU\).
  Passing to this subsequence, we may assume that \(V^\lambda_{q_n} \notin \cU\) for all \(n\).

  Since \(\Gr(V)\) is compact, a subsequence of \((V^\lambda_{q_n})_{n=1}^\infty\) converges to some point \(W \notin \cU\).
  By the continuity statement in \cref{prop:universal-model-for-non-simple-rep}, \(W\) consists of highest weight vectors of weight \(\lambda\) for the action of \(\uqzero\).
  Hence \(W \subseteq V^\lambda_{q_0}\).
  But since \(W\) and \(V^\lambda_{q_0}\) are in the same component of \(\Gr(V)\), their dimensions coincide, and hence \(W = V^\lambda_{q_0}\), which contradicts \(W \notin \cU\).
\end{proof}

For later use we will also require:
\begin{prop}
  \label{prop:decomp-of-cts-family-is-constant}
  Let \(q > 0\) and \(V \in \oq\).
  Replace \(V\) by a universal model as in \cref{prop:universal-model-for-non-simple-rep}.
  Suppose that \(q \mapsto W_q\) is a continuous map \((0,\infty) \to \Gr(V)\) such that for each \(q\), \(W_q\) is a \(\uqg\)-submodule of \(V\).
  Let \(\lambda \in \pplus\), and for each \(q > 0\) let \(V^\lambda_q\) be the space of highest weight vectors in \(V\) for the action of \(\uqg\).
  Then \(q \mapsto V^\lambda_q \cap W_q\) is a continuous map \((0,\infty) \to \Gr(V)\).
\end{prop}

\begin{proof}
  We claim first that we can find a continuous family \(q \mapsto X_q \subseteq V\) such that for each \(q > 0\), \(X_q\) is a \(\uqg\)-submodule of \(V\) and \(V = W_q \oplus X_q\).
  Indeed, for each \(q > 0\) there is a (Hermitian, positive definite) inner product \(\langle \cdot,\cdot \rangle_q\) on \(V\) that is invariant for the action of \(\uqg\) as described in \cref{sec:representations-of-quea}.
  Moreover, these inner products can be chosen to be continuous in \(q\); see \cite{ChaPre95}*{Proposition 10.1.21}.
  Let \(\perp_q\) denote orthogonal complement in \(V\) with respect to \(\langle \cdot,\cdot \rangle_q\).
  Then the family \(q \mapsto X_q \eqdef W_q^{\perp_q}\) is the desired family of complements.
  (Note that \(X_q\) is a \(\uqg\)-submodule since the inner product \(\langle \cdot,\cdot \rangle_q\) is invariant.)

  Now for each \(q>0\) let \(P_q\) be the orthogonal projection of \(V\) onto \(W_q\) with respect to \(\langle \cdot,\cdot \rangle_q\).
  Then \(q \mapsto P_q\) is a continuous family.
  Similarly, let \(Q_q\) be the orthogonal projection of \(V\) onto \(V^\lambda_q\); then \(q \mapsto Q_q\) is continuous by \cref{lem:continuity-of-highest-weight-spaces}.

  We claim now that \(P_qQ_q = Q_qP_q\) for each \(q\).
  Indeed, for \(v \in V\), we can write \(v = w + x\) uniquely, with \(w \in W_q\) and \(x \in X_q\).
  As both \(W_q\) and \(X_q\) are \(\uqg\)-modules, they both decompose into a direct sum of weight spaces, so we can write 
  \[
  w = w^\lambda + w', \quad x = x^\lambda + x',
  \]
  where \(w^\lambda \in V^\lambda_q \cap W_q\), \(x^\lambda \in V^\lambda_q \cap X_q\), \(w'\in (V^\lambda_q)^{\perp_q} \cap W_q \) and \(x' \in (V^\lambda_q)^{\perp_q} \cap X_q\).
  Then we have
  \[
  P_q Q_q v = P_q (w^\lambda + x^\lambda) = w^\lambda,
  \]
  while
  \[
  Q_q P_q v = Q_q (w^\lambda + w') = w^\lambda.
  \]
  Since the projections \(P_q\) and \(Q_q\) commute, their product \(P_qQ_q\) is the projection onto \(V^\lambda_q \cap W_q\).
  But then \(q \mapsto P_q Q_q\) is a continuous path of projections, so \(q \mapsto V^\lambda_q \cap W_q\) is a continuous map to the Grassmann manifold.
\end{proof}

\begin{rem}
  \label{rem:on-continuity-of-decomps}
  While the statement of \cref{prop:decomp-of-cts-family-is-constant} is technical, the meaning is important.
  The proposition implies that the multiplicity in \(W_q\) of the irreducible representation of \(\uqg\) with highest weight \(\lambda\) is independent of \(q\).
  This should be viewed as a form of rigidity.
\end{rem}

\subsection{The cactus group action on quantum symmetric tensors}
\label{sec:cactus-gp-action-on-q-symmetric-tensors}

As we discussed in \cref{sec:cactus-group}, for \(V \in \oq\) and \(n \geq 2\), the cactus group \(J_n\) acts on the tensor power \(V^{\otimes n}\).
In this section we show that this action is continuous in \(q\) (as usual, we tacitly replace \(V\) by a universal model as in \cref{conv:dropping-upper-c}).
We then use this continuity to show that the cactus group acts trivially on the space \(\symq^n V\) of quantum symmetric tensors.
\emph{A priori} it is not even obvious that \(\symq^n V\) is invariant under the action of \(J_n\), but this is a consequence of our proof.

\begin{notn}
  \label{notn:parametrized-cactus-group-action}
  Let \(V \in \oq\) for some \(q > 0\), replace \(V\) by a universal model, and form the distinguished basis \(B(V)\).
  For each \(q > 0\) there is an action of the cactus group \(J_n\) on \(V^{\otimes n}\), constructed from the commutors \(\sigmauvq{V^{\otimes k},V^{\otimes l}}\) via the process in \cref{sec:coboundary-categories}.
  For \(1 \leq p \leq r < t \leq n\) we let \(\sigma_{p,r,t}(q)\) \index[notn]{sigmaprtq@\(\sigma_{p,r,t}(q)\)} and \(s_{p,t}(q)\) \index[notn]{sptq@\(s_{p,t}(q)\)} be the operators on \(V^{\otimes n}\) constructed as in \eqref{eq:sigma-prq-def} and \eqref{eq:s-pq-def}, respectively.
  For each \(q > 0\) we denote the group homomorphism associated to the action of \(J_n\) by \(\rho_q : J_n \to GL(V^{\otimes n})\). \index[notn]{rhoq@\(\rho_q\)}

  As in \cref{dfn:symmetric-and-antisymmetric-tensors} we define the spaces \(\symq^nV\) and \(\extq^nV\) of quantum symmetric and antisymmetric tensors, respectively.
  Since we can do this for each \(q > 0\), we view \((\symq^n(V))_{q > 0}\) as a  parametrized family of subspaces of the ambient universal model \(V^{\otimes n}\), and similarly for \((\extq^nV)_{q > 0}\).
\end{notn}

\begin{cor}[to \cref{prop:commutors-act-by-rational-functions}]
  \label{cor:all-coboundary-operators-act-by-rational-functions}
  Let \(B(V^{\otimes n})\) be the distinguished basis for \(V^{\otimes n}\) constructed from \(B(V)\).
  Then the matrix coefficients of all operators \(\sigma_{p,r,t}(q)\) and \(s_{p,t}(q)\) with respect to \(B(V^{\otimes n})\) lie in \(\Qq\), and their denominators do not vanish for \(q > 0\).
\end{cor}

\begin{proof}
  This follows from \cref{prop:commutors-act-by-rational-functions} together with the constructions of the operators in \cref{sec:coboundary-categories}.
\end{proof}

\begin{rem}
  \label{rem:on-continuity-of-quantum-symmetric-tensors}
  If we denote by \(\Gr(V^{\otimes n})\) the disjoint union of Grassmann manifolds \(\Gr_k(V^{\otimes n})\) for all values of \(k\), then we can view the family of spaces of quantum symmetric tensors as the map \((0,\infty) \to \Gr(V^{\otimes n}): q \mapsto \symq^nV\).
  We caution that this map is not necessarily continuous (although it is continuous for \(n=2\); see \cref{prop:continuity-of-symmetric-and-exterior-squares}).
  This has much to do with the property of \emph{flatness} of the quantum symmetric algebra \(\symq(V)\) (see \cref{sec:flatness-and-related-properties}).
  If \(\symq(V)\) is flat, then \(q \mapsto \symq^nV\) is continuous for each \(n \geq 2\) (and similarly for the quantum antisymmetric tensors), but if \(\symq(V)\) is not flat, then the dimension of \(\symq^n V\) can jump at certain algebraic values of \(q\); see \cref{rem:flatness-and-transcendentality}.
\end{rem}

With notation as above, we have:
\begin{lem}
  \label{lem:continuity-of-cactus-gp-action}
  For any \(x \in J_n\), \(q \mapsto \rho_q(x)\) is a continuous map \((0,\infty) \to GL(V^{\otimes n})\).
\end{lem}

\begin{proof}
  As the endomorphisms \(s_{p,t}(q)\) are constructed from the commutors \(\sigmauvq{V^{\otimes k},V^{\otimes l}}\), continuity for the \(s_{p,t}(q)\) follows from \cref{prop:continuity-of-commutors}.
\end{proof}

We now describe the action of the cactus group on \(V^{\otimes n}\) at \(q = 1\):
\begin{lem}
  \label{lem:cactus-gp-action-at-q=1}
  At \(q = 1\), the action of the cactus group \(J_n\) on \(V^{\otimes n}\) factors through the canonical action of the symmetric group via the homomorphism \(x \mapsto \hat x\) from \eqref{eq:cactus-gp-hom-to-symmetric-gp}.
  In other words, for \(x \in J_n\) and \(v_1, \dots, v_n\in V\) we have 
  \begin{equation}
    \label{eq:cactus-gp-action-at-q=1}
    \rho_1(x) (v_1 \otimes \dots \otimes v_n) = v_{\hat x (1)} \otimes \dots \otimes v_{\hat x (n)}.
  \end{equation}
\end{lem}

\begin{proof}
  This follows from the definition \(\sigma_{UV}(1) = \tau_{UV}\) (see \cref{dfn:numerical-coboundary-operators}).
\end{proof}

We will require the following technical lemma for the proof of \cref{prop:cactus-gp-fixes-symmetric-tensors}:
\begin{lem}
  \label{lem:continuity-for-eigenvectors}
  Let \(W\) be a finite-dimensional vector space over \(\bbC\).
  Let \(I\) be an open interval in \(\bbR\), and let \(T : I \to \End_\bbC(W)\) be a continuous function.
  Let \(D \subseteq I\) be a dense subset, and suppose that there is a continuous function \(c : I \to \bbC\) such that \(c(t)\) is an eigenvalue of \(T(t)\) for all \(t \in D\).
  Then \(c(t)\) is an eigenvalue of \(T(t)\) for all \(t \in I\).
\end{lem}

\begin{proof}
  This follows immediately from the fact that the invertible operators form an open set in \(\End_\bbC(W)\).
\end{proof}

Now we prove the main result of this section:
\begin{prop}
  \label{prop:cactus-gp-fixes-symmetric-tensors}
  Let \(q_0 > 0\) be transcendental, let \(V \in \cO_{q_0}\), and let \(n \in \bbN\) with \(n \geq 2\).
  \begin{enumerate}[(a)]
  \item The space \(\symq[q_0]^nV\) of quantum symmetric \(n\)-tensors is fixed pointwise by the action of \(J_n\) on \(V^{\otimes n}\), i.e.\ \(\rho_{q_0}(x)w = w\) for all \(x \in J_n\) and \(w \in \symq[q_0]^nV\).
  \item Let \(\cS \cO_{q_0}\) be the category of super-representations of \(U_{q_0}(\fg)\) as in \cref{sec:coboundary-super-representations}, and give \(\cS \cO_{q_0}\) the coboundary structure as in \cref{dfn:commutors-on-super-representations}.
    Then the space \(\extq[q_0]^n V_\o\) is fixed pointwise by the action of the cactus group on \(V_\o^{\otimes n}\).
  \end{enumerate}  
\end{prop}

\begin{proof}
  We focus on (a), as the proof for (b) is similar.
  Begin by replacing \(V\) by a universal model.
  We prove the result by induction on \(n\).
  The base case is \(n = 2\).
  Since \(\symq[q_0]^2 V\) is defined as the fixed points of \(s_{1,2}(q_0)\), and \(J_2 = \{ e, s_{1,2} \}\), the result holds trivially.
  
  Now suppose that \(n > 2\).
  It suffices to show that \(\symq[q_0]^nV\) is fixed pointwise by each \(s_{p,t}(q_0)\).
  If \(t-p<n-1\), then \(s_{p,t}(q_0)\) acts on at most \(n-1\) consecutive tensor factors of \(V^{\otimes n}\).
  We assume without loss of generality that it acts on the first \(n-1\) factors, so we can view \(s_{p,t}\) as an element of \(J_{n-1}\).
  Since 
  \[\symq[q_0]^n V = (\symq[q_0]^{n-1}V \otimes V) \cap (V^{\otimes n-2} \otimes \symq[q_0]^2V),\]
  the inductive hypothesis implies that \(s_{p,t}(q_0)\) acts as the identity on \(\symq[q_0]^{n-1} V \otimes V\), and hence on \(\symq[q_0]^{n} V\) as well.

  It is now left to show that \(s_{1,n}(q_0)\) acts as the identity on \(\symq[q_0]^nV\).
  We begin by showing that \(s_{1,n}(q_0)\symq[q_0]^nV \subseteq \symq[q_0]^nV\).
  According to \cref{dfn:cactus-group}(c), for \(1 \leq j \leq n-1\) the relation
  \[
   s_{j,j+1}s_{1,n} = s_{1,n}s_{n-j,n-j+1}
  \]
  holds in \(J_n\).
  Hence for any \(w \in \symq[q_0]^n V\) and any \(1 \leq j \leq n-1\), we have
  \[
   s_{j,j+1}(q_0)s_{1,n}(q_0)w = s_{1,n}(q_0)s_{n-j,n-j+1}(q_0) w = s_{1,n}(q_0)w.
  \]
  Since \(s_{1,n}(q_0) w\) is fixed by all \(s_{j,j+1}(q_0)\), we conclude that \(s_{1,n}(q_0) w \in \symq[q_0]^nV\), and hence \(\symq[q_0]^nV\) is invariant under \(s_{1,n}(q_0)\).

  Now we want to show that \(s_{1,n}(q_0)\) acts as the identity on \(\symq[q_0]^n V\).
  Since \(s_{1,n}(q_0)\) is an involution, its eigenvalues can be only \(\pm 1\).
  So we need to show that \(-1\) does not occur as an eigenvalue of \(s_{1,n}(q_0)\) on \(\symq[q_0]^n V\).
  Our strategy is to show that if \(s_{1,n}(q_0)\) does have \(-1\) as an eigenvalue on \(\symq[q_0]^nV\), then \(s_{1,n}(q)\) has \(-1\) as an eigenvalue on \(\symq^n V\) for all \(q > 0\).
  We will obtain a contradiction by taking \(q = 1\).
  Indeed, \cref{lem:cactus-gp-action-at-q=1} implies that \(S_1^nV = S^nV\), the classical space of symmetric tensors, and that \(s_{1,n}(1)\) is just the permutation reversing the interval from \(1\) to \(n\).
  But this permutation acts as the identity on \(S^nV\), and so cannot have \(-1\) as an eigenvalue.

  So, suppose that \(s_{1,n}(q_0)\) has an eigenvalue \(-1\) on \(\symq[q_0]^nV\).
  By \cref{cor:all-coboundary-operators-act-by-rational-functions}, there is an eigenvector \(v(q_0) \in \symq[q_0]^nV\) with eigenvalue \(-1\) such that the coordinates of \(v(q_0)\) in the basis \(B(V^{\otimes n})\) are rational functions in \(q_0\) whose denominators have no roots on the positive real line.
  We can view the equality \(s_{1,n}(q_0)v(q_0) = -v(q_0)\) as a system of rational equations in \(q_0\).
  Since \(q_0\) is transcendental, this system of equations holds wherever it is defined.
  Thus if we define \(v(q)\) by replacing \(q_0\) with \(q\) in the coordinates of \(v(q_0)\), we have \(s_{1,n}(q) v(q) = - v(q)\) for all \(q > 0\).
  A similar argument shows that \(s_{j,j+1}(q)v(q) = v(q)\) for \(1 \leq j \leq n-1\), so that \(v(q) \in \symq^nV\) for all \(q > 0\).

  If \(q\) is transcendental, then \(v(q) \neq 0\), and hence it is an eigenvector of \(s_{1,n}(q)\) with eigenvalue \(-1\).
  Now we apply \cref{lem:continuity-for-eigenvectors} to the family of operators \(q \mapsto s_{1,n}(q)\) (with \(c\) the constant function \(-1\)) to conclude that \(s_{1,n}(q)\) has \(-1\) as an eigenvalue for all \(q > 0\).

  We do not know \emph{a priori} that the \(-1\)-eigenvector for \(s_{1,n}(q)\) lies in \(\symq^nV\) for algebraic values of \(q\).
  But if in the proof of \cref{lem:continuity-for-eigenvectors} we take the sequence of unit eigenvectors for the various \(s_{1,n}(t_m)\) to lie in \(\symq[t_m]^nV\), then that lemma implies also that the limit eigenvector will lie in \(\symq[q_0]^nV\).

  As indicated above, we get a contradiction when \(q = 1\) because the permutation group acts trivially on the symmetric tensors.
  This completes the proof.
\end{proof}

\begin{rem}
  \label{rem:on-cactus-gp-fixing-symm-tensors-prop}
  Note that the quantum symmetric vectors \(\symq^nV\) are defined to be the space of common fixed points of the elements \(s_{j,j+1} \in J_n\).
  These elements do not generate \(J_n\), but nevertheless \cref{prop:cactus-gp-fixes-symmetric-tensors} implies that \(\symq^nV\) is fixed by the entire cactus group.
  There is no reason to expect that this holds for an arbitrary coboundary structure on a monoidal category.

  We expect this result to hold for all algebraic \(q>0\) as well (and more generally for all non-roots of unity \(q \in \bbC^\times\)), but our method of proof clearly does not extend to cover that situation.
\end{rem}

\section{Quantum symmetric algebras are commutative}
\label{sec:quantum-symmetric-algebras-are-commutative}

In this section we show that quantum symmetric algebras \(\symq(V)\) are analogous to classical ones in that they are ``commutative'' in a certain precise sense, and moreover that \(\symq(V)\) is the \emph{universal} commutative algebra that is the target of a module map from \(V\).

\subsection{Commutative algebras in coboundary categories}
\label{sec:commutative-algebras-in-coboundary-categories}

\index[term]{algebra in a monoidal category}
Recall that an \emph{algebra} (or \emph{monoid}) in a monoidal category \((\cC, \otimes, 1_\cC)\) is an object \(A \in \cC\) equipped with morphisms \(m : A \otimes A \to A\) and \(u : 1_\cC \to A\) satisfying the usual associativity and unit axioms, i.e.\ such that the diagrams
\begin{equation*}
  \begin{CD}
    A \otimes A \otimes A @>{m \otimes \id}>> A \otimes A\\
    @V{id \otimes m}VV @VV{m}V\\
    A \otimes A @>>{m}> A
  \end{CD}
\end{equation*}
and 
\begin{equation*}
  \begin{tikzpicture}[description/.style={fill=white,inner sep=2pt}]
    \matrix (m) [matrix of math nodes, row sep=3em,
    column sep=3em, text height=1.5ex, text depth=0.25ex]
    { 1_\cC \otimes A & A \otimes A & A \otimes 1_\cC\\
      & A &\\  };
    \path[->,font=\scriptsize]
    (m-1-1) edge node[auto] {\( u \otimes \id \)} (m-1-2)
    (m-1-3) edge node[auto,swap] {\( \id \otimes u \)} (m-1-2)
    (m-1-2) edge node[auto] {\(m\)} (m-2-2)
    (m-1-1) edge node[auto,swap] {\(\cong\)} (m-2-2)
    (m-1-3) edge node[auto] {\(\cong\)} (m-2-2);
  \end{tikzpicture}
\end{equation*}
commute, where the diagonal arrows in the second diagram are the unitor morphisms in \(\cC\).

\begin{dfn}
  \label{dfn:commutative-alg-in-coboundary-category}
  \index[term]{algebra in a monoidal category!commutative}
  Suppose that \((\cC,\otimes,1_\cC)\) is a monoidal category and that \((A,m,u)\) is an algebra object in \(\cC\).
  If \((\sigma_{VW})_{V,W \in \cC}\) is a coboundary structure on \(\cC\), we say that A is \emph{commutative} if \(m = m \circ \sigma_{AA}\).
  For an object \(V\) of \(\cC\), an \emph{enveloping commutative algebra} of \(V\) in \(\cC\) is a commutative algebra \(A\) in \(\cC\) together with a morphism \(V \to A\) such that any morphism from \(V\) to a commutative algebra \(B\) in \(\cC\) factors uniquely through a morphism of algebras \(A \to B\).
\end{dfn}

As usual with such universal constructions, if an enveloping commutative algebra of \(V\) exists, it is unique up to unique isomorphism.

\begin{eg}
  \label{eg:symmetric-and-exterior-algebras-are-commutative}
  Via the homomorphism \(J_n \to S_n\) introduced in \eqref{eq:cactus-gp-hom-to-symmetric-gp}, any symmetric monoidal category can be viewed also as a coboundary category.
  If we take \(\cC = \vect\) with the symmetric monoidal structure given by the usual tensor flip, then for any vector space \(V\), the symmetric algebra \(S(V)\) is an enveloping commutative algebra of \(V\).

  Let us take instead \(\cC = \svect\), the symmetric monoidal category of super-vector spaces with the symmetric monoidal structure defined analogously to \cref{dfn:commutors-on-super-representations}.
  For any (ungraded) vector space \(V\), the exterior algebra \(\ext(V)\) with its parity grading is an enveloping commutative algebra of \(V_\o\) in \(\svect\).
  Thus, we can view exterior algebras as symmetric algebras in a different category.
\end{eg}

\begin{rem}
  \label{rem:why-not-braided-commutative}
  There is an analogous notion of a commutative algebra object in a \emph{braided} monoidal category. 
  Namely, we say that \(A\) is commutative if \(m \circ \rhat_{AA} = m\), where \(m\) is the multiplication of \(A\) and \(\rhat_{AA}\) is the braiding of \(A\) with itself.  
  While this notion may be useful in some settings, it is not appropriate for the quantum symmetric algebras.
  
  Indeed, suppose that \(V \in \oq\), that \(R \subseteq V \otimes V\), and that \(A = T(V)/\langle R \rangle\) is a (braided) commutative quadratic \(\uqg\)-module algebra generated by \(V\).
  By commutativity of \(A\) and naturality of the braidings, we have 
  \[
  m|_{V \otimes V} = m \circ \rhat_{AA}|_{V \otimes V} = m \circ \rhat_{VV}
  \] 
  when we restrict the multiplication of \(A\) to \(V \otimes V \subseteq A \otimes A\).
  Thus we have
  \begin{equation}
    \label{eq:why-not-braided-commutative}
    m|_{V \otimes V} \circ (\id - \rhat_{VV}) = 0.
  \end{equation}
  In general, the braiding \(\rhat_{VV}\) does not have \(1\) as an eigenvalue, so in this case \eqref{eq:why-not-braided-commutative} implies that \(m|_{V \otimes V}=0\).
  This means that the product of any two generators is zero, and hence \(A = \bbC \oplus V\), which is not a very interesting algebra.
  When working with the coboundary structure, since \(\sigma_{VV}^2 = \id\), we avoid this problem.
\end{rem}

\subsection{Commutativity and universality}
\label{sec:commutativity-and-universality}

In this section we show that for any \(V \in \oq\), the quantum symmetric algebra \(\symq(V)\) is commutative in the sense of \cref{dfn:commutative-alg-in-coboundary-category}.
As \(\symq(V)\) is always infinite-dimensional \cite{BerZwi08}*{Proposition 2.13}, \(\symq(V)\) is not an object in the category \(\oq\).
Thus, in order to show that \(\symq(V)\) is a commutative algebra, we must first define an appropriate coboundary category in which this algebra is an object.

\begin{notn}
  \label{dfn:coboundary-structure-for-quantum-sym-alg}
  We denote by \(\oqint\) \index[notn]{Oqint@\(\oqint\)} the category of \emph{integrable modules}, i.e.\ the full subcategory of \(\uqg\)-\textbf{Mod} whose objects are arbitrary direct sums of the \(V(\lambda)\).
  We denote by \(\soqint\)\index[notn]{SOqint@\(\soqint\)} the \(\bbZ/2\)-graded version of \(\oqint\) category, as in \cref{sec:coboundary-super-representations}.
\end{notn}

It is clear that \(\oqint\) is closed under tensor products, and also that the commutors from \(\oq\) extend naturally to \(\oqint\).
Thus \(\oqint\) has the structure of a coboundary category.
Applying similar constructions as in \cref{sec:coboundary-super-representations}, we obtain also a coboundary category structure on \(\soqint\).

Now we prove the commutativity property:
\begin{prop}
  \label{prop:quantum-symmetric-algebra-is-commutative}
  Let \(q > 0\) be transcendental and \(V \in \oq\).  Then:
  \begin{enumerate}[(a)]
  \item The quantum symmetric algebra \(\symq(V)\) is a commutative algebra in \(\oqint\).
  \item The quantum exterior algebra \(\extq(V_\o)\) is a commutative algebra in \(\soqint\).
  \end{enumerate}
\end{prop}

\begin{proof}
  As with \cref{prop:cactus-gp-fixes-symmetric-tensors}, we prove the statement only for the quantum symmetric algebra; the proof for the quantum exterior algebra is essentially identical, since the signs are absorbed into the definition of the coboundary structure on \(\soqint\).

    We show that \(\symq(V)\) is commutative by lifting everything to the tensor algebra. 
  Let \(n \ge 2\) and \(1 \leq r \leq n-1\). 
  We want to show that the lower triangle in the diagram  
  \[
  \begin{tikzpicture}[anchor=base,cross line/.style={preaction={draw=white,-,line width=6pt}}]
    \path (0,0) node (1) {\(\scriptstyle V^{\otimes r} \otimes V^{\otimes (n-r)}\)} +(6,0) node (2) {\(\scriptstyle V^{\otimes n}\)} +(2,-1) node (3) {\(\scriptstyle V^{\otimes (n-r)}\otimes V^{\otimes r}\)} +(0,-4) node (4) {\(\scriptstyle S^r\otimes S^{n-r}\)} +(6,-4) node (5) {\(\scriptstyle S^n\)} +(2,-5) node (6) {\(\scriptstyle S^{n-r}\otimes S^r\)}; 
    \draw[->] (1) -- (2) node[pos=.5,auto] {\(\scriptstyle m\)};
    \draw[->] (1) -- (3) node[pos=.5,auto,swap] {\(\scriptstyle \tilde\gamma_q\)};
    \draw[->] (3) -- (2) node[pos=.5,auto,swap] {\(\scriptstyle m\)};
    \draw[->] (4) -- (5) node[pos=.5,auto] {\(\scriptstyle m\)};
    \draw[->] (4) -- (6) node[pos=.5,auto,swap] {\(\scriptstyle \gamma_q\)};
    \draw[->] (6) -- (5) node[pos=.5,auto,swap] {\(\scriptstyle m\)};
    \draw[->] (1) -- (4) node[pos=.5,auto,swap] {\(\scriptstyle \pi\otimes\pi\)};
    \draw[->] (2) -- (5) node[pos=.5,auto] {\(\scriptstyle \pi\)};
    \draw[->,cross line] (3) -- (6) node[pos=.5,auto] {\(\scriptstyle \pi\otimes\pi\)};
  \end{tikzpicture}
  \]
  commutes.
  Here we denote \(\sym^j = \symq^j(V)\),
   \(\pi : V^{\otimes j} \to \sym^j\) are the quotient maps, \(m\) denotes multiplication in both the tensor algebra and the quantum symmetric algebra, and \(\tilde\gamma_q, \gamma_q\) are the commutors
  \[
  \tilde\gamma_q \eqdef \sigma_{V^{\otimes r}, V^{\otimes (n-r)}} = \sigma_{1,r,n}(q): V^{\otimes r} \otimes V^{\otimes (n-r)} \to V^{\otimes (n-r)} \otimes V^{\otimes r}
  \]
  and
  \[
  \gamma_q \eqdef \sigma_{\sym^r,\sym^{n-r}} : \symq^r(V) \otimes \symq^{n-r}(V) \to \symq^{n-r}(V) \otimes \symq^{r}(V),
  \]
  respectively (see \cref{notn:parametrized-cactus-group-action}).

  All of the squares in the diagram commute.
  Indeed, the back square and the front right square commute by definition of multiplication in \(\symq(V)\), and the front left square commutes by naturality of the coboundary structure.

  Since \(\pi \otimes \pi\) is surjective, a diagram chase shows that it suffices to prove that 
  \begin{equation*}
    \pi \circ m \circ \tilde\gamma_q = \pi \circ m : V^{\otimes r} \otimes V^{\otimes (n-r)} \to S^n.
  \end{equation*}
  The two arrows labeled \(m\) at the top of the diagram are the canonical identifications of \(V^{\otimes r} \otimes V^{\otimes  (n-r)}\) and \(V^{\otimes (n-r)} \otimes V^{\otimes r}\) with \(V^{\otimes n}\), so we can regard \(\tilde\gamma_q\) as an operator on \(V^{\otimes n}\).
  With this identification in mind, we need to prove that \(\pi \circ \tilde\gamma_q = \pi\), or in other words that \(\ran (\tilde\gamma_q - \id) \subseteq \ker (\pi)\).
  Equivalently, we will show that \(\ker(\pi)^\perp \subseteq \ran (\tilde\gamma_q - \id)^\perp\).
  
  In the proof of \cref{prop:symmetrization-and-antisymmetrization} we saw that \(\ker(\pi)^\perp = \symq^n V\) (note that \(\ker(\pi) = J_n\) in the notation of that proposition), so we need to show that \(\symq^n V \subseteq \ran (\tilde\gamma_q - \id)^\perp\).
  Since \(\tilde\gamma_q\) is unitary, we have
  \[ 
  \ran(\tilde\gamma_q - \id)^\perp = \ker(\tilde\gamma_q^* - \id) = \ker(\tilde\gamma_q^{-1} - \id),    
  \]
  so we have reduced the problem to showing that \(\symq^nV \subseteq \ker(\tilde\gamma_q^{-1} - \id)\).
  But \(\tilde\gamma_q\) acts as the identity on \(\symq^nV\) by \cref{prop:cactus-gp-fixes-symmetric-tensors}.
  This completes the proof.
\end{proof}

Now we are ready to prove that the quantum symmetric and exterior algebras have the universal properties we claimed:

\begin{thm}
  \label{thm:quantum-symmetric-algebra-is-enveloping-comm-alg}
  Let \(q > 0\) be transcendental and \(V \in \oq\).  Then:
  \begin{enumerate}[(a)]
  \item The natural inclusion \(V \hookrightarrow \symq(V)\) makes  \(\symq(V)\) an enveloping commutative algebra of \(V\) in \(\oqint\).
  \item The natural inclusion \(V \hookrightarrow \extq(V)\) makes  \(\extq(V)\) an enveloping commutative algebra of \(V_\o\) in \(\soqint\).
  \end{enumerate}
\end{thm}

\begin{proof}
  Again, we focus on (a), as the proof of (b) is similar.

  We know from \cref{prop:quantum-symmetric-algebra-is-commutative} that \(\symq(V)\) is commutative, so it remains only to verify the universality property. 
  Suppose that \(A\) is a commutative algebra in \(\oqint\) and \(f : V \to A\) is a module map. 
  By the universal property of the tensor algebra, \(f\) lifts uniquely to a morphism of algebras \(\tilde{f} : T(V) \to A\). 
  We only need to show that \(\tilde{f}\) factors as 
  \[   
  T(V) \overset{\pi}{\longrightarrow} \symq(V)  \overset{\hat{f}}{\longrightarrow}  A,   
  \]
  or in other words that \(\langle \extq^2 V \rangle \subseteq \ker(\tilde{f})\). 
  The uniqueness of \(\hat{f}\) then follows from the uniqueness of \(\tilde{f}\) and surjectivity of the quotient map \(T(V) \to \symq(V)\).

  We now show that \(\extq^2 V \subseteq \ker(\tilde{f})\). 
  On the degree two component of \(T(V)\), \(\tilde{f}\) is defined by \(\tilde{f}(v \otimes w) = m_A(f(v) \otimes f(w))\), where \(m_A\) is the multiplication map of \(A\). 
  By the naturality of the coboundary operators, \((f \otimes f) (\extq^2V )\subseteq \extq^2 A\). 
  Hence it suffices to show that \(m_A\) vanishes on \(\extq^2 A\). 
  But this is immediate from the definition of commutativity.
\end{proof}

\begin{rem}
  \label{rem:on-the-commutativity-property}
  As with \cref{prop:cactus-gp-fixes-symmetric-tensors}, we expect that the conclusions of \cref{prop:quantum-symmetric-algebra-is-commutative} and \cref{thm:quantum-symmetric-algebra-is-enveloping-comm-alg} hold also for all algebraic \(q > 0\), and in general for all non-roots of unity \(q \in \bbC^{\times}\).
\end{rem}

\section{Flatness and related properties}
\label{sec:flatness-and-related-properties}

It is natural to ask how closely the quantum symmetric and exterior algebras of finite-dimensional \(\uqg\)-modules mirror their classical counterparts.
Let \(V \in \oq\), and consider the quantum symmetric algebra \(\symq(V)\) and the classical symmetric algebra \(\sym(V)\). 
One way to compare these two algebras is to examine the dimensions of their graded components to see if they coincide.
We refer to this property as \emph{flatness}; see \cref{dfn:flatness-for-algebras,dfn:flatness-for-modules}.
In \cref{sec:flatness-and-parameter-values} we examine the effect of varying the deformation parameter on the size of the quantum symmetric and exterior algebras.
In \cref{sec:pbw-bases-and-koszul-property} we discuss PBW bases and the homological condition known as the Koszul property.
Then in \cref{sec:flatness-and-cominuscule-parabolics,sec:quantum-sym-algs-as-quantum-schubert-cells} we review some results of Zwicknagl from \cite{Zwi09} concerning the flat quantum symmetric and exterior algebras.
In \cref{sec:filtrations-on-flat-quantum-algebras} we define certain natural filtrations on these flat algebras and examine the associated graded algebras, and in \cref{sec:quantum-exterior-algebras-are-frobenius} we use these filtrations to show that flat quantum exterior algebras are Frobenius algebras.

\subsection{Definition of flatness}
\label{sec:definition-of-flatness}

The following definition allows us to encode the dimensions of the homogeneous components of a graded algebra into a single object:
\begin{dfn}[\cite{PolPos05}*{Ch.~2, \S 2}]
  \label{dfn:hilbert-series}
  \index[notn]{hAz@\(h_A(z)\)}
  \index[term]{Hilbert series}
  Let \(A = \bigoplus_{n=0}^{\infty} A_n\) be a \(\zp\)-graded algebra over \(\bbC\) such that each \(A_n\) is finite-dimensional.
  The \emph{Hilbert series} of \(A\) is the generating function \(h_A(z)\) defined by 
  \begin{equation*}
    \label{eq:hilbert-series-definition}
    h_A(z) \eqdef \sum_{n=0}^{\infty} \dim (A_n) z^n.
  \end{equation*}
\end{dfn}

\begin{eg}
  \label{eg:hilbert-series-for-sym-and-ext-algs}
  Let \(V\) be a vector space over \(\bbC\) with \(\dim (V) = d < \infty\).
  Then the Hilbert series of the symmetric and exterior algebras of \(V\) are given by
  \begin{equation*}
    h_{\sym(V)}(z) = \sum_{n=0}^\infty \binom{d+n-1}{n}z^n = \frac{1}{(1-z)^d}
  \end{equation*}
  and
  \begin{equation*}
    h_{\ext(V)}(z) = \sum_{n=0}^d \binom{d}{n}z^n = (1+z)^d,
  \end{equation*}
  respectively.
\end{eg}

The following definition is the fundamental notion that we consider in \cref{sec:flatness-and-related-properties}:
\begin{dfn}[\cite{BerZwi08}*{Definition 2.27}]
  \label{dfn:flatness-for-algebras}
  Let \(V \in \oq\) be a finite-dimensional Type 1 \(\uqg\)-module.
  We say that \(\symq(V)\) \emph{is a flat deformation of \(\sym(V)\)} if 
  \[
  \dim \symq^n(V) = \dim \sym^n(V)
  \]
  for all \(n \geq 0\), i.e.\ if \(h_{\symq(V)}(z) = h_{\sym(V)}(z)\).
  Similarly, we say that \(\extq(V)\) \emph{is a flat deformation of \(\ext(V)\)} if
  \[
  \dim \extq^n(V) = \dim \ext^n(V)
  \]
  for all \(n \geq 0\), i.e.\ if \(h_{\extq(V)}(z) = h_{\ext(V)}(z)\).
  We may also just say that \(\symq(V)\) or \(\extq(V)\) is \emph{flat}.
  \index[term]{flat quantum symmetric/exterior algebra}
  See also \cref{dfn:flatness-for-modules} below.
\end{dfn}

\begin{eg}
  \label{eg:flatness-for-symmetric-algebra-of-2-dim-rep}
  Consider the quantum symmetric algebra \(\symq(V)\) of the two-dimensional representation of \(\uqsl\) that we constructed in \cref{eg:quantum-symmetric-algebra-for-2-dim-rep}.
  It was given by
  \[
  \symq(V) = \bbC \langle x_1,x_2 \rangle / \langle x_2 x_1 = q^{-1}x_1x_2 \rangle.
  \]
  It is clear from the relations that the set \(\{ x_1^a x_2^b \mid  a,b \geq 0 \}\) spans \(\symq(V)\).
  It follows from Bergman's Diamond Lemma \cite{Ber78}*{Theorem 1.2} that these elements are independent (there are no ambiguities in the reduction system) and hence they form a basis, so \(\symq(V)\) is flat.
  
  The following elementary argument also shows that \(\symq(V)\) is flat in this instance.
  Let \(W\) be a complex vector space with a countable basis \((e_n)_{n = 0}^\infty\).
  Define operators \(X_1,X_2\) on \(W\) by
  \[
  X_1e_n = e_{n-1}, \quad X_2e_n = q^n e_n, 
  \]
  where we denote \(e_{-1} = 0\).
  It is straightforward to verify that \(X_2 X_1 = q^{-1}X_1 X_2\), so that \(x_1 \mapsto X_1, x_2 \mapsto X_2\) gives a representation of \(\symq(V)\) on \(W\).
  It is also straightforward to see that \(\{ X_1^a X_2^b \mid a,b \geq 0 \}\) is a linearly independent set of operators on \(W\).
\end{eg}

Not all quantum symmetric algebras are flat.
The following example shows that the quantum symmetric algebra of a direct sum of modules need not be flat, even if the quantum symmetric algebras of the two modules are flat:
\begin{eg}
  \label{eg:non-flatness-for-2-copies-of-2-dim-rep}
  Take again \(\fg = \fsl_2\).
  This time we consider the direct sum of two copies of the two-dimensional representation \(V\).
  For clarity we denote the second copy of the representation by \(V'\), and the corresponding basis by \(\{ y_1, y_2 \}\).
  With respect to the decomposition 
  \[  
  (V \oplus V')^{\otimes 2} \cong V^{\otimes 2} \oplus (V \otimes V') \oplus (V' \otimes V) \oplus (V')^{\otimes 2},
  \]
  the commutor of \(V \oplus V'\) with itself is given by
  \[
  \sigma_{V \oplus V',V \oplus V'} =
  \begin{pmatrix}
    \sigma_{V,V} & 0 & 0 & 0 \\
    0 & 0 & \sigma_{V',V} & 0 \\
    0 & \sigma_{V,V'} & 0 & 0 \\
    0 & 0 & 0 & \sigma_{V',V'}
  \end{pmatrix},
  \]
  where all four commutors in this block matrix equal the commutor \(\sigma_{VV}\) from \eqref{eq:braiding-and-commutor-for-2-dim-rep}.
  Diagonalizing this matrix and taking the \(-1\) eigenspace, we obtain the following set of relations for the quantum symmetric algebra \(\symq(V \oplus V')\):
  \begin{equation}
    \label{eq:non-flat-quantum-symmetric-algebra-example}
    \begin{gathered}
      x_2 x_1 = q^{-1}x_1 x_2, \quad y_1 x_1 = x_1 y_1, \quad y_2 x_1
      = \frac{2 q}{1+q^2}x_1 y_2 + \frac{1-q^2}{1+q^2}x_2 y_1,\\
      y_2 y_1 = q^{-1}y_1 y_2, \quad y_2 x_2 = x_2 y_2, \quad y_1 x_2 =  \frac{q^2-1}{1+q^2}x_1 y_2 + \frac{2q}{1+q^2}x_2 y_1.
    \end{gathered}
  \end{equation}
  To see that \(\symq(V \oplus V')\) is not flat, we take the monomial \(y_1 x_2 x_1\) and reduce it in two ways according to the relations \eqref{eq:non-flat-quantum-symmetric-algebra-example}.
  On the one hand we have
  \begin{align*}
    y_1 (x_2 x_1) & = q^{-1} (y_1 x_1) x_2 \\
    & = q^{-1}x_1 (y_1 x_2) \\
    & =   \frac{q-q^{-1}}{1 + q^2}x_1^2 y_2 + \frac{2}{1 + q^2} x_1 x_2 y_1 .
  \end{align*}
  On the other hand we have
  \begin{align*}
    (y_1 x_2) x_1 & = \left( \frac{q^2-1}{1+q^2}x_1 y_2 + \frac{2q}{1+q^2}x_2 y_1 \right)x_1 \\
    & = \frac{q^2-1}{1+q^2} x_1 \left( \frac{2 q}{1+q^2}x_1 y_2 + \frac{1-q^2}{1+q^2}x_2 y_1 \right) + \frac{2q}{1+q^2} x_2 x_1 y_1\\
    & = \frac{2q(q^2-1)}{(1+q^2)^2} x_1^2 y_2 + \left( \frac{2}{1+q^2} - \frac{(1-q^2)^2}{(1+q^2)^2} \right)x_1x_2y_1.
  \end{align*}
  Equating the two expressions for \(y_1x_2x_1\), we see that the two terms \(\frac{2}{1+q^2}x_1 x_2 y_1\) cancel.
  Simplifying the remaining terms (using the fact that \(q\) is not a root of unity), we are left with the relation
  \[
  (q-q^{-1})x_1^2y_2 = (q^2-1) x_1x_2y_1,
  \]
  which simplifies further to \(x_1^2 y_2 = q x_1 x_2 y_1\).
  This shows that the degree three component of \(\symq(V \oplus V')\) has strictly smaller dimension than that of \(\sym(V \oplus V')\), and hence \(\symq(V \oplus V')\) is not flat.
\end{eg}

\begin{rem}
  \label{rem:flatness-and-transcendentality}
  Examining \cref{eg:non-flatness-for-2-copies-of-2-dim-rep} more closely, we see that flatness can also be influenced by the value of the deformation parameter \(q\).
  Indeed, the relation \(  (q-q^{-1})x_1^2y_2 = (q^2-1) x_1x_2y_1\) shows that \(\symq(V \oplus V')\) has more relations than \(\sym(V \oplus V')\) in general.
  However, when \(q = \pm 1\), this relation vanishes, and so is not an obstruction to flatness (and, of course, at \(q=1\) the algebra is flat by definition).
  We conclude that for a fixed \(V \in \oq\) and \(n \geq 2\) the value of \(\dim \symq^n(V)\) can vary with \(q\).
\end{rem}

\subsection{Flatness and parameter values}
\label{sec:flatness-and-parameter-values}

In this section we prove a number of results concerning the effect of the parameter value on the Hilbert series of the quantum symmetric and exterior algebras of a fixed module.
This issue has been considered in \cite{Dri92,Ber00}; see also \cite{PolPos05}*{Ch.~6, Theorem 2.1}.

We now introduce formal versions of the quantum symmetric and exterior algebras:

\begin{dfn}
  \label{dfn:formal-quantum-symmetric-algebra}
  Let \(\uvq\) be as in \cref{sec:formal-and-integral-versions}, and let \(V \in \ov\).
  Let \(\tilde{\sigma}_{VV}(\nu)\) be the commutor from \eqref{eq:coboundary-operator-for-formal-situation}.
  Then in analogy with \cref{dfn:quantum-symmetric-and-exterior-algebras} we define
  \[
  \symq[\nu](V) \eqdef T(V)/ \langle \extq[\nu]^2 V \rangle, \quad 
  \extq[\nu](V) \eqdef T(V)/ \langle \symq[\nu]^2 V \rangle,
  \]
  where 
  \[
  \extq[\nu]^2 V \eqdef \ker \left( \id + \tilde{\sigma}_{VV}(\nu) \right) = \im \left( \id - \tilde{\sigma}_{VV}(\nu) \right)
  \]
  and
  \[
  \symq[\nu]^2 V \eqdef \ker \left( \id - \tilde{\sigma}_{VV}(\nu) \right) = \im \left( \id + \tilde{\sigma}_{VV}(\nu) \right).
  \]
\end{dfn}

Our first main result of this section says that if we specialize the formal variable \(\nu\) to a \emph{transcendental} number \(q > 0\), the formal quantum symmetric algebra ``becomes'' the one defined over \(\bbC\) in \cref{dfn:quantum-symmetric-and-exterior-algebras}.
We require a preparatory lemma:

\begin{lem}
  \label{lem:setup-for-transcendental-specialization}
  Let \(V \in \ov\).
  Let \(q > 0\) be transcendental, and let \(V_q = \bbC \otimes_{\Qv} V\), where we view \(\Qv\) as a subfield of \(\bbC\) via the isomorphism \(\Qv \cong \bbQ(q)\).
  Then:
  \begin{enumerate}[(a)]
  \item  \(V_q\) has a natural structure of a \(\uqg\)-module, and the decomposition of \(V_q\) into highest weight modules over \(\uqg\) is the same as that of \(V\) over \(\uvq\).
  \item If \(\tilde{\sigma}_{VV}(\nu)\) denotes the commutor of \(\uvq\)-modules as above, then \(\id_\bbC \otimes \tilde{\sigma}_{VV}(\nu) = \sigma_{V_qV_q}(q)\).
  \end{enumerate}
\end{lem}

\begin{proof}
  Part (a) is exactly the construction of the universal model for \(V\) as in \cref{sec:canonical-bases}, and part (b) follows from \cref{lem:specialized-commutors-agree-with-numerical-commutors}.
\end{proof}

This allows us to prove:

\begin{prop}
  \label{prop:hilbert-series-constant-for-transcendental-q}
  Let \(V \in \ov\), let \(q > 0\) be transcendental, and define the \(\uqg\)-module \(V_q\) as in \cref{lem:setup-for-transcendental-specialization}.
  Then the isomorphism \(V_q = \bbC \otimes_{\Qv} V\) extends to isomorphisms
  \begin{equation}
    \label{eq:transcendental-sym-alg-tensor-equality}
    \symq(V_q) \cong \bbC \otimes_{\bbQ(\nu)} \symq[\nu](V), \quad \extq(V_q) \cong \bbC \otimes_{\bbQ(\nu)} \extq[\nu](V)
  \end{equation}
  of graded algebras.
\end{prop}

\begin{proof}
  We prove the result just for the quantum symmetric algebras, as the argument for the exterior algebras is analogous.
  We have
  \begin{align*}
    \bbC \otimes_{\bbQ(\nu)} \symq[\nu](V) & = \bbC \otimes_{\bbQ(\nu)} \left( T(V)/ \langle \extq[\nu]^2 V \rangle \right)\\
    & \cong \left( \bbC \otimes_{\bbQ(\nu)} T(V) \right) / \left( \bbC \otimes_{\bbQ(\nu)} \langle \extq[\nu]^2 V \rangle\right)\\
    & \cong T \left( \bbC \otimes_{\bbQ(\nu)} V \right) / \langle \bbC \otimes_{\bbQ(\nu)}  \extq[\nu]^2 V \rangle\\
    & \cong T \left( \bbC \otimes_{\bbQ(\nu)} V \right) / \langle \bbC \otimes_{\bbQ(\nu)} \im \left( \id - \tilde{\sigma}_{V V}(\nu) \right) \rangle\\
    & \cong T \left( \bbC \otimes_{\bbQ(\nu)} V \right) / \langle \im \left( \id_\bbC \otimes (\id - \tilde{\sigma}_{V V}(\nu)) \right) \rangle\\
    & \cong T(V_q) / \langle \im ( \id - \sigma_{V_q V_q}(q) ) \rangle\\
    & = \symq(V_q).
  \end{align*}
  The last equality is by definition.
  The remaining isomorphisms essentially all follow from the fact that taking the tensor product with a field over a field is an exact functor.
\end{proof}

The Hilbert series \(h_A(z)\) of a \(\zp\)-graded algebra \(A\) over \(\Qv\) is defined in the same way as for a graded algebra over \(\bbC\) (see \cref{dfn:hilbert-series} above).
With this in mind, we have:

\begin{cor}[to \cref{prop:hilbert-series-constant-for-transcendental-q}]
  \label{cor:hilbert-series-constant-for-transcendental-q}
  Let \(q, q' > 0\) be transcendental and let \(V\) be a universal model for a module in \(\oq\).
  Then \(h_{\symq(V)}(z) = h_{\symq[q'](V)}(z)\).
\end{cor}

Next we show that for transcendental values of \(q\), the graded components of the quantum symmetric and exterior algebras are no larger than those of their classical counterparts.
This was also proved by different methods by Berenstein and Zwicknagl in \cite{BerZwi08}*{Main Theorem 2.21}.

\begin{prop}
  \label{prop:transcendental-quantum-algs-are-at-most-classical}
  Let \(q_0 > 0\) and let \(V \in \oq[q_0]\).
  Replace \(V\) by a universal model as in \cref{prop:universal-model-for-non-simple-rep}, and let \(B(V) = \{ x_1, \dots, x_n \}\) be the distinguished basis. 
  Let \(m \geq 2\) be a positive integer.
  Then for all \(q > 0\) except possibly a finite set of algebraic numbers, the following hold:
  \begin{enumerate}[(a)]
  \item The set of monomials \(\{ x_1^{a_1}\dots x_n^{a_n} \mid a_1, \dots, a_n \in \zp, \ \sum_j a_j = m\}\) spans \(\symq^m(V)\).
  \item The set of monomials \(\{ x_1^{a_1} \wedge \dots \wedge x_n^{a_n} \mid a_1, \dots, a_n \in \{ 0,1 \}, \ \sum_j a_j = m \}\) spans \(\extq^m(V)\).
  \end{enumerate}
\end{prop}

\begin{proof}
  We prove (a) only; the proof for (b) is similar.
  For \(1 \leq i, j \leq n\), we say that the quadratic monomial \(x_i x_j\) is \emph{ordered} if \(i \leq j\), and \emph{disordered} if \(i > j\).
  It suffices to show that the relations in \(\symq(V)\) permit us to express the disordered quadratic monomials in terms of the ordered ones.

  Recall that the space of relations for \(\symq(V)\) is given by \(\extq^2V = \im \left( \id - \sigma_{VV}(q) \right) \subseteq V \otimes V\).
  By \cref{prop:continuity-of-symmetric-and-exterior-squares}, the map \(q \mapsto \extq^2V\) is continuous.
  This implies that \(\dim \extq^2V = \binom{n}{2}\), and moreover that there are continuous functions \(w_1, \dots, w_{\binom{n}{2}} : (0,\infty) \to V \otimes V\) such that, for each \(q>0\), the set \(\{ w_1(q), \dots, w_{\binom{n}{2}}(q) \}\) forms a basis for \(\extq^2V\).

  By \cref{prop:commutors-act-by-rational-functions}, each \(w_i\) can be written in the form
  \[
  w_i(q) = \sum_{j,k} c_i^{jk}(q) x_j x_k,
  \]
  where the \(c_i^{jk}\) lie in \(\bbQ(q)\), and the denominators do not vanish on \(\bbR\).
  Rearranging the terms so that all disordered monomials appear on the left and ordered ones appear on the right, we see that the relations in \(\symq(V)\) are of the form
  \begin{equation}
    \label{eq:coefficient-matrix-for-symqv-reordering}
    \sum_{j > k} c_i^{jk}(q) x_j x_k = - \sum_{j \leq k} c_i^{jk}(q) x_j x_k, \quad 1 \leq i \leq \binom{n}{2}.
  \end{equation}
  Viewing the disordered monomials as variables, this is a system of \(\binom{n}{2}\) linear equations in \(\binom{n}{2}\) unknowns.
  The determinant of the coefficient matrix of the system lies in \(\bbQ(q)\), and it does not vanish identically because the system is solvable at \(q = 1\) (because the ordered monomials \emph{do} span \(\sym^2V\)).
  Thus there are at most finitely many (algebraic) values of \(q\) where the determinant vanishes.
  At all other values of \(q\), the system is solvable, and hence we can write the disordered quadratic monomials in terms of the ordered ones.
  For higher degree monomials the result follows from the quadratic case.
\end{proof}

\begin{dfn}
  \label{dfn:generic-set-for-V}
  \index[term]{generic set for a representation}
  With notation as in \cref{prop:transcendental-quantum-algs-are-at-most-classical}, we define the \emph{generic set for \(V\)} to be the set of positive real numbers \(q\) such that the conclusion of the proposition holds, i.e.~such that the ordered monomials in the distinguished basis span the algebras \(\symq(V)\) and \(\extq(V)\).
\end{dfn}

\begin{rem}
  \label{rem:on-size-of-quantum-algebras-at-algebraic-q}
  In all examples that I have investigated, the dimensions of the spaces \(\symq^m(V)\) and \(\extq^m(V)\) are at most the classical values, regardless of \(q\).
  In other words, the generic set consists of \emph{all} positive real \(q\).

  Indeed, in every example the coefficient matrix for the system \eqref{eq:coefficient-matrix-for-symqv-reordering} has determinant of the form \(\frac{1}{f(q)}\) (with \(f(q) \in \bbQ[q]\) nonvanishing on \(\bbR_{>0}\)) and hence it does not vanish anywhere.
  A more careful analysis of the form of the relations using the construction of the commutor is necessary to prove in general that the generic set for every representation is all of \(\bbR_{>0}\).
  Alternatively, it may be possible to exploit the fact that the map \(q \mapsto \sigma_{VV}(q)\) is actually analytic in order to establish this result; see \cite{Ber00}*{Theorem 4.7}.
\end{rem}

Our next main result shows that the dimensions of the graded components of the quantum symmetric and exterior algebras are minimal at transcendental parameter values.
We require the following technical lemma:

\begin{lem}
  \label{lem:lower-semicontinuity-for-subspace-sums}
  Let \(W\) be a finite-dimensional vector space over \(\bbF = \bbR\) or \(\bbC\).
  Let \(I \subseteq \bbR\) be an open interval and let \(L_1, \dots, L_k : I \to \Gr(W)\) be continuous functions to the full Grassmann manifold of \(W\) (we do not assume that the images lie in the same component of \(\Gr(W)\)).
  Define a function \(\ell : I \to \bbR\) by
  \[
  \ell(t) \eqdef \dim \left( \sum_{j=1}^k L_j(t) \right),
  \]
  where the sum is the sum of subspaces of \(W\), i.e.\ the set of all possible linear combinations of vectors from these subspaces.
  Then the function \(\ell\) is lower semicontinuous on \(I\).
\end{lem}

\begin{proof}
  It suffices to prove the statement for \(W = \bbF^n\).
  For each \(j\), since \(L_j\) is continuous, we can choose continuously parametrized bases for the \(L_j(t)\).
  More precisely, we can find a continuous function \(B_j : I \to \bbM_{n \times d_j}(\bbF)\), where \(d_j = \dim L_j(t)\), such that the columns of \(B_j(t)\) form a basis for \(L_j(t)\) for each \(t \in I\).
  Then \(\ell(t)\) is the rank of the combined matrix
  \[
  M(t) =
  \renewcommand\arraystretch{0.6}
  \begin{bmatrix}
    B_1(t) & \dots & B_k(t)
  \end{bmatrix}.
  \]
  Now \(t \mapsto M(t)\) is continuous, and matrix rank is lower semicontinuous (i.e.~for a fixed rank \(d\), the set of matrices of rank at least \(d\) is open).
  Thus \(\ell\) is lower semicontinuous.
\end{proof}

\renewcommand\arraystretch{1.2}

This lemma allows us to prove:

\begin{prop}
  \label{prop:upper-semicontinuity-for-quantum-symmetric-algs}
  Let \(q_0 > 0\) and let \(V \in \oq[q_0]\).
  Replace \(V\) be a universal model as in \cref{prop:universal-model-for-non-simple-rep}, and let \(m \geq 2\) be a positive integer.
  Then
  \begin{enumerate}[(a)]
  \item The functions \(q \mapsto \dim \symq^m(V)\) and \(q \mapsto \dim \extq^m(V)\) are upper semicontinuous.
  \item Both functions are constant on an open dense subset of \((0,\infty)\) that contains all positive transcendental numbers.
  \item Both functions take their minima on the dense open set from (b).
  \end{enumerate}
\end{prop}

\begin{proof}
  We prove the result just for \(\symq^m(V)\), as the proof for \(\extq^m(V)\) is similar.
  In view of \cref{prop:symmetrization-and-antisymmetrization}, for (a) it suffices to prove that \(q \mapsto \dim \symq^m V\) is upper semicontinuous.
  From the proof of \cref{prop:symmetrization-and-antisymmetrization} we know that \(V^{\otimes m} = \symq^mV \oplus \langle \extq^2V \rangle_m\), where \(\langle \extq^2V \rangle_m\) is the degree \(m\) component of the ideal generated by the quantum exterior square of \(V\).
  Thus it suffices to prove that \(\dim \langle \extq^2V \rangle_m\) is lower semicontinuous as a function of \(q\).

  For \(1 \leq j \leq m-1\), define \(L_j : (0,\infty) \to \Gr(V^{\otimes m})\) by
  \[
  L_j(q) = V^{\otimes j-1} \otimes \extq^2V \otimes V^{m-j-1}.
  \]
  Each \(L_j\) is continuous by \cref{prop:continuity-of-symmetric-and-exterior-squares}, and moreover we have \(\langle \extq^2V \rangle_m = \sum_j L_j(q)\).
  Then the result follows from \cref{lem:lower-semicontinuity-for-subspace-sums} (taking \(W = V^{\otimes m}\)).

  For parts (b) and (c), let \(M\) denote the common dimension of \(\symq^m(V)\) for transcendental \(q\) (recall \cref{prop:hilbert-series-constant-for-transcendental-q} above).
  By part (a), we know that the set
  \[
  \cU \eqdef \{ q > 0 \mid \dim \symq^m(V) < M \}
  \]
  is open.
  But \(\cU\) does not contain any transcendental numbers, and hence \(\cU\) is empty.
  Thus \(M\) is the minimal value of \(\dim \symq^m(V)\).
  Again using upper semicontinuity (and the fact that the dimension is integer-valued), we see that the set
  \[
  \cV \eqdef \{ q > 0 \mid \dim \symq^m(V) \leq M \}
  \]
  is open, and we know it contains all of the transcendental numbers.
  Thus \(\cV\) is dense in \((0,\infty)\), and since \(\cU = \emptyset\) we see that the dimension of \(\symq^m(V)\) is \(M\) for all \(q \in \cV\).
\end{proof}

\subsection{PBW bases and the Koszul property}
\label{sec:pbw-bases-and-koszul-property}

In this section we discuss the relation between flatness of the quantum symmetric and exterior algebras, PBW bases, and the homological condition known as the Koszul property.
We refer the reader to the monograph \cite{PolPos05} for more information on PBW algebras and Koszul algebras.

We begin with the definition of a PBW algebra.
This concept takes its name from the Poincar\'e-Birkhoff-Witt Theorem for universal enveloping algebras of Lie algebras.

\begin{dfn}[\cite{PolPos05}*{Ch.~4, \S 1}]
  \label{dfn:pbw-algebra}
  \index[term]{PBW!algebra}
  \index[term]{PBW!basis}
  Let \(A = T(V)/\langle R \rangle\) be a quadratic algebra, where \(V\) is a complex vector space of finite dimension and \(R \subseteq V \otimes V\) is the space of relations.
  Fix a basis \(\{ x_1,\dots,x_n \}\) of \(V\).
  Let \(S_1 = \{ 1,\dots,n \}\), and order \(S_1 \times S_1\) lexicographically.
  Let \(S \subseteq S_1 \times S_1\) be the set of indices \((i,j)\) such that the residue class of \(x_i \otimes x_j\) in \((V \otimes V)/R\) is not in the span of the classes of \(x_r \otimes x_s\) with \((r,s) < (i,j)\).
  Then the relations in \(A\) can be written in the form
  \begin{equation}
    \label{eq:quadratic-rels-in-pbw-definition}
    x_k x_l \ = \sum_{S \ni (i,j) < (k,l)} c^{ij}_{kl} x_i x_j
  \end{equation}
  for \((k,l) \notin S\).
  Denote \(S_0 = \emptyset\), and for \(k \geq 2\) denote
  \[
  S_k \eqdef \{ (j_1, \dots, j_k) \mid (j_i,j_{i+1}) \in S, \ 1 \leq i \leq k-1 \}.
  \]
  For \(\bj = (j_1,\dots,j_k) \in S_k\), set \(x_\bj \eqdef x_{j_1} \dots x_{j_k} \in A_k\).
  Finally, the elements \(x_1,\dots,x_n\) are said to be \emph{PBW generators} of \(A\) if \(\{ x_\bj \mid \bj \in S_k \}\) is a basis for \(A_k\) for each \(k \geq 0\).
  In that case, the set \(\bigcup_{k \geq 0} \{ x_\bj \mid \bj \in S_k \}\) is called a \emph{PBW basis} for \(A\), and \(A\) is said to be a \emph{PBW algebra}.
\end{dfn}

\begin{rem}
  \label{rem:on-def-of-pbw-algebra}
  The definition of a PBW algebra is, evidently, somewhat technical to state formally.
  However, the idea of the definition is relatively straightforward.
  The relations \eqref{eq:quadratic-rels-in-pbw-definition} form a reduction system in Bergman's sense \cite{Ber78}*{\S 1}, and the monomials \(x_\bj\) for \(\bj \in S_k\) are precisely the monomials that are irreducible under the reduction system.
  The monomials in \(S_k\) always span \(A_k\), so the real question is whether or not they are independent.
\end{rem}

Bergman's Diamond Lemma \cite{Ber78}*{Theorem 1.2} implies that, for a quadratic algebra, independence of the irreducible monomials only has to be checked for the cubic ones:
\begin{lem}
  \label{lem:diamond-lemma-for-quadratic-algebras}
  With notation as in \cref{dfn:pbw-algebra}, if the set \(\{ x_\bj \mid \bj \in S_3 \}\) of irreducible cubic monomials in \(A\) is linearly independent, then \(\{ x_\bj \mid \bj \in S_k \}\) is linearly independent in \(A_k\) for all \(k \geq 3\).
\end{lem}

This immediately implies:
\begin{prop}
  \label{prop:flatness-in-degree-three-implies-flatness}
  Let \(V \in \oq\), and replace \(V\) by a universal model.
  Let \(B(V) = \{ x_1, \dots, x_n \}\) be the distinguished basis for \(V\).
  Suppose that \(q > 0\) is in the generic set for \(V\).
  \begin{enumerate}[(a)]
  \item If \(\dim \symq^3(V) = \dim \sym^3(V)\), then \(\symq(V)\) is flat, and moreover \(\{ x_1, \dots, x_n \}\) is a set of PBW generators for \(\symq(V)\).
  \item If \(\dim \extq^3(V) = \dim \ext^3(V)\), then \(\extq(V)\) is flat, and moreover \(\{ x_1, \dots, x_n \}\) is a set of PBW generators for \(\extq(V)\).
  \end{enumerate}
\end{prop}

\begin{proof}
  We prove only (a), as the proof for (b) is similar.
  The fact that \(q\) is in the generic set for \(V\) means that the set \(S\) from \cref{dfn:pbw-algebra} is exactly the set of pairs \((i,j)\) with \(i \leq j\).
  If \(\dim \symq^3(V) = \dim \sym^3(V)\), then \cref{lem:diamond-lemma-for-quadratic-algebras} implies that the ordered monomials of degree \(m\) are independent in \(\symq^m(V)\) for each \(m \geq 3\).
  Since we already know that these monomials span, they are a basis, and the conclusion follows.
\end{proof}

Now we show that the various quantum algebras associated to a module and its dual are flat simultaneously:

\begin{prop}
  \label{prop:equivalence-of-flatness-for-various-algebras}
  Let \(V \in \oq\), and suppose that \(q > 0\) is in the generic set for \(V\).
  Then the following are equivalent:
  \begin{enumerate}[(a)]
  \item \(\symq(V)\) is a flat deformation of \(\sym(V)\).
  \item \(\symq(V^\ast)\) is a flat deformation of \(\sym(V^\ast)\).
  \item \(\extq(V)\) is a flat deformation of \(\extq(V)\).
  \item \(\extq(V^\ast)\) is a flat deformation of \(\extq(V^\ast)\).
  \end{enumerate}
\end{prop}

\begin{proof}
  First we prove (a) \(\iff\) (b).
  According to \cite{BerZwi08}*{Equation (2.3)}, we have \(\symq^n(V)^\ast \cong \symq^n V^\ast\) (recall \cref{notn:subobject-vs-quotient}).
  From \cref{prop:symmetrization-and-antisymmetrization} (applied to \(V^\ast\)) we have \(\symq^nV^\ast \cong \symq^n(V^\ast)\), so we conclude that \(\symq^n(V)^\ast \cong \symq^n(V^\ast)\) as \(\uqg\)-modules.
  Hence \(\symq(V)\) and \(\symq(V^\ast)\) have the same Hilbert series, so they are flat simultaneously.

  Replacing \(\symq(V)\) by \(\extq(V)\), the analogous argument proves (c) \(\iff\) (d).
  To complete the proof, we prove (a) \(\iff\) (c).

  By \cref{prop:flatness-in-degree-three-implies-flatness}, we need only look at the degree three components of the algebras, and we use the equality
  \[
  V^{\otimes 3} = \symq^3V \oplus (V \otimes \extq^2V + \extq^2V \otimes V)
  \]
  from the proof of \cref{prop:symmetrization-and-antisymmetrization}.
  From this we see that the dimension of \(\symq^3V\) is classical if and only if the dimension of the complement \((V \otimes \extq^2V + \extq^2V \otimes V)\) is classical.
  In view of \cref{prop:continuity-of-symmetric-and-exterior-squares}, the dimension of \(\extq^2V\) is classical.
  Since \(\dim(X + Y) = \dim X + \dim Y - \dim(X \cap Y)\), we see that the dimension of the sum \((V \otimes \extq^2V + \extq^2V \otimes V)\) is classical if and only if the dimension of the intersection \((V \otimes \extq^2V) \cap (\extq^2V \otimes V)\) is classical.
  But this latter space is exactly \(\extq^3V\).
  We conclude that \(\dim \symq^3V\) attains its classical value if and only if \(\dim \extq^3V\) does.
\end{proof}

\begin{dfn}
  \label{dfn:flatness-for-modules}
  \index[term]{generically flat \(\uqg\)-module}
  Following \cite{BerZwi08}*{Definition 2.27}, and in light of \cref{prop:equivalence-of-flatness-for-various-algebras}, we say that a module \(V \in \oq[q_0]\) is \emph{generically flat} if (after replacing \(V\) by a universal model) any (and hence all) of the algebras \(\symq(V)\), \(\symq(V^\ast)\), \(\extq(V)\), \(\extq(V^\ast)\) are flat deformations of their classical counterparts when \(q\) is transcendental.
\end{dfn}

Note that if \(V\) is generically flat, then \cref{prop:upper-semicontinuity-for-quantum-symmetric-algs}(b) implies that \(\symq(V)\) and \(\extq(V)\) are flat for all \(q\) in a dense open subset of \((0,\infty)\).

Now we briefly define the notion of Koszul algebra.
As this requires some homological algebra that we do not need for our purposes, we merely state the definition and refer the reader to \cite{PolPos05} for an explanation of the terms involved.

\begin{dfn}
  \label{dfn:koszul-algebra}
  \index[term]{Koszul algebra}
  Let \(A = \bigoplus_{k=0}^\infty A_k\) be a \(\zp\)-graded algebra with \(A_0 = \bbC\) and \(\dim A_k < \infty\) for all \(k\).
  Then \(A\) is said to be \emph{Koszul} if the following equivalent conditions hold:
  \begin{enumerate}[(a)]
  \item Both \(A\) and the \emph{Yoneda algebra} \(\mathrm{Ext}_A(\bbC,\bbC)\) are generated in degree one.
  \item \(A\) is a quadratic algebra, and the natural map \(A^! \to \mathrm{Ext}_A(\bbC,\bbC)\) is an isomorphism, where \(A^!\) is the quadratic dual algebra to \(A\) as in \cref{dfn:quadratic-dual-algebra}.
  \item \(\mathrm{Ext}_A^{ij}(\bbC,\bbC) = 0\) for \(i \neq j\).
  \item The Koszul complex \(K(A,\bbC)\) is acyclic.
  \item The Koszul complex \(K(A \otimes A^{\mathrm{op}},A)\) is acyclic.
  \end{enumerate}
\end{dfn}

Now we list the results about Koszul algebras that we require for our investigation of quantum symmetric and exterior algebras:

\begin{prop}
  \label{prop:results-about-koszul-algebras}
  Let \(A\) be a quadratic algebra.
  \begin{enumerate}[(a)]
  \item \cite{PolPos05}*{Ch.~4, Theorem 3.1} If \(A\) is a PBW algebra, then \(A\) is Koszul.
  \item \cite{PolPos05}*{Ch.~2, Corollary 3.3} \(A\) is Koszul if and only if the quadratic dual algebra \(A^!\) is Koszul.
  \item \cite{PolPos05}*{Ch.~2, Corollary 2.2} Suppose that \(A\) and \(A^!\) are Koszul.  Then we have
    \begin{equation}
      \label{eq:hilbert-series-of-koszul-duals}
      h_A(z) h_{A^!}(-z) = 1,
    \end{equation}
    where \(h_A(z)\) and \(h_{A^!}(z)\) are the Hilbert series of \(A\) and \(A^!\), respectively.
  \end{enumerate}
\end{prop}

\begin{prop}
  \label{cor:flatness-implies-koszulity}
  If \(V \in \oq\) is generically flat, and \(q\) lies in the generic set for \(V\), then both \(\symq(V)\) and \(\extq(V)\) are Koszul algebras.
\end{prop}

\begin{proof}
  From \cref{prop:flatness-in-degree-three-implies-flatness} we see that the algebras being flat implies that they have PBW bases.
  Then the Koszul property follows from \cref{prop:results-about-koszul-algebras}(a).
\end{proof}

\begin{quest}
  \label{quest:does-koszul-imply-flat}
  We have seen that \(\symq(V)\) is Koszul when it is flat and \(q\) is in the generic set.
  Therefore we ask whether the converse holds: if \(\symq(V)\) is a Koszul algebra, is it necessarily a flat deformation of \(\sym(V)\)?
\end{quest}

\subsection{Flatness and cominuscule parabolics}
\label{sec:flatness-and-cominuscule-parabolics}

Building on the preliminary work in \cite{BerZwi08}, Zwicknagl has completely classified the flat \emph{simple} \(\uqg\)-modules \cite{Zwi09}*{Theorem 3.14, Theorem 4.23}.
Using a classical \(r\)-matrix for \(\fg\), one can define a bracket on the classical symmetric algebra \(S(V_\lambda)\) that is antisymmetric and satisfies the Leibniz rule.
It turns out that the Jacobi identity holds (and hence \(S(V_\lambda)\) is a Poisson algebra with respect to the bracket) if and only if the quantum symmetric algebra \(\symq(V_\lambda)\) is flat.

The important fact for our purposes, which will be used in \cref{chap:quantum-clifford-algebras}, is that the representations corresponding to the (abelian) nilradicals of cominuscule parabolics are flat.
In order to make this statement more precise, we need the following notation:
\begin{notn}
  \label{notn:quantized-enveloping-alg-of-levi-factor}
  Let \(\fg\) be a simple Lie algebra and let \(\fp \subseteq \fg\) be a parabolic subalgebra of cominuscule type.
  Let \(\fp = \fl \oplus \up\) be the decomposition of \(\fp\) into its Levi factor and nilradical as in \cref{sec:standard-parabolics}, and denote by \(s\) the index of the simple root \(\alpha_s\) that is not included in \(\rtsys(\fl)\), as in \cref{prop:cominuscule-parabolic-conditions}(f).
  For \(q > 0\), we denote
  \begin{equation}
    \label{eq:uql-definition}
    \uql \eqdef \langle E_j, F_j, K_\lambda \mid j \neq s, \ \lambda \in \cP(\fg) \rangle \subseteq \uqg.
  \end{equation}
  Note that \(\uql\) is a Hopf \(\ast\)-subalgebra of \(\uqg\) and that \(K_{\omega_s}\) is a central element therein.
  \index[notn]{Uql@\(\uql\)}

  Since \(\up\) is an irreducible representation of \(\fl\) (see \cref{prop:cominuscule-parabolic-conditions}), there is a corresponding irreducible representation of \(\uql\), which by abuse of notation we denote also by \(\up\), and similarly for \(\um\).
  \index[notn]{u@\(\upm,\fu\)} 
  As \(\fl\) is just reductive, not semisimple, we need to specify the action of the central element \(K_{\omega_s}\) of \(\uql\): \(K_{\omega_s}\) acts as \(q_s^{\pm 1} = q^{\pm d_s}\) in \(\upm\).
  We fix a nondegenerate \(\uql\)-invariant pairing \(\langle \cdot,\cdot \rangle : \um \otimes \up \to \bbC\), which induces an isomorphism of \(\uql\)-modules \(\um \cong \up^\ast\).
\end{notn}

\begin{rem}
  \label{rem:on-uniqueness-of-pairing}
  Note that such the invariant pairing from \cref{notn:quantized-enveloping-alg-of-levi-factor} is unique up to a scalar factor.
  Indeed, an invariant pairing is a \(\uql\)-module map \(\um \otimes \up \to \bbC\).
  According to \cref{prop:multiplicity-one-decomps}, \(\um \otimes \up\) decomposes with multiplicity one, so there is at most a one-dimensional space of such module maps.
  (We know that one exists because the classical analogues of these two representations are dual to each other, so the quantized versions are as well.)
\end{rem}

We have:
\begin{prop}
  \label{prop:cominuscule-nilradicals-are-flat}
  Let \(\fg\) be a simple Lie algebra and let \(\fp \subseteq \fg\) be a parabolic subalgebra of cominuscule type.
  Let \(\fp = \fl \oplus \up\) be the decomposition of \(\fp\) into its Levi factor and nilradical as in \cref{sec:standard-parabolics}, and denote also by \(\up\) the corresponding irreducible representation of \(\uql\).
  Then \(\upm\) are generically flat in the sense of \cref{dfn:flatness-for-modules}.
\end{prop}

\begin{proof}
  This is part of Zwicknagl's classification in \cite{Zwi09}*{Theorem 3.14, Theorem 4.23}.
  Flatness of the quantum symmetric algebras for these modules was also proved by Heckenberger and Kolb in \cite{HecKol04}*{Corollary 2}, and flatness for the quantum exterior algebras was proved in \cite{HecKol06}*{Proposition 3.6}.
\end{proof}

In \cite{Zwi09}, Corollary 4.26 gives a list of flat quantum symmetric algebras and describes many of them in terms of previously studied algebras.

\subsection{Flat quantum symmetric algebras as twisted quantum Schubert cells}
\label{sec:quantum-sym-algs-as-quantum-schubert-cells}

In addition to the classification of flat modules, \cite{Zwi09} also gives a remarkable construction of quantum symmetric algebras of abelian nilradicals as twisted quantum Schubert cells.
As we need this embedding in order to establish the form of the relations in the quantum symmetric algebra, we go into some detail on this now.
We keep the notation from \cref{sec:flatness-and-cominuscule-parabolics}, and recall the relevant notions from \cref{sec:lie-background-parabolics} regarding parabolic subalgebras and from \cref{sec:quantum-schubert-cells-and-their-twists} concerning the quantum Schubert cells and their twists.
In particular:

\begin{itemize}
\item \(\fg\) is simple and \(\fp \subseteq \fg\) is a parabolic subalgebra of cominuscule type (see \cref{prop:cominuscule-parabolic-conditions}).
\item \(\fp = \fl \oplus \up\) is the decomposition of \(\fp\) into its Levi factor \(\fl\) and nilradical \(\up\).
\item \(\fg = \um \oplus \fp\) as \(\fl\)-modules, \(\up\) and \(\um\) are mutually dual via the Killing form of \(\fg\), and \(\up \oplus \um = \fl^\perp\).
\item \(s\) denotes the index of the simple root \(\alpha_s\) that is not included in the root system \(\rtsys(\fl)\) of \(\fl\).
\item \(\Wl\) is the parabolic subgroup of the Weyl group \(W(\fg)\) of \(\fg\) generated by the simple reflections \(s_j\) for \(j \neq s\).
\item \(\wzl\) is the longest word of \(\Wl\), and the \emph{parabolic element} \(\wl \in W(\fg)\) is defined by \(\wl = \wzl \wz\), where \(\wz\) is the longest word in \(W(\fg)\).
\item We have \(\wz = \wzl \wl\), and this factorization is compatible with the word-length function: \(\ell(\wz) = \ell(\wzl) + \ell(\wl)\).
\item As in \eqref{eq:reduced-expression-for-longest-word} we fix a reduced expression \(\wz = s_{i_1} \dots s_{i_M} s_{i_{M+1}} \dots s_{i_{M+N}}\)
for \(\wz = \wzl \wl\) such that \(\wzl = s_{i_1} \dots s_{i_M}\) and \(\wl = s_{i_{M+1}} \dots s_{i_{M+N}}\).
\item This reduced expression defines quantum root vectors \(E_\beta, F_\beta \in \uqg\) associated to each \(\beta \in \posrts(\fg)\).  (See \cref{sec:quantum-root-vectors}.)
\item From \cref{lem:radical-roots-recipe} we know that the set of radical roots is given by \(\rtsys(\up) = \{ \xi_1, \dots, \xi_N \}\), where in \eqref{eq:radical-root-recipe} we defined
\[
  \xi_k \eqdef s_{i_1} \dots s_{i_M} s_{i_{M+1}} \dots s_{i_{k-1}}(\alpha_{i_k}) = \wzl s_{i_{M+1}} \dots s_{i_{k-1}}(\alpha_{i_k}).
\]
\item In \cref{eg:twisted-quantum-schubert-cell-for-parabolic} we showed that the twisted quantum Schubert cell \(U'(\wl)\) is generated by the quantum root vectors \(E_{\xi_j}\) associated to the radical roots \(\xi_j \in \rtsys(\up)\).
\end{itemize}

\begin{rem}
  \label{rem:on-schubert-cell-for-cominuscule-case}
  So far we have not said anything that is specific to the cominuscule situation.
  One reason that this case is special is that the nilradical \(\up\) of \(\fp\) is abelian (see \cref{prop:cominuscule-parabolic-conditions}).
  Hence \(\up\) generates a \emph{commutative} subalgebra of the universal enveloping algebra \(U(\fg)\); by the Poincar\'e-Birkhoff-Witt Theorem, this subalgebra is isomorphic to the symmetric algebra \(S(\up)\).
  Moreover, this is an isomorphism of \(U(\fl)\)-module algebras.
  Here the action of \(U(\fl)\) on \(S(\up)\) is induced by the action of \(\fl\) on \(\up\), and the action of \(U(\fl)\) on \(U(\fg)\) is the restriction of the adjoint action of \(U(\fg)\) on itself.
  Zwicknagl has shown \cite{Zwi09}*{Main Theorem 5.6} that the analogous result holds in the quantum setting.
  We recall the precise results in \cref{prop:cominuscule-twisted-schubert-cell-is-quadratic,twisted-schubert-cell-is-quantum-symmetric-alg} below.
\end{rem}

First we examine the relations in the twisted quantum Schubert cell \(U'(\wl)\) associated to the cominuscule parabolic subalgebra \(\fp\):
\begin{prop}
  \label{prop:cominuscule-twisted-schubert-cell-is-quadratic}
  The twisted quantum Schubert cell \(U'(\wl)\) is a quadratic algebra with the generators \(\{ E_{\xi_j} \}\) in degree one and relations of the form
    \begin{equation}
      \label{eq:commutation-rels-for-cominuscule-quantum-root-vectors}
      E_{\xi_l} E_{\xi_k} - q^{-(\xi_k,\xi_l)}E_{\xi_k} E_{\xi_l} = \sum_{k < i \leq j < l} c^{ij}_{kl}E_{\xi_i} E_{\xi_j}
    \end{equation}
    for \(k < l\), where \(c^{ij}_{kl} \in \bbQ[q^{\pm 1}]\).
\end{prop}

\begin{proof}
  From \cref{prop:commutation-rels-for-quantum-root-vectors} we know that the \(q\)-commutator \(E_{\xi_l} E_{\xi_k} - q^{-(\xi_k,\xi_l)}E_{\xi_k} E_{\xi_l}\) can be expressed as a \(\bbQ[q^{\pm 1}]\)-linear combination of products \(E_{\xi_{j_1}} \dots E_{\xi_{j_m}}\), where \(k < j_1 \leq \dots \leq j_m < l\) and \(\sum_{i=1}^m \xi_{j_i} = \xi_k + \xi_l\).
  Since we are in the cominuscule situation, by \cref{lem:roots-of-uplus-and-l}(a) we know that each \(\xi_{j_i}\) contains the distinguished root \(\alpha_s\) with coefficient 1, as do \(\xi_k\) and \(\xi_l\).
  Thus we must have \(m=2\), and so the relation is quadratic.
  The PBW theorem for \(\uqg\) implies that there are no more relations in \(U'(\wl)\).
\end{proof}

Recall from \cite{KliSch97}*{\S~1.3.4} that if \(H\) is any Hopf algebra, the \emph{(left) adjoint action} of \(H\) on itself is given by
\index[term]{adjoint action of Hopf algebra on itself}
\begin{equation}
  \label{eq:adjoint-action-definition}
  x \rhd a \eqdef x_{(1)}aS(x_{(2)}),
\end{equation}
where, as usual, we use Sweedler notation with implied summation, i.e.~the coproduct of \(x \in H\) is given by \(\Delta(x) = x_{(1)} \otimes x_{(2)}\).
The adjoint action is compatible with multiplication in the sense that
\begin{equation}
  \label{eq:adjoint-action-compatibility}
  x \rhd (ab) = (x_{(1)}\rhd a)(x_{(2)} \rhd b)
\end{equation}
for any \(x,a,b \in H\).

\begin{prop}[\cite{Zwi09}*{Main Theorem 5.6}]
  \label{twisted-schubert-cell-is-quantum-symmetric-alg}
  With all notation as above, the following hold:
  \begin{enumerate}[(a)]
  \item The subspace \(\spn \{ E_{\xi_1}, \dots, E_{\xi_N} \} \subseteq \uqg\) is invariant under the adjoint action of \(\uql\) on \(\uqg\).
  \item Let \(v_0 \in \up\) be a highest weight vector for the action of \(\uql\).  
    Then \(v_0 \mapsto E_\theta \in \uqg\) induces an isomorphism of \(\uql\)-modules \(\up \stackrel{\sim}{\longrightarrow} \spn \{ E_{\xi_1}, \dots, E_{\xi_N} \}\), where \(\theta\) is the highest root of \(\fg\).
  \item The isomorphism from (b) extends to an isomorphism \(\symq(\up) \stackrel{\sim}{\longrightarrow} U'(\wl)\) of \(\zp\)-graded \(\uql\)-module algebras.
  \end{enumerate}
\end{prop}

We compute a simple example to illustrate these results:
\begin{eg}
  \label{eg:embedding-of-sym-alg-as-twisted-schubert-cell}
  We take \(\fg = \fsl_3\) with simple roots \(\alpha_1\) and \(\alpha_2\).
  The highest root is \(\theta = \alpha_1 + \alpha_2\).
  We take \(s = 1\), so that the simple root \(\alpha_1\) is excluded from \(\rtsys(\fl)\).
  Thus we have
  \[
  \rtsys(\fl) = \{ \pm \alpha_2 \}, \quad \rtsys(\up) = \{ \alpha_1, \alpha_1 + \alpha_2 \}.
  \]
  In this case \(\fl \cong \fgl_2\) and \(\up\) is the natural two-dimensional irreducible representation.
  
  The Weyl group \(W(\fsl_3)\) is generated by \(s_1,s_2\), and the longest word is \(\wz = s_1s_2s_1 = s_2s_1s_2\).
  The parabolic subgroup associated to this situation is \(\Wl = \{ e, s_2 \}\), so \(\wzl = s_2\) and \(\wl = \wzl \wz = s_1s_2\).
  Thus we have \(\xi_1 = s_2(\alpha_1) = \alpha_1 + \alpha_2\) and \(\xi_2 = s_2 s_1 (\alpha_2) = \alpha_1\).
  The corresponding quantum root vectors are 
  \[
  E_{\xi_1} = T_2(E_1) = q^{-1}E_1 E_2 - E_2 E_1 \quad \text{and} \quad E_{\xi_2} = T_2T_1(E_2) = E_1,
  \]
  and by definition these generate the twisted quantum Schubert cell \(U'(s_1s_2)\).
  In \cref{eg:quantum-symmetric-algebra-for-2-dim-rep} we saw that the quantum symmetric algebra of the two-dimensional irreducible representation \(V\) of \(\uqsl\) is given by
  \[
  \symq(V) = \bbC \langle x_1,x_2 \rangle / \langle x_2 x_1 = q^{-1}x_1x_2 \rangle,
  \]
  where \(\{ x_1, x_2 \}\) is a weight basis for \(V\) with \(x_1\) the highest weight vector.
  Now we check that the same relation holds between \(E_{\xi_1}\) and \(E_{\xi_2}\) in \(\uqsl[3]\).
  Since \(E_{\xi_1}\) is the quantum root vector associated to the highest weight, we want to show that \(E_{\xi_2}E_{\xi_1} - q^{-1}E_{\xi_1}E_{\xi_2}=0\).
  Indeed, we have:
  \begin{align*}
    E_{\xi_2}E_{\xi_1} - q^{-1}E_{\xi_1}E_{\xi_2} & = (q^{-1} E_1^2E_2 - E_1E_2E_1) - q^{-1}(q^{-1} E_1 E_2E_1 - E_2E_1^2)\\
    & = q^{-1} \left( E_1^2E_2 - (q+q^{-1})E_1E_2E_2 + E_2E_1^2 \right),
  \end{align*}
  and this vanishes by the quantum Serre relation \eqref{eq:quantum-serre-relations} in \(\uqsl[3]\).

  Finally, we show that \(U'(\wl)\) is invariant under the adjoint action of \(\uql\), which in this case is generated by \(E_2, F_2\), and the Cartan elements \(K_\lambda\).
  In fact, invariance under the Cartan elements is straightforward, as the quantum root vectors lie in the weight spaces defining the \(\cQ\)-grading on \(\uqsl[3]\) (see \cref{sec:grading-by-root-lattice}).
  First we check that \(E_{\xi_1}\) is a highest weight vector, i.e.\ that it is annihilated by the adjoint action of \(E_2\).
  Using the definition \eqref{eq:adjoint-action-definition} of the adjoint action together with the structure maps for \(\uqsl[3]\) defined in \cref{sec:quea-hopf-structure}, we get:
  \begin{align*}
    E_2 \rhd E_{\xi_1} & = E_2 E_{\xi_1} - K_2 E_{\xi_1} K_2^{-1}E_2\\
    & = E_2 E_{\xi_1} - q E_{\xi_1} E_2\\
    & = E_2(q^{-1} E_1 E_2 - E_2 E_1) - q (q^{-1} E_1 E_2 - E_2 E_1) E_2\\
    & = - E_2^2 E_1 + (q+q^{-1})E_2E_1E_2 - E_1 E_2^2,
  \end{align*}
  which also vanishes according to the quantum Serre relation.
  Next we check that the adjoint action of \(F_2\) takes \(E_{\xi_1}\) to a multiple of \(E_{\xi_2}\):
  \begin{align*}
    F_2 \rhd E_{\xi_1} & = F_2 E_{\xi_1} K_2 - E_{\xi_1} F_2 K_2\\
    & = F_2 (q^{-1} E_1 E_2 - E_2 E_1) K_2 - (q^{-1} E_1 E_2 - E_2 E_1) F_2 K_2\\
    & = q^{-1}E_1 (F_2 E_2 - E_2 F_2)K_2 + (E_2 F_2 - F_2 E_2)E_1 K_2\\
    & = \frac{1}{q-q^{-1}} \left( q^{-1}E_1(K_2^{-1} - K_2)K_2 + (K_2 - K_2^{-1})E_1 K_2\right)\\
    & = \frac{1}{q-q^{-1}} \left( q^{-1} E_1 - q^{-1}E_1 K_2^2 + K_2 E_1 K_2 - K_2^{-1}E_2 K_2\right)\\
    & = \frac{1}{q-q^{-1}} (q-q^{-1})E_1\\
    & = - E_{\xi_2}.
  \end{align*}
  Then the commutation relations between \(E_2\) and \(F_2\) complete the proof that \(\spn \{ E_{\xi_1}, E_{\xi_2} \}\) is invariant under the adjoint action.
  Since the span of the generators is invariant, the whole algebra \(U'(s_1s_2)\) is invariant because of the compatibility \eqref{eq:adjoint-action-compatibility}.
\end{eg}

\subsection{Filtrations on flat quantum symmetric and exterior algebras}
\label{sec:filtrations-on-flat-quantum-algebras}

In this section we define certain filtrations on the quantum symmetric and exterior algebras associated to abelian nilradicals.
The associated graded algebras will have particularly simple relations, and we will use this information to show that the quantum exterior algebras are Frobenius in \cref{sec:quantum-exterior-algebras-are-frobenius}.
First we recall from \cite{PolPos05}*{Ch.~4,~\S~7} the notion of a filtration by an arbitrary graded ordered semigroup:

\index[term]{graded ordered semigroup}
\begin{dfn}
  \label{dfn:graded-ordered-semigroup}
  A \emph{graded ordered semigroup} is a semigroup \(\Gamma\) with unit \(e\), together with a homomorphism \(g : \Gamma \to \zp\) and a collection of total orders on the fibers \(\Gamma_n = g^{-1}(n)\) such that
  \[
  \alpha < \beta \implies \alpha \gamma < \beta \gamma, \quad \gamma \alpha < \gamma \beta \quad \text{for } \alpha,\beta \in \Gamma_n, \, \gamma \in \Gamma_k.
  \]
  We assume also that \(\Gamma_0 = g^{-1}(0) = \{ e \}\), and moreover that each \(\Gamma_n\) is finite.
\end{dfn}

Now we define the notion of a filtration of an algebra by a graded ordered semigroup.
Note that the algebra is assumed to be \(\zp\)-graded to start with.

\index[term]{graded ordered semigroup!filtration by}
\begin{dfn}
  \label{dfn:filtration-by-graded-ordered-semigroup}
  Let \(A = \bigoplus_{n=0}^\infty A_n\) be a \(\zp\)-graded algebra, and let \(\Gamma\) be a graded ordered semigroup as above.
  A \emph{\(\Gamma\)-valued filtration on \(A\)} is a collection of subspaces \(\{ F_\alpha = F_\alpha A_n \mid n \in \zp, \, \alpha \in \Gamma_n \}\) of \(A\) satisfying the following conditions:
  \begin{enumerate}[(i)]
  \item \(\alpha \leq \beta\) in \(\Gamma_n\) implies that \(F_\alpha \subseteq F_\beta\).
  \item \(F_{\gamma_n} = A_n\) for the maximal element \(\gamma_n \in \Gamma_n\).
  \item \(F_\alpha \cdot F_\beta \subseteq F_{\alpha \beta}\) for all \(\alpha,\beta \in \Gamma\).
  \end{enumerate}
  The filtration is said to be \emph{one-generated} if for all \(n\) and all \(\alpha \in \Gamma_n\) we have
  \begin{equation}
    \label{eq:one-generated-filtration}
    F_\alpha A_n = \sum_{i_1 \dots i_n \leq \alpha} F_{i_1}A_1 \dots F_{i_n}A_1,
  \end{equation}
  where \(i_1, \dots, i_s \in \Gamma_1\).
  Such a filtration is determined by its restriction to \(A_1\).

  The \emph{associated \(\Gamma\)-graded algebra} \(\grf A\) is defined by
  \[
  \grf A \eqdef \bigoplus_{\alpha \in \Gamma} F_{\alpha} / F_{\alpha'},
  \]
  where \(\alpha'\) immediately precedes \(\alpha\) in the total order; we set \(F_{\alpha'} = 0\) in case \(\alpha\) is the minimal element in \(\Gamma_n\).
  As \(\grf A\) is \(\Gamma\)-graded, it is also \(\zp\)-graded via the homomorphism \(g\); it is generated in degree one if and only if the filtration is one-generated.
\end{dfn}

We now introduce a distinguished set of generators for the quantum symmetric and exterior algebras \(\symq(\up)\) and \(\extq(\um)\) associated to a cominuscule parabolic subalgebra:

\begin{notn}
  \label{notn:generators-of-quantum-symmetric-algebra}
  With notation as in \cref{sec:quantum-sym-algs-as-quantum-schubert-cells},   we fix a weight basis \(\{ x_1, \dots, x_N \}\) for the \(\uql\)-module \(\up\) such that \(x_j\) maps to the quantum root vector \(E_{\xi_j}\) under the isomorphism \(\symq(\up) \cong U'(\wl)\) from \cref{twisted-schubert-cell-is-quantum-symmetric-alg}(c).
  \cref{prop:cominuscule-twisted-schubert-cell-is-quadratic} together with flatness implies that \(\{ x_1, \dots, x_N \}\) is a set of PBW generators for \(\symq(\up)\).

  We denote the dual basis for \(\um \cong \up^\ast\) (recall \cref{notn:quantized-enveloping-alg-of-levi-factor}) by \(\{ y_1, \dots, y_N \}\), so that \(\langle y_i, x_j \rangle = \delta_{ij}\).
  Then \cite{PolPos05}*{Ch.~4, Theorem 4.1} implies that \(\{ y_1, \dots, y_N \}\) is a set of PBW generators for \(\extq(\um) \cong \symq(\up)^!\).
\end{notn}

The relations \eqref{eq:commutation-rels-for-cominuscule-quantum-root-vectors} for the quantum root vectors immediately give the following relations among the PBW generators of \(\symq(\up)\):
\begin{equation}
  \label{eq:commutation-rels-for-generators-of-symq-uplus}
  x_l x_k - q^{-(\xi_k,\xi_l)}x_k x_l = \sum_{k < i \leq j < l} c^{ij}_{kl}x_i x_j, \quad k < l.
\end{equation}

Now we introduce a graded ordered semigroup \(\Gamma\) and a \(\Gamma\)-valued filtration on the quantum symmetric algebra \(\symq(\up)\):

\begin{dfn}
  \label{dfn:filtration-of-quantum-symmetric-algebra}
  Let \(\Gamma = \zp^N\), and denote by \(\delta_j\) the element with a \(1\) in the \(j^{\text{th}}\) position and zeros elsewhere.
  Define a semigroup homomorphism \(g : \Gamma \rightarrow \zp\) by
  \[
  g(k_1,\dots, k_N) \eqdef \sum_j k_j.
  \]
  For \(l \in \zp\) denote \(\Gamma_l = g^{-1}(l)\).
  We give each \(\Gamma_l\) the lexicographic order, i.e.\ we say that \((k_1,\dots,k_N) < (k_1',\dots,k_N')\) if there is an index \(j\) such that \(k_i = k_i'\) for \(i < j\) and \(k_j < k_j'\).
  (Note that \(\delta_1 > \dots > \delta_N\) in this ordering.)
  For \(\bk \in \Gamma_l\), we define the subspace \(F_\bk = F_\bk \symq(\up) \subseteq \symq^l(\up)\) by
  \begin{equation}
    \label{eq:filtration-of-quantum-symmetric-algebra}
    F_\bk = F_\bk \symq(\up) \eqdef \spn \{ x_{\bk'} \mid \bk' \in \Gamma_l \text{ and } \bk' \leq \bk\} \subseteq \symq^l(\up).
  \end{equation}
\end{dfn}

\begin{lem}
  \label{lem:filtration-is-one-generated}
  The set \(\{ F_{\bk} \mid \bk \in \Gamma \}\) is a \emph{\(\Gamma\)-valued filtration} of \(\symq(\up)\), i.e.\ the following hold:
  \begin{enumerate}[(a)]
  \item \(F_{\bj} \subseteq F_{\bk}\) when \(\bj \leq \bk \in \Gamma_l\).
  \item \(F_{(l,0,\dots,0)} = S^l_q(\up)\) (note that \((l,0,\dots,0)\) is the maximal element in \(\Gamma_l\)).
  \item \(F_{\bk} F_{\bk'} \subseteq F_{\bk + \bk'}\).
  \end{enumerate}
  Moreover, the filtration is \emph{one-generated}, i.e.\ for every \(\bk \in \Gamma_l\) we have
  \begin{equation}
    \label{eq:one-generation-of-filtration}
    F_{\bk} = \sum_{\delta_{i_1} + \dots + \delta_{i_l} \leq \bk}F_{\delta_{i_1}} \dots F_{\delta_{i_l}}.
  \end{equation}
\end{lem}

\begin{proof}
  Parts (a) and (b) follow immediately from the definitions.
  Part (c) follows from the commutation relations \eqref{eq:commutation-rels-for-generators-of-symq-uplus}.
  Finally, the filtration is one-generated because the \(x_j\) generate \(\symq(\up)\).
\end{proof}

The associated \(\Gamma\)-graded algebra \(\grf \symq(\up)\) is also \(\zp\)-graded via the homomorphism \(g : \Gamma \to \zp\).
Its relations are particularly simple:

\begin{prop}
  \label{prop:associated-graded-of-quantum-symmetric-algebra}
  The associated \(\Gamma\)-graded algebra \(\grf \symq(\up)\) is generated by the elements \(\{ \ox_j \}_{j=1}^N\) and has only the relations
  \begin{equation}
    \label{eq:commutation-rels-in-associated-graded}
    \ox_l \ox_k - q^{-(\xi_k,\xi_l)}\ox_k \ox_l = 0, \quad k < l,    
  \end{equation}
  where \(\ox_k\) is the image of \(x_k\) in \(F_{\delta_k} / F_{\delta_{k+1}}\).
\end{prop}

\begin{proof}
  The elements \(\{ \ox_k \}\) generate the associated graded algebra because the filtration is one-generated.
  In the relations \eqref{eq:commutation-rels-for-generators-of-symq-uplus}, the terms on the right-hand side are all lower in the filtration than those on the left-hand side.
  Thus the terms on the right become zero in the associated graded algebra, which gives \eqref{eq:commutation-rels-in-associated-graded}.
  These are the only relations in \(\grf \symq(\up)\) because the \(x_\bk\) for \(\bk \in \Gamma\) form a PBW basis of \(\symq(\up)\).
\end{proof}

We now define the \emph{dual filtration} (see \cite{PolPos05}*{Ch.~4,~\S~8}) on the quadratic dual algebra \(\extq(\um)\) to \(\symq(\up)\):

\begin{dfn}
  \label{dfn:filtration-of-quantum-exterior-algebra}
  Let \(\gc \eqdef \zp^N\) denote the same semigroup as \(\Gamma\), but with the opposite ordering on each fiber \(\gc_l = g^{-1}(l)\), so that \(\delta_1 < \dots < \delta_N\).
  We give \(\extq(\um) = \symq(\up)^!\) the one-generated \(\gc\)-valued filtration \(\Fc\) determined by
  \begin{equation}
    \label{eq:def-of-filtration-of-quantum-exterior-algebra}
    \Fc_{\delta_k} \eqdef \spn \{ y_j \mid j \leq k\}, \quad 1 \leq k \leq N.
  \end{equation}
  Then for arbitrary \(\bk \in \gc_l\) the subspace \(\Fc_{\bk} = \Fc_{\bk} \extq(\um) \subseteq \extq^l(\um)\) is defined by the analogous formula to \eqref{eq:one-generation-of-filtration}, keeping in mind that we use the ordering of \(\gc\).
\end{dfn}

\begin{prop}
  \label{prop:associated-graded-of-quantum-exterior-algebra}
  The associated \(\gc\)-graded algebra \(\grfc \extq(\um)\) is generated by the elements \(\{ \oy_j \}_{j=1}^N\) and has only the relations
  \begin{equation}
    \label{eq:commutation-rels-in-associated-graded-exterior-algebra}
    \oy_l \wedge \oy_k +  q^{-(\xi_k,\xi_l)} \oy_k \wedge \oy_l = 0, \quad k \leq l,
  \end{equation}
  where \(\oy_k\) is the image of \(y_k\) in \(\Fc_{\delta_k} / \Fc_{\delta_{k-1}}\).
\end{prop}

\begin{proof}
  The elements \(\{ \oy_k \}\) generate the associated graded algebra because the filtration \(\Fc\) is one-generated.
  According to Corollary 7.3 in Chapter 4 of \cite{PolPos05}, since we know that \(\grf \symq(\up)\) is PBW (and hence Koszul) we have
  \[
  \grfc \extq(\um) = \grfc \left( \symq(\up)^! \right) \cong \left( \grf \symq(\up) \right)^!,
  \]
  and it is straightforward to show that the quadratic dual relations to \eqref{eq:commutation-rels-in-associated-graded} are exactly \eqref{eq:commutation-rels-in-associated-graded-exterior-algebra}.
\end{proof}

\begin{rem}
  \label{rem:heckenberger-kolb-already-proved-this}
  \cref{prop:associated-graded-of-quantum-exterior-algebra} was proved by different methods in \cite{HecKol06}, Proposition 3.7.
\end{rem}

\begin{dfn}
  \label{dfn:basis-of-quantum-exterior-algebra}
  For any subset \(J \subseteq \{ 1, \dots, N \}\) with \(\# J = l\), we define elements \(y_J \in \extq^l(\um)\) and \(x_J \in \extq(\up)\) by
  \begin{equation}
    \label{eq:basis-of-quantum-exterior-algebra}
    y_{J} \eqdef y_{j_1} \wedge \dots \wedge y_{j_l}, \quad x_{J} \eqdef x_{j_1} \wedge \dots \wedge x_{j_l}
  \end{equation}
  where \(J = \{ j_1, \dots, j_l \}\) and \(j_1 < \dots < j_l\).
  We denote \([N] = \{ 1,\dots,N \}\).
  \index[notn]{N@\([N]\)}
\end{dfn}

\begin{cor}
  \label{cor:basis-for-quantum-exterior-power}
  The elements \(y_J\) with \(\# J = l\) form a basis for \(\extq^l(\um)\).
\end{cor}

\begin{proof}
  It follows from \eqref{eq:commutation-rels-in-associated-graded-exterior-algebra} that the elements \(y_J\) with \(\# J = l\) span \(\extq^l(\um)\).
  \cref{prop:equivalence-of-flatness-for-various-algebras} implies that the dimension of \(\extq^l(\um)\) is \(\binom{N}{l}\), so these \(y_J\) are linearly independent, and hence form a basis.
\end{proof}

\subsection{Flat quantum exterior algebras are Frobenius algebras}
\label{sec:quantum-exterior-algebras-are-frobenius}

In \cref{chap:quantum-clifford-algebras} we will construct a quantum Clifford algebra via creation and annihilation operators on \(\extq(\um)\).
In order to show that the Clifford algebra acts irreducibly via these operators, we use the fact that \(\extq(\um)\) and \(\extq(\up)\) are Frobenius algebras.
The theory of Frobenius algebras is fairly well-developed.
We introduce only the definitions and basic aspects of this theory that we require for our purposes, and refer the reader to \cite{Koc04} for more details on Frobenius algebras.

\begin{prop}
  \label{prop:equiv-frobenius-alg-conditions}
  Let \(A\) be a finite-dimensional unital associative algebra over \(\bbC\).
  The following are equivalent:
  \begin{enumerate}[(a)]
  \item There exists an isomorphism \(\Psi : A \stackrel{\sim}{\longrightarrow} A^\ast\) of left \(A\)-modules, where \(A^\ast\) as a left \(A\)-module is the dual of the right regular representation of \(A\).
  \item There exists an element \(\psi \in A^\ast\) such that \(\ker \psi\) contains no nonzero left ideal of \(A\).
  \item There exists a nondegenerate bilinear form \((\cdot,\cdot) : A \times A \to \bbC\) such that \((ab,c) = (a,bc)\) for all \(a,b,c\in A\).
  \end{enumerate}
\end{prop}

\begin{proof}
  (a) \(\Longrightarrow\) (b)
  For \(\Psi\) to be an isomorphism of left \(A\)-modules means that \(\Psi\) is a linear isomorphism such that
  \begin{equation}
    \label{eq:frobenius-isomorphism}
    \left[ \Psi(xa) \right](b) = \left[ \Psi(a) \right](bx)
  \end{equation}
  for all \(a,b,x \in A\).
  Then define \(\psi \in A^\ast\) by \(\psi = \Psi(1_A)\).
  Taking \(a=1\) in \eqref{eq:frobenius-isomorphism} gives \(\psi(bx) = [\Psi(x)](b)\) for \(b,x \in A\).
  If the left ideal \(Ax\) is contained in \(\ker \psi\), then for all \(b \in A\) we have
  \[
  0 = \psi(bx) = [\Psi(x)](b),
  \]
  and hence \(\Psi(x)=0\).
  Since we assumed that \(\Psi\) is an isomorphism, we conclude that \(x=0\).
  \vspace{0.05in}  

  \noindent (b) \(\Longrightarrow\) (c)
  Given \(\psi \in A^\ast\), we define the bilinear form \((\cdot,\cdot) : A \times A \to \bbC\) by \((a,b) = \psi(ab)\).
  Then the equality \((ab,c)=(a,bc)\) follows from the fact that \(A\) is associative.
  To see that the form is nondegenerate, suppose that \((\cdot,b)=0\) in \(A^\ast\).
  Then for all \(a \in A\) we have
  \[
  0 = (a,b) = \psi(ab),
  \]
  and hence the left ideal \(Ab\) is contained in \(\ker \psi\), so \(b=0\) according to our assumption.
  (Nondegeneracy on the left side follows from nondegeneracy on the right by a simple dimension argument.)
  \vspace{0.05in}

  \noindent (c) \(\Longrightarrow\) (a)
  Given the bilinear form \((\cdot,\cdot) : A \times A \to \bbC\), define the map \(\Psi : A \to A^\ast\) by \(\Psi(a) = (\cdot,a)\).
  The map \(\Psi\) is a linear isomorphism precisely because the form is nondegenerate.
  Furthermore, we have
  \[
  [\Psi(xa)](b) = (b,xa) = (bx,a) = [\Psi(a)](bx),
  \]
  which is exactly the condition \eqref{eq:frobenius-isomorphism} asserting that \(\Psi\) is a morphism of left \(A\)-modules.
\end{proof}

\begin{dfn}
  \label{dfn:frobenius-algebra}
  \index[term]{Frobenius algebra}
  Let \(A\) be a finite-dimensional unital associative algebra.
  We say that \(A\) is a \emph{Frobenius algebra} if the equivalent conditions of \cref{prop:equiv-frobenius-alg-conditions} hold.
  In that case, we refer to the functional \(\psi\) as the \emph{Frobenius functional} and to the bilinear form \((\cdot,\cdot)\) as the \emph{Frobenius form}.
\end{dfn}

\begin{rem}
  \label{rem:hilbert-algebras-are-like-frobenius-algebras}
  In the context of \(\ast\)-algebras, there is a somewhat related notion known as a \emph{Hilbert algebra}.
  In that case the algebra is equipped with a \(\ast\)-structure and an inner product satisfying \(\langle a,b \rangle = \langle b^\ast, a^\ast \rangle\) and the compatibility criterion \(\langle xy,z \rangle = \langle y, x^\ast z \rangle\).
  See \cite{Dix81} for details.
\end{rem}

\begin{rem}
  \label{rem:on-frobenius-isomorphism}
  We emphasize that a Frobenius functional on a Frobenius algebra \(A\) is the image of the unit of \(A\) under the isomorphism \(\Psi\).
\end{rem}

\begin{eg}
  \label{eg:exterior-algebra-is-frobenius}
  Let \(V\) be a finite-dimensional vector space.
  Then the exterior algebra \(\ext(V)\) is a Frobenius algebra with Frobenius functional given by projection onto the top-degree component.
  More precisely, let \(\{ x_1, \dots, x_n \}\) be a basis for \(V\).
  Then define \(\psi : \ext(V) \to \bbC\) by setting \(\psi(x_1 \wedge \dots \wedge x_n) = 1\) and declaring that \(\psi\) vanishes on all lower-degree components of \(\ext(V)\).
  The kernel of \(\psi\) contains no nonzero left ideal: any left ideal necessarily contains \(\ext^n(V)\), and \(\psi\) does not vanish on \(\ext^n(V)\).
\end{eg}

The fact that ordinary exterior algebras are Frobenius algebras relies essentially on their graded structure.
As indicated above, the key fact is that any nonzero left ideal contains the top-degree component, and this top-degree component is one-dimensional.
In other words, for any \(x \in \ext(V)\), there is an element \(y\) such that \(y \wedge x = z\), where \(z\) is a fixed generator of \(\ext^{\mathrm{top}}(V)\) (see \cref{prop:characterization-of-zp-graded-frob-algs} below).

Now consider the quantum exterior algebra \(\extq(\um)\), where \(\um\) is associated to a cominuscule parabolic as in \cref{sec:quantum-sym-algs-as-quantum-schubert-cells}, so that \(\extq(\um)\) is flat.
As the graded components of \(\extq(\um)\) have the same dimensions as those of \(\ext(\um)\), we may hope that \(\extq(\um)\) is also a Frobenius algebra.
While this is true, it is more difficult to prove in the quantum case because we do not have explicit relations (although in any given example it is not difficult to prove).
However, using the filtration \(\Fc\) on \(\extq(\um)\) defined in \cref{sec:filtrations-on-flat-quantum-algebras}, it is not difficult to prove that the associated graded algebra \(\grfc \extq(\um)\) is Frobenius.

Bongale proved that if \(A\) is a finite-dimensional \(\zp\)-filtered algebra whose associated graded algebra is Frobenius, then \(A\) is Frobenius as well.
We generalize this result to filtrations by arbitrary graded ordered semigroups, which allows us to prove that \(\extq(\um)\) is Frobenius.
We begin by recalling Bongale's result:

\begin{prop}[\cite{Bon67}*{Theorem 2}]
  \label{prop:grA-frobenius-implies-A-frobenius-zp}
  If \(A = \bigcup_{k \geq 0} F_k A\) is a \(\zp\)-filtered algebra with \(F_0A = \bbC\) such that the associated graded algebra \(\grf A\) is a Frobenius algebra, then \(A\) itself is a Frobenius algebra.
\end{prop}

This result relies on the following characterization of \(\zp\)-graded Frobenius algebras, which is a slightly modified version of \cite{Bon67}*{Theorem 1}:

\begin{prop}
  \label{prop:characterization-of-zp-graded-frob-algs}
  Suppose that \(A = \bigoplus_{k=0}^n A_k\) is a \(\zp\)-graded algebra with \(A_0 = \bbC\) and \(A_n \neq (0)\).
  Define \(\cD \eqdef \bigoplus_{k=0}^{n-1} A_k\).
  Then:
  \begin{enumerate}[(a)]
  \item \(A\) is a Frobenius algebra if and only if both \(\dim A_n = 1\) and \(\cD\) contains no nonzero left ideal of \(A\); equivalently, \(\dim A_n = 1\), and if \(z \in A_n\) is any generator, then for all nonzero \(x \in A_k\), there exists \(y \in A_{n-k}\) with \(yx = z\).
  \item In that case, every Frobenius functional \(\psi\) on \(A\) satisfies \(\psi|_\cD = 0\), and every such functional is a Frobenius functional.
  \item Given a Frobenius functional \(\psi\), the corresponding isomorphism \(\Psi : A \to A^\ast\) carries \(A_k\) isomorphically onto \((A_{n-k})^\ast\) for each \(k\).
  \end{enumerate}
\end{prop}

We have the following straightforward generalization of \cref{prop:grA-frobenius-implies-A-frobenius-zp}:

\begin{prop}
  \label{prop:grA-frobenius-implies-A-frobenius-arbitrary-Gamma}
  Let \(A = \bigoplus_{k=0}^n\) be a finite-dimensional graded algebra with \(A_0 = \bbC\), and assume that \(A\) is equipped with a \(\Gamma\)-valued filtration \(\{ F_\alpha A_k \mid k \in \zp, \alpha \in \Gamma_k \}\) for some graded ordered semigroup \(\Gamma\).
  If the \(\Gamma\)-graded algebra \(\grf A\) is Frobenius, then \(A\) is Frobenius.
\end{prop}

\begin{proof}
  Note first that if \(B = \bigoplus_{\gamma \in \Gamma} B_\gamma\) is any \(\Gamma\)-graded algebra, then \(B\) is also \(\zp\)-graded via the homomorphism \(g : \Gamma \to \zp\): for \(k \in \zp\) define \(B_k = \bigoplus_{\gamma \in \Gamma_k} B_\gamma\).
  Thus \(\grf A\) is a \(\zp\)-graded Frobenius algebra, so by \cref{prop:grA-frobenius-implies-A-frobenius-zp} we see that
  \[
  \left( \grf A \right)_n = \bigoplus_{\gamma \in \Gamma_n} F_\gamma A_n / F_{\gamma'}A_n
  \]
  (recall the definition of \(\grf A\) from \cref{dfn:filtration-by-graded-ordered-semigroup}) is one-dimensional, and hence \(A_n\) itself is one-dimensional.
  Moreover, \cref{prop:grA-frobenius-implies-A-frobenius-zp} implies that \(\bigoplus_{k=0}^{n-1} (\grf A)_k\) contains no nontrivial left ideals.
  
  Denote \(\cD = \bigoplus_{k=0}^{n-1} A_k\).
  As \(\cD\) has codimension one in \(A\), by \cref{prop:grA-frobenius-implies-A-frobenius-zp} it suffices to show that \(\cD\) contains no nontrivial left ideals in order to conclude that \(A\) is Frobenius.
  To that end, suppose that \(L \subseteq \cD\) is a left ideal.
  Then define
  \[
  P(L) \eqdef \bigoplus_{k=0}^{n-1} p_k(L),
  \]
  where \(p_k : A \to A_k\) is the natural projection; \(P(L)\) is a \emph{graded} left ideal of \(A\) contained in \(\cD\).
  
  Now give \(P(L)\) the induced \(\Gamma\)-valued filtration coming from the filtration of \(A\).
  Then \(\grf P(L)\) is a left ideal of \(\grf A\) contained in \(\bigoplus_{k=0}^{n-1} (\grf A)_k\), so we conclude that \(\grf P(L) = (0)\).
  Hence \(P(L)=(0)\), which implies that \(L=(0)\).
  This completes the proof.
\end{proof}

Now we can use \cref{prop:grA-frobenius-implies-A-frobenius-arbitrary-Gamma} together with the filtration from \cref{dfn:filtration-of-quantum-exterior-algebra} to show that the quantum exterior algebra \(\extq(\um)\) is a Frobenius algebra:

\begin{prop}
  \label{prop:flat-quantum-exterior-algebra-is-frobenius}
  With all notation as in \cref{sec:quantum-sym-algs-as-quantum-schubert-cells}, we have:
  \begin{enumerate}[(a)]
  \item \(\extq(\um)\) is a Frobenius algebra.
  \item The map \(\psi : \extq(\um) \to \bbC\) given by \(\psi \equiv 0\) on \(\bigoplus_{k=0}^{N-1} \extq^k(\um)\) and \(\psi(y_{[N]}) = 1\) is a Frobenius functional for \(\extq(\um)\).
  \end{enumerate}
\end{prop}

\begin{proof}
  As indicated above, we prove part (a) by showing that \(\grfc \extq(\um)\) is a Frobenius algebra first.
  Let \(\oy_1, \dots, \oy_N\) denote the generators of \(\grfc \extq(\um)\) as in \cref{prop:associated-graded-of-quantum-exterior-algebra}, and let \(\oy_J = \oy_{j_1} \wedge \dots \oy_{j_k}\) if \(J = \{ j_1, \dots, j_k \} \subseteq [N]\) with \(j_1 < \dots < j_k\).

  Let \(a \in \grfc \extq(\um)\).
  We claim that there is an element \(b \in \grfc \extq(\um)\) such that \(b \wedge a = \oy_{[N]}\).
  It suffices to prove this for a \emph{homogeneous} element \(a \in \grfc \extq^k(\um)\) (for a general element \(a\), use the \(b\) such that \(b \wedge a_d = \oy_{[N]}\), where \(a_d\) is the nonzero homogeneous component of \(a\) of lowest degree).
  So, suppose
  \[
  a = \sum_J c_J \oy_J,
  \]
  where the sum runs over subsets \(J \subseteq [N]\) with \(\# J = k\).
  Choose some \(J_0\) with \(c_{J_0} \neq 0\), and let \(J' = [N] \setminus J_0\).
  Then the relations \eqref{eq:commutation-rels-in-associated-graded-exterior-algebra} imply that
  \[
  \oy_{J'} \wedge a = \pm q^{c} c_{J_0} \cdot \oy_{[N]}
  \]
  for some \(c \in \bbQ^\times\).
  
  Thus every left ideal in \(\grfc \extq(\um)\) contains the top-degree component, so \cref{prop:characterization-of-zp-graded-frob-algs} implies that \(\grfc \extq(\um)\) is a Frobenius algebra.
  Then \cref{prop:grA-frobenius-implies-A-frobenius-arbitrary-Gamma} implies that \(\extq(\um)\) itself is Frobenius.
  From the proof of that result we see that the kernel of the corresponding Frobenius functional \(\psi\) on \(\extq(\um)\) must be \(\bigoplus_{k=0}^{N-1} \extq^k(\um)\); setting \(\psi(y_{[N]})=1\) then determines \(\psi\) completely.
\end{proof}

\begin{rem}
  \label{rem:q-exterior-alg-is-not-frobenius-algebra-in-category}
  While \(\extq(\um)\) is both a Frobenius algebra and a \(\uql\)-module algebra, it is not a Frobenius algebra internal to the category of \(\uql\)-modules.
  The reason is that the Frobenius functional \(\psi\) is not a morphism to the trivial module.
  Indeed, the central element \(K_{\omega_s}\) acts as the scalar \(q^{-N d_s}\) on \(\extq^N(\um)\), but as the identity on \(\bbC\).
\end{rem}

\begin{rem}
  \label{rem:lambda-q-uplus-is-also-frobenius}
  The analogous argument shows that \(\extq(\up)\) is a Frobenius algebra as well, with Frobenius functional given by projection onto the top-degree component.
  This fact will be used later on in \cref{thm:quantum-gamma-factorization-isomorphism}, where we prove that the spinor representation of the quantum Clifford algebra is irreducible.
\end{rem}

\section{Collapsing in degree three}
\label{sec:collapsing-in-degree-three}

Zwicknagl's classification of flat (in the sense of \cref{dfn:flatness-for-modules}) simple highest-weight \(\uqg\)-modules in \cite{Zwi09}, which we discussed in \cref{sec:flatness-and-cominuscule-parabolics}, shows that most \(\uqg\)-modules are not flat.
Thus the following question naturally arises: given \(V \in \oq\), by how much does \(V\) fail to be flat?

We make a first step toward answering this question here.
In \cref{sec:collapsing-in-degree-three-motivation} we recall a conjecture of Berenstein and Zwicknagl that relates the amounts of ``collapsing'' in the quantum symmetric and exterior algebras.
In \cref{sec:grothendieck-ring} we introduce the Grothendieck ring \(\kg\) as a convenient setting in which to address this conjecture.
In \cref{sec:three-fruit-cactus-group} we prove a number of technical results.
Then in  \cref{sec:numerical-koszul-duality} we prove our main result, \cref{thm:collapsing-in-degree-three}, and show how it implies Berenstein and Zwicknagl's conjecture.
Finally in \cref{sec:conj-on-q-symm-and-ext-cubes} we discuss a further conjecture of Zwicknagl and its relation to \cref{thm:collapsing-in-degree-three}. 

\subsection{Motivation}
\label{sec:collapsing-in-degree-three-motivation}

Berenstein and Zwicknagl make the following conjecture regarding the amount of collapsing in the quantum symmetric and exterior algebras:
\begin{conj}[\cite{BerZwi08}*{Conjecture 2.26}]
  \label{conj:numerical-koszul-duality}
  For each \(\lambda \in \pplus\), we have
  \begin{equation}
    h_{\symq(V(\lambda))}(z) \cdot h_{\extq(V(\lambda)^\ast)}(-z) = 1 + O(z^4),\label{eq:numerical-koszul-duality}
  \end{equation}
  where \(  h_{\symq(V(\lambda))}(z)\) and \(h_{\extq(V(\lambda)^\ast)}(z)\) are the Hilbert series of \(\symq(V(\lambda))\) and \(\extq(V(\lambda)^\ast)\), respectively.
  Equivalently, 
  \begin{equation}
    \label{eq:numerical-koszul-duality-alternate}
    \dim \symq^3V(\lambda) - \dim \extq^3 V(\lambda)^\ast = \dim V(\lambda)^2.
  \end{equation}
\end{conj}

\begin{rem}
  \label{rem:numerical-koszul-duality}
  Berenstein and Zwicknagl describe \eqref{eq:numerical-koszul-duality} as a form of ``numerical Koszul duality'' between the quantum symmetric and exterior algebras.
  Indeed, in the case when \(V(\lambda)\) is flat, the two Hilbert series in question coincide with the Hilbert series of the classical symmetric and exterior algebras.
  The algebras \(\sym(V(\lambda))\) and \(\ext(V(\lambda))\) are Koszul, and hence \eqref{eq:numerical-koszul-duality} holds without the \(O(z^4)\) term in view of \cref{prop:results-about-koszul-algebras}(c).
\end{rem}

We prove a stronger version of this conjecture in \cref{thm:collapsing-in-degree-three} below, accounting for all finite-dimensional modules rather than just the simple ones.
This result means that, in the non-flat cases, the quantum symmetric and exterior algebras exhibit the same amount of collapsing in their degree three components.

\subsection{The Grothendieck ring}
\label{sec:grothendieck-ring}

\index[term]{Grothendieck ring}

A convenient setting in which to address \cref{conj:numerical-koszul-duality} is that of the Grothendieck ring of the category \(\oq\).
We now recall:
\begin{dfn}
  \label{dfn:grothendieck-ring}
  Let \(q > 0\).
  The \emph{Grothendieck semiring} of \(\oq\) is the set \(\kqpg\) of isomorphism classes of modules in \(\oq\), endowed with the commutative binary  operations
  \[
  [V] + [W] \eqdef [V \oplus W], \quad [V] \cdot [W] \eqdef [V \otimes W]
  \]
  for \(V,W \in \oq\).
  Then \emph{Grothendieck ring} of \(\oq\) is the Grothendieck group \(\kqg\) of the abelian semigroup \((\kqpg,+)\) together with the induced multiplication.
  (Note that this multiplication is commutative because the braidings (or the commutors) give isomorphisms of \(\uqg\)-modules \(V \otimes W \cong W \otimes V\) for all \(V,W \in \oq\).)
\end{dfn}

\begin{rem}
  \label{rem:on-grothendieck-ring}
  Multiplication in the Grothendieck ring \(\kqg\) encodes the decomposition of tensor products into simple submodules.
  Indeed, suppose that
  \[
  V \otimes W \cong \bigoplus_{j} m_j V(\lambda_j)
  \]
  in \(\oq\), where \(m_j V(\lambda_j) \eqdef V(\lambda_j)^{\oplus m_j}\).
  Then we have the equality
  \[
  [V] \cdot [W] = \sum_j m_j [V(\lambda_j)]
  \]
  in \(\kqg\).
\end{rem}

We have the following description of the Grothendieck rings:
\begin{lem}
  \label{lem:common-grothendieck-rings}
  Let \(q,q' > 0\).
  \begin{enumerate}[(a)]
  \item As an additive abelian semigroup, the Grothendieck semiring \(\kqpg\) is freely generated by the classes of the simple highest weight modules \(V(\lambda)\) for \(\lambda \in \pplus\).
  \item The function \([V(\lambda)] \mapsto [V(\lambda)]\) extends to an isomorphism of semirings \(\kqpg \to \kqpg[q']\), and hence to an isomorphism of rings \(\kqg \to \kqg[q']\).
  \end{enumerate}
\end{lem}

\begin{proof}
  By definition (see \cref{sec:representations-of-quea}), every object \(V \in \oq\) has a decomposition as a direct sum of modules \(V(\lambda)\).
  This shows that the classes \([V(\lambda)]\) for \(\lambda \in \pplus\) generate \(\kqpg\).
  These classes are independent because the decomposition of any \(V \in \oq\) into isotypical components is unique.
  This proves (a).

  The fact that \([V(\lambda)] \mapsto [V(\lambda)]\) extends to an isomorphism of abelian groups is the content of part (a).
  Part (b) follows because tensor products of simple \(\uqg\)-modules decompose into simple submodules in the same way that they do classically.
\end{proof}

With \cref{lem:common-grothendieck-rings} in mind, we introduce:

\index[notn]{Kg@\(\kg,\kpg\)}
\begin{notn}
  \label{notn:grothendieck-group-of-reps}
  We denote by \(\kpg\) the common Grothendieck semiring of the categories \(\oq\) for \(q > 0\), and the common Grothendieck ring by \(\kg\).
  We denote multiplication in \(\kg\) by \(\cdot\) or simply by juxtaposition when appropriate.
  For a module \(V \in \oq\), we denote its image in \(\kpg\) also by \(V\).
\end{notn}
  
Finally, we require the following:
\begin{dfns}
  \label{dfn:structures-on-grothendieck-ring}
  Let \(\kpg\) and \(\kg\) be as in \cref{notn:grothendieck-group-of-reps}.
  \begin{enumerate}[(1)]
  \item Viewing \(\kpg\) as the positive cone induces a partial order on \(\kg\).
    More precisely, for \(X,Y \in \kg\) we say that \(X \leq Y\) if \(Y - X \in \kpg\).
  \item For an element \(X \in \kg\) and a dominant integral weight \(\lambda \in \pplus\), the \emph{multiplicity of \(V(\lambda)\) in \(X\)} is the coefficient of \([V(\lambda)]\) in the unique expression for \(X\) as a \(\bbZ\)-linear combination of the generators for \(\kg\) discussed in \cref{lem:common-grothendieck-rings}(a).
    We denote this multiplicity by \(m_\lambda(X)\).
  \item For each \(q > 0\), the endofunctors \(V \mapsto \symq^n V\) and \(V \mapsto \extq^n V\) of the category \(\oq\) descend to functions \(\symq^n, \extq^n : \kpg \to \kpg\).
    Note that \(\symq[1]^n\) and \(\extq[1]^n\) are induced by the usual symmetric and exterior powers of \(\fg\)-modules, and we denote these functions just by \(\sym^n\) and \(\ext^n\), respectively.
  \end{enumerate}
\end{dfns}

\begin{rem}
  \label{rem:on-the-structures-on-grothendieck-ring}
  Note that the multiplicity of \(V(\lambda)\) in an element of \(\kg\) can be negative, and that this notion generalizes the notion of the multiplicity of \(V(\lambda)\) in a representation.
\end{rem}

\begin{lem}
  \label{lem:partial-order-on-grothendieck-ring}
  The partial order \(\leq\) on \(\kg\) introduced in \cref{dfn:structures-on-grothendieck-ring}(a) is compatible with the ring structure on \(\kg\).
  More precisely, for \(X,Y,Z \in \kg\), if \(X \leq Y\) then \(X + Z \leq Y + Z\), and moreover if \(Z \in \kpg\) then \(X \cdot Z \leq Y \cdot Z\).
\end{lem}

\begin{proof}
  This follows immediately from the fact that \(\kpg\) is closed under the addition and multiplication in \(\kg\).
\end{proof}

\subsection{On the three-fruit cactus group}
\label{sec:three-fruit-cactus-group}

In order to prove \cref{thm:collapsing-in-degree-three}, we will require some technical results concerning the three-fruit cactus group \(J_3\) and its action on the tensor cube of a representation (see \cref{sec:cactus-group} for the definitions).
First we introduce some notation:
\begin{notn}
  \label{notn:cactus-commutors-for-collapsing-in-degree-three}
  Let \(V \in \oq\) for some \(q > 0\), and replace \(V\) by a universal model as in \cref{prop:universal-model-for-non-simple-rep}.
  With the parametrized commutors as in \cref{dfn:numerical-coboundary-operators}, for each \(q>0\) we define operators \(a(q),b(q)\), and \(\psi(q)\) on \(V^{\otimes 3}\) by
  \begin{equation}
    \label{eq:cactus-commutors-for-collapsing}
    a(q) \eqdef \sigma_{VV}(q) \otimes \id, \quad b(q) \eqdef \id \otimes \sigma_{VV}(q), \quad \psi(q) \eqdef \sigma_{V \otimes V \, V}(q).
  \end{equation}
\end{notn}

\begin{rem}
  \label{rem:on-cactus-commutors-for-collapsing}
  Recall that we introduced notation for the parametrized cactus group operators in \cref{notn:parametrized-cactus-group-action}.
  In that notation, we have
  \begin{equation}
    a(q) = s_{1,2}(q), \quad b(q) = s_{2,3}(q), \quad \psi(q) = \sigma_{1,2,3}(q) = s_{1,3}(q) s_{1,2}(q).\label{eq:cactus-commutors-for-collapsing-alternate}
  \end{equation}
\end{rem}

With these definitions, we have:
\begin{lem}
  \label{lem:rels-for-cactus-commutors-for-collapsing}
  For all \(q > 0\), the following relations hold in \(GL(V^{\otimes 3})\):
  \begin{enumerate}[(a)]
  \item \(a(q)^2 = b(q)^2 = 1\).
  \item \(\psi(q) a(q) = b(q) \psi(q)\).
  \item \(a(q) \psi(q) a(q) = \psi(q)^{-1}\).
  \end{enumerate}
\end{lem}

\begin{proof}
  Part (a) follows from the involutivity of the generators \(s_{p,t}\) of the cactus group.
  By definition, we have
  \[
  b(q) \psi(q) a(q) = s_{2,3}(q) s_{1,3}(q) s_{1,2}(q) s_{1,2}(q) = s_{2,3}(q) s_{1,3}(q) = s_{1,3}(q) s_{1,2}(q) = \psi(q),
  \]
  where we used the relation \(s_{2,3}s_{1,3} = s_{1,3}s_{1,2}\) in \(J_3\) for the third equality.
  This proves (b).
  Finally, we have
  \[
  (\psi(q) a(q))^2 = s_{1,3}(q)^2 = \id,
  \]
  which proves (c).
\end{proof}

We require another technical lemma describing the quantum symmetric and exterior cubes of the module \(V\)  in terms of fixed points of the operators \(a(q)\) and \(\psi(q)\):

\begin{lem}
  \label{lem:sym-and-ext-cubes-as-fixed-points-of-commutors}
  Let \(q > 0\) be transcendental or 1.
  With notation as above, the quantum symmetric and exterior cubes of the module \(V\) can be described as follows:
  \begin{equation}
    \label{eq:sym-and-ext-cubes-as-fixed-points}
    \begin{gathered}
      \symq^3V=\left\{v\in V^{\otimes 3} \mid a(q)v=v,\ \psi(q) v=v\right\},\\
      \extq^3V = \left \{v\in V^{\otimes 3} \mid a(q)v=-v,\ \psi(q) v=v \right\}.
    \end{gathered}
  \end{equation}
\end{lem}

\begin{proof}
  We prove the result just for the quantum symmetric cube, as the argument for the quantum exterior cube is analogous.
  In \cref{prop:cactus-gp-fixes-symmetric-tensors} we showed that the space of quantum symmetric tensors is fixed by the entire cactus group, so \(\symq^3V\) is contained in the right-hand side of \eqref{eq:sym-and-ext-cubes-as-fixed-points}.
  On the other hand, suppose that \(a(q)v = v\) and \(\psi(q)v=v\) for some \(v \in V^{\otimes 3}\).
  Then by \cref{lem:rels-for-cactus-commutors-for-collapsing}(b), we have
  \[
  b(q) v = \psi(q) a(q) \psi(q)^{-1} v = v.
  \]
  By definition we have \(\symq^3V = \ker(a(q) - \id) \cap \ker(b(q)-\id)\), so the displayed equation implies that \(v \in \symq^3V\).
  This completes the proof of the lemma.
\end{proof}

Now we collect some further results concerning the operators \(\psi(q)\) on \(V^{\otimes 3}\):
\begin{lem}
  \label{lem:facts-on-two-over-one-commutor}
  With notation as above, we have:
  \begin{enumerate}[(a)]
  \item \(\psi(q)\) is diagonalizable for all \(q > 0\).
  \item If \(q>0\) is transcendental, then \(\psi(q)\) does not have \(-1\) as an eigenvalue.
  \end{enumerate}
\end{lem}

\begin{proof}
  According to \cref{prop:coboundary-struct-on-uqg-modules}(b), \(\psi(q) = \sigma_{V \otimes V \, V}(q)\) is a unitary operator, and hence it is diagonalizable.
  For part (b), an argument similar to the proof of \cref{prop:cactus-gp-fixes-symmetric-tensors} shows that if \(q\) is transcendental and \(\psi(q)\) has \(-1\) as an eigenvalue, then \(\psi(1)\) also has \(-1\) as an eigenvalue.
  But according to \cref{dfn:numerical-coboundary-operators}, \(\psi(1)\) is a cyclic permutation of \(V^{\otimes 3}\) of order 3, and hence cannot have \(-1\) as an eigenvalue.
\end{proof}

The last technical results we need for \cref{thm:collapsing-in-degree-three} involve restricting the operators \(a(q),b(q)\), and \(\psi(q)\) to spaces of highest weight vectors in \(V^{\otimes 3} \).
First we prove invariance:
\begin{lem}
  \label{lem:psi-restricted-to-highest-weight-space}
  \index[notn]{Wlambdaq@\(W^\lambda_q\)}
  Let \(\lambda \in \pplus\).
  For each \(q > 0\), the space of highest weight vectors of weight \(\lambda\) in \(V^{\otimes 3}\) for the action of \(\uqg\) is invariant under the operators \(a(q),b(q)\), and \(\psi(q)\).
\end{lem}

\begin{proof}
  The operators \(a(q),b(q)\), and \(\psi(q)\) are all \(\uqg\)-module maps, so they preserve highest weight spaces.
\end{proof}

We thus introduce:
\begin{notn}
  \label{notn:restriction-of-commutors-to-highest-weight-spaces}
  For \(\lambda \in \pplus\) and for each \(q > 0\), we denote by \(W^\lambda_q\) the space of highest weight vectors of weight \(\lambda\) in \(V^{\otimes 3}\).
  By \cref{lem:psi-restricted-to-highest-weight-space} the space \(W^\lambda_q\) is invariant under \(a(q),b(q)\), and \(\psi(q)\), and we denote their restrictions to \(W^\lambda_q\) by \(\aql,\bql\), and \(\psiql\), respectively.
\end{notn}

Now we use the relation between \(a(q)\) and \(\psi(q)\) to obtain some information about the spectral decomposition of \(\psiql\):
\begin{lem}
  \label{lem:eigenvalues-of-psilambda-q}
  Suppose that \(c \in \bbC^\times\) is an eigenvalue of \(\psiql\).
  Then \(c^{-1}\) is also an eigenvalue of \(\psiql\), and the operator \(\aql\) implements an isomorphism
  \begin{equation}
    \aql : \ker(\psiql - c) \overset{\sim}{\longrightarrow} \ker(\psiql - c^{-1}).\label{eq:eigenspaces-of-psilambda-q}
  \end{equation}
\end{lem}

\begin{proof}
  This follows immediately from \cref{lem:rels-for-cactus-commutors-for-collapsing}(c).
\end{proof}

\begin{notn}
  \label{notn:eigenspaces-of-psilambda-q}
  According to \cref{lem:eigenvalues-of-psilambda-q}, the eigenvalues \(c\) of \(\psiql\) such that \(c \neq c^{-1}\) come in pairs.
  We denote by
  \begin{equation}
    \label{eq:paired-eigenspaces-of-psilambda-q}
    M^\lambda_q \eqdef \bigoplus_{c \neq c^{-1}} \ker(\psiql - c)
  \end{equation}
  the sum of the paired eigenspaces.
\end{notn}

Recall that the commutor \(\sigma_{VV}(q)\) is a self-adjoint unitary operator, and hence \(\frac{1+\sigma_{VV}(q)}{2}\) is the orthogonal projection onto \(\ker(1 - \sigma_{VV}(q)) \eqdef \symq^2 V\).
This implies that \(\frac{1+a(q)}{2}\) is the orthogonal projection onto \(\symq^2V \otimes V\), and similarly \(\frac{1-a(q)}{2}\) is the orthogonal projection onto \(\extq^2V \otimes V\).
The final result we need examines the restrictions of these two projections to the space \(M^\lambda_q\).
For convenience, we denote these restrictions by
\begin{equation}
  \label{eq:notn-for-projs-onto-sym-ext-tensors-in-wlambda-q}
  \ppql = \frac{1+\aql}{2}, \quad \pmql = \frac{1-\aql}{2}.
\end{equation}
Then we have:

\begin{lem}
  \label{lem:ranks-of-projections-on-mlambda-q}
  With notation as above, the projections \(\ppql\) and \(\pmql\) preserve the space \(M^\lambda_q \subseteq W^\lambda_q\), and their ranks when restricted to \(M^\lambda_q\) are equal.
\end{lem}

\begin{proof}
  The two projections preserve \(M^\lambda_q\) by \cref{lem:eigenvalues-of-psilambda-q}.
  According to that lemma, we can find a set \(\{ c_j \}\) of eigenvalues of \(\psiql\) such that
  \[
  M^\lambda_q = \bigoplus_j \ker(\psiql - c_j) \oplus \ker(\psiql - c_j^{-1}).
  \]
  It follows from \cref{lem:eigenvalues-of-psilambda-q} that both projections are injective when restricted to the space
  \[
  N^\lambda_q \eqdef \bigoplus_j \ker(\psiql - c_j).
  \]
  Moreover, it is straightforward to show that
  \[
  \ppql(N^\lambda_q) = \ppql(M^\lambda_q), \quad \pmql(N^\lambda_q) = \pmql(M^\lambda_q).
  \]
  Hence the restrictions of both \(\ppql\) and \(\pmql\) to \(M^\lambda_q\) have rank equal to \(\dim N^\lambda_q\).
  This completes the proof of the lemma.
\end{proof}

Finally, combining several of the preceding results, we obtain:
\begin{prop}
  \label{prop:decomp-of-Wlambda-q-as-eigenspaces-of-psilambda-q}
  With all notation as above, for each transcendental \(q > 0\) we have the decomposition
  \begin{equation}
    \label{eq:decomp-of-Wlambda-q-as-eigenspaces-of-psilambda-q}
    W^\lambda_q = \ker (\psiql - \id) \oplus M^\lambda_q.
  \end{equation}
  Both components are invariant under the projections \(\ppql\) and \(\pmql\).
\end{prop}

\begin{proof}
  We know from \cref{lem:facts-on-two-over-one-commutor}(a) that \(\psi(q)\) is diagonalizable.
  As \(\psiql\) is the restriction of \(\psi(q)\) to the invariant subspace \(W^\lambda_q\) (see \cref{lem:psi-restricted-to-highest-weight-space} and \cref{notn:restriction-of-commutors-to-highest-weight-spaces}), \(\psiql\) is also diagonalizable.
  This means that \(W^\lambda_q\) decomposes as the direct sum of eigenspaces of \(\psiql\).
  By \cref{lem:facts-on-two-over-one-commutor}(b), \(-1\) does not occur as an eigenvalue of \(\psiql\).
  Thus the only eigenspace of \(\psiql\) that does not occur in \(M^\lambda_q\) is the one associated to the eigenvalue \(1\).
  This gives the decomposition \eqref{eq:decomp-of-Wlambda-q-as-eigenspaces-of-psilambda-q}.
  We know that \(M^\lambda_q\) is invariant under \(\ppql\) and \(\pmql\) by \cref{lem:ranks-of-projections-on-mlambda-q}, and \cref{lem:rels-for-cactus-commutors-for-collapsing}(c) implies that \(\ker(\psiql - \id)\) is invariant under these operators as well.
  This completes the proof.
\end{proof}

\subsection{Numerical Koszul duality}
\label{sec:numerical-koszul-duality}

In this subsection we prove \cref{thm:collapsing-in-degree-three}, the main result of \cref{sec:collapsing-in-degree-three}.
The following technical result is the cornerstone of the proof:
\begin{prop}
  \label{prop:difference-between-symmetric-and-exterior-cubes}
  Let \(q > 0\) be transcendental or 1 and let \(V \in \oq\).
  Then the equation
  \begin{equation}
    \label{eq:difference-between-symmetric-and-exterior-cubes}
    [\symq^3V] - [\extq^3V] = \left( [\symq^2V] - [\extq^2V] \right) \cdot [V]
  \end{equation}
  holds in \(\kg\).
\end{prop}

\begin{proof}
  We retain the notation from \cref{notn:restriction-of-commutors-to-highest-weight-spaces}, \cref{notn:eigenspaces-of-psilambda-q}, and \eqref{eq:notn-for-projs-onto-sym-ext-tensors-in-wlambda-q}.
  We prove the result by showing that \(V(\lambda)\) has the same multiplicity (in the sense of \cref{dfn:structures-on-grothendieck-ring}(2)) in both sides of \eqref{eq:difference-between-symmetric-and-exterior-cubes}, for each \(\lambda \in \pplus\).

  First we compute the multiplicity of \(V(\lambda)\) in \(\symq^3V\).
  By \cref{lem:sym-and-ext-cubes-as-fixed-points-of-commutors}, we have
  \[
  \symq^3V = \ker(a(q)-\id) \cap \ker(\psi(q)-\id),
  \]
  so the space of highest weight vectors of weight \(\lambda\) in \(\symq^3V\) is given by 
  \[
  \symq^3V \cap W^\lambda_q =  \ker(\aql-\id) \cap \ker(\psiql-\id) = \ker \pmql \cap \ker(\psiql-\id).
  \]
  Similarly, the space of highest weight vectors of weight \(\lambda\) in \(\extq^3V\) is
  \[
  \extq^3V \cap W^\lambda_q = \ker(\aql+\id) \cap \ker(\psiql-\id) = \ker \ppql \cap \ker(\psiql-\id).
  \]
  Combining these observations, we conclude that the multiplicity of \(V(\lambda)\) in the left-hand side of \eqref{eq:difference-between-symmetric-and-exterior-cubes} is 
  \begin{equation}
    \label{eq:mult-of-vlambda-in-sym3-minus-ext3}
    \begin{aligned}
      m_\lambda \left( [\symq^3V] - [\extq^3V] \right) & = \dim \left( \ker \pmql \cap \ker(\psiql-\id) \right) \\
      & \phantom{===} - \dim \left( \ker \ppql \cap \ker(\psiql-\id) \right)\\
      & = \dim \ker \pmql\mid_{\ker(\psiql-\id)}  - \dim \ker \ppql\mid_{\ker(\psiql-\id)}\\
      & = \rk \ppql \mid_{\ker(\psiql-\id)} - \rk \pmql \mid_{\ker(\psiql-\id)}.
    \end{aligned}
  \end{equation}

  Now we examine the right-hand side of \eqref{eq:difference-between-symmetric-and-exterior-cubes}.
  According to the definitions of addition and multiplication in \(\kg\), we have 
  \[
  \left( [\symq^2V] - [\extq^2V] \right) \cdot [V] = [\symq^2 V \otimes V] - [\extq^2 V \otimes V].
  \]
  Since \(\symq^2V \otimes V = \ker(a(q) - \id)\), the space of highest weight vectors of weight \(\lambda\) in \(\symq^2V \otimes V\) is given by
  \[
  \left( \symq^2 V \otimes V \right) \cap W^\lambda_q = \ker(\aql - \id) = \ker \pmql,
  \]
  and similarly the space of highest weight vectors of weight \(\lambda\) in \(\extq^2V \otimes V\) is given by
  \[
  \left( \extq^2 V \otimes V \right) \cap W^\lambda_q = \ker(\aql + \id) = \ker \ppql.
  \]
  Hence the multiplicity of \(V(\lambda)\) in the right-hand side of \eqref{eq:difference-between-symmetric-and-exterior-cubes} is
  \begin{equation}
    \label{eq:mult-of-vlambda-in-(sym2V-ext2V)V}
    \begin{aligned}
    m_\lambda \left( \left( [\symq^2V] - [\extq^2V] \right) \cdot [V] \right) & = \dim \ker \pmql - \dim \ker \ppql\\
    & = \rk \ppql - \rk \pmql.
    \end{aligned}
  \end{equation}

  We want to show that \(m_\lambda \left( [\symq^3V] - [\extq^3V] \right) = m_\lambda \left( \left( [\symq^2V] - [\extq^2V] \right) \cdot [V] \right)\).
  In view of \eqref{eq:mult-of-vlambda-in-sym3-minus-ext3} and \eqref{eq:mult-of-vlambda-in-(sym2V-ext2V)V}, we need to prove that the difference between the ranks of the projections \(\ppql\) and \(\pmql\) does not change when we restrict them to the invariant subspace \(\ker (\psiql - \id)\) of \(W^\lambda_q\).
  From \cref{prop:decomp-of-Wlambda-q-as-eigenspaces-of-psilambda-q} we get
  \[
  \rk \ppql = \rk \ppql\mid_{\ker(\psiql - \id)} + \rk \ppql\mid_{M^\lambda_q},
  \]
  and similarly
  \[
  \rk \pmql = \rk \ppql\mid_{\ker(\psiql - \id)} + \rk \pmql\mid_{M^\lambda_q}.
  \]
  Subtracting the second equation from the first and invoking \cref{lem:ranks-of-projections-on-mlambda-q}, we obtain the desired equality.
  This completes the proof.
\end{proof}

Finally, we are ready to prove our main result:

\begin{thm}
  \label{thm:collapsing-in-degree-three}
  Let \(q > 0\) be transcendental and let \(V \in \oq\).
  Then we have
  \begin{equation}
    \label{eq:collapsing-in-degree-three}
    [\symq^3V] - [\extq^3V] = [\sym^3V] - [\ext^3V]
  \end{equation}
  in the Grothendieck ring \(\kg\).
\end{thm}

\begin{proof}
  Replace \(V\) by a universal model as in \cref{prop:universal-model-for-non-simple-rep}.
  By \cref{prop:difference-between-symmetric-and-exterior-cubes}, we have
  \begin{equation}
    \label{eq:collapsing-in-degree-three-step1}
    [\symq^3V] - [\extq^3V] = \left( [\symq^2V] - [\extq^2V] \right) \cdot [V].
  \end{equation}
  By \cref{prop:continuity-of-symmetric-and-exterior-squares}, we know that \(q \mapsto \symq^2V\) is continuous.
  Then \cref{prop:decomp-of-cts-family-is-constant} implies that \([\symq^2V] = [\sym^2V]\) in \(\kg\).
  Similarly we obtain \([\extq^2V] = [\ext^2 V]\).
  Thus
  \[
  \left( [\symq^2V] - [\extq^2V] \right) \cdot [V] = \left( [\sym^2V] - [\ext^2V] \right) \cdot [V] = [\symq^3V] - [\extq^3V],
  \]
  where we used \cref{prop:difference-between-symmetric-and-exterior-cubes} again for the second equality.
  In combination with \eqref{eq:collapsing-in-degree-three-step1}, this completes the proof.
\end{proof}

This immediately implies:

\begin{cor}
  \label{cor:collapsing-in-degree-three}
  Let \(q > 0\) be transcendental and let \(V \in \oq\).
  Then \cref{conj:numerical-koszul-duality} (numerical Koszul duality) holds for \(V\), i.e.\ we have
  \begin{equation}
    \label{eq:hilb-series-prod-for-collapsing-in-degree-three}
    h_{\symq(V)}(z) \cdot h_{\extq(V^\ast)}(-z) = 1 + O(z^4).
  \end{equation}
\end{cor}

\begin{proof}
  Let \(a_n = \dim \symq^n(V)\) and \(b_n = \dim \extq^n(V^\ast) = \dim \extq^n(V)\) for each \(n \geq 1\).
  Then
  \begin{equation}
    \label{eq:product-of-hilbert-series}
    \begin{aligned}
      h_{\symq(V)}(z) \cdot h_{\extq(V^\ast)}(-z) & = (1 + a_1z + a_2z^2 + a_3z^3 + O(z^4)) \, \cdot\\
      & \phantom{===} (1 - b_1z + b_2z^2 - b_3z^3 + O(z^4))  \\
      & = 1 + (a_1 - b_1)z + (a_2 - a_1b_1 + b_2)z^2 + \\
      & \phantom{===} (a_3 - a_2b_1 + a_1b_2  - b_3)z^3 + O(z^4).
    \end{aligned}
  \end{equation}
  Let \(m = \dim V\).
  Then we have \(a_1 = b_1 = n\).
  Moreover \(a_2 = \binom{m+1}{2}\) and \(b_2 = \binom{m}{2}\) since the quantum symmetric and exterior squares are flat deformations of the classical versions by \cref{prop:continuity-of-symmetric-and-exterior-squares}.
  Then the coefficient of \(z^2\) in the right-hand side of \eqref{eq:product-of-hilbert-series} is
  \[
  a_2 - a_1b_1 + b_2 = \binom{m+1}{2} - m^2 + \binom{m}{2} = 0.
  \]
  The coefficient of \(z^3\) is 
  \[
  a_3 - a_2b_1 + a_1b_2  - b_3 = a_3 - m^2 - b_3.
  \]
  Thus we see that \eqref{eq:hilb-series-prod-for-collapsing-in-degree-three} holds if and only if \(a_3 - b_3 = m^2\), i.e.\ if and only if
  \begin{equation}
    \label{eq:difference-in-degree-three-dimensions}
    \dim \symq^3(V) - \dim \extq^3(V) = (\dim V)^2.
  \end{equation}
  But then we have
  \[
  \dim \symq^3(V) - \dim \extq^3(V) = \dim \sym^3(V) - \dim \ext^3(V) = (\dim V)^2;
  \]
  the first equality is a consequence of \cref{thm:collapsing-in-degree-three}, while the second is elementary.
  This verifies \eqref{eq:difference-in-degree-three-dimensions} and completes the proof.
\end{proof}

\begin{rem}
  \label{rem:on-collapsing-in-degree-three}
  This verifies Berenstein and Zwicknagl's conjecture.
  However, \cref{thm:collapsing-in-degree-three} is a much stronger result.
  Indeed, \cite{BerZwi08}*{Main Theorem 2.21} shows that \([\symq^n (V)] \leq [\sym^n (V)]\) and \([\extq^n (V)] \leq [\ext^n (V)]\) in \(\kg\).
  This means that the quantum symmetric and exterior powers are potentially ``missing'' some irreducible submodules as compared to their classical counterparts.
  \cref{cor:collapsing-in-degree-three} says that the same number of \emph{dimensions} are missing from \(\symq^3(V)\) and \(\extq^3(V)\), while \cref{thm:collapsing-in-degree-three} says that the same \emph{irreducible submodules} are missing.
\end{rem}

\subsection{Evidence for a conjecture on the quantum symmetric and exterior cubes}
\label{sec:conj-on-q-symm-and-ext-cubes}

In this section we discuss a conjecture of Zwicknagl describing further the quantum symmetric and exterior cubes of a module.
Although we cannot prove the conjecture, it is consistent with  \cref{thm:collapsing-in-degree-three}, and so we view \cref{thm:collapsing-in-degree-three} as positive evidence for it.
First we introduce:

\begin{dfn}[\cite{Zwi09}*{Definition 6.18}]
  \label{dfn:sym3-low-and-ext3-low}
  It is not hard to see that \(\kpg\) is a lattice with respect to the partial order \(\leq\) introduced in \cref{dfn:structures-on-grothendieck-ring}(1).
  For \(V \in \oq[1]\) we define two representations of \(\fg\), viewed as elements of \(\kpg\), by
  \[
  [\sym^3_{\low} V] \eqdef [\sym^2V \otimes V] - X, \quad [\ext^3_{\low} V] \eqdef [\ext^2V \otimes V] - X,
  \]
  where \(X = \inf \left\{ [\sym^2V \otimes V], [\ext^2V \otimes V] \right\}\).
\end{dfn}

This is not how the definition is framed in \cite{Zwi09}, but it is equivalent.
The proof of \cite{BerZwi08}*{Lemma 2.30} shows that \([\sym^3_{\low}V] \leq [\symq^3V]\) and \([\ext^3_{\low}V] \leq [\extq^3V]\).
The conjecture of interest, rephrased in our language, states:

\begin{conj}[\cite{Zwi09}*{Conjecture 7.3}]
  \label{conj:on-sym3-low-and-ext3-low}
  When \(V = V(\lambda)\) and when \(q\) is transcendental, we have the equalities
  \begin{equation}
    \label{eq:on-sym3-low-and-ext3-low}
    [\sym^3_{\low}V(\lambda)] = [\symq^3V(\lambda)], \quad [\ext^3_{\low}V(\lambda)] = [\extq^3V(\lambda)]
  \end{equation}
  in \(\kg\).
\end{conj}

Denoting \(V = V(\lambda)\), by construction, we have
\begin{align*}
  [\sym^3_{\low}V] - [\ext^3_{\low}V] & = [\sym^2 V \otimes V] - [\ext^2V \otimes V]\\
  & = \left( [\sym^2 V] - [\ext^2 V] \right) \cdot [V]\\
  & = [\sym^3V] - [\ext^3V]\\
  & = [\symq^3V] - [\extq^3V],
\end{align*}
where we used \cref{prop:difference-between-symmetric-and-exterior-cubes} for the third equality and \cref{thm:collapsing-in-degree-three} for the last one.
This equality would follow immediately from \cref{conj:on-sym3-low-and-ext3-low}; thus, as indicated above, our results are consistent with the conjecture.

\chapter{Quantum Clifford algebras and flag manifolds}
\label{chap:quantum-clifford-algebras}

In this chapter we use the theory developed in \cref{chap:quantum-symmetric-and-exterior-algebras} to introduce quantum Clifford algebras.
In \cref{sec:creation-and-annihilation} we begin by recalling briefly the situation for classical Clifford algebras, and we give a nonstandard proof that the spinor representation is irreducible.
This proof uses the fact that exterior algebras are Frobenius algebras.
Then we introduce a quantum analogue of the Clifford algebra associated to the hyperbolic space \(\up \oplus \um\), where \(\up\) is the nilradical of a cominuscule parabolic subalgebra \(\fp \subseteq \fg\).
In \cref{sec:dolbeault-operator-on-quantized-flag-manifolds} we use this quantum Clifford algebra to construct a Dolbeault-Dirac operator over the quantization of the flag manifold \(G/P\).
Finally in \cref{sec:cliff-alg-for-cp2} we work out the details of this construction for a simple example.

\section{Creation and annihilation operators}
\label{sec:creation-and-annihilation}

\subsection{The classical Clifford algebra}
\label{sec:clifford-alg-classical-case}

We now recall some elements from the classical theory of Clifford algebras.
An excellent reference for the general theory (over arbitrary fields, including in characteristic two) is \cite{Che97}; a more modern and concrete treatment of Clifford algebras over \(\bbR\) and \(\bbC\) can be found in \cite{GraVarFig01}*{Ch.~5}, for instance.
As the results we need are well-established, we state them here without proof.
We begin with some definitions:

\begin{dfns}
  \label{dfns:bilinear-form-stuff}
  Let \(W,X,Y\) be finite-dimensional vector spaces over \(\bbC\).
  \begin{enumerate}[(1)]
  \item If \(g : W \times W \to \bbC\) is a symmetric bilinear form, then we call the pair \((W,g)\) a \emph{quadratic space}\index[term]{quadratic space}.
  \item A bilinear form \(g : X \times Y \to \bbC\) is said to be \emph{nondegenerate}\index[term]{nondegenerate form} if either (and hence both) of the maps 
    \[
     X \ni x \mapsto g(x,\cdot) \in Y^\ast, \quad Y \ni y \mapsto g(\cdot,y) \in X^\ast
     \]
    are isomorphisms.
  \item If \((W,g)\) is a quadratic space, a subspace \(V \subseteq W\) is said to be \emph{isotropic}\index[term]{isotropic subspace} if \(g(V,V)=0\).
  \item The \emph{hyperbolic space}\index[term]{hyperbolic space}\index[notn]{HV@\(H(V)\)} \(H(V)\) is the quadratic space \((V \oplus V^\ast,h_V)\), where \(h_V\)\index[notn]{hV@\(h_V\)} is the canonical (nondegenerate) symmetric bilinear form
    \[
    h_V \left( (v,\phi), (w, \psi) \right) \eqdef \phi(w) + \psi(v)
    \]
    for \(v,w \in V\) and \(\phi,\psi \in V^\ast\).
  \end{enumerate}
\end{dfns}

With these definitions, we have:
\begin{prop}
  \label{prop:symmetric-blfs}
  Let \((W,g)\) be a quadratic space with \(W\) even-dimensional and \(g\) nondegenerate.
  Then there are isotropic subspaces \(X,Y \subseteq W\) such that \(W = X \oplus Y\) and such that \(g|_{X \times Y}\) is nondegenerate.
  Moreover, the map
  \[
  W \cong X \oplus Y \to X \oplus X^\ast, \quad (x,y) \mapsto (x,g(\cdot,y))
  \]
  induces an isomorphism of quadratic spaces \((W,g) \cong H(X)\).
\end{prop}

For our purposes we will quantize only the Clifford algebra of a hyperbolic space.
In view of \cref{prop:symmetric-blfs}, any even-dimensional quadratic space \((W,g)\) with \(g\) nondegenerate is of this form, so this is not much of a restriction.
Now we recall the definition of the Clifford algebra associated to a symmetric bilinear form:
\begin{dfn}
  \label{dfn:clifford-algebra}
  Let \((W,g)\) be a quadratic space.
  The associated \emph{Clifford algebra}\index[term]{Clifford algebra} \(\cl(W,g)\) is the (complex, unital, associative) algebra
  \[
  \cl(W,g) \eqdef T(W)/\langle x \otimes y + y \otimes x - g(x,y) \mid x,y,\in W \rangle,
  \]
  where \(T(W)\) is the tensor algebra of \(W\).
\end{dfn}

Now we define \emph{creation and annihilation operators}, which give rise to a representation of \(\cl(H(V))\) on \(\ext(V)\):
\begin{dfn}
  \label{dfn:creation-and-annihilation-ops}
  For \(x \in V\), define the \emph{creation operator}\index[term]{creation operators!classical}\index[notn]{lambdax@\(\lambda_x\)} \(\lambda_x\) on \(\ext(V)\) by
  \begin{equation}
    \label{eq:classical-creation-op-def}
    \lambda_x (w)  = x \wedge w
  \end{equation}
  for \(w \in \ext(V)\).
  For \(f \in V^\ast\), define the \emph{annihilation operator}\index[term]{annihilation operators!classical}\index[notn]{deltaf@\(\delta_f\)} \(\delta_f\) on \(\ext(V)\) by
  \begin{equation}
    \label{eq:classical-annihilation-op-def}
    \delta_f (v_1 \wedge \dots \wedge v_k) = \sum_{j=1}^k (-1)^{j-1} \langle f,v_j \rangle v_1 \wedge \dots \wedge \hat{v}_j \wedge \dots \wedge v_k
  \end{equation}
  for \(v_1, \dots, v_k \in V\), where \(\langle f,v_j \rangle\) is the dual  pairing, and \(\hat{v}_j\) means that \(v_j\) is omitted.
\end{dfn}

\begin{rem}
  \label{rem:on-the-creation-and-annihilation-operators}
  The creation operators are easy to define: they are just multiplication operators coming from the left regular representation of \(\ext(V)\) on itself.
  On the other hand, the annihilation operators are less transparent.
  One way to motivate their definition is to use the antisymmetrizer to embed \(\ext^k(V)\) as the space \(\ext^k V \subseteq V^{\otimes k}\) of totally antisymmetric \(k\)-tensors.
  Then the map \(\delta_f\) can be interpreted as contraction of an antisymmetric tensor with \(f\) in the first component.
  One sees immediately that the result lies in \(\ext^{k-1} V\), so then composing with the quotient map \(V^{\otimes (k-1)} \to \ext^{k-1}(V)\), we obtain (up to a scalar factor coming from the relation between the antisymmetrizers in degrees \(k\) and \(k-1\)) the map \(\delta_f\) as defined in \eqref{eq:classical-annihilation-op-def}.
  We will see another perspective on the annihilation operators in \cref{sec:frobenius-viewpoint-on-clifford-algebra}.
\end{rem}

\begin{prop}
  \label{prop:rep-of-classical-clifford-algebra}
  With definitions as above, we have:
  \begin{enumerate}[(a)]
  \item \(\lambda_x^2 = 0\) for all \(x \in V\).
  \item \(\delta_f^2 = 0\) for all \(f \in V^\ast\).
  \item \(\lambda_x \delta_f + \delta_f \lambda_x = \langle f,x \rangle\) for all \(x \in V\) and \(f \in V^\ast\), where \(\langle f,x \rangle\) is the dual pairing.
  \item The map \((x,f) \mapsto \lambda_x + \delta_f\) for \(x \in V\) and \(f \in V^\ast\) extends to a homomorphism of algebras \(\cl(H(V)) \to \End(\ext(V))\).
  \end{enumerate}
\end{prop}

\index[term]{spinor representation!classical}
\begin{notn}
  \label{notn:spinor-representation-classical}
  We refer to the representation of \(\cl(H(V))\) on \(\ext(V)\) coming from \cref{prop:rep-of-classical-clifford-algebra}(d) as the \emph{spinor representation} of the Clifford algebra.  
\end{notn}

Leaving aside the Clifford algebra for a moment, the creation and annihilation operators give rise to actions of the exterior algebras of \(V\) and \(V^\ast\) on \(\ext(V)\).
Indeed, by \cref{prop:rep-of-classical-clifford-algebra}(a) the linear map \(x \mapsto \lambda_x\) extends to an algebra homomorphism 
\index[notn]{gamma@\(\gpm\)}
\begin{equation}
  \label{eq:classical-gamma-plus-def}
  \gp : \ext(V) \to \End(\ext(V))
\end{equation}
which is easily seen to be injective.
This embeds \(\ext(V)\) into \(\End(\ext(V))\) as the subalgebra generated by the creation operators.

Likewise, \cref{prop:rep-of-classical-clifford-algebra}(b) implies that \(f \mapsto \delta_f\) extends to an algebra homomorphism 
\begin{equation}
  \label{eq:classical-gamma-minus-def}
  \gm : \ext(V^\ast) \to \End(\ext(V)),
\end{equation}
also injective, so that \(\ext(V^\ast)\) is embedded as the subalgebra of \(\End(\ext(V))\) generated by the annihilation operators.

The commutation relations from \cref{prop:rep-of-classical-clifford-algebra}(c) then imply that the image of the linear map
\begin{equation}
  \label{eq:factorization-of-classical-clifford-algebra}
  \gamma : \ext(V^\ast) \otimes \ext(V) \to \End(\ext(V)), \quad \gamma(y \otimes x) \eqdef \gm(y) \gp(x)
\end{equation}
is a subalgebra of \(\End(\ext(V))\).
One then shows that the spinor representation is irreducible, i.e.~that the algebra \(\im(\gamma)\) acts irreducibly on \(\ext(V)\).
Then Burnside's theorem on matrix algebras \cite{LomRos04} implies that the image of \(\gamma\) is the entire endomorphism algebra \(\End(\ext(V))\).
Then for dimension reasons \(\gamma\) must be injective, so we conclude:

\begin{prop}
  \label{prop:factorization-of-clifford-algebra}
  The linear map \(\gamma\) is an isomorphism, and hence is a factorization of the algebra \(\End(\ext(V))\) into the product of the two subalgebras \(\gm(\ext(V^\ast))\) and \(\gp(\ext(V))\).
  This implies that the spinor representation induces an isomorphism \(\cl(H(V)) \overset{\sim}{\longrightarrow} \End(\ext(V))\).
\end{prop}

\subsection{A Frobenius algebra viewpoint}
\label{sec:frobenius-viewpoint-on-clifford-algebra}

The irreducibility of the spinor representation is the key part of the proof of \cref{prop:factorization-of-clifford-algebra}.
The usual proof is a computation in coordinates: beginning with any nonzero element \(a \in \ext(V)\), one finds an element \(b \in \ext(V)\) such that \(\gp(b) a = b \wedge a\) is nonzero and lies in the top-degree component \(\ext^{\mathrm{top}}(V)\).
Then, by acting with an appropriate generator \(z\) of \(\ext^{\mathrm{top}}(V^\ast)\), we find that \(\gm(z) \gp(b) a = 1_{\ext(V)}\).
Finally, for arbitrary \(x \in \ext(V)\), we obtain \(\gp(x)\gm(z) \gp(b) a = x\).
This shows that every nonzero vector in \(\ext(V)\) is cyclic, and hence the algebra \(\im \gamma\) acts irreducibly, which implies that \(\gamma\) is an isomorphism.

In this section we reformulate this argument using the fact that both \(\ext(V)\) and \(\ext(V^\ast)\) are Frobenius algebras, as we noted in \cref{eg:exterior-algebra-is-frobenius}.
This will allow us to prove the analogue of \cref{prop:factorization-of-clifford-algebra} in the quantum setting.
The method of proof will differ from the classical one, however: rather than using the commutation relations between creation and annihilation operators, we will prove directly that the map \(\gamma\) is a linear isomorphism.
Then this implies that the image of \(\gamma\) is an algebra, and hence once can \emph{deduce} the existence of commutation relations; see \cref{rem:on-proof-of-classical-gamma-isomorphism}.

We begin by identifying \(\ext(V^\ast)\) as the linear dual of \(\ext(V)\) via a bilinear pairing that extends the dual pairing \(\langle \cdot,\cdot \rangle : V^\ast \times V \to \bbC\):

\begin{dfn}
  \label{dfn:bilinear-pairing-of-classical-ext-algs}
  For \(1 \leq k \leq \dim V\), define \(\langle \cdot,\cdot \rangle : \ext^k(V^\ast) \times \ext^k(V) \to \bbC\) by
  \begin{equation}
    \label{eq:determinant-def-of-ext-alg-pairing}
    \langle y_k \wedge \dots \wedge y_1, x_1 \wedge \dots \wedge x_k \rangle \eqdef \det (\langle y_i, x_j \rangle)
  \end{equation}
  for \(y_1, \dots, y_k \in V^\ast\) and \(x_1, \dots, x_k \in V\).
  Taking the direct sum of these pairings (so that \(\ext^k(V^\ast)\) pairs trivially with \(\ext^l(V)\) for \(k \neq l\)), we obtain a bilinear form \(\langle \cdot,\cdot \rangle : \ext(V^\ast) \times \ext(V) \to \bbC\).
\end{dfn}

It is standard that the pairing \eqref{eq:determinant-def-of-ext-alg-pairing} is nondegenerate, so we have:
\begin{lem}
  \label{lem:duality-of-classical-exterior-algebras}
  The map \(\eta : \ext(V) \to {^\ast}\ext(V^\ast)\) given by \(\eta(x) = \langle \cdot, x \rangle\) is an isomorphism, i.e.~\(\eta\) identifies \(\ext(V)\) as the \emph{right dual} of \(\ext(V^\ast)\).
\end{lem}

\begin{rem}
  \label{rem:on-left-vs-right-duals}
  The distinction between left and right duals is not important here, as we deal only with vector spaces.
  In the quantum setting we will carry this all out in the framework of the category of modules over a Hopf algebra, where left and right duals carry different actions and so must be distinguished.
\end{rem}

In \cref{lem:duality-of-classical-exterior-algebras} we identified \(\ext(V)\) linearly with the dual space of \(\ext(V^\ast)\).
As \(\ext(V^\ast)\) is an algebra, it acts on itself by right multiplication, and the dual action turns \({^\ast}\ext(V^\ast)\) into a left \(\ext(V^\ast)\)-module.
Translating that action through the isomorphism \(\eta\) turns \(\ext(V)\) into a left \(\ext(V^\ast)\)-module.
We denote this action by \(\rho : \ext(V^\ast) \to \End(\ext(V))\); it is determined via the dual pairing by
\begin{equation}
  \label{eq:dual-action-of-exterior-algebra}
  \langle y, \rho(z)x \rangle \eqdef \langle y \wedge z, x \rangle
\end{equation}
for \(x \in \ext(V)\) and \(y,z \in \ext(V^\ast)\).

In order to compute this action explicitly, take \(x = x_1 \wedge \dots \wedge x_k \in \ext^k(V)\), \(y = y_k \wedge \dots \wedge y_2 \in \ext^{k-1}(V^\ast)\), and \(z \in V^\ast\).
Then by definition we have
\renewcommand\arraystretch{0.8}
\begin{align*}
  \langle y, \rho(z) x \rangle & = \langle y \wedge z, x \rangle\\
  & = \langle y_k \wedge \dots \wedge y_2 \wedge z, x_1 \wedge \dots \wedge x_k \rangle\\
  & = \det
  \begin{pmatrix}
    \langle z, x_1 \rangle & \dots & \langle z, x_k \rangle\\
    \langle y_2, x_1 \rangle & \dots & \langle y_2, x_k \rangle\\
    \vdots & & \vdots\\
    \langle y_k, x_1 \rangle & \dots & \langle y_k, x_k \rangle
  \end{pmatrix}.
\end{align*}
Expanding the determinant along the first row and using the definition \eqref{eq:determinant-def-of-ext-alg-pairing} of the pairing \(\ext^{k-1}(V^\ast) \times \ext^{k-1}(V)\to \bbC\), we see that
\begin{align*}
  \langle y, \rho(z) x \rangle & = \sum_{j=1}^k (-1)^{j-1} \langle z,x_j \rangle \langle y_k \wedge \dots \wedge y_2, x_1 \wedge \dots \wedge \hat{x}_j \wedge \dots \wedge x_k \rangle\\
  & = \left\langle y_k \wedge \dots \wedge y_2, \sum_{j=1}^k (-1)^{j-1} \langle z,x_j \rangle x_1 \wedge \dots \wedge \hat{x}_j \wedge \dots \wedge x_k  \right\rangle\\
  & = \langle y, \delta_z(x) \rangle,
\end{align*}
where \(\delta_z\) is the annihilation operator defined in \eqref{eq:classical-annihilation-op-def}.
Thus we have \(\rho(z) = \delta_z\) for \(z \in V^\ast\)
This gives:

\begin{prop}
  \label{prop:annihilation-ops-are-dual-frobenius-rep}
  Identify \(\ext(V)\) with the linear dual of \(\ext(V^\ast)\) via the pairing \eqref{eq:determinant-def-of-ext-alg-pairing}.
  Then the action \(\rho\) of \(\ext(V^\ast)\) on the dual of its right regular representation coincides with the action \(\gm\) of \(\ext(V^\ast)\) on \(\ext(V)\) defined in \eqref{eq:classical-gamma-minus-def}.
\end{prop}

\begin{proof}
  We showed above that \(\rho(z) = \delta_z = \gm(z)\) for \(z \in V^\ast\).
  Since \(\rho\) and \(\gm\) are algebra maps that that agree on the generators, we conclude \(\rho = \gm\).
\end{proof}

\begin{notn}
  \label{notn:multiwedge-basis-for-ext-alg}
  Fix a basis \(\{ x_1, \dots, x_n \}\) for \(V\).
  For a subset \(J \subseteq [n]\), define
  \[
  x_J \eqdef x_{j_1} \wedge \dots \wedge x_{j_k}
  \]
  if \(J = \{ j_1 < \dots < j_k \}\).
\end{notn}

It is a standard fact that the set \(\{ x_J \}_{J \subseteq [n]}\) is a basis for \(\ext(V)\).
Moreover, there is another basis for \(\ext(V)\) that is ``dual'' to this one in an appropriate sense.
We formulate this construction in terms of an arbitrary \(\zp\)-graded Frobenius algebra, as this will allow us to obtain an analogous basis in the quantum setting as well.

\begin{lem}
  \label{lem:dual-basis-for-graded-frobenius-algebra}
  Let \(A = \bigoplus_{k=0}^n A_k\) be a \(\zp\)-graded Frobenius algebra with \(A_0 = \bbC\), and denote \(\cA = \bigoplus_{k=0}^{n-1} A_k\).
  Let \(\{ x_i \}_{i=1}^m\) be a basis for \(A\) consisting of homogeneous elements, and fix a generator \(x_0\) for the top-degree component \(A_n\) of \(A\).
  Then there is a unique basis \(\{ z_i \}_{i=1}^m\) of homogeneous elements of \(A\) such that
  \begin{equation}
    \label{eq:frobenius-alg-dual-basis}
    x_i z_j = \delta_{ij} x_0 \quad \text{ if } \deg x_i + \deg z_j = n, \text{ or equivalently if } \deg x_i = \deg x_j.
  \end{equation}
  Moreover, the following hold:
  \begin{enumerate}[(a)]
  \item \(\deg z_i = n - \deg x_i\) for all i.
  \item \(x_i z_j = 0\) if \(\deg x_i + \deg z_j > n\), or equivalently if \(\deg x_i > \deg x_j\).
  \item \(x_i z_j \in \cA\) if \(\deg x_i + \deg z_j < n\), or equivalently if \(\deg x_i < \deg x_j\).
  \end{enumerate}
\end{lem}

\begin{proof}
  Part (a) follows from \eqref{eq:frobenius-alg-dual-basis}, and (b) and (c) are automatic from grading considerations.
  So we just need to show that there is a unique basis satisfying \eqref{eq:frobenius-alg-dual-basis}.

  By \cref{prop:characterization-of-zp-graded-frob-algs}(b), the linear map  \(\psi : A \to \bbC\) defined by
  \[
  \psi(x_0) = 1, \qquad \psi|_{\cA} = 0,
  \]
  is a Frobenius functional on \(A\).
  Let \(\Psi : A \to A^\ast\) be the corresponding isomorphism of left \(A\)-modules given by \(\Psi(x) = \psi( \cdot \; x)\).
  Let \(\{ \tilde{z}_i \}_{i=1}^m\) be the basis for \(A^\ast\) dual to \(\{ x_i \}_{i=1}^m\), so that \(\langle \tilde{z}_i, x_j \rangle = \delta_{ij}\), and then define \(z_i \eqdef \Psi^{-1}(\tilde{z}_i) \in A\) for each \(i\).

  By \cref{prop:characterization-of-zp-graded-frob-algs}(c), the elements \(z_i\) are homogeneous with \(\deg z_i = n - \deg x_i\).
  If \(\deg x_i + \deg z_j = n\), then we have \(x_i z_j \in A_n\) and \(\psi(x_i z_j) = 1\); since \(A_n\) is one-dimensional by \cref{prop:characterization-of-zp-graded-frob-algs}(a), this means that \(x_i z_j = x_0\).
  Hence \eqref{eq:frobenius-alg-dual-basis} is satisfied.
  
  Finally, the basis \(\{ z_i \}_{i=1}^m\) is unique because it maps to the uniquely defined dual basis to \(\{ x_i \}_{i=1}^m\) under the map \(\Psi\).
\end{proof}

Now we give an alternative proof that the map \(\gamma\) defined in \eqref{eq:factorization-of-classical-clifford-algebra} is an isomorphism:
\begin{proof}[Proof of \cref{prop:factorization-of-clifford-algebra}]
  As in \cref{notn:multiwedge-basis-for-ext-alg}, let \(\{ x_1, \dots, x_n \}\) be a basis for \(V\), and let \(\{ x_I \}_{I \subseteq [n]}\) be the corresponding basis of \(\ext(V)\).
  Let \(\{ z_I \}_{I \subseteq [n]}\) be the ``dual basis'' whose existence is guaranteed by \cref{lem:dual-basis-for-graded-frobenius-algebra}, where we choose \(x_{[n]}\) as the distinguished generator of \(\ext^n(V)\).
  Note that \(\deg x_I = \abs{I}\) and \(\deg z_I = n - \abs{I}\).
  
  Suppose that 
  \[
  c = \sum_{I \subseteq [n]} c_I \otimes x_I \in \ext(V^\ast) \otimes \ext(V)
  \]
  lies in the kernel of \(\gamma\) for some elements \(c_{I} \in \ext(V^\ast)\).
  Applying \(\gamma(c) = \sum_I \gm(c_I) \gp(x_I)\) to \(z_J\) and using the properties of the basis from \cref{lem:dual-basis-for-graded-frobenius-algebra}, we get
  \begin{equation}
    \label{eq:classical-gamma-isomorphism}
    0 = \sum_{I \subseteq [n]} \gm (c_I) x_I \wedge z_I = \gm(c_J)x_{[n]} + \sum_{\abs{J} < \abs{I}} \gm(c_I)x_I \wedge z_J.
  \end{equation}
  
   Now we claim that if \(\gm(y)x_{[n]} = 0\) for some \(y \in \ext(V^\ast)\), then \(y = 0\).
   Indeed, if \(\gm(y)x_{[n]} = 0\) then for any \(w \in \ext(V^\ast)\) we have
   \[
   0 = \langle w, \gm(y) x_{[n]} \rangle \eqdef \langle w \wedge y, x_{[n]} \rangle.
   \]
   However, since \(\ext(V^\ast)\) is a graded Frobenius algebra, if \(y \neq 0\) then by \cref{prop:characterization-of-zp-graded-frob-algs}(a) there exists \(w \in \ext(V^\ast)\) such that \(w \wedge y = y_{[n]}\), where \(y_{[n]}\) is some fixed nonzero element of \(\ext^n(V^\ast)\).
   But then we obtain \(\langle y_{[n]},x_{[n]} \rangle = 0\), which is impossible because the pairing \eqref{eq:determinant-def-of-ext-alg-pairing} identifies the (one-dimensional) top-degree components of \(\ext(V)\) and \(\ext(V^*)\) as mutually dual.
   Hence we must have \(y = 0\).
   
   Applying this claim together with induction on \(\abs{J}\) to \eqref{eq:classical-gamma-isomorphism}, we conclude that \(c_J = 0\) for all \(J\), and hence \(c = 0\).
   Thus \(\gamma\) is injective, and hence is an isomorphism for dimensional reasons.
\end{proof}

\begin{rem}
  \label{rem:on-proof-of-classical-gamma-isomorphism}
  As we noted in the beginning of this section, this proof relied on the fact that both \(\ext(V)\) and \(\ext(V^\ast)\) are Frobenius algebra.
  We did not use the commutation relations between the creation and annihilation operators.

  This is significant because the fact that \(\gamma\) is a linear isomorphism allows us to \emph{deduce} the existence of commutation relations.
  Indeed, the fact that \(\gamma\) is a linear isomorphism implies that the image of \(\gamma\) is an algebra; this fact is not clear \emph{a priori} unless commutation relations are already known.
  Then for \(x \in V\) and \(y \in V^\ast\), the operator \(\gp(x) \gm(y)\) must lie in the image of \(\gamma\), and hence can be represented (uniquely) as a linear combination of operators of the form \(\gm(w) \gp(z)\).

  Of course, this approach is not constructive; it does not tell us what the commutation relations actually \emph{are}, and in any case we already know them from \cref{prop:rep-of-classical-clifford-algebra}(c).
  The point here is that we will use an analogous argument to the one in this section for the quantum case; there the relations between creation and annihilation operators are not easily computed in general.
\end{rem}

\subsection{The quantum Clifford algebra}
\label{sec:clifford-alg-quantum-case}

Now we show how to adapt the definitions given in \cref{sec:frobenius-viewpoint-on-clifford-algebra} to obtain creation and annihilation operators on quantum exterior algebras when the underlying module is generically flat in the sense of \cref{dfn:flatness-for-modules}.

We restrict to the following situation: \(\fg\) is simple, \(\fp \subseteq \fg\) is a parabolic subalgebra of cominuscule type, and \(\fl\), \(\up\), and \(\um\) are the Levi factor, nilradical, and opposite nilradical as in \cref{sec:standard-parabolics}.
We define \(\uql\) and the corresponding irreducible representations \(\upm\)  as in \cref{notn:quantized-enveloping-alg-of-levi-factor}.
Recall that we have fixed a \(\uql\)-invariant pairing, i.e.\ a \(\uql\)-module map
\begin{equation}
  \label{eq:dual-pairing-of-nilradicals}
  \langle \cdot,\cdot \rangle : \um \otimes \up \to \bbC
\end{equation}
that identifies \(\um \cong (\up)^\ast\) and hence \(\up \cong {}^\ast(\um)\).
We also assume that \(q > 0\) is in the generic set for \(\upm\).

Our goal is to obtain quantum creation and annihilation operators on the (flat) quantum exterior algebra \(\extq(\up)\).
We begin with the creation operators.
As in the classical case, these are easy to define.
We begin with:

\index[notn]{gamma@\(\gpm\)}
\index[term]{creation operators!quantum}
\begin{lem}
  \label{lem:gamma-plus-is-module-alg-map}
  The left regular representation
  \begin{equation}
    \label{eq:gamma-plus-quantum-definition}
    \gp : \extq(\up) \to \End(\extq(\up)), \quad x \mapsto x \wedge (\cdot),
  \end{equation}
  is a homomorphism of \(\uql\)-module algebras.
\end{lem}

\begin{proof}
  It is clear that \(\gp\) is an algebra homomorphism.
  It is compatible with the \(\uql\)-action because \(\extq(\up)\) is a \(\uql\)-module algebra, i.e.\ the multiplication map in \(\extq(\up)\) is a morphism of modules.
\end{proof}

\begin{dfn}
  \label{dfn:quantum-creation-operators}
  We define the \emph{quantum creation operators} on \(\extq(\up)\) to be the operators \(\gp(x)\) for \(x \in \up\).
\end{dfn}

To define the quantum annihilation operators, we mimic the definition \eqref{eq:dual-action-of-exterior-algebra}.
In order to do so, we need to identify the quantum exterior algebras \(\extq(\upm)\) as being mutually dual.
We will use the isomorphism \(\extq^kV \cong \extq^k(V)\) from \cref{prop:symmetrization-and-antisymmetrization} to do this.
First we extend the dual pairing \eqref{eq:dual-pairing-of-nilradicals}:

\begin{dfn}
  \label{dfn:extension-of-dual-pairing-of-nilradicals}
  For each \(k \geq 1\) we define \(\langle \cdot, \cdot \rangle : \um^{\otimes k} \otimes \up^{\otimes k} \to \bbC\) by
  \begin{equation}
    \label{eq:extension-of-dual-pairing-of-nilradicals}
    \langle y_k \otimes \dots \otimes y_1, x_1 \otimes \dots \otimes x_k \rangle \eqdef \langle y_1, x_1 \rangle \dots \langle y_k, x_k \rangle.
  \end{equation}
  For \(k = 0\), we have \(\extq^0 \um \cong \extq^0 \up \cong \bbC\), and we take the bilinear pairing to be the ordinary multiplication of complex numbers (which is a \(\uql\)-module map).
\end{dfn}

Then we have:
\begin{lem}
  \label{lem:dual-pairing-of-alternating-tensors}
  The bilinear pairing \eqref{eq:extension-of-dual-pairing-of-nilradicals} is nondegenerate, and restricts to a nondegenerate pairing
  \begin{equation}
    \label{eq:dual-pairing-of-alternating-tensors}
    \langle \cdot,\cdot \rangle : \extq^k \um \otimes \extq^k \up \to \bbC.
  \end{equation}
\end{lem}

\begin{proof}
  It is straightforward that \eqref{eq:extension-of-dual-pairing-of-nilradicals} is nondegenerate since the original pairing of \(\um\) with \(\up\) was.
  From the proof of \cref{prop:symmetrization-and-antisymmetrization} we see that
  \[
  \um^{\otimes k} = \extq^k \um \oplus \langle \symq^2 \um \rangle_{k},
  \]
  and likewise
  \[
  \up^{\otimes k} = \extq^k \up \oplus \langle \symq^2 \up \rangle_{k}.
  \]
  Note that these are decompositions of \(\uql\)-modules.
  It follows from \cref{prop:quadratic-duality-of-quantum-symmetric-and-exterior-algebras}(a) that
  \[
  \langle \extq^k \um, \langle \symq^2 \up \rangle_{k} \rangle = 0 = \langle \langle \symq^2 \um \rangle_{k}, \extq^k \up  \rangle.
  \]
  Since the original form on \(\um^{\otimes k} \otimes \up^{\otimes k}\) was nondegenerate, its restriction to \(\extq^k \um \otimes \extq^k \up\) must be nondegenerate as well.
\end{proof}

Now we use \cref{prop:symmetrization-and-antisymmetrization} to translate the pairing \eqref{eq:dual-pairing-of-alternating-tensors} to the quantum exterior algebras.
\begin{dfn}
  \label{dfn:dual-pairing-of-exterior-algebras}
  For each \(k\), define maps 
  \begin{equation*}
    \label{eq:isos-of-alt-tensors-with-ext-algs}
    \pi_\pm^k : \extq^k \upm \to \extq^k(\upm)
  \end{equation*}
  to be the compositions of the natural inclusions \(\extq^k \upm \hookrightarrow \upm^{\otimes k}\) with the quotient maps \(\upm^{\otimes k} \twoheadrightarrow \extq^k(\upm)\).
  By \cref{prop:symmetrization-and-antisymmetrization} these maps are both isomorphisms of \(\uql\)-modules, so we can define a pairing \(\langle \cdot,\cdot \rangle : \extq^k(\um) \otimes \extq^k(\up) \to \bbC\) by
  \begin{equation}
    \label{eq:dual-pairing-of-exterior-alg-components-definition}
    \langle y, x \rangle \eqdef \langle (\pi_-^k)^{-1}(y), (\pi_+^k)^{-1}(x) \rangle
  \end{equation}
  for \(y \in \extq^k(\um)\), \(x \in \ext^k(\up)\), where the right-hand side is the pairing \eqref{eq:dual-pairing-of-alternating-tensors}.
  Then we define the pairing
  \begin{equation}
    \label{eq:dual-pairing-of-exterior-algebras-definition}
    \langle \cdot, \cdot \rangle : \extq(\um) \otimes \extq(\up) \to \bbC
  \end{equation}
  to be the direct sum of the pairings \eqref{eq:dual-pairing-of-exterior-alg-components-definition} for \(0 \leq k \leq N\), where \(N = \dim \upm\).
\end{dfn}

With this in place, we have the following immediate consequence of \cref{lem:dual-pairing-of-alternating-tensors}:

\begin{cor}
  \label{cor:dual-pairing-of-quantum-exterior-algebras}
  The pairing \eqref{eq:dual-pairing-of-exterior-algebras-definition} is \(\uql\)-invariant and nondegenerate, and hence identifies \(\extq(\um)\) as the left dual of \(\extq(\up)\), and \(\extq(\up)\) as the right dual of \(\extq(\um)\).
\end{cor}

We can now use this identification to define the quantum annihilation operators, in analogy to \eqref{eq:dual-action-of-exterior-algebra}:
\index[notn]{gamma@\(\gpm\)}
\begin{dfn}
  \label{dfn:quantum-annihilation-operators}
  For \(y \in \extq(\um)\) and \(x \in \extq(\up)\), we define \(\gm(y)x\) to be the unique element of \(\extq(\up)\) satisfying
  \begin{equation}
    \label{eq:quantum-annihilation-operator-definition}
    \langle w, \gm(y)x \rangle = \langle w \wedge y, x \rangle
  \end{equation}
  for all \(w \in \extq(\um)\).
\end{dfn}

We have the analogous result to \cref{lem:gamma-plus-is-module-alg-map}:

\begin{lem}
  \label{lem:gamma-minus-is-module-alg-map}
  The map \(\gm : \extq(\um) \to \End (\extq(\up))\) determined by \eqref{eq:quantum-annihilation-operator-definition} is a homomorphism of \(\uql\)-module algebras.
\end{lem}

\begin{proof}
  The action \(\gm\) is the dual of the right regular representation of \(\extq(\um)\) on itself; this shows that \(\gm\) is an algebra homomorphism.
  To show that \(\gm\) is a map of \(\uql\)-modules, it suffices to show that \(X \rhd (\gm(y)x) = \gm(X_{(1)}\rhd y) (X_{(2)} \rhd x)\) for all \(X \in \uql\), \(x \in \extq(\up)\), and \(y \in \extq(\um)\).
  (Here \(\rhd\) stands for the action of \(\uql\) on both \(\extq(\upm)\).)
  For any \(w \in \extq(\um)\) we have
  \begin{align*}
    \left\langle w, \gm(X_{(1)} \rhd y)(X_{(2)} \rhd x) \right\rangle & = \left\langle w \cdot (X_{(1)} \rhd y), X_{(2)} \rhd x \right\rangle\\
    & = \left\langle S^{-1}(X_{(2)}) \rhd [w \cdot (X_{(1)} \rhd y)], x \right\rangle\\
    & = \left\langle (S^{-1}(X_{(3)}) \rhd w) \cdot (S^{-1}(X_{(2)})X_{(1)} \rhd y), x \right\rangle\\
    & = \left\langle (S^{-1}X_{(2)} \rhd w) \cdot \counit(X_{(1)}) y, x \right\rangle\\
    & = \left\langle (S^{-1}(X) \rhd w)\cdot y, x \right\rangle\\
    & = \left\langle S^{-1}(X) \rhd w, \gm(y)x \right\rangle\\
    & = \left\langle w, X \rhd (\gm(y)x) \right\rangle.
  \end{align*}
  Here we have used the Hopf algebra axioms, the fact that \(\extq(\um)\) is a \(\uql\)-module algebra, and the fact that the pairing is \(\uql\)-invariant.
  Since this holds for all \(w\) and since the pairing is nondegenerate, the desired conclusion follows.
\end{proof}

Now we can define the quantum analogue of the map \eqref{eq:factorization-of-classical-clifford-algebra}:

\begin{dfn}
  \label{dfn:gamma-map-quantum-case}
  We define map \(\gamma\) by
  \begin{equation}
    \label{eq:factorization-of-quantum-clifford-algebra}
    \gamma : \extq(\um) \otimes \extq(\up) \to \End(\extq(\up)), \quad \gamma(y \otimes x) \eqdef \gm(y)\gp(x),
  \end{equation}
  where \(\gp\) is as in \eqref{eq:gamma-plus-quantum-definition} and \(\gm\) is as in \eqref{eq:quantum-annihilation-operator-definition}.
\end{dfn}

We can now prove our first main result of \cref{chap:quantum-clifford-algebras}:

\begin{thm}
  \label{thm:quantum-gamma-factorization-isomorphism}
  The map \(\gamma\) from \cref{dfn:gamma-map-quantum-case} is an isomorphism of \(\uql\)-modules.
\end{thm}

\begin{proof}
  Both \(\gpm\) are module maps, by \cref{lem:gamma-plus-is-module-alg-map,lem:gamma-minus-is-module-alg-map}, respectively.
  Moreover, as \(\extq(\up)\) is a \(\uql\)-module algebra, its multiplication map is a module map.
  Thus \(\gamma\) is a map of \(\uql\)-modules.

  By \cref{prop:flat-quantum-exterior-algebra-is-frobenius}, the quantum exterior algebra \(\extq(\up)\) is a \(\zp\)-graded Frobenius algebra, so \cref{lem:dual-basis-for-graded-frobenius-algebra} applies.
  Then, replacing \(V\) by \(\up\) and \(V^\ast\) by \(\um\), the proof of \cref{prop:factorization-of-clifford-algebra} goes through in the quantum situation as well.
\end{proof}

\begin{dfn}
  \label{dfn:quantum-clifford-algebra}
  We define the \emph{quantum Clifford algebra}\index[term]{quantum Clifford algebra} to be the endomorphism algebra \(\clq \eqdef \End(\extq(\up))\) together with the factorization \(\gamma\) from \cref{thm:quantum-gamma-factorization-isomorphism}.
\end{dfn}

\begin{rem}
  \label{rem:on-def-of-quantum-clifford-algebra}
  \cref{dfn:quantum-clifford-algebra} is, evidently, somewhat unsatisfactory.
  The quantum Clifford algebra is defined only implicitly, not by a universal property or as a quotient of the tensor algebra of \(\up \oplus \um\).
  We will see in \cref{sec:cliff-alg-for-cp2} an explicit example of a quantum Clifford algebra, and we will compute relations between the creation and annihilation operators.
  In particular, the relations between quantum creation and annihilation operators are not necessarily quadratic-constant, as in the classical case.
  See \cref{sec:cp2-commutation-relations} for an explicit example.
  It is a problem for further investigation to determine how to characterize the cross-relations in general.
  See also \cref{sec:intro-qcas-previous-work}, where we discuss previous versions of quantum Clifford algebras introduced by other authors.
\end{rem}

\cref{rem:on-def-of-quantum-clifford-algebra} notwithstanding, it follows from \cref{thm:quantum-gamma-factorization-isomorphism} that commutation relations \emph{exist} in \(\clq\) for any set of generators:

\begin{cor}
  \label{cor:there-are-generators-and-relations}
  Let \(\{ x_i \}_{i=1}^N\) and \(\{ y_i \}_{i=1}^N\) be any bases for \(\up\) and \(\um\), respectively.
  Then the operators \(\{ \gp(x_i), \gm(y_i) \mid 1 \leq i \leq N \}\) generate \(\clq\) as an algebra.
  Moreover, for each \(i,j\), there are elements \(w^{ij}_k \in \extq(\um)\) and \(z^{ij}_k \in \extq(\up)\) such that
  \[
  \gp(x_i) \gm(y_j) = \sum_k \gm(w^{ij}_k) \gp(z^{ij}_k).
  \]
\end{cor}

\subsection{\(\ast\)-structures on the quantum Clifford algebra}
\label{sec:star-structure-on-quantum-clifford-algebra}

For the purpose of constructing the Dolbeault-Dirac operator in \cref{sec:dolbeault-operator-on-quantized-flag-manifolds}, we require a \(\ast\)-structure on the quantum Clifford algebra \(\clq\).
By definition, \(\clq\) is just the endomorphism algebra \(\End(\extq(\up))\), so we obtain a \(\ast\)-structure from any inner product on \(\extq(\up)\).
Moreover, we can equip \(\uql\) with the \(\ast\)-structure we have called the compact real form, and we have the following:

\begin{lem}
  \label{lem:star-struct-on-quantum-clifford-algebra}
  Let \((\cdot,\cdot)\) be a positive-definite Hermitian inner product on \(\extq(\up)\), conjugate-linear in the first variable, satisfying
  \[
  (X \rhd x, z) = (x, X^* \rhd z)
  \]
  for all \(X \in \uql\) and \(x,z \in \extq(\up)\).
  Then the induced \(\ast\)-structure on \(\clq\) satisfies
  \begin{equation}
    \label{eq:star-structure-compatibility}
    (X \rhd a)^\ast = S(X)^\ast \rhd a^\ast
  \end{equation}
  for all \(X \in \uql\) and \(a \in \clq\).
  (Here \(\rhd\) denotes the action \(X \rhd a = X_{(1)}aS(X_{(2)})\) of a Hopf algebra on the endomorphism algebra of a representation.)
\end{lem}

\begin{proof}
  See \cite{Tuc12}*{Proposition 8.14}.
\end{proof}

There is still the question of how to construct the invariant inner product on \(\extq(\up)\).
We outline one way now.
First, up to a positive scalar factor there is a unique \(\uql\)-invariant inner product on \(\up\) (since \(\up\) is irreducible) as discussed in \cref{sec:representations-of-quea}.
We fix a normalization, and denote this inner product by \((\cdot,\cdot)\) (to distinguish it from the bilinear pairing \(\langle \cdot,\cdot \rangle\) between \(\extq(\um)\) and \(\extq(\up)\)).
This extends to a \(\uql\)-invariant inner product on each \(\up^{\otimes k}\), determined by
\begin{equation*}
  \label{eq:inner-product-extension}
  (v_1 \otimes \dots \otimes v_k, w_1 \otimes \dots \otimes w_k) \eqdef (v_1, w_1) \dots (v_k, w_k)
\end{equation*}
for \(v_i,w_i \in \up\).
Then the restriction defines an invariant inner product on the submodule \(\extq^k \up \subseteq \up^{\otimes k}\).
By \cref{prop:symmetrization-and-antisymmetrization}, the quotient map \(\up^{\otimes k} \to \extq^k(\up)\) restricts to an isomorphism \(\extq^k \up \overset{\sim}{\longrightarrow} \extq^k(\up)\), so we obtain an invariant inner product on \(\extq^k(\up)\).
Finally, the direct sum of these inner products yields an invariant inner product on \(\extq(\up)\).

\begin{rem}
  \label{rem:on-non-uniqueness-of-star-structure}
  The construction described above is a natural way to obtain an invariant inner product on \(\extq(\up)\) and hence a \(\ast\)-structure on \(\clq\), but it is not the only way.
  Indeed, we can decompose \(\extq(\up)\) into irreducible components, and then we have the freedom to rescale the inner product on any component, which maintains invariance.
  The results of this thesis do not depend on the choice.
  However, for potential applications of these results the \(\ast\)-structure on \(\clq\) does matter; we discuss this issue further in \cref{rem:on-choice-of-star-structure-for-dirac-operator}.
\end{rem}

\section{The Dolbeault-Dirac operator on quantized flag manifolds}
\label{sec:dolbeault-operator-on-quantized-flag-manifolds}

We retain all notation from \cref{sec:creation-and-annihilation}.
In this section we construct an algebraic quantum analogue \(\dsl\) of the Dolbeault-Dirac operator on a certain flag manifold related to the pair \((\fg,\fp)\).
In \cref{sec:abstract-dolbeault-operator} we define \(\dsl\) and prove our final main result, \cref{thm:square-of-dirac-operator}.
Then in \cref{sec:dolbeault-operator-classical-motivation} we discuss the geometry of the flag manifold and the relation of our algebraic Dolbeault-Dirac operator to the classical one coming from the theory of complex manifolds.
Finally, in \cref{sec:relation-to-ulis-paper} we discuss how this relates to previous work of Kr\"ahmer on Dirac operators on quantum flag manifolds.

\subsection{The abstract Dolbeault-Dirac operator}
\label{sec:abstract-dolbeault-operator}

Our ``Dolbeault-Dirac operator'' will be a certain element \(\dsl\) of the algebra \(\uqg \otimes \clq\).
It will be constructed as \(\dsl = \eth + \eth^\ast\), where \(\eth\) is a canonically defined element of the algebra \(\symq(\up)^{\op} \otimes \extq(\um)\), which is then embedded into \(\uqg \otimes \clq\).
As we will see in \cref{sec:dolbeault-operator-classical-motivation} below, \(\eth\) is closely related to the Dolbeault differential operator from the theory of complex manifolds.
And indeed, \(\eth\) also has a homological interpretation, although in an algebraic rather than a geometric context: it is the boundary operator for the Koszul complex of the quadratic algebra \(\symq(\up)\).

We begin by recalling the relevant notions from the theory of quadratic algebras:
\begin{dfn}[\cite{PolPos05}*{Ch.~2,~\S 3}]
  \label{dfn:koszul-boundary-operator}
  Let \(A = T(V)/ \langle R \rangle\) be a quadratic algebra, and let \(A^! = T(V^\ast)/ \langle R^\circ \rangle\) be the quadratic dual algebra as in \cref{dfn:quadratic-dual-algebra}.
  Then the \emph{Koszul boundary operator} is the canonical element \(e_A \in A^{\op} \otimes A^!\) given by
  \begin{equation}
    \label{eq:koszul-boundary-operator-definition}
    e_A \eqdef \sum_i x_i \otimes x^i,
  \end{equation}
  where \(\{ x_i \}\) is any basis for \(V\) and \(\{ x^i \}\) is the dual basis for \(V^\ast\).
\end{dfn}

\begin{rem}
  \label{rem:on-koszul-boundary-operator}
  This definition of \(e_A\) differs slightly from the usual one given in \cite{PolPos05}.
  In the usual definition the boundary operator is defined to be an element of \(A \otimes A^!\) rather than \(A^{\op} \otimes A^!\).
  Recall from \cref{rem:on-quadratic-duality} that the algebra we have denoted by \(A^!\) is the opposite of the one defined in \cite{PolPos05} due to our differing choice of convention for the identification of \((V \otimes V)^\ast\) with \(V^\ast \otimes V^\ast\).
  The algebra that we denote by \(A^{op} \otimes A^!\) would be called \((A \otimes A^!)^{\op}\) in the conventions of \cite{PolPos05}.
\end{rem}

The following result is immediate from the definition of \(e_A\):
\begin{lem}[\cite{PolPos05}*{Ch.~2,~\S 3}]
  \label{lem:boundary-operator-squares-to-zero}
  With \(e_A\) defined as above, we have:
  \begin{enumerate}[(a)]
  \item \(e_A\) is independent of the choice of basis for \(V\).
  \item With respect to the ordinary tensor product algebra structure on \(A^{op} \otimes A^!\), we have \(e_A^2 = 0\).
  \end{enumerate}
\end{lem}

Recall from \cref{prop:quadratic-duality-of-quantum-symmetric-and-exterior-algebras} that \(\extq(\um)\) is the quadratic dual of \(\symq(\up)\).

\begin{dfn}
  \label{dfn:eth}
  Let \(\eth \in \symq(\up)^{\op} \otimes \extq(\um)\) be the Koszul boundary operator for \(\symq(\up)\).
\end{dfn}

Now let \(\{ x_1, \dots, x_N \}\) and \(\{ y_1, \dots, y_N \}\) be the dual bases for \(\up\) and \(\um\) introduced in \cref{notn:generators-of-quantum-symmetric-algebra}.
By definition, we have
\begin{equation}
  \label{eq:eth-with-explicit-bases}
  \eth = \sum_{i=1}^N x_i \otimes y_i.
\end{equation}
By construction, \(x_i \mapsto E_{\xi_i}\) extends to the isomorphism \(\symq(\up) \overset{\sim}{\longrightarrow} U'(\wl) \subseteq \uqg\) from \cref{twisted-schubert-cell-is-quantum-symmetric-alg}(c).
The antipode \(S\) is an anti-automorphism of \(\uqg\), and hence there is an embedding of algebras
\begin{equation}
  \label{eq:kappa-embedding-symq-uplus-opposite}
  \kappa \colon \symq(\up)^{\op} \hookrightarrow \uqg, \qquad \kappa(x_i) \eqdef S^{-1}(E_{\xi_i}).
\end{equation}
Recall from \cref{lem:gamma-minus-is-module-alg-map} that we have a homomorphism \(\gm : \extq(\um) \to \clq\), and in the last part of the proof of \cref{thm:quantum-gamma-factorization-isomorphism} we saw that \(\gm\) is injective.
Thus 
\[\kappa \otimes \gm : \symq(\up)^{\op} \otimes \extq(\um) \to \uqg \otimes \clq\]
is an embedding of algebras.

\begin{notn}
  \label{notn:embedded-dolbeault-operator}
  By slight abuse of notation we also denote by \(\eth\) the element
  \begin{equation}
    \label{eq:embedded-dolbeault-operator}
    \eth \eqdef (\kappa \otimes \gm)(\eth) = \sum_{i=1}^N S^{-1}(E_{\xi_i}) \otimes \gm(y_i)
  \end{equation}
  in \(\uqg \otimes \clq\).
\end{notn}

Recall from \cref{sec:quea-compact-real-form} that we gave \(\uqg\) the \(\ast\)-structure known as the compact real form.
Fixing a \(\uql\)-invariant inner product on \(\extq(\up)\), we obtain from \cref{lem:star-struct-on-quantum-clifford-algebra} a \(\uql\)-invariant \(\ast\)-structure on the quantum Clifford algebra \(\clq\).
This gives a \(\ast\)-structure on the tensor product \(\uqg \otimes \clq\) in the evident way.
Thus we make the following

\begin{dfn}
  \label{dfn:dolbeault-dirac-operator}
  We define the \emph{Dolbeault-Dirac operator} to be the element
  \begin{equation}
    \label{eq:dolbeault-dirac-operator-definition}
    \dsl \eqdef \eth + \eth^\ast \in \uqg \otimes \clq.
  \end{equation}
\end{dfn}

\begin{rem}
  \label{rem:on-dolbeault-dirac-definition}
  The Dolbeault-Dirac operator \(\dsl\) that we have constructed is not, properly speaking, an operator.
  However, as we will see in \cref{sec:dolbeault-operator-classical-motivation}, the \(q=1\) analogue of \(\dsl\) implements the classical Dolbeault-Dirac operator on a certain flag manifold associated to the pair \((\fg,\fp)\).
  This is analogous to Kostant's use of the term ``Dirac operator'' in \cite{Kos99}; see also the ``formal Dirac operator'' defined in \cite{Kna01}*{Equation (12.38)}, as well as \cite{Agr03}.
\end{rem}

The following result is an immediate corollary of \cref{lem:boundary-operator-squares-to-zero}(b):

\begin{thm}
  \label{thm:square-of-dirac-operator}
  The Dolbeault-Dirac operator \(\dsl\) satisfies
  \begin{equation}
    \label{eq:square-of-dirac-operator}
    \dsl^2 = \eth \eth^\ast + \eth^\ast \eth.
  \end{equation}
\end{thm}

\begin{rem}
  \label{rem:on-choice-of-star-structure-for-dirac-operator}
  In \cref{rem:on-non-uniqueness-of-star-structure}, we discussed the freedom that exists in the choice of \(\ast\)-structure on \(\clq\).
  Clearly, this choice affects \(\eth^\ast\), and hence \(\dsl\).
  As we explain in \cref{sec:relation-to-ulis-paper}, the ultimate goal is to find an explicit formula for \(\dsl^2\) in terms of quantum Casimir elements of \(\uqg\) and \(\uql\).
  It seems likely that such a formula will exist only for a specific choice of \(\ast\)-structure.

  In particular, the right-hand tensor components of the summands of \(\dsl^2\) will contain operators of the form \(\gm(y_i)^\ast \gm(y_j)\) or \(\gm(y_i) \gm(y_j)^\ast\).
  In order for any cancellation to occur, we will thus need quadratic relations to hold among the \(\gm(y_i)\)'s and their adjoints (or, equivalently, the \(\gp(x_i)\)'s and their adjoints).
  While this question requires further investigation, it is important to note that the lack of quadratic relations between the creation and annihilation operators does not present an obstacle to the computation of \(\dsl^2\).
  
  See \cref{sec:cp2-inner-product-and-star-structure} for an explicit example.
\end{rem}

\subsection{Motivation from the classical setting}
\label{sec:dolbeault-operator-classical-motivation}

Now we explain the geometric background that relates the algebraic object \(\dsl\) to the classical Dolbeault-Dirac operator.

First of all, one can view a pair of a complex semisimple Lie algebra \(\fg\) and a para\-bolic Lie subalgebra \(\fp\) as an infinitesimal description of the complex manifold \(G/P\), where \(G\) is the (connected, simply connected) Lie group corresponding to \(\fg\) and \(P\) is the parabolic subgroup having Lie
algebra \(\fp\).
These spaces are referred to as the \emph{generalized flag manifolds}, and (with respect to a Hermitian metric induced by the Killing form of \(\fg\)) they exhaust the compact homogeneous K\"ahler manifolds \cite{Wan54} as well as the coadjoint orbits of the compact semisimple Lie groups.
This leads to a wealth of applications in geometry, physics, and representation theory; see e.g.~\cite{BasEas89,ChrGin10}, \cite{Bes08}*{Ch.~8}.   
The case in which \(\fp\) is of cominuscule type as in \cref{prop:cominuscule-parabolic-conditions} corresponds to \(G/P\) being a symmetric space; see e.g.~\cite{Kos61}*{Proposition 8.2}. 
This symmetric space is irreducible precisely when \(\fg\) is simple, so we can alternatively interpret the pairs \((\fg,\fp)\) that we consider throughout the paper as the irreducible compact Hermitian symmetric spaces.
A general compact Hermitian symmetric space is just a product of irreducible ones. 

As a real manifold \(G/P\) is diffeomorphic to \(G_0/L_0\), where \(G_0\) is the compact real form of \(G\) and \(L_0 = L \cap G_0\) is the compact real form of \(L\) \cite{BasEas89}*{\S 6.4}. 
If \(Q\) is the parabolic subgroup of \(G\) with Lie algebra \(\fq = \fl \oplus \um\), then \(G/Q\) is also diffeomorphic to \(G_0/L_0\) and hence to \(G/P\).
However, the two induced complex structures on \(G_0/L_0\) are inverse to each other.

Our next aim is to describe the Dolbeault complex \((\Omega^{(0,\bullet)}, \delbar)\) of the complex manifold \(G/Q\). 
To this end, identify \(\fg/\fq\) with the holomorphic tangent space of \(G/Q\) at the identity coset.
This identifies the adjoint representation of \(\fq\) on \(\fg/\fq\) with the isotropy representation.
As representations of \(L_0\) we have \(\fg/\fq \cong \up\).
Hence the smooth sections of the holomorphic tangent bundle \(T^{(1,0)}\) can be identified with the \(L_0\)-equivariant smooth functions
\begin{equation}
  \psi \colon G_0 \rightarrow \up,\quad
  \psi(gh)=h^{-1} \psi(g) \quad \forall g \in G_0, h \in L_0,\label{eq:sections-of-t10}
\end{equation}
as \(T^{(1,0)}\) is associated to the \(L_0\)-principal fiber bundle \(G_0 \to G_0 / L_0\) by the isotropy representation.
We fix an \(L_0\)-invariant Hermitian inner product \(\langle \cdot,\cdot \rangle\) on \(\up\), which induces an isomorphism of complex vector bundles \(T^{(1,0)} \cong \Omega^{(0,1)}\) (these sub-bundles are isotropic for the underlying complexified Riemannian metric, and hence mutually dual since the metric is nondegenerate).
Hence from now on we view functions \(\psi\) as in \eqref{eq:sections-of-t10} as smooth \((0,1)\)-forms on \(G/Q\).
Similarly, \((0,n)\)-forms are identified with \(L_0\)-equivariant smooth functions from \(G_0\) to \(\ext^n(\up)\).
Recall from \cite{Wel08}*{Ch.~I, \S 3} that the \emph{Dolbeault operator}
\begin{equation}
  \label{eq:delbar-definition}
  \delbar : \Omega^{(0,n)} \rightarrow 
  \Omega^{(0,n+1)}
\end{equation}
is obtained by taking Cartan's differential \(\mathrm{d}\) of a \((0,n)\)-form and projecting onto \(\Omega^{(0,n+1)}\).

Now we construct the Hilbert space of square-integrable sections of the bundle \(\Omega^{(0,\bullet)}\).
By embedding \(\ext^n(\up)\) into \(\up^{\otimes n}\) we extend the inner product on \(\up\) to one on \(\ext^n(\up)\), for each \(n\).
There is an inner product on smooth sections of \(\Omega^{(0,\bullet)}\), defined for smooth \(L_0\)-equivariant functions \(\phi,\psi : G_0 \to \ext(\up)\) by
\begin{equation}
  \label{eq:inner-product-on-forms}
  (\phi,\psi) \eqdef \int_{G_0} \langle \phi(g), \psi(g) \rangle dg,
\end{equation}
where we integrate with respect to the (normalized) Haar measure of \(G_0\).
We denote by \(\cH\) the Hilbert space completion of \(\Omega^{(0,\bullet)}\) with respect to this inner product.

The universal enveloping algebra \(U(\fg)\) acts on the algebra \(C^\infty(G_0)\) of smooth complex-valued functions on \(G_0\) by extending the action of \(U(\fg_0)\) by differential operators.
If \(\cl\) is the Clifford algebra of \(\up \oplus \um\) with respect to the canonical symmetric bilinear form, then \(\cl\) acts naturally on \(\ext(\up)\) via the representation constructed in \cref{sec:clifford-alg-classical-case}.
Hence
\[
U(\fg) \otimes \cl  
\]
acts on the algebra \(C^\infty(G_0) \otimes \ext(\up)\) of all smooth functions from \(G_0\) to \(\ext(\up)\).
Under this action, the classical analogue of our \(\eth\) is easily seen to leave the subalgebra of \(L_0\)-equivariant functions invariant, and we have
\begin{equation}
  \label{eq:delbar-eth-adjoint}
  (\delbar \phi, \psi) = (\phi, \eth \psi)
\end{equation}
for smooth sections \(\phi,\psi\) of \(\Omega^{(0,\bullet)}\).

The choice of compact real form of \(G\) induces a \(\ast\)-structure on \(U(\fg)\), and the Hermitian inner product on \(\ext(\up)\) induces a \(\ast\)-structure on \(\cl\).
By \eqref{eq:delbar-eth-adjoint} the element
\[
\dsl \eqdef \eth + \eth^\ast \in U(\fg) \otimes \cl
\]
acts, up to a scalar, as the Dolbeault-Dirac operator on \(G/Q\) formed with respect to the canonical spin\(^c\)-structure \cite{Fri00}*{Section~3.4} 
or \cite{BerGetVer04}*{Section~3.6}.
Analogously, the same algebraic element \(\dsl\) implements the Dolbeault-Dirac operator of \(G/P\) formed with respect to the \emph{anti-canonical} spin\(^c\)-structure.

The point of the algebraic description of the 
Dolbeault-Dirac operator is that it leads
to a computation of its square and of its spectrum,
based on the celebrated Parthasarathy formula, which
expresses \(\dsl^2\) as a linear combination of Casimir
elements in \(U(\fg),U(\fl)\) and
constants; see
e.g.~\cite{Par72,Agr03,Kos99} and also \cite[Lemma~12.12]{Kna01}
for this algebraic approach to Dirac operators and the
Parthasarathy formula, \cite{CahFraGut89,CahGut88,Rie09,Sem93}
for the construction of spinors on symmetric spaces and
the application of Parthasarathy's formula in explicit
computations of spectra.
As we will explain next, 
the present article is meant as a step toward
a quantum analogue of these results.

\subsection{Relation to previous work of Kr\"ahmer}
\label{sec:relation-to-ulis-paper}

Recall that the matrix coefficients of the Type 1 representations of \(\uqg\) 
generate a Hopf \(\ast\)-algebra \(\bbC_q[G]\) that deforms the complex
coordinate ring of the real affine algebraic group \(G_0\); see \cite{KliSch97}*{Ch.~9} or \cite{Jan96}*{Ch.~7}.
The universal \(C^\ast\)-completion of \(\bbC_q[G]\)
is the fundamental example of 
a compact quantum group in the sense of Woronowicz \cite{Wor87}. 
Together with \(G_0\), one can quantize 
\(L_0\) in the form of a quotient Hopf *-algebra 
\(\bbC_q[L]\), and also \(G_0/L_0\) in the form of a right coideal subalgebra 
\(A\) of \(\bbC_q[G]\) \cite{Dij96,NouSug94,MulSch99}.
Associated vector bundles such as 
\(\Omega^{(0,\bullet)}\) can be quantized in the form of finitely generated
projective \(A\)-modules, which admit  
Hilbert space completions \(\cH\) using 
the Haar measure of the 
\(C^\ast\)-completion of \(\bbC_q[G]\). 
See e.g.~\cite{GovZha99,Kra04} and the references therein 
for these topics. 
The paradigmatic example of such a quantized symmetric space is the standard Podle\'s quantum sphere \cite{Pod87}. 

These structures all arise naturally from the theory of quantum groups.
An obvious question to ask is whether there is also a Dirac-type operator \(\dsl\) on a quantized spinor module that produces a spectral triple \((A,\cH,\dsl)\) in the sense of Connes.
This would provide a quantization of the metric structure of \(G/P\).

D{\c{a}}browski and Sitarz in \cite{DabSit03} constructed such a spectral triple over the Podle\'s sphere by deforming the Dirac operator with respect to the standard spin structure and Levi-Civita connection.
In \cite{Kra04}, an abstract argument was given that a quantization of the
Dolbeault-Dirac operator on all symmetric \(G/P\) exists.
It was shown that the commutators \([\dsl,a]\) between algebra elements \(a \in A\) and the Dirac operator are given by bounded operators, which is the first axiom of a spectral triple.
However, the implicit nature of the construction meant that it was not possible to compute the spectrum of \(\dsl\), nor even to prove that \(\dsl\) had compact resolvent.
The latter is the second axiom for a spectral triple, and is a key condition for it to define a K-homology class for the \(C^\ast\)-algebra completion of 
the quantization \(A\) of \(G_0/L_0\).
Up to now, the only cases in which this has been carried out are the projective spaces; see \cite{DanDabLan08}.
The approach is by direct computation, and relies on the Hecke condition for the relevant braidings, so it seems difficult to generalize these methods to arbitrary \(G/P\).

The main motivation for us is that our new approach to the construction of \(\clq\), and hence a quantization of the Dolbeault-Dirac operator in terms of the algebras of Berenstein and Zwicknagl, might lead to new techniques for its study and ultimately to a quantum version of the Parthasarathy formula.

\section{The quantum Clifford algebra for $\bbC \bbP^2$}
\label{sec:cliff-alg-for-cp2}

In this section we examine the quantum Clifford algebra arising from the pair \((\fsl_3,\fp)\), where \(\fp\) is the cominuscule parabolic from \cref{eg:embedding-of-sym-alg-as-twisted-schubert-cell}.
The corresponding simple Lie group is \(SL(3,\bbC)\), and the corresponding parabolic subgroup is
\[
P = \left\{
  \begin{pmatrix}
    * & * & *\\
    0 & * & *\\
    0 & * & *
  \end{pmatrix}
\right\} \subseteq SL(3,\bbC).
\]
Now, \(SL(3,\bbC)\) acts smoothly and transitively on the set of (complex) lines through the origin in \(\bbC^3\), and \(P\) is the stabilizer of the line spanned by the first coordinate vector.
Thus the quotient \(SL(3,\bbC)/P\) is the complex projective space \(\bbC\bbP^2\).

\subsection{Basic setup}
\label{sec:cp2-clifford-setup}

We have \(\fg = \fsl_3\) with simple roots \(\Pi = \{ \alpha_1, \alpha_2 \}\).
We construct the cominuscule parabolic subalgebra \(\fp\) as in \cref{eg:embedding-of-sym-alg-as-twisted-schubert-cell}, and we briefly recall the setup there:
\[
\rtsys(\fl) = \{ \pm \alpha_2 \}, \qquad \rtsys(\up) = \{ \alpha_1, \alpha_1 + \alpha_2 \}.
\]
The ambient and parabolic Weyl groups are
\[
W(\fsl_3) = \langle s_1, s_2 \rangle, \qquad \Wl = \langle s_2 \rangle,
\]
respectively, and the relevant elements are
\[
\wz = s_2s_1s_2, \qquad \wzl = s_2, \qquad \wl = s_1 s_2,
\]
so that \(\wz = \wzl \wl\).
The radical roots are 
\[
\xi_1 = s_2(\alpha_1) = \alpha_1 + \alpha_2, \qquad s_2s_1(\alpha_2) = \alpha_1.
\]
The associated quantum root vectors are
\[
E_{\xi_1} = T_2(E_1) = q^{-1}E_1 E_2 - E_2 E_1 \quad \text{and} \quad E_{\xi_2} = T_2T_1(E_2) = E_1,
\]
and we showed in \cref{eg:embedding-of-sym-alg-as-twisted-schubert-cell} that they satisfy the commutation relation \(E_{\xi_2} E_{\xi_1} - q^{-1}E_{\xi_1} E_{\xi_2} = 0\) in \(\uqsl[3]\).

\subsection{The quantum exterior algebras}
\label{sec:cp2-quantum-exterior-algebras}

The subalgebra \(\uql \subseteq \uqsl[3]\) is generated by \(E_2,F_2\), and all \(K_\lambda\).
We now describe the representations \(\upm\) of \(\uql\).
In the classical setting, both \(\up\) and \(\um\) are two-dimensional irreducible representations of \(\fl \cong \fgl_2\); the central element \(H_{\omega_1}\) acts as \(\pm 1\) in \(\upm\), respectively.

The quantization of \(\up\) (viewed as a representation of \(\uqsl \subseteq \uql\)) is the two-dimensional representation described in \cref{eg:quantum-symmetric-algebra-for-2-dim-rep}.
We retain the notation from that example.
We showed there that the quantum exterior algebra of \(\up\) is generated by \(x_1\) and \(x_2\) with the relations
\begin{equation}
  \label{eq:cp2-ext-alg-of-uplus}
  x_1 \wedge x_1= x_2\wedge x_2 = x_2 \wedge x_1 + q x_1 \wedge x_2 = 0.
\end{equation}

Let \(\{ y_1,y_2 \}\) be the dual basis for \(\um\).
To find the relations for \(\extq(\um)\) we need the braiding or coboundary operator of \(\um\) with itself.
We can find it in two ways.
The first is to use \cref{prop:coboundary-struct-on-uqg-modules}(a).
The second is to note that in this case \(\um \cong \up\) as \(\uqsl\)-modules via the map \(y_1 \mapsto -qx_2\), \(y_2 \mapsto x_1\), and then use naturality of the braidings or commutors.
Either way, we find that the relations for \(\extq(\um)\) are
\begin{equation}
  \label{eq:cp2-ext-alg-of-uminus}
  y_1 \wedge y_1= y_2\wedge y_2 = y_2 \wedge y_1 + q^{-1} y_1 \wedge y_2 = 0.
\end{equation}

\subsection{The dual pairing of the exterior algebras}
\label{sec:cp2-dual-pairing}

Recall from \cref{dfn:dual-pairing-of-exterior-algebras} that in order to compute the dual pairing of \(\extq(\um)\) with \(\extq(\up)\) we need to invert the isomorphisms \(\pi_\pm^k : \extq^k \upm \to \extq^k(\upm)\).
In this example \(\dim \upm = 2\), so this amounts to finding elements in \(\extq^2\up\) and \(\extq^2 \um\) lifting \(x_1 \wedge x_2\) and \(y_1 \wedge y_2\), respectively.

In \cref{eg:quantum-symmetric-algebra-for-2-dim-rep} we saw that \(\extq^2 \up\) is spanned by the element
\[
-q^{-1}x_1 \otimes x_2 + x_2 \otimes x_1,
\]
and we have
\begin{align*}
  \pi_+^2(-q^{-1}x_1 \otimes x_2 + x_2 \otimes x_1) & = -q^{-1} x_1 \wedge x_2 + x_2 \wedge x_1\\
  & = -(q + q^{-1}) x_1 \wedge x_2.
\end{align*}
Thus if we define
\[
a_+ \eqdef \frac{1}{q + q^{-1}} \left( q^{-1} x_1 \otimes x_2 - x_2 \otimes x_1 \right) \in \extq^2 \up
\]
then we have \(\pi_+^2(a_+) = x_1 \wedge x_2\).
Playing the same game with \(\um\), we find that if we define
\[
a_- \eqdef \frac{1}{q+q^{-1}} \left( q y_1 \otimes y_2 - y_2 \otimes y_1 \right) \in \extq^2 \um
\]
then we have \(\pi_-^2(a_-) = y_1 \wedge y_2\).

Then according to the definition  \eqref{eq:dual-pairing-of-exterior-alg-components-definition} the pairing between \(y_1 \wedge y_2\) and \(x_1 \wedge x_2\) is:
\begin{equation}
  \label{eq:cp2-pairing-of-top-degree-guys}
  \begin{aligned}
    \langle y_1 \wedge y_2, x_1 \wedge x_2 \rangle & \eqdef \langle a_-, a_+ \rangle\\
    & = \frac{1}{(q+q^{-1})^2} \langle q y_1 \otimes y_2 - y_2 \otimes y_1, q^{-1} x_1 \otimes x_2 - x_2 \otimes x_1 \rangle\\
    & = - \frac{1}{q+q^{-1}}.
  \end{aligned}  
\end{equation}

\subsection{Creation and annihilation operators}
\label{sec:cp2-creation-and-annihilation-operators}

Now we compute the matrices expressing the action of the quantum creation operators \(\gp(x_1),\gp(x_2)\) and the quantum annihilation operators \(\gm(y_1),\gm(y_2)\) with respect to the ordered basis \(\{ 1, x_1, x_2, x_1 \wedge x_2 \}\) of \(\extq(\up)\).

The creation operators are easy, as they are just given by left-multiplication.
Using the relations \eqref{eq:cp2-ext-alg-of-uplus} we find that 
\renewcommand\arraystretch{0.7}
\begin{equation}
  \label{eq:cp2-creation-operators}
  \gp(x_1) =
  \begin{pmatrix}
    0 & 0 & 0 & 0\\
    1 & 0 & 0 & 0\\
    0 & 0 & 0 & 0\\
    0 & 0 & 1 & 0 
  \end{pmatrix}, 
  \qquad
  \gp(x_2) =
  \begin{pmatrix}
    0 & 0 & 0 & 0\\
    0 & 0 & 0 & 0\\
    1 & 0 & 0 & 0\\
    0 & -q & 0 & 0
  \end{pmatrix}.
\end{equation}

Now we compute the action of the annihilation operators.
It follows immediately from \eqref{eq:quantum-annihilation-operator-definition} that \(\gm(y)1=0\) for any \(y \in \extq(\um)\), and also that \(\gm(y)x = \langle y,x \rangle\) for any \(x \in \up\) and \(y \in \um\).
Thus it is just left to compute \(\gm(y_1) (x_1 \wedge x_2)\) and \(\gm(y_2) (x_1 \wedge x_2)\).

We consider the former one first.
From grading considerations we know that \(\gm(y_1) (x_1 \wedge x_2)\) will be a degree one element in \(\extq(\up)\), so it is determined by its pairings with \(y_1\) and \(y_2\).
Using the relations in the exterior algebras and the pairing computed in \eqref{eq:cp2-pairing-of-top-degree-guys}, we obtain
\[
\langle y_1, \gm(y_1)(x_1 \wedge x_2) \rangle \eqdef \langle y_1 \wedge y_1, x_1 \wedge x_2 \rangle = 0
\]
and
\[
\langle y_2, \gm(y_1)(x_1 \wedge x_2) \rangle \eqdef \langle y_2 \wedge y_1, x_1 \wedge x_2 \rangle = \langle -q^{-1} y_1 \wedge y_2 , x_1 \wedge x_2 \rangle = \frac{1}{1 + q^2},
\]
and hence we find that \(\gm(y_1) (x_1 \wedge x_2) = \frac{1}{1 + q^2} x_2\).

Using the same method to compute \(\gm(y_2)(x_1 \wedge x_2)\), we find that
\[
\langle y_1, \gm(y_2)(x_1 \wedge x_2) \rangle \eqdef \langle y_1 \wedge y_2, x_1 \wedge x_2 \rangle = -\frac{1}{q+q^{-1}}
\]
and
\[
\langle y_2, \gm(y_2)(x_1 \wedge x_2) \rangle \eqdef \langle y_2 \wedge y_2, x_1 \wedge x_2 \rangle = 0,
\]
so we conclude that \(\gm(y_2) (x_1 \wedge x_2) = \frac{-1}{q + q^{-1}} x_1\).
Therefore the annihilation operators are given in matrix form by
\begin{equation}
  \label{eq:cp2-annihilation-operators}
  \gm(y_1) =
  \begin{pmatrix}
    0 & 1 & 0 & 0\\
    0 & 0 & 0 & 0\\
    0 & 0 & 0 & \frac{1}{1+q^2}\\
    0 & 0 & 0 & 0
  \end{pmatrix}, 
  \qquad
  \gm(y_2) = 
  \begin{pmatrix}
    0 & 0 & 1 & 0\\
    0 & 0 & 0 & \frac{-1}{q+q^{-1}}\\
    0 & 0 & 0 & 0\\
    0 & 0 & 0 & 0
  \end{pmatrix}.
\end{equation}

\subsection{Commutation relations}
\label{sec:cp2-commutation-relations}

It is straightforward to verify that \(\gp(x_1)\) and \(\gp(x_2)\) satisfy the relations of \(\extq(\up)\), and similarly that \(\gp(y_1)\) and \(\gp(y_2)\) satisfy the relations of \(\extq(\um)\).
What is more interesting is to compute the commutation relations between the creation and annihilation operators (recall \cref{cor:there-are-generators-and-relations}).
The following relations can be easily verified using the matrix representations of the operators given in \cref{sec:cp2-creation-and-annihilation-operators}:
\begin{equation}
  \label{eq:cp2-commutation-relations}
  \begin{gathered}
    x_1y_1 + x_2y_2 = 1 + (q+q^{-1})y_1y_2x_1x_2,\\
    q x_1y_1 -q^{-1} x_2y_2 = (q+q^{-1})(y_2x_2 - y_1x_1),\\
    x_1y_2 = -(q+q^{-1})y_2x_1,\\
    x_2y_1 = -(q+q^{-1})y_1x_2,
  \end{gathered}
\end{equation}
where we have omitted the \(\gpm\) symbols for readability.
As we noted in \cref{rem:on-def-of-quantum-clifford-algebra}, the relations between the generators are not quadratic-constant.

\subsection{Inner products and \(\ast\)-structures}
\label{sec:cp2-inner-product-and-star-structure}

As we noted in \cref{rem:on-non-uniqueness-of-star-structure}, there is some choice involved in the construction of an inner product on \(\extq(\up)\).
Here we examine some particular choices.

First we describe a general \(\uql\)-invariant inner product on \(\up\).
As \(\up\) is irreducible, this is unique up to a positive scalar factor, so fixing the value of \((x_1,x_1)\) determines it entirely.
Using the explicit action of \(\uql\) given in \eqref{eq:generators-of-uqsl2-in-2dim-rep} and the \(\ast\)-structure on \(\uql\) induced by the compact real form of \(\uqg\), we obtain
\[
(x_2,x_2) = (F_2 \rhd x_1, x_2) = (x_1, F_2^\ast \rhd x_2) = (x_2, E_2K_2^{-1}\rhd x_2) = q (x_1, E_2 \rhd x_2) = q(x_1, x_1).
\]
Moreover, we must have \((x_1,x_2)=0\) as these elements have different weights.
For now we just set \((x_1,x_1) = \alpha\) with \(\alpha > 0\).

We want to extend this inner product to the entire quantum exterior algebra \(\extq(\up)\).
As we have the freedom to scale the overall inner product by any positive real number without affecting the adjoints, we set \((1,1)=1\).
Then we set \((x_1 \wedge x_2, x_1 \wedge x_2) = \gamma\), so the matrix of pairwise inner products of the ordered basis vectors is
\begin{equation}
  \label{eq:cp2-inner-product-matrix}
  M \eqdef
  \begin{pmatrix}
    1 & 0 & 0 & 0\\
    0 & \alpha & 0 & 0\\
    0 & 0 & q \alpha & 0\\
    0 & 0 & 0 & \gamma
  \end{pmatrix}.
\end{equation}
Suppose that \(T\) is the matrix representation of a linear operator on \(\extq(\up)\).
Elementary linear algebra shows that the adjoint of \(T\) with respect to the inner product given by \eqref{eq:cp2-inner-product-matrix} is \(T^\ast = M^{-1}T^\dagger M\), where \(T^\dagger\) is the conjugate transpose of \(T\).
Thus the adjoints of the creation operators are
\begin{equation}
  \label{eq:cp2-adjoints-of-creation-operators}
  \gp(x_1)^\ast =
  \begin{pmatrix}
    0 & \alpha & 0 & 0\\
    0 & 0 & 0 & 0\\
    0 & 0 & 0 & \frac{\gamma }{q \alpha}\\
    0 & 0 & 0 & 0
  \end{pmatrix},
  \quad
  \gp(x_2)^\ast =
  \begin{pmatrix}
    0 & 0 & q \alpha & 0\\
    0 & 0 & 0 & -\frac{q \gamma }{\alpha}\\
    0 & 0 & 0 & 0\\
    0 & 0 & 0 & 0
  \end{pmatrix}.
\end{equation}

Comparing \eqref{eq:cp2-adjoints-of-creation-operators} with \eqref{eq:cp2-annihilation-operators}, we obtain:
\begin{prop}
  \label{prop:cp2-first-try-at-inner-products}
  \begin{enumerate}[(a)]
  \item If we set \(\alpha = 1\) and \(\gamma = \frac{1}{q+q^{-1}}\), then with respect to the inner product determined by \eqref{eq:cp2-inner-product-matrix} we have \(\gp(x_1)^\ast = \gm(y_1)\) and \(\gp(x_2)^\ast = q \gm(y_2)\).
  \item If we set \(\alpha = q^{-1}\) and \(\gamma = \frac{q^{-2}}{q+q^{-1}}\), then with respect to the inner product determined by \eqref{eq:cp2-inner-product-matrix} we have \(\gp(x_1)^\ast = q \gm(y_1)\) and \(\gp(x_2)^\ast = \gm(y_2)\).
  \end{enumerate}
\end{prop}

\begin{rem}
  \label{rem:cp2-on-quadratic-commutation-relations}
  We cannot have both \(\gp(x_1)^\ast = \gm(y_1)\) and \(\gp(x_2)^\ast = \gm(y_2)\) at the same time because \(\{ x_1,x_2 \}\) is merely an orthogonal basis, not an orthonormal one.
  However, with either of the inner products described in \cref{prop:cp2-first-try-at-inner-products}, the commutation relations between the \(\gp(x_i)\)'s and their adjoints will not be quadratic-constant.
  For this reason, I believe that these inner products are not appropriate for the potential application outlined in \cref{sec:relation-to-ulis-paper}.
  See \cref{rem:on-choice-of-star-structure-for-dirac-operator} for further explanation.
\end{rem}




\printindex[notn]
\printindex[term]

\end{document}